\documentclass{memo-l}

\usepackage{amsmath}
\usepackage{amsthm}
\usepackage{amssymb}
\usepackage[left=3cm,right=3cm,top=3cm,bottom=3cm]{geometry}
\usepackage{amscd}
\usepackage{stmaryrd}
\usepackage[normalem]{ulem}
\usepackage{MnSymbol}
\usepackage{bbm}
\usepackage[arrow, curve, matrix]{xy}
\usepackage{tikz-cd}
\usepackage{accents}

\usepackage[bbgreekl]{mathbbol}

\usepackage[utf8]{inputenc}
\usepackage[cyr]{aeguill}
\usepackage{xspace}

\usepackage{tabu, multirow}

\usepackage{listings}
\usepackage{verbatim}

\usepackage{mathbbol}
\usepackage{amssymb}             

\DeclareSymbolFontAlphabet{\mathbb}{AMSb}%
\DeclareSymbolFontAlphabet{\mathbbl}{bbold}

\usepackage{enumitem}
\usepackage{array}

\theoremstyle{thm} \newtheorem*{thm*}{Theorem}
\newtheorem{prop}{Proposition} [section]
\newtheorem{lem}[prop]{Lemma}
\newtheorem{corol}[prop]{Corollary}

\theoremstyle{definition} 
\theoremstyle{remark} \newtheorem{ex}[prop]{Example}
\theoremstyle{remark} \newtheorem{rem}[prop]{Remark}
\theoremstyle{definition} \newtheorem{defi}[prop]{Definition}
\theoremstyle{definition} \newtheorem{nota}[prop]{Notation}
\theoremstyle{definition} \newtheorem{setup}[prop]{Setup}
\theoremstyle{thm} \newtheorem{thm}[prop]{Theorem}
\theoremstyle{thm} \newtheorem{claim}[prop]{Claim}
\newtheorem{introconjecture}{Conjecture} 
\newtheorem{introthm}{Theorem}

\numberwithin{section}{chapter}
\numberwithin{equation}{chapter}

\makeindex

\newcommand{\quotientd}[2]{{\left.\raisebox{.2em}{$#1$}\middle\slash\raisebox{-.2em}{$#2$}\right.}}
\newcommand{\quotientGIT}[2]{{\left.\raisebox{.2em}{$#1$}\middle\slash\hspace{-0.2em}\middle\slash\raisebox{-.2em}{$#2$}\right.}}

\newcommand{\norm}[1] {\| #1 \| }
\newcommand{\lcm} {\mathrm{lcm}}

\usepackage{mathbbol}
\usepackage{bbm}

\newcounter{stepcc}

\begin{document}

\frontmatter

\title{Generalized algebraic Morse inequalities and Hasse-Schmidt jet differentials}


\author{Beno\^it Cadorel}
\address{Institut \'Elie Cartan de Lorraine \\ UMR 7502 \\ Universit\'e de Lorraine, Site de Nancy \\ B.P. 70239, F-54506 Vandoeuvre-l\`es-Nancy Cedex}

\email{benoit.cadorel@univ-lorraine.fr}
\thanks{}

\author{}
\address{}
\curraddr{}
\email{}
\thanks{}

\date{}

\subjclass[2020]{Primary 14A15, 32Q45}

\keywords{Jet differentials, Hasse-Schmidt differentials, complex hyperbolicity}

\begin{abstract}
We give a fully algebraic proof of an important theorem of Demailly, stating the existence of many Green-Griffiths jet differentials on a complex projective manifold of general type. To this end, we introduce a new algebraic version of the Morse inequalities, which we use in our proof as an algebraic counterpart to Demailly's and Bonavero's holomorphic Morse inequalities.

Our proof also applies in positive characteristic, giving the existence of Hasse-Schmidt jet differentials for a smooth projective variety of general type over an arbitrary algebraically closed field.
\end{abstract}

\maketitle

\tableofcontents

\chapter*{Introduction}

\section{Jet differentials and hyperbolicity}

\subsection*{The Green-Griffiths-Lang conjecture}
{\em Jet differential equations} are certainly one of the most important tool in the study of complex hyperbolicity. Recall that a complex manifold \(X\) is said to be {\em Brody hyperbolic} if it does not contain any {\em entire curve}, namely, any image of a non-constant holomorphic map \(f : \mathbb{C} \to X\). In the projective setting, the tantalizing {\em Green-Griffiths-Lang conjecture} predicts that this property should essentially be equivalent to the algebro-geometric property of being of {\em general type}~:

\begin{introconjecture}[Green-Griffiths \cite{GG80}, Lang \cite{lang87}] \label{conjGGL} Let $X$ be a complex projective manifold. The following are equivalent :
	\begin{enumerate}
		\item \(X\) is of general type ;
		\item there exists a proper algebraic subset $\mathrm{Exc}(X) \subsetneq X$, such that $X$ is \emph{Brody hyperbolic modulo $\mathrm{Exc}(X)$} i.e. for any non constant holomorphic map $f : \mathbb C \longrightarrow X$, we have $f(\mathbb C) \subseteq \mathrm{Exc}(X)$.
	\end{enumerate}
\end{introconjecture}

Here, we say that a projective manifold $X$ is of {\em general type} if its canonical line bundle $K_X$ is big, i.e. if we have the maximal growth $h^0(X, K_X^{\otimes m}) \geq C \; m^{\dim X}$ for some $C > 0$. 
\medskip

If \(C\) is a {\em curve}, then \(C\) is of general type if and only if \(g(C) \geq 2\), and in this case, the result is a consequence of Riemann's uniformization theorem. Indeed, any entire curve has to lift to the universal covering \(\Delta \to C\), and thus must be constant by Liouville's theorem. In higher dimension, the conjecture is still largely open.

\subsection*{The bundles of Green-Griffiths jet differentials}
The relevance of jet differential techniques for the study of complex hyperbolicity problems has first been evidenced by Green and Griffiths \cite{GG80}. For any integers $k, m \geq 1$, these authors have constructed a vector bundle $E_{k,m}^{GG} \Omega_X \longrightarrow X$, the so-called \emph{bundle of Green-Griffiths jet differentials}, whose local sections represent \emph{holomorphic ordinary differential equations of order $k$ and degree $m$} on $X$ (acting on germs of holomorphic curves with values in \(X\)). The main interest of these jet differentials in view of Conjecture \ref{conjGGL} comes from the following fundamental theorem. 

\begin{introthm} [\cite{SY96, dem97}] \label{thmannfund} Let $A$ be an ample line bundle on $X$. Let $k, m \in \mathbb N^\ast$, and let $P \in H^0(X, E_{k, m}^{GG} \Omega_X \otimes \mathcal O(- A))$. Then, any non constant holomorphic map $f : \mathbb C \longrightarrow X$ is a solution to the differential equation $P$, i.e.  $P(f; f', ..., f^{(k)}) \equiv 0$. 
\end{introthm}

The theorem above is the first step of a general strategy to find restrictions on the geometry of entire curves, in the hope of eventually proving the Green-Griffiths-Lang conjecture. In many contexts, this strategy has permitted to obtain strong hyperbolicity results e.g. for hypersurfaces of $\mathbb P^{n+1}$ (see e.g. \cite{DMR10, ber15, bro17, siu15, deng16, dem18, BK19, Cad24}, complements of hypersurfaces of $\mathbb P^n$ (see \cite{Dar15, BD19}), surfaces of general type (see \cite{bog77, McQ98, RR12})... The field of application of these jet differential techniques has also been recently extended to the orbifold setting by Campana, Darondeau, Demailly and Rousseau in \cite{CDR18, CDDR24}.
\medskip

The strategy sketched above will work if we are able to prove the existence of many global jet differential equations on a given manifold of general type. The most general result in this direction is due to Demailly \cite{dem11}, and can be stated as follows.

\begin{introthm} \label{thmprinc} Let $X$ be a complex projective manifold of general type. Then, for sufficiently large $k \in \mathbb N$, the Green-Griffiths sheaf of algebras $E_{k, \bullet}^{GG} \Omega_X$ is \emph{big}, i.e. there is maximal growth 
	\[
		h^0(X, E_{k, m}^{GG} \Omega_X) \geq C m^{n + nk - 1}
	\]
with $C > 0$, if $m \gg 1$ goes to infinity while being sufficiently divisible.

In particular, if $A$ is an ample line bundle on $X$, and if $m \gg k \gg 1$, we have
$$
H^0(X, E_{k, m}^{GG} \Omega_X \otimes \mathcal O(- A)) \neq 0.
$$
\end{introthm}

The proof of Theorem \ref{thmprinc} given by Demailly in \cite{dem11} is fundamentally analytic in nature: it is based on the \emph{holomorphic Morse inequalities} he introduced in \cite{dem85}, and that were later extended to the singular setting by Bonavero \cite{bonavero98}. 
\medskip

\section{Main results}

Our purpose in the current notes is twofold:
\begin{enumerate}
	\item give an fully algebraic (and hopefully, essentially self-contained) proof of Theorem~\ref{thmprinc} ;
	\item extend the field of validity of Demailly's result to positive characteristic.
\end{enumerate}
\medskip

Let us give some details.
\medskip

\subsection*{Demailly's proof. Holomorphic and algebraic Morse inequalities} Before explaining what needs to be done to pass to arbitrary characteristic, let us a brief account on Demailly's proof, and how we can hope to convert it to an algebraic setting.
\medskip

The proof given by Demailly in \cite{dem11} consists in applying his holomorphic Morse inequalities to the tautological line bundle on the {\em Green-Griffiths jet spaces}. If $X$ is a complex manifold, these jet spaces are projective fiber bundles $X_k^{GG} \overset{\pi_k}{\longrightarrow} X$ (for each $k \in \mathbb N^\ast$), each one endowed with an (orbifold) tautological line bundle $\mathcal O_k^{GG}(1)$, such that:\footnote{This line bundle might be better seen in the setting of {\em stacks}. See Annex~\ref{annex:stacks} for more details.}
\[
	(\pi_k)_\ast \mathcal O_k^{GG}(m) = E_{k, m}^{GG} \Omega_X
	\quad
	\text{for} \; m \in \mathbb{N}.
\]

Showing that $E_{k, \bullet}^{GG} \Omega_X$ is big for some $k$, amounts to showing that $\mathcal O_k^{GG}(1)$ is big. If \(h\) is a smooth hermitian metric on \(\mathcal{O}_{k}^{GG}(1)\), then (at least for \(m\) divisible enough), we have the {\em holomorphic Morse inequality} :
\[
	(h^{0} - h^{1})(X_{k}^{GG}, \mathcal{O}_{k}^{GG}(m))
	\geq
	\frac{m^{n+nk-1}}{(n+ nk-1)!} \int_{X_{k}^{GG}(\leq 1)} \left(\frac{i}{2\pi} \Theta(h)\right)^{n + nk - 1}
	+ o(m^{n+nk-1})
\]
where the integral on the right is performed on the locus where \(i \Theta(h)\) has less than \(1\) negative eigenvalue. Demailly's proof then boils down to showing that if \(X\) is of general type, we may construct a (singular) metric on \(\mathcal{O}_{k}^{GG}(1)\) for which the integral on the right hand side becomes positive.
\bigskip

 The main candidate for an algebraic version of the holomorphic Morse inequalities are the \emph{algebraic Morse inequalities} of Demailly \cite{dem96} and Angelini \cite{ang96}. In their simplest form, due to Siu \cite{siu93}, they can be stated as follows. Let $X$ be a complex projective manifold of dimension $n$, and let $L$ be a line bundle on $X$. Assume that $L = \mathcal O(A - B)$, where $A, B$ are nef divisors on $X$. Then, for $m \gg 1$, we have
$$
h^0(X, L^{\otimes m}) - h^1(X, L^{\otimes m}) \geq \frac{m^n}{n!} \left(A^n - n B \cdot A^{n-1} \right) + O(m^{n-1}).
$$
In particular, if $A^n > n B \cdot A^{n-1}$, then $h^0(X, L^{\otimes m})$ has maximal growth, and $L$ is big.
\medskip

We can certainly try to apply the previous algebraic Morse inequalities to the jet bundle situation. To do this, we need to write $\mathcal O_k^{GG}(1) = \mathcal O(A - B)$, where $A, B$ are nef \(\mathbb{Q}\)-divisors on $X_k^{GG}$, and then to compute $A^N - N A^{N-1} \cdot B$, where $N = \dim X_k^{GG}$. A natural way to proceed is to remark that $\mathcal O_k^{GG}(1)$ is relatively ample, and to choose $B = \pi_k^\ast H$, where $H$ is an ample divisor on $X$, sufficiently positive so that $\mathcal O (A) = \mathcal O_k^{GG}(1) \otimes \mathcal \pi_k^\ast \mathcal O(H)$ is itself nef.
\medskip

This general strategy has been followed by several authors (starting with Diverio \cite{div08, div09}, see also \cite{BK19, Cad24}) to the case where $X$ is a hypersurface of $\mathbb P^{n+1}$ of degree $d$. In this situation, it is possible to show the existence of jet differential equations of order $k \geq n$, as soon as $d \geq d(n)$, for some constant $d(n) \in \mathbb N$. Unfortunately, this algebraic method does not seem to give the bigness of $E_{k, \bullet}^{GG} \Omega_X$ when $d \geq n + 3$ and $k \gg 1$, which would be the expected bound on $d$ according to Theorem \ref{thmprinc}. There seems here to be a discrepancy between the results the algebraic and analytic methods can provide, at least with this specific method. 
\footnote{Note however that the optimal bound for hypersurfaces has been provided by Merker \cite{mer15}, using a careful analysis of the even degree cohomology groups for \(\mathcal{O}_{X_{k}^{GG}}(m)\).}
\medskip 

\subsection*{A new version of the algebraic Morse inequalities} It seems to us the discrepancy discussed above comes from a rather restrictive setting for the statement of the algebraic Morse inequalities. Rather than dealing only with the case where $L = A - B$, with $A, B$ nef (or even very ample), it would be much more flexible to be able to deal with any difference of \emph{effective} divisors $A, B$.

In the following discussion, we propose to follow this idea, and to give accordingly a new algebraic version of the Morse inequalities. 
\medskip

In the case where \(A\) and \(B\) are very ample, the standard algebraic Morse inequalities can be proved by induction: writing \(L = \mathcal{O}(A - B)\), we can use an exact sequence argument to reduce the problem to showing the inequalities for both \(L|_{A}\) on the lower dimensional variety \(A\) (resp. for \(L|_{B}\) on \(B\)). Actually, assuming that \(A\) and \(B\) are merely effective does not present much difficulty for this argument. However, for the induction to work, one sees that we need to assume now that we also have a decomposition on \(A\) of the form:
\[
	L|_{A} \cong \mathcal{O}_{A}(A_{A} - B_{A}),
\]
with \(A_{A}, B_{A}\) effective on \(A\), and similarly for \(L|_{B}\).
\medskip

This remark quickly leads us to trying to state Morse inequalities in terms of \emph{stratifications} on our manifolds, in the following sense. Let $L \longrightarrow X$ be a line bundle over a complex manifold, and let $e$ be a trivialization of $L$ over some open dense subset $U \subseteq X$, where $X \setminus U$ is the support of a Cartier divisor $D$ such that $L= \mathcal O(D)$. We can now extract a particular combinatorial data of this situation: 
\begin{enumerate}[label=(\arabic*)]
	\item write $D = D^+ - D^-$, where $D^+$, $D^-$ are effective; 
	\item note for later use the multiplicities of $e$ along the irreducible components of $D^+$ and $D^-$; 
	\item restrict $L$ to $D$, and find a trivialization of this restriction on a Zariski dense open subset of $D$;
	\item repeat this operation with $X$ replaced by each component of $D$. 
\end{enumerate}

Inductively, this defines a stratification $\Sigma$ on $X$, with the data of a trivialization $\mathbf e$ over $\Sigma$ (see Section~\ref{sec:stratifications} for the complete definitions). Then the data of $\underline{\Sigma} = (\Sigma, \mathbf e)$ and of the multiplicities computed along the way permits to define \emph{truncated Chern intersection numbers} $\deg c_1(X, \underline{\Sigma})^{n}_{[\leq i]}$, for all $0 \leq i \leq n$. 
\medskip

The previous definitions actually make sense over an arbitrary field; we may also replace \(L\) by a \(\mathbb{Q}\)-line bundle. In the end, this allows to formulate generalized Morse inequalities as follows.

\begin{introthm}  [= Theorem~\ref{thmmorse}]\label{thmintromorse} Let $X$ be a projective variety of dimension $n$ over an arbitrary field \(\mathbbm{k}\). Let $L$ be a $\mathbb Q$-line bundle on $X$, and let $\underline{\Sigma}$ be a trivialized stratification adapted to $L$. Let $M$ be another line bundle on $X$. Then, for each integer $0 \leq i \leq n$, and any $m$ divisible enough, we have
\begin{enumerate}[label=(\roman*)]
\item (Strong Morse inequalities) 
$$
\sum_{0 \leq j \leq i}  (-1)^{j + i} h^j (X, M \otimes L^{\otimes m}) \leq (-1)^i \left( \deg c_1(L, \underline{\Sigma})^n_{[\leq i]} \right) \frac{m^n}{n!} + O(m^{n-1})
$$

\item (Weak Morse inequalities) 
$$
h^i (X, M \otimes L^{\otimes m}) \leq (-1)^i \left( \deg c_1(L, \underline{\Sigma})^n_{[i]} \right) \frac{m^n}{n!} + O(m^{n-1})
$$

\item (Asymptotic Riemann-Roch formula)
$$
\chi(X, M \otimes L^{\otimes m}) = \left( \deg c_1(L, \underline{\Sigma})^n_{[\leq n]} \right) \frac{m^n}{n!} + O(m^{n-1})
$$
\end{enumerate}
\end{introthm}

The core of our proof is very similar in spirit to the one of Angelini \cite{ang96}: it is an induction on $\dim X$, using quite standard dimensional considerations on long exact sequences of coherent sheaves. The main technical difficulty comes from the fact that the induction step requires several clean-up operations to be able to make sense of the exact sequence argument:
\begin{enumerate}
	\item the fact that we deal with \(\mathbb{Q}\)-line bundles requires us to take cyclic covers on \(X\). If \(\mathrm{char}(\mathbbm{k}) > 0\), this actually does not add many complications, but some care needs to be applied;
	\item the main technical difficulty comes from the fact that the irreducible components of \(X \setminus U\) above are not necessarily {\em Cartier}, so we need to perform a modification on \(X\) to get back to this situation. In characteristic \(0\), this could be easily achieved by Hironaka's resolution of singularities, but this is a bit more technical in positive characteristic: we need to blow-up Weil divisors, and deal carefully with the exceptional components that may appear in the process. 
\end{enumerate}

\subsection*{Weighted projectivized bundles and jet spaces} Let us go back to the complex setting for the next few paragraphs. To prove Theorem \ref{thmprinc}, a natural idea would be to proceed as in \cite{dem11}, and to apply the Morse inequalities to the line bundles $\mathcal O_k^{GG}(1) \longrightarrow X_k^{GG}$. In these notes, we will follow a rather different strategy, which is based on two simplifying ideas (which would also be relevant to simplify a bit Demailly's analytic proof). 
\medskip

First of all, following a remark we made in \cite{cad17}, we can reduce the study of the jet spaces $X_k^{GG}$ to the one of a \emph{weighted projectivized bundle}, by using a construction that was essentially used implicitely in \cite{dem11}. Its geometric interpretation is the existence of a deformation of $X_k^{GG}$ into the {\em weighted projectivized bundle} $P_k := \mathbb{P}_X(\Omega_X^{(1)} \oplus ... \oplus \Omega_X^{(k)})$ (see Section \ref{sectweightproj} for a proper definition), and of an (orbifold) line bundle over this family of deformations, restricting to the tautological line bundles $\mathcal O_k^{GG}(1) \longrightarrow X_k^{GG}$ and $\mathcal O_k(1) \longrightarrow P_k$. In this situation, the semi-continuity properties of $h^0 - h^1$ (see \cite{dem95}) yields
\begin{equation} \label{eqineqfund}
h^0 ( X_k^{GG}, \mathcal O_k^{GG}(m)) - h^1 ( X_k^{GG}, \mathcal O_k^{GG}(m)) \geq  h^0(P_k^{GG}, \mathcal O_k(m)) - h^1(P_k^{GG}, \mathcal O_k(m)).
\end{equation}

This inequality can also be shown using a very simple argument due to Merker, and which was communicated to me by L. Darondeau. To show that the left hand side is large, we just have to bound from below the right hand side, which is in fact the $h^0 - h^1$ of a natural graded algebra on $E_{k,m}^{GG} \Omega_X$: this right hand side of the previous equation is actually equal to  
\begin{equation} \label{eqchi1sym}
(h^0 - h^1) \left( \sum_{l_1 + 2 l_2 + ... + k l_k = m} S^{l_1} \Omega_X \otimes ... \otimes S^{l_r} \Omega_X \right), 
\end{equation}
we get to our first reduction step, which states that it is enough to bound from below the quantity \eqref{eqchi1sym}, which is completely described \emph{only in terms of the cotangent bundle $\Omega_X$}.
\medskip
 
As another important reduction step, we can use a version of the \emph{splitting principle} implying that to get an asymptotic lower bound for \eqref{eqchi1sym}, it is actually enough to deal with the case where $\Omega_X = L_1 \oplus ... \oplus L_n$ is a direct sum of line bundles. We can then expend \eqref{eqchi1sym} in terms of direct sums of various products of powers of the $L_i$, and apply the Morse inequalities to each one of the direct factors so obtained. By a Riemann sum argument, this computation eventually makes appear an integral over the standard $(kn-1)$-dimensional simplex $\Delta^{kn-1}$; we can actually state the following general result (see Section \ref{sectstatement} for more precise and general statements, in particular when the base field is arbitrary). 

\begin{introthm} [see Theorem~\ref{thmineqintegral}] \label{thmintroint} Let $X$ be a complex projective manifold of dimension $n$, and let $L_1, ..., L_r$ be line bundles on $X$. Let $a_1, ..., a_r \in \mathbb N_{\geq 1}$. Fix a stratification $\Sigma$ on $X$, and let $\mathbf{e}_1, ..., \mathbf{e}_r$ be trivializations of the $L_i$ on $\Sigma$. Then, for all $i$ $(0 \leq i \leq n)$, there exists a piecewise polynomial function $\upsilon_{[\leq i]} :\Delta^{r-1} \longrightarrow \mathbb R$, constructed explicitly in terms of $(\Sigma, \mathbf{e}_1, ..., \mathbf{e}_r)$, such that for all $m \in \mathbb N$:
\begin{align} \nonumber
\chi^{[i]} (X, \bigoplus_{a_1 l_1 + ... + a_r l_r = m} & L_1^{\otimes l_1} \otimes ...  \otimes L_r^{\otimes l_r} )  \\ \label{eqintrointegral}
& \leq  \frac{\mathrm{gcd}(a_1, ..., a_r)}{a_1 ... a_r} \binom{n + r - 1}{r - 1} \left[ \int_{\Delta^{r-1}} \upsilon_{[\leq i]} dP \right] \frac{m^{n+ r - 1}}{(n+ r - 1)!} + o(m^{n+r - 1}),
\end{align}
where $dP$ is the invariant probability measure on the simplex $\Delta^{r-1}$.
\end{introthm}

Suppose now that we are given a trivialized stratification $\underline{\Sigma} = (\Sigma, \mathbf{e})$, adapted to $K_X = L_1 \otimes ... \otimes L_n$. We can then construct trivializations $\mathbf{e}_i$ of the $L_i$ on $\Sigma$ (after possibly refining $\Sigma$), whose tensor product gives back $\mathbf{e}$. The next part of the work is to estimate the integral in \eqref{eqintrointegral} as we apply this inequality to \eqref{eqchi1sym}, letting $k \longrightarrow + \infty$. This estimation is very close to the computations done in \cite{dem11}: we can actually present them in a probabilistic manner, further elaborating on Demailly's "Monte-Carlo method". Exactly as in \cite{dem11}, we observe an "averaging" phenomenon: as $k \longrightarrow + \infty$, the integral of \eqref{eqintrointegral} gets closer to its mean value over the simplex $\Delta^{r-1}$; a simple computation shows that this mean value is proportional to $\deg c_1(K_X, \underline{\Sigma})^n_{[\leq i]}$.
\bigskip

Letting $i = 1$, this implies that Theorem \ref{thmprinc} will be proved if we can find a stratification for which $\deg c_1(K_X, \underline{\Sigma})^n_{[\leq 1]} > 0$, whenever $K_X$ is big. As often when passing analytic arguments to an algebraic context, this stratification will actually exist only on some ramified covering $X' \overset{p}{\longrightarrow} X$, produced using Kawamata and Bloch-Gieseker lemmas. Also, $\Sigma$ will actually be adapted to an ample subsheaf of $\mathcal O(p^\ast K_X)$ rather that to $K_X$ itself.

Putting everything together, we get the following result, which implies immediately Theorem \ref{thmprinc}.

\begin{introthm} [consequence of Proposition~\ref{propfinal}] \label{thmfulldetail} Let $X$ be a projective manifold of general type over \(\mathbb{C}\), of dimension $n$. Then, for all small $\epsilon > 0$, there exists 
\begin{enumerate} \itemsep=0em
		\item a generically finite, proper morphism $p: X' \longrightarrow X$;
		\item a decomposition $p^\ast K_X = A + E$ into ample and effective divisors;
		\item a trivialized stratification $\underline{\Sigma}$ on $X'$, adapted to $A$, such that $$\deg c_1(A, \underline{\Sigma})^n_{[\leq 1]} > (\deg p) (\mathrm{vol}(K_X) - \epsilon) > 0.$$
\end{enumerate}

	Moreover, when $m \gg k \gg 1$, and $m$ is divisible enough, we have
\begin{equation} \label{eq:lowerboundvol}
h^0(X', p^\ast E_{k,m}^{GG} \Omega_X) \geq \frac{(\log k)^n}{n! (k!)^n} \left( \deg c_1(A, \underline{\Sigma})^n_{[\leq 1]} - O(\frac{1}{ \log k}) \right) \frac{m^{n + nk - 1}}{(n+kr - 1)!}  + o(m^{n + kr - 1}).
\end{equation}
This implies that for $k \gg 1$, $p^\ast E_{k, \bullet}^{GG} \Omega_X$, and hence $E_{k, \bullet}^{GG} \Omega_X$, is big.
\end{introthm}

\bigskip

\subsection*{Arbitrary characteristic and Hasse-Schmidt differentials} The algebraicity of the proof discussed above opens the door to proving corresponding existence results over an arbitrary field. We believe this might be relevant in particular for a future study of hyperbolicity in non-archimedean settings (e.g. over local fields such as \(\mathbb{F}_{p}((t))\) or \(\mathbb{Q}_{p}\)).\footnote{See \cite{JV18} for an discussion of these notions, mostly in characteristic zero.}
\medskip

The bundles of Green-Griffiths jet differential have a natural replacement in arbitrary characteristic, namely the bundles of {\em Hasse-Schmidt differentials} (discussed in particular by Vojta in \cite{Voj04}). If \(X\) is a smooth variety over a field \(\mathbbm{k}\), these vector bundles \(E_{k, m}^{HS}\Omega_{X}\) also have an interpretation in terms of differential equations for infinitesimal curves
\[
	\mathrm{Spec}\, \mathbbm{k}[t]/(t^{k+1})
	\longrightarrow
	X.
\]
Over \(\mathbb{C}\), there is a canonical isomorphism \(E_{k, m}^{HS}\Omega_{X} \cong E_{k, m}^{GG} \Omega_{X}\). Essentially, the sections of \(E_{k, m}^{HS}\Omega_{X}\) admit natural frames written in terms of {\em divided powers}
\[
	\frac{f^{(j)}}{j!}
\]
for germs of curves \(f : \Delta \to X\) (see Chapter~\ref{chap:jetdiff} for the definitions).
\medskip

The neat part is that everything that has been said in the previous discussion translates readily if we replace \(E_{k,m}^{GG} \Omega_{X}\) by \(E_{k, m}^{HS}\Omega_{X}\). In particular, this vector bundles also have a natural filtration whose graded object again identifies with
\[
	\bigoplus_{l_{1} + 2 l_{2} + \dotsc + kl_{k} = m}
	S^{l_{1}} \Omega_{X} \otimes \dotsc \otimes S^{l_{r}} \Omega_{X};
\]
the proof is essentially the same as in the complex case, using {\em étale coordinate charts} instead of {\em euclidean} ones.
\medskip

Once we have this grading, the algebraic proof of Theorem~\ref{thmfulldetail} also applies, and show that \(E_{k, \bullet}^{HS} \Omega_{X}\) is big if \(X\) is a geometrically smooth variety of general type over \(\mathbbm{k}\).

\section{Organization of the text} These notes are organized in five chapters. Three appendices contain several technical results or some background used in the rest of the work. 
\medskip

\noindent
{\bf Chapter 1.} The main purpose of this chapter is to give an introductory presentation of the Hasse-Schmidt jet differentials, aimed mainly at researchers with a more complex geometric background. We do not claim a lot of originality for this part, which mostly consist in a sanity check that everything relevant works as intended. In particular, we give a detailed presentation of the standard filtration on Hasse-Schmidt differentials, crucial in the rest of the text. We give also a very short presentation of the possible use of these jet differentials to prove algebraic hyperbolicity results in positive characteristic.
\medskip

\noindent
{\bf Chapter 2.} This chapter gives a recollection on {\em weighted Segre classes}, used by the author in \cite{cad17, Cad24} to perform intersection computations in explicit cases. Our basic idea is that Theorem~\ref{thmintroint} should be interpreted as a truncated version of a corresponding Riemann-Roch formula ; accordingly, the integral in \eqref{eqintrointegral} can be thought of as a "truncated Segre class". The purpose of this chapter is mainly to rederive a proof of the Whitney formula for truncated Segre classes (see Proposition~\ref{prop:whitneyformula}) with an argument making appear an integral over a simplex, exactly as in Theorem~\ref{thmintroint}. Hopefully, this will help the reader to see more closely the argument behind the proof of this latter result.
\medskip

\noindent
{\bf Chapter 3.} This chapter is the heart of the text. After some introductory definitions, we give the proof of our algebraic version of the Morse inequalities. The proof is performed in Section~\ref{sectmorseineq}; as explained above, it is an induction using several clean-up operations. In principle, it is self contained and uses only elementary results of algebraic geometry, most of which were gathered in Annex~\ref{annex:alggeom}.
\medskip

\noindent
{\bf Chapter 4.} This chapter presents the proof of Theorem~\ref{thmintroint} and several related results. As we explained above, the idea is basically to follow the proof of Whitney formula discussed in Chapter~\ref{chap:weighted}, but replacing the use of the Hirzebruch-Riemann-Roch theorem with the algebraic Morse inequalities.

There is a technical difficulty related to the fact that we work with \(\mathbb{Q}\)-line bundles instead of {\em standard ones.} when taking symmetric powers of a direct sum of line bundles, this has the consequence that our vector bundles might be defined only up to tensoring with a torsion line bundle; so it might not be clear that the left hand side of \eqref{eqintrointegral} is well defined if we replace the \(L_{i}\) with {\em \({\mathbb{Q}}\)-line bundles}. The remedy is to choose representatives for the image of \(\mathrm{Pic}(X) \to \mathrm{Pic}(X)_{\mathbb{Q}}\) in a subgroup fixed in advance; this is discussed in Section~\ref{eq:symmqline}.
\medskip

\noindent
{\bf Chapter 5.} Finally, we give a proof of the main Theorem~\ref{thm:existence}. By what we have explained above, we are lead back to Theorem~\ref{thmintroint} after using first a reduction to the graded object of \(E_{k, m}^{HS}\Omega_{X}\), and then a splitting principle for \(\Omega_{X}\). Next, we have to produce a stratification on a variety \(X'\) admitting a suitable generically finite morphism \(X' \to X\); this is done using some of the results of Chapter~\ref{chap:truncatedstrat}.

The last step is to estimate the integral giving the leading term of the inequalities. For this, we expand on Demailly's {\em Monte-Carlo} interpretation, and show that this integral can be interpreted as the expectancy value of a random variable drawn uniformly in a simplex. As \(k\) goes to infinity, a sort of "law of large numbers" gives the asymptotic estimate appearing in the right hand side of \eqref{eq:lowerboundvol}.
\medskip

\noindent
{\bf Annex A.} This annex gathers several algebraic geometry lemmas, used throughout the text. Most of them are elementary in nature. 
\medskip

\noindent
{\bf Annex B.} We have included here several basic results related to volume computations of simplexes, and random variables in these sets, for lack of a proper reference. These results are used in particular in Chapter~\ref{chap:weighted} and Chapter~\ref{chap:existence}.
\medskip

\noindent
{\bf Annex C.} This annex presents several basic results concerning stacks in positive characteristic, and in particular the weighted projective spaces that appear as fibers of the Green-Griffiths jet spaces. The results of this annex are quite classical and well known to experts;  they are mostly intended to the reader wondering about the consistency of the definitions, in particular in the positive characteristic setting.

\section*{Acknowledgments.}

I would like to thank Damian Brotbek, Fr\'ed\'eric Campana, Julien Grivaux, Henri Guenancia, Gianluca Pacienza and Erwan Rousseau for our enriching discussions around this subject. I am also grateful to Jean-Pierre Demailly for several useful explanations about his original article, to which I am of course very indebted. I would also like to thank Jean-Beno\^{i}t Bost for pointing out several mistakes in the earlier version of this work. Finally, I address special thanks to Lionel Darondeau for many enlightening and motivating discussions all along the conception of these notes.
\medskip

\medskip
During the preparation of this work, the author was partially supported by the following ANR programs: D\'efi de tous les savoirs (DS10) 2015, "GRACK", Project ID : ANR-15-CE40-0003, ANR JCJC "Karmapolis", Project ID : ANR-21-CE40-0010, and ANR "GAG", Project ID : ANR-24-CE40-3526.

\chapter*{Notation and terminology}

\subsection*{Scheme theory}

\begin{itemize}
\setlength{\itemsep}{0.5em}

\item We fix an algebraically closed field \(\mathbbm{k}\) of arbitrary characteristic. In some parts of the text, we will drop this assumption that \(\mathbbm{k}\) is algebraically closed; we will always mention this explicitely.

\item A \emph{scheme} will be a separated scheme of finite type over \(\mathbbm{k}\). A {\em proper} scheme is a scheme admitting a proper morphism to \(\mathrm{Spec}\, \mathbbm{k}\). Unless otherwise mentioned, all fiber products will be taken over \(\mathrm{Spec}\, \mathbbm{k}\).

\item A \emph{variety} will be an irreducible and reduced scheme. A scheme is {\em pure dimensional} if all its irreducible components are varieties of the same dimension.

\item A scheme is {\em normal} if all its local rings are normal. This is equivalent to asking that all its connected components are normal varieties.

\item If \(X\) is a scheme, we will use the notation \(Z_{k}(X)\) (resp. \(Z_{k}(X)_{\mathbb{Q}}\)) to denote the cycle group of dimension \(k\) cycles on \(X\) with coefficients in \(\mathbb{Z}\) (resp. \(\mathbb{Q})\). Similary \(A_{k}(X)\) (resp. \(A_{k}(X)_{\mathbb{Q}}\)) will denote the Chow group of \(k\)-th dimensional cycles class with integral (resp. rational) coefficients. We refer to \cite{ful98} for more details. 

\item A {\em \(\mathbb{Q}\)-line bundle} on a scheme \(X\) is an element of \(\mathrm{Pic}(X) \otimes_{\mathbb{Z}} \mathbb{Q}\). We will sometimes qualify elements of \(\mathrm{Pic}(X)\) as {\em standard line bundles}, as opposed to \(\mathbb{Q}\)-line bundles.

\item If \(V, W\) are varieties, a dominant map \(f : V \to W\) is {\em birational} if it is an isomorphism above a neighborhood of the generic point of \(W\). 

\item If \(V\) is a variety, a {\em modification} of \(V\) is a proper birational morphism \(W \longrightarrow V\).
\end{itemize}
\medskip

\subsection*{Trees}

\begin{itemize}
\setlength{\itemsep}{0.5em}

\item A {\em rooted tree} (or {\em tree}, for short) is a finite connected acyclic undirected graph, with a distinguished node called the {\em root}. We will sometimes speak of \(\mathcal{T}\) as the set of its nodes, implying the additional data formed by the edges.
	
\item In a tree, a node \(a\) is the {\em parent} of an other node \(b\), if it is the closest node on the unique path connecting \(b\) to the root. In this case \(b\) is called a {\em child} of \(a\). We say that \(b\) is an {\em ancestor} of \(a\) if there is a path leading from \(a\) to the root passing through \(b\). Conversely, we say that \(a\) is a {\em descendant} of \(b\).
	
	\item A node with no children is called a {\em leaf}.

	\item A {\em root-to-leaf path} is a path leading the root to any leaf. These paths are in bijection with the leaves.

	\item An {\em embedding of trees} \(\iota : \mathcal{T} \hookrightarrow \mathcal{T'}\) is a injective map between the set of nodes, that preserves the parent-child relation induced by the edges. We will say also that \(\mathcal{T}\) embeds as a subtree of \(\mathcal{T'}\). If \(\iota\) preserves the root, we say that it is a {\em rooted embedding of trees}.

	\item If \(\mathcal{T}\) is a rooted tree and \(\mathbf{a} \in \mathcal{T}\) is a node, then the {\em subtree of \(\mathcal{T}\) based at \(\mathbf{a}\)} is the subtree of \(\mathcal{T}\) formed by elements of \(\mathcal{T}\) that have \(\mathbf{a}\) as an ancestor. Its root identifies with \(\mathbf{a}\).
\end{itemize}

\mainmatter

\chapter{Green-Griffiths and Hasse-Schmidt jet differentials} \label{chap:jetdiff}

In this chapter, \(\mathbbm{k}\) will be a field of arbitrary characteristic, {\em not necessarily algebraically closed}. A {\em scheme} will mean a separated scheme of finite type above \(\mathrm{Spec}\, \mathbbm{k}\) (see the section on notation).

\section{Hasse-Schmidt differentials}

It is a classical fact that differential calculus is better behaved in positive characteristic if one accepts to work with {\em divided} differentials. When it comes to defining the theory of jet differentials, one way to proceed is to use the theory of {\em Hasse-Schmidt higher derivations}, as advocated by Vojta \cite{Voj04}.
\medskip

We refer to Vojta's notes for an axiomatic presentation of the construction. In the following section, we will present a more geometric approach, perhaps closer to its complex geometric equivalent: we will first describe these jet differentials on (open subsets) of affine spaces, and then use \'{e}tale charts to define them locally and glue them on any given {\em smooth} scheme. Note that our discussion ultimately relies in Vojta's proofs; we hope it should be sufficiently self contained nonetheless.
\bigskip

\subsection{Main definitions} Let us recall the concept of (divided) higher derivations, instrumental in everything that follows.

\begin{defi} \label{defi:HS}
	Let \(A, B\) be two \(\mathbbm{k}\)-algebras, and let \(l \in \mathbb{N}\) be an integer. A {\em higher derivation of order} \(l\) from \(A\) to \(B\) is a sequence \((D_{0}, \dotsc, D_{l})\) where \(D_{0} : A \to B\) is a \(\mathbbm{k}\)-algebra homomorphism and \(D_{1}, \dotsc, D_{l} : A \to B\) are morphisms of additive groups, subject to the following requirements:
	\begin{enumerate}
		\item for all \(a \in \mathbbm{k}\) and \(i \in \llbracket 1, l \rrbracket\), one has \(D_{i}(a) = 0\);
		\item for all \(x, y \in A\), and \(m \in \llbracket 0, l\rrbracket\), one has the {\em Leibniz rule}:
			\[
				D_{m}(xy)
				=
				\sum_{i + j = m} D_{i}(x) D_{j}(y).
			\]
	\end{enumerate}

	Let us denote by \(\mathrm{Der}_{\mathbbm{k}}^{l}(A, B)\) the \(\mathbbm{k}\)-vector space of higher derivations of order \(l\) from \(A\) to \(B\).
\end{defi}

\begin{rem}
	To fix the ideas, let us assume that \(\mathbbm{k} = \mathbb{C}\), and let \(A = \mathbb{C}[X_{1}, \dotsc, X_{n}]\), \(B = \mathbb{C}\). Then, for any germ of analytic curve \(f : (\mathbb{C}, 0) \to \mathbb{A}^{n}_{\mathbb{C}}\), one gets such a sequence of derivations by letting, for any \(P \in A\):
	\begin{equation} \label{eq:divided}
		D_{i}(P) = \frac{1}{i!}\frac{\partial^{i} (P \circ f)}{\partial t^{i}}(0)
	\end{equation}

	It is essentially tautological that there is a one-to-one correspondence between \(l\)-th Taylor expansions of such germs, and elements of \(\mathrm{Der}_{\mathbb{C}}^{l}(A, \mathbb{C})\). The next proposition shows that this holds in general.
\end{rem}

\begin{prop} \label{prop:kjetsderiv}
	Let \(A\) be a \(\mathbbm{k}\)-algebra. Let \(k \in \mathbb{N}\) be an integer, and let \(R := \mathbbm{k}[t]/(t^{k+1})\).  Then there is a bijective map
	\[
		\begin{array}{cccc}
			\mathrm{Der}_{\mathbbm{k}}^{k}(A, \mathbbm{k})
			&
			\longrightarrow
			& 
			\mathrm{Hom}_{\mathbbm{k}}(A, R) \\
			(D_{0}, D_{1}, \dotsc, D_{k})
			&
			\longmapsto
			&
			\left(a 
			\;\longmapsto\; 
			D_{0}(a) + D_{1}(a) t + \dotsc + D_{l}(a)t^{k}\right)
		\end{array}
	\]
	If one lets \(\Delta_{k} := \mathrm{Spec}(R)\), and \(X := \mathrm{Spec}(A)\), the set on the right hand side also identifies to \(\mathrm{Hom}_{\mathbbm{k}}(\Delta_{k}, X)\), the set of {\em \(k\)-th order jets on \(X\)} i.e. of \(\mathbbm{k}\)-scheme morphisms from \(\Delta_{k}\) to \(X\).
\end{prop}

The Hasse-Schmidt algebra can be defined from the following universal property.

\begin{thm} \label{thm:defiuniversalHS}
	Let \(A\) be a \(\mathbbm{k}\)-algebra, and let \(k \in \mathbb{N}\). Then there exists an \(A\)-algebra, denoted by \(\mathrm{HS}_{A/\mathbbm{k}}^{k}\), and an element  \((D_{0}, \dotsc, D_{k}) \in \mathrm{Der}_{\mathbbm{k}}^{k}(A, \mathrm{HS}^{k}_{A/\mathbbm{k}})\) that for any \(\mathbbm{k}\)-algebra \(B\), the application
	\begin{equation} \label{eq:defipropunivHS}
		\begin{array}{ccc}
			\mathrm{Hom}_{A}(\mathrm{HS}_{A/\mathbbm{k}}^{k}, B)
			&
		\longrightarrow
			&
			\mathrm{Der}_{\mathbbm{k}}(A, B) \\
			\Psi
			&
			\longmapsto
			&
			(\Psi \circ D_{0}, \Psi \circ D_{1}, \dotsc, \Psi \circ D_{k})
		\end{array}
	\end{equation}
	is a bijection.
\end{thm}

In particular, \(\mathbbm{k}\)-points in the \(A\)-scheme \(\mathrm{Spec}(\mathrm{HS}_{A/\mathbbm{k}}^{k})\) are in bijection with elements of \(\mathrm{Der}_{\mathbbm{k}}^{k}(A, K)\). By Proposition~\ref{prop:kjetsderiv}, the latter identifies with the set of \(k\)-jets in \(\mathrm{Spec}(A)\).

\begin{rem}
	The algebra \(\mathrm{HS}_{A/\mathbbm{k}}^{l}\) is the \(A\)-algebra generated by the formal elements  \(D_{j}(x)\)  for \(j \in \llbracket 0, l\rrbracket\) and \(x \in A\), subject to some natural relations (see \cite[Definition 1.3]{Voj04}). If we let
	\[
		\deg D_{j}(x) = j
	\]
	for all such \(j\) and \(x\), one can show that this gives \(\mathrm{HS}_{A/K}^{k}\) a well-defined natural structure of \(\mathbb{N}\)-graded algebra.
\end{rem}

The last definition can be used locally on a scheme to endow it with a sheaf of algebras.

\begin{defi} Let \(X\) be a \(\mathbbm{k}\)-scheme. The {\em sheaf of Hasse-Schmidt algebras of order \(k\)} on \(X\) is the \(\mathcal{O}_{X}\)-algebra \(E_{k}^{HS}\Omega_{X}\) obtained by patching together the sheaves associated with the \(\mathrm{HS}_{A/\mathbbm{k}}^{l}\), where \(\mathrm{Spec}\, A\) runs in a covering of \(X\) by affine charts.
\end{defi}

\begin{rem}
	The notation \(E_{k}^{GG}\Omega_{X}\) has become quite standard to denote the vector bundle of (Green-Griffiths) jet differentials of order \(k\) on a complex manifold \(X\). We decided to stick to this notation (replacing the letters \(GG\) with \(HS\)) to evidence the link more closely. We will later use also the notation \(E_{k, \bullet}^{GG} \Omega_{X}\) over a smooth scheme, once we have evidenced a natural grading in this setting.
\end{rem}

\begin{ex}
	If \(k = 1\), the Hasse-Schmidt algebra \(E^{HS}_{1}\Omega_{X}\) identifies with the symmetric algebra of the sheaf of K\"{a}hler differentials \(S^{\bullet} \Omega_{X}^{1}\), via the identification
	\[
		D_{1}(x) \equiv dx 
	\]
	for any local section \(x\) of \(\mathcal{O}_{X}\).
\end{ex}

The algebra of Hasse-Schmidt differentials satisfies several functoriality properties. Let us simply mention that for any morphism \(f : X \to Y\), we can define a pullback morphism
\[
	f^{\ast} (E_{k}^{HS}\Omega_{Y}) \longrightarrow E_{k}^{HS}\Omega_{X},
\]
where both sides are \(\mathcal{O}_{X}\)-algebras. We end this section with two important properties concerning the behaviour of this map if \(f\) is smooth or \'{e}tale.

\begin{prop} \label{prop:etalesmooth}
	If \(f : X \to Y\) is an embedding (resp.\ smooth, resp.\ \'{e}tale) morphism of schemes, then the pullback morphism
	\[
		f^{\ast} (E_{k}^{HS}\Omega_{Y})
		\longrightarrow
		E_{k}^{HS}\Omega_{X}	
	\]
	is surjective (resp.\ is injective, resp.\ is an isomorphism).
\end{prop}
\begin{proof}
	We refer to the following three results:
	\begin{enumerate}
		\item if \(f\) is an embedding, see \cite[Theorem 2.2]{Voj04};
		\item if \(f\) is smooth, see \cite[Lemma 3.5]{Voj04};
		\item if \(f\) is \'{e}tale, see \cite[Theorem 3.6]{Voj04}.
	\end{enumerate}
\end{proof}

\subsection{Smooth schemes and local \'{e}tale charts} \label{section:localetalecharts}
In this section, we will show how to describe the Hasse-Schmidt differentials using \'{e}tale coordinate charts on a given smooth scheme; the description will be very close to what we can obtain in local coordinates in the complex analytic topology.
\medskip

Let us first describe what happens on the {\em affine space}: fix an integer \(n \in \mathbb{N}\), and let \(X = \mathbb{A}^{n}_{\mathbbm{k}}\). Let \(A := \mathbbm{k}[x_{1}, \dotsc, x_{n}]\) be the ring of regular functions on \(X\). In this case, we have the following proposition. 

\begin{prop} \label{prop:defiHSalgebraaffine} For any \(l \in \mathbb{N}\), the {\em Hasse-Schmidt ring of jet differentials of order \(k\) on \(\mathbb{A}^{n}_{\mathbbm{k}}\)} identifies with the free sheaf of graded \(\mathcal{O}_{X}\)-algebras associated to the free \(A\)-algebra
	\[
		A[x_{i}^{(j)}]_{\substack{1 \leq i \leq n \\ 1 \leq j \leq l}},
	\]
	where the \(x_{i}^{(j)}\) are indeterminates, and where the \(\mathbb{N}\)-grading is defined by 
	\[
		\deg x_{i}^{(j)} = j \quad (1 \leq i \leq n,  1 \leq j \leq m).
	\]
	The universal derivation \((D_{0}, D_{1}, \dotsc, D_{n})\) is given by its action on the indeterminates \(x_{i}\), as follows :
	\[
		D_{j} (x_{i}) = x_{i}^{(j)}.
	\]
\end{prop}

To prove Proposition~\ref{prop:defiHSalgebraaffine}, it suffices to show that the derivation \((D_{0}, D_{1}, \dotsc, D_{l})\) satisfies the universal property evoked in Theorem~\ref{thm:defiuniversalHS}, which is a matter of routine.
\medskip

\begin{rem}
	Just to fix the ideas, let us show how to compute such differentials in the \(2\)-dimensional ring \(A = \mathbbm{k}[x, y]\). We have for example
			\begin{equation*}
				D_{2}(x^{2} y) = (x^{(1)})^{2} y + 2 x x^{(2)} y + 2 x x^{(1)} y^{(1)} + x^{2} y^{(2)}. 
			\end{equation*}

	This formula can be obtained by dividing by \(2!\) the second derivative computed formally as follows (compare e.g. with \eqref{eq:divided}):
	\[
		(x^{2}y)'' = 2 (x')^{2} y + 2 x x'' y + 4 x x' y' + x^{2} y''.
	\]
	Of course, this division operation does not make sense if \(\mathrm{char}(\mathbbm{k}) = 2\), and we must resort to the Leibniz rule of Definition~\ref{defi:HS} to perform the computation in full generality.
\end{rem}
\medskip

The description in affine coordinates can be used to provide trivializations in {\em étale charts}: this is what we are now going to explain. 
\medskip

In the rest of this section, we fix a smooth \(\mathbbm{k}\)-scheme \(X\) of dimension \(n\). It is a well-known fact (apply for example \cite[Lemma 054L]{stacks} with \(Y = V = \mathrm{Spec}\, \mathbbm{k}\)) that \(X\) is covered by affine open subsets \(U \subset X\), sitting in diagrams
\begin{equation}
	\begin{tikzcd}
		U 
			\arrow[r, hook] 
			\arrow[d, "\pi"] 
		& X 
		\\
		\mathbb{A}_{\mathbbm{k}}^{n}
	\end{tikzcd}
\end{equation}
where the map \(\pi\) is {\em \'{e}tale}. One can even assume that the map assumes the following form.

\begin{prop} \label{prop:standardopen}
	Any smooth scheme \(X\) is covered by open affine subsets \(U \subset X\), that sit in diagrams
\[
	\begin{tikzcd}
		& 
		U 
			\arrow[r, hook, "\iota"] 
			\arrow[d, "\pi"] 
		& 
		V \times \mathbb{A}_{\mathbbm{k}}^{m} 
			\arrow[dl, "p"]
		\\
		\mathbb{A}_{\mathbbm{k}}^{n}	
		& 
		V 
		\arrow[l, hook, "j"] 
	\end{tikzcd}
\]
	where \(\iota\) is a closed embedding, \(p\) is the projection, \(j\) is an open embedding, and \(\pi\) is {\em finite \'{e}tale surjective}. Denote by \(\mathbbm{k}[x]\) (resp. \(\mathbbm{k}[y]\)) the ring of regular functions on \(\mathbb{A}_{\mathbbm{k}}^{n}\) (resp. \(\mathbb{A}_{\mathbbm{k}}^{m}\)), with \(x = (x_{1}, \dotsc, x_{n})\) (resp. \(y = (y_{1}, \dotsc, y_{m})\)). Then, we can assume the following.
	\begin{enumerate}
		\item \(V = D(g)\) is the non-vanishing open subset of some regular function \(g \in \mathbbm{k}[x]\);
		\item \(U\) is a complete intersection in \(V \times \mathbb{A}_{\mathbbm{k}}^{m}\), with ideal generated by \(m\) elements \(f_{1},\dotsc, f_{m} \in \mathbbm{k}[x,y]_{g}\).
		\item The Jabobian matrix
			\bgroup
			\renewcommand{\arraystretch}{1.5}
			\begin{equation} \label{eq:jacobian}
		J
		:=
			\left(
			\begin{array}{cccc}
				\frac{\partial f_{1}}{\partial y_{1}} 
				&
				\frac{\partial f_{1}}{\partial y_{2}}
				& 
				\dotsc 
				& 
				\frac{\partial f_{1}}{\partial y_{m}}
				\\
				\frac{\partial f_{2}}{\partial y_{1}} 
				&
				\frac{\partial f_{2}}{\partial y_{2}}
				& 
				\dotsc 
				& 
				\frac{\partial f_{2}}{\partial y_{m}}
				\\
				\dotsc
				&
				\dotsc
				& 
				\dotsc 
				& 
				\dotsc
				\\
				\frac{\partial f_{m}}{\partial y_{1}} 
				&
				\frac{\partial f_{m}}{\partial y_{2}}
				& 
				\dotsc 
				& 
				\frac{\partial f_{m}}{\partial y_{m}}
				\\
			\end{array}
			\right)
			\end{equation}
			\egroup
			maps to an invertible matrix in \(\mathrm{M}_{m}(R)\), where \(R = \Gamma(U, \mathcal{O}_{U}) \cong \mathbbm{k}[x, y]_{g}/(f_{1}, \dotsc, f_{m})\). In other words, the regular function \(\mathrm{det}(J)\) is invertible in restriction to \(U\).
	\end{enumerate}
\end{prop}
\begin{proof}
	We refer to the Stacks project (cf. \cite[Lemma 01V7]{stacks}. (See also \cite[Definition 00T6]{stacks} for the definition of a standard smooth morphism.
\end{proof}

With the previous notation, Proposition~\ref{prop:etalesmooth} asserts that the sheaf of algebras \(E_{k}^{HS}\Omega_{U}^{l} = E_{k}^{HS}\Omega_{X}|_{U}\) identifies with the pullback \(\pi^{\ast} (E_{k}^{HS}\Omega_{\mathbb{A}_{\mathbbm{k}}^{n}})\), and as such, identifies in turn with the free \(\mathcal{O}_{U}\)-algebra
\[
	\mathcal{O}_{U}\, [(\pi^{\ast} x_{i}^{(j)})_{1 \leq i \leq n, 1 \leq j \leq k}],
\]
We can actually use basic differential algebra to describe this identification, using the invertibility of the Jacobian matrix~\eqref{eq:jacobian}. Let us end this section with a description of this computation.
\medskip

In fact, Proposition~\ref{prop:etalesmooth} applied to the embedding \(\iota : U \hookrightarrow V \times \mathbb{A}_{\mathbbm{k}}^{m}\) shows that \(E_{k}^{HS}\Omega_{U}\) is generated as an \(\mathcal{O}_{U}\)-algebra by the elements
\[
	\iota^{\ast} x_{i}^{(q)},
	\quad
	\iota^{\ast} y_{j}^{(q)}
\]
where \(i \in \llbracket 1, n \rrbracket\), \(j \in \llbracket 1, m \rrbracket\), and \(q \in \llbracket 1, l\rrbracket\). It also follows from \cite[Theorem 2.2]{Voj04} that the kernel of the map 
\[
	\iota^{\ast} (E_{k}^{HS} \Omega_{V \times \mathbb{A}_{\mathbbm{k}}^{m}})
	\longrightarrow
	E_{k}^{HS}\Omega_{U}
\] 
is generated by the derivatives of the equation \(f_{1}, \dotsc, f_{m}\), i.e. by the elements
\[
	D_{i}\left(f_{j}\right)
	\quad
	(1 \leq i \leq k, 1 \leq j \leq m).
\]
We can use inductively Definition~\ref{defi:HS} to compute these differentials: using the Leibniz rule, we see that they appear as the coefficients of the formal expression:
\begin{equation} \label{eq:formal}
	f_{j}(x + \epsilon x^{(1)} + \epsilon^{2} x^{(2)} + \dotsc + \epsilon^{k} x^{(k)}, 
	      y + \epsilon y^{(1)} + \epsilon^{2} y^{(2)} + \dotsc + \epsilon^{k} y^{(k)})
	=
	\sum_{j = 1}^{k} D_{i}(f_{j})\, \epsilon^{i},
\end{equation}
where we denoted \(x^{(i)} = (x_{1}^{(i)}, \dotsc, x_{n}^{(i)})\) and \(y^{(i)} = (y_{1}^{(i)}, \dotsc, y_{m}^{(i)})\). The expression above should be seen as an element of \(\mathrm{HS}_{\mathbbm{k}[x, y]_{g}/\mathbbm{k}} \otimes_{\mathbbm{k}} \mathbbm{k}[\epsilon]\), where \(\epsilon\) is a nilpotent element satisfying \(\epsilon^{k+1} = 0\).
\medskip

Truncating this formula at any order \(q \in \llbracket 1, k \rrbracket\) and expressing the vanishing of the higher order coefficient gives the following equality, for any \(j \in \llbracket 1, m\rrbracket\):
\begin{equation} \label{eq:jacobian1}
	\sum_{i = 1}^{n} \frac{\partial f_{j}}{\partial y_{i}}(x, y)\, y_{i}^{(q)}
	=
	\sum_{i=1}^{n} \frac{\partial f_{j}}{\partial x_{i}}(x, y)\, x_{i}^{(q)}
	+
	\widehat{Q}_{q}(x, y, x^{(1)}, \dotsc, x^{(q-1)}, y^{(1)}, \dotsc, y^{(q-1)}), 
\end{equation}
where \(\widehat{Q}_{q}\) a polynomial in the lower order derivatives \(x^{(i)}\) and \(y^{(i)}\) (for \(i \in \llbracket 1, q - 1\rrbracket\)), with coefficients in \(\mathbbm{k}[x, y]_{g}\). Note that the equation above is homogeneous, i.e. all monomials in \(\widehat{Q}_{q}\) are of weighted degree \(q\) in the variables \(x^{(j)}\) and \(y^{(j)}\).
\medskip

Finally, since the Jabobian matrix \eqref{eq:jacobian} is invertible on \(U\), we may express the \(y^{(q)}\) in terms of the \(x^{(q)}\) inductively. Summing up, we find the following:

\begin{prop} \label{prop:jacobianjets}
	Let \(U\) be as in Proposition~\ref{prop:standardopen}. Let us denote by the same symbols \(y_{i}^{(j)}\), \(x_{i}^{(j)}\) the restriction of these elements to \(U\), implying the pullback \(\iota^{\ast}\). Let \(q \in \mathbb{N}\).
	\smallskip
	
	Then, for any \(i \in \llbracket 1, m\rrbracket\), there exists a polynomial \(Q_{i}^{q}(x^{(1)}, \dotsc, x^{(q-1)})\) with coefficients in \(\Gamma(U, \mathcal{O}_{U})\), that is homogeneous of weighted degree \(q\), and such that:
	\[
		y_{i}^{(q)}
		=
		\sum_{j = 1}^{n} J_{ij}\; x_{i}^{(q)}
		+
		Q_{i}^{q}(x^{(1)}, x^{(2)}, \dotsc, x^{(q-1)}).
	\]
	where \((J_{i,j})_{1 \leq i \leq m, 1 \leq j \leq n}\) denote the matrix
	\[
	J^{-1}
	\cdot
	\left(
			\begin{array}{cccc}
				\frac{\partial f_{1}}{\partial x_{1}} 
				&
				\frac{\partial f_{1}}{\partial x_{2}}
				& 
				\dotsc 
				& 
				\frac{\partial f_{1}}{\partial x_{n}}
				\\
				\frac{\partial f_{2}}{\partial x_{1}} 
				&
				\frac{\partial f_{2}}{\partial x_{2}}
				& 
				\dotsc 
				& 
				\frac{\partial f_{2}}{\partial x_{n}}
				\\
				\dotsc
				&
				\dotsc
				& 
				\dotsc 
				& 
				\dotsc
				\\
				\frac{\partial f_{m}}{\partial x_{1}} 
				&
				\frac{\partial f_{m}}{\partial x_{2}}
				& 
				\dotsc 
				& 
				\frac{\partial f_{m}}{\partial x_{n}}
				\\
			\end{array}
			\right)
	\]
	with coefficients in \(\mathbbm{k}[x, y]_{g}/(f_{1}, \dotsc, f_{m}) \cong \Gamma(U, \mathcal{O}_{U})\).	
\end{prop}

\subsection{The Fa\`{a} di Bruno formula}

In this section, we explain how to use the description given by Proposition~\ref{prop:jacobianjets} to express the base change formula between two sets of \'{e}tale coordinates on a smooth scheme \(X\). This will allow us to define a standard canonical filtration of \(E_{k}^{HS}\Omega_{X}\) in this context.
\medskip

We fix a smooth scheme \(X\). Let \(U_{1}, U_{2} \hookrightarrow X\) be open subsets as in Proposition~\ref{prop:standardopen}, sitting in diagrams
\[
	\begin{tikzcd}
		U_{i} 
		\arrow[r, hook] 
		\arrow[d, "\pi_{i}"] 
		& 
		X 
		\\ 
		\mathbb{A}^{n}_{\mathbbm{k}}
	\end{tikzcd}
\]
and denote by \(x = (x_{1}, \dotsc, x_{n})\) (resp. \(y = (y_{1}, \dotsc, y_{n})\)) the coordinates on \(U_{1}\) (resp. \(U_{2}\)) pulled back from \(\mathbb{A}^{n}_{\mathbbm{k}}\) by \(\pi_{1}\) (resp. \(\pi_{2}\)).
\smallskip

With this notation, we have the following analog of the Fa\`{a} di Bruno formula.

\begin{prop}  [Fa\'{a} di Bruno formula] \label{prop:faadibruno}
	With the notation above, \(U_{1} \cap U_{2}\) is covered by open affine subsets \(U \subset U_{1} \cap U_{2}\), such that the following hold:
	\begin{enumerate}
		\item there exists an invertible matrix \((J_{ij})_{1 \leq i, j \leq n}\) with coefficients in \(\Gamma(U, \mathcal{O}_{U})\) such that
			\[
				y_{i}^{(1)} = \sum_{j=1}^{n} J_{ij}\, x_{j}^{(1)}
			\]
			for all \(i \in \llbracket 1, n \rrbracket\).
		\item fix an integer \(q \geq 1\). Then, for any \(i \in \llbracket 1, n \rrbracket\), there exists a polynomial \(Q_{i}^{q}(x^{(1)}, \dotsc, x^{(q-1)})\) with coefficients in \(\Gamma(U, \mathcal{O}_{U})\) such that
			\begin{equation} \label{eq:faadibruno}
				y_{i}^{(q)} 
				= 
				\sum_{j=1}^{n} J_{ij}\, x_{j}^{(q)}
				+
				Q_{i}^{q}(x^{(1)}, \dotsc, x^{(q-1)}).
			\end{equation}
	The polynomials \(Q_{i}^{q}\) are homogeneous of weighted degree \(q\).
	\end{enumerate}
\end{prop}

\begin{rem}
	In particular the matrix \((J_{ij})\) does not depend on \(q\). It is actually the matrix of change of coordinates for the sheaf of K\"{a}hler differentials, under the composition of the identification morphisms
	\[
		\pi_{2}^{\ast} \Omega_{\mathbb{A}^{n}_{\mathbbm{k}}}|_{U}
		\overset{\cong}{\longrightarrow}
		\Omega_{U}
		\overset{\cong}{\longrightarrow}
		\pi_{1}^{\ast} \Omega_{\mathbb{A}^{n}_{\mathbbm{k}}}|_{U}.
	\]
\end{rem}

\begin{proof}
	For each of the open subsets \(U_{1}, U_{2}\), we can find a diagram as in Proposition~\ref{prop:standardopen}:
\[
	\begin{tikzcd}
		& U_{i} 
			\arrow[r, hook, "\iota_{i}"] 
			\arrow[d, "\pi_{i}"] 
		& V_{i} \times \mathbb{A}_{\mathbbm{k}}^{m_{i}} 
			\arrow[dl, "p_{i}"] 
		\\
		\mathbb{A}_{\mathbbm{k}}^{n}	
		& V_{i}
			\arrow[l, hook, "j_{i}"]
	\end{tikzcd}.
\]
	with \(V_{i} = D(g_{i})\) for some \(g_{i} \in \mathbbm{k}[\mathbb{A}^{n}_{\mathbbm{k}}]\). Now, letting \(m = m_{1} + m_{2}\), we can put the intersection \(U_{1} \cap U_{2}\) in a diagram of the form
\[
	\begin{tikzcd}
	& 
		U_{1} \cap U_{2} 
		\arrow[ddr, "\pi_{1}"]
		\arrow[dd, hook,  "{\iota}"] 
	& 
	\\
	\\
	& V_{1} \times V_{2} \times \mathbb{A}_{\mathbbm{k}}^{m} 
		\arrow[r, swap, "q_{1}"]
	& V_{1} 
	\end{tikzcd}
\]

	Denote the coordinates on the open subset \(V_{1} \times V_{2} \times \mathbb{A}_{\mathbbm{k}}^{m} \subset \mathbb{A}^{2n+m}_{\mathbbm{k}}\) by \((x, y, z)\) (with \(x = (x_{1}, \dotsc, x_{n})\), \(y = (y_{1}, \dotsc, y_{n})\) and \(z = (z_{1}, \dotsc, z_{m})\)). In the diagram above, the morphism \(q_{1}\) is the projection \(q_{1}(x, y, z) = x\). The vertical arrow \(\iota = (\iota_{1}, \iota_{2})\) is a closed embedding by Lemma~\ref{lem:closedembedding} below.
\medskip

	Let \(Z\) be the schematic image of \(\iota\). Since \(V_{1} \times V_{2} \times \mathbb{A}_{\mathbbm{k}}^{m}\) is affine, the ideal sheaf \(\mathcal{I}_{Z}\) is generated by global sections 
	\[
		F_{i}(x, y, z)
		\in
		\mathbbm{k}[V_{1} \times V_{2} \times \mathbb{A}_{\mathbbm{k}}^{m}]
		= 
		\mathbbm{k}
		[x, y, z]_{g_{1}, g_{2}}
		\quad
		\quad
		\left(\text{for}\; 1 \leq i \leq r\right).
	\]

	Fix \(q \in \mathbb{N}\). We may proceed as in Section~\ref{section:localetalecharts}: differentiate the above equation as in \eqref{eq:formal} gives the existence of an element \(Q_{q}(x, y, z, x^{(\bullet)}, y^{(\bullet)}, z^{(\bullet)}) \in \mathrm{HS}_{\mathbbm{k}[x, y, z]_{g_{1},g_{2}}/\mathbbm{k}}\), depending only on the derivatives up to order \(q-1\), such that
	\begin{align*}
	\sum_{i = 1}^{n} \frac{\partial F_{j}}{\partial y_{i}}(x, y, z)\, y_{i}^{(q)}
		+
	\sum_{i = 1}^{m} \frac{\partial F_{j}}{\partial z_{i}}(x, y, z)\, z_{i}^{(q)}
		& = 
	\sum_{i=1}^{n} \frac{\partial F_{j}}{\partial x_{i}}(x, y, z)\, x_{i}^{(q)} \\
		& + Q_{q}(x, y, z,x^{(1)}, \dotsc, x^{(q-1)}, y^{(1)}, \dotsc, y^{(q-1)}, z^{(1)}, \dotsc, z^{(q-1)}).
	\end{align*}

	\medskip
	The projection map \(q_{1} : Z \to V_{1}\) identifies with \(\pi_{1}\), and thus is \'{e}tale. The jacobian criterion for \'{e}taleness implies that \(Z\) is covered by open affine subsets \(U \subset Z\), for which we may choose indexes \(j_{1}, \dotsc, j_{n + m} \in \llbracket 1, r\rrbracket\) such that the matrix
	\[
	\left(
	\begin{matrix}
		\frac{\partial F_{j_{1}}}{\partial y_{1}}
		& \dotsc
		& \frac{\partial F_{j_{1}}}{\partial y_{n}}
		& \frac{\partial F_{j_{1}}}{\partial z_{1}} 
		& \dotsc
		& \frac{\partial F_{j_{1}}}{\partial z_{m}} 
		\\
		\frac{\partial F_{j_{2}}}{\partial y_{1}}
		& \dotsc
		& \frac{\partial F_{j_{2}}}{\partial y_{n}}
		& \frac{\partial F_{j_{2}}}{\partial z_{1}} 
		& \dotsc
		& \frac{\partial F_{j_{2}}}{\partial z_{m}} 
		\\
		\dotsc
		& \dotsc
		& \dotsc
		& \dotsc
		& \dotsc
		& \dotsc
		\\
		\frac{\partial F_{j_{n+m}}}{\partial y_{1}}
		& \dotsc
		& \frac{\partial F_{j_{n+m}}}{\partial y_{n}}
		& \frac{\partial F_{j_{n+m}}}{\partial z_{1}} 
		& \dotsc
		& \frac{\partial F_{j_{n+m}}}{\partial z_{m}} 
	\end{matrix}
	\right)
	\]
	induces an invertible matrix with coefficients in \(\Gamma(U, \mathcal{O}_{U})\). Thus, we can argue as in the discussion preceding of Proposition~\ref{prop:jacobianjets}, and deduce that there exists a matrix
	\(
		J \in \mathrm{M}_{n}(\Gamma(U, \mathcal{O}_{U}))
	\)
 depending only on the first order derivatives of the \(F_{j}\), and for any \(i \in \llbracket 1, n\rrbracket\) and \(p \in \mathbb{N}\), a polynomial \({Q}_{i}^{p}(x^{(1)}, \dotsc, x^{(p-1)})\) with coefficients in \(\Gamma(U, \mathcal{O}_{U})\) such that:\footnote{We may also express the \(z_{j}^{(q)}\) in terms of the \(x_{(j)}\), which is not relevant for our purposes.}
	\[
		y_{i}^{(q)}
		=
		\sum_{j = 1}^{n} J_{ij}\; x_{i}^{(q)}
		+
		{Q}_{i}^{q}(x^{(1)}, x^{(2)}, \dotsc, x^{(q-1)}).
	\]
	Note in particular that the matrix \((J_{ij})\) is independent of \(q\). We may also reverse the role of \(x\) and \(y\) in this discussion, which yields an inverse for \(J\). This gives the result.
\end{proof}

\begin{lem} \label{lem:closedembedding}
	Let \(X\) be a separated scheme, and let \(U_{1}, U_{2} \subset X\) be open subsets. For \(i = 1, 2\), assume we are given a closed embedding \(\iota_{i} : U_{i} \rightarrow V_{i}\) in some scheme \(V_{i}\). Then the product morphism \((\iota_{1}, \iota_{2}) : U_{1} \cap U_{2} \to V_{1} \times V_{2}\) is a closed embedding.
\end{lem}
\begin{proof}
	The open subscheme \(U_{1} \cap U_{2}\) identifies with \((U_{1} \times U_{2}) \cap \Delta_{X}\) in \(X \times X\), and \(\Delta_{X}\) is closed since \(X\) is separated. This implies that \(\iota\) can be written as the composition of two closed embeddings
	\[
		U_{1} \cap U_{2}
			\longrightarrow
		U_{1} \times U_{2}
			\overset{\iota_{1} \times \iota_{2}}{\longrightarrow}
		V_{1} \times V_{2}.
	\]
\end{proof}

\begin{rem} \label{rem:degree}
	Since the polynomials \(Q_{i}^{q}\) appearing in Proposition~\ref{prop:faadibruno} are homogeneous of degree \(q\), we retrieve the fact that the degree of a section on \(U_{1} \cap U_{2}\) coincide if we express it in terms of the \(x_{i}^{(q)}\) or the \(y_{i}^{(q)}\).
	\smallskip

	On any smooth scheme, this gives a globally defined decomposition of the algebra \(E_{k}^{HS}\Omega_{X}\) as	     \[
		E_{k}^{HS}\Omega_{X}
		=
		\bigoplus_{m \in \mathbb{N}} E_{k, m}^{HS}\Omega_{X},
	\]
	where the \(E_{k, m}^{HS}\Omega_{X}\) are the locally free sheaves whose sections have pure weighted degree \(m\). We will call \(E_{k, m}^{HS}\Omega_{X}\) the {\em Hasse-Schmidt vector bundle of order \(k\) and degree \(m\)}.
\end{rem}

\subsection{The canonical filtration}
As we will see now, Proposition~\ref{prop:faadibruno} above allows us to define the canonical filtration on the Hasse-Schmidt algebra of any smooth scheme, exactly as in the complex analytic setting (cf. \cite{GG80} and \cite[§7.A.]{dem12a}). Let us start by defining this filtration on an affine space: it will be provided by choosing a natural ordering in the monomials in the formal derivatives. We need the following definition.
\medskip

\begin{defi} \label{defi:lexicographic}
	Let \(k \in \mathbb{N}\). Let \(\alpha, \beta \in \mathbb{N}^{k}\). We write \(\alpha \leq \beta\) if \(\alpha\) comes before \(\beta\) for the lexicographic order starting from the right. In other words
	\[
		(\alpha_{1}, \dotsc, \alpha_{k})
		<	
		(\beta_{1}, \dotsc, \beta_{k})
		\quad
		\Leftrightarrow		
		\quad
		\left\{
			\begin{array}{l}
		\alpha_{k} < \beta_{k}
		\quad
		\text{or}
				\\
		\alpha_{k} = \beta_{k}
		\quad
		\text{and}
		\quad
		(\alpha_{1}, \dotsc, \alpha_{k-1})
		<
		(\beta_{1}, \dotsc, \beta_{k-1}).
		\end{array}
			\right.
	\]
\end{defi}

\begin{defi} \label{defi:affinefilt}
	Let \(n, k, m \in \mathbb{N}\). By Proposition~\ref{prop:defiHSalgebraaffine} and Remark~\ref{rem:degree}, the Hasse-Schmidt vector bundle \(E_{k,m}^{HS} \Omega_{\mathbb{A}_{\mathbbm{k}}^{n}}^{l}\) is the \(\mathcal{O}_{\mathbb{A}^{n}_{\mathbbm{k}}}\)-module generated by the monomials in the variables
	\[
		x_{i}^{(p)}
		\quad
		\quad
		(1 \leq i \leq n, 1 \leq p \leq k),
	\]
	with total weighted degree \(m\). For such monomial \(\mu\) in the \(x_{i}^{(p)}\), of the form
	\[
		\mu	
		=
		\prod_{i, p} (x_{i}^{(p)})^{\alpha_{i, p}},
	\]
	we let \(\alpha(m) = (\alpha_{1}, \dotsc, \alpha_{l})\) be the tuple of integers given by \(\alpha_{p} = \sum_{i = 1}^{n} \alpha_{i, p}\) for all \(p \in\llbracket 1, l\rrbracket\).
	\medskip

	Let now be a tuple \(\beta = (\beta_{1}, \dotsc, \beta_{p}) \in \mathbb{N}^{p}\). We define the \(\mathcal{O}_{\mathbb{A}_{\mathbbm{k}}^{n}}\)-submodule
	\[
		F^{\beta}
		(E_{k, m}^{HS} \Omega_{\mathbb{A}_{\mathbbm{k}}^{n}}^{l})
		\subset
		E_{k, m}^{HS} \Omega_{\mathbb{A}_{\mathbbm{k}}^{n}}^{l}
	\]
	as the subsheaf generated by the monomials \(\mu\) satisfying \(\alpha(\mu) \leq \beta\) accordingly to Definition~\ref{defi:lexicographic}.
\end{defi}

By the definition above, we obtain a filtration \(F^{\bullet} (E_{k}^{HS} \Omega_{\mathbb{A}_{\mathbbm{k}}^{n}})\), indexed by \(\mathbb{N}^{k}\), for any dimension \(n\) and any order \(k\). We will now see that Proposition~\ref{prop:faadibruno} can be used to glue these \'{e}tale local definitions on any smooth scheme.
\medskip

\begin{prop}
	Let \(X\) be a smooth scheme. Let \(k \in \mathbb{N}\) and let \(\beta \in \mathbb{N}^{k}\). Let \(U_{1}, U_{2} \subset X\) be open subsets as in Proposition~\ref{prop:faadibruno}, and denote by \(\pi_{i} : U_{i} \to V_{i}\) the corresponding projections.
	\medskip

	Then, one has, for any \(m \in \mathbb{N}\):
	\[
		\pi_{1}^{\ast}
		(F^{\beta} E_{k,m}^{HS} \Omega_{\mathbb{A}^{n}_{\mathbbm{k}}})|_{U_{1} \cap U_{2}}
		=
		\pi_{2}^{\ast}
		(F^{\beta} E_{k,m}^{HS} \Omega_{\mathbb{A}^{n}_{\mathbbm{k}}})|_{U_{1} \cap U_{2}}	
	\]
	as subsheaves of \(E_{k,m}^{HS} \Omega_{X}|_{U_{1} \cap U_{2}}\).
\end{prop}
\begin{proof}
	We may use the same notation as in Proposition~\ref{prop:faadibruno}. Then the sheaf \(\pi_{1}^{\ast} (E_{k}^{HS} \Omega_{\mathbb{A}^{n}_{\mathbbm{k}}})\) (resp. \(\pi_{2}^{\ast} (E_{k}^{HS} \Omega_{\mathbb{A}^{n}_{\mathbbm{k}}})\)) identifies with the \(\mathcal{O}_{U_{1}}\)-module generated by the \(x_{i}^{(j)}\) (resp. the \(\mathcal{O}_{U_{1}}\)-module generated by the \(y_{i}^{(j)}\)).
	\medskip

	By Proposition~\ref{prop:faadibruno} now allows us to cover \(U_{1} \cap U_{2}\) by open affine subsets on which \eqref{eq:faadibruno} holds. With the notation of Definition~\ref{defi:affinefilt}, this shows that any monomial \(\mu\) in the \(y_{i}^{(p)}\) can be expressed as an \(\Gamma(U, \mathcal{O}_{U})\)-linear combination of monomials in the \(x_{i}^{(p)}\). Besides, we see from \eqref{eq:faadibruno} that if \(\alpha(\mu) \leq \beta\), then each monomial \(\mu'\) appearing in this expression also satisfies \(\alpha(\mu') \leq \beta\).
	\medskip

	This proves that
	\[
		\pi_{2}^{\ast}
		(F^{\beta} E_{k,m}^{HS} \Omega_{\mathbb{A}^{n}_{\mathbbm{k}}})|_{U}
		\subset	
		\pi_{1}^{\ast}
		(F^{\beta} E_{k,m}^{HS} \Omega_{\mathbb{A}^{n}_{\mathbbm{k}}})|_{U}.	
	\]
	We have the converse inclusion by symmetry. Since the open subsets \(U\) cover \(U_{1} \cap U_{2}\), this gives the result.
\end{proof}

\begin{defi} \label{defi:canonicalfilt}
	Let \(X\) be a smooth scheme. Let \(k, m \in \mathbb{N}\). The {\em canonical filtration} \(F^{\bullet} E_{k, m}\Omega_{X}\) is the filtration indexed by \(\mathbb{N}^{l}\), and defined by
	\[
		F^{\bullet} E_{k, m}^{HS}\Omega_{X}
		=
		\pi^{\ast} 
		(F^{\bullet} E_{k,m}^{HS}\Omega_{\mathbb{A}_{\mathbbm{k}}^{k}})
	\]
	for any local chart \(\pi : U \to \mathbb{A}^{n}_{\mathbbm{k}}\) as in Proposition~\ref{prop:standardopen}.
\end{defi}

\begin{rem}
	The canonical filtration is compatible with the structure of \(\mathcal{O}_{X}\)-algebra. Namely, the product of elements in \(F^{\beta_{1}} E_{k, m_{1}}^{HS} \Omega_{X}\) and \(F^{\beta_{2}} E_{k, m_{2}}^{HS} \Omega_{X}\) lands in \(F^{\beta_{1} + \beta_{2}} E_{k, m_{1} + m_{2}}^{HS}\Omega_{X}\). To see this, it suffices to check it on the affine space, where it comes from a straightforward computation.
	\medskip

	Thus, we obtain a structure of \(\mathbb{N}^{l}\)-{\em filtered} \(\mathcal{O}_{X}\)-algebra:
	\[
		F^{\bullet} E_{k, \bullet}^{HS}\Omega_{X}
		=
		\bigoplus_{m \in \mathbb{N}}
		F^{\bullet} E_{k, m}^{HS}\Omega_{X}.
	\]
\end{rem}

\subsection{Weighted symmetric products and the graded algebra} \label{sec:weighted} Exactly as in the complex analytic case, the Fa\`{a} di Bruno formula allows us to describe the graded sheaf of algebras associated with the filtration \(F^{\bullet} E_{k}^{HS}\Omega_{X}\) as a symmetric algebra constructed in terms of the K\"{a}hler differentials. To define it more precisely, we reintroduce the following definition, already used in \cite{cad17}.
\medskip

\begin{defi}
	Let \(X\) be a scheme. A {\em weighted direct sum} is the data of vector bundles \(E_{1}, E_{2}\, \dotsc, E_{r} \to X\) and positive integers \(a_{1}, a_{2}, \dotsc, a_{r} \in \mathbb{N}_{>0}\). We will sometimes denote this data by a sum with exponents, and use most of the time boldface letters to denote such objects, e.g.
	\[
		\mathbf{E}
		:=
		E_{1}^{(a_{1})}
		\oplus
		E_{2}^{(a_{2})}
		\oplus
		\dotsc
		\oplus
		E_{r}^{(a_{r})}.
		\]
\end{defi}

\begin{rem}
	We could of course deal with more general weighted direct sums, considering arbitrary coherent sheaves as our direct factors. In all the rest of these notes, only locally free sheaves will be of interest to us.
\end{rem}

Given a weighted direct sum, we can naturally define various linear operations:

\begin{defi} \label{eq:deftensorweighted}
	Let \(X\) be a scheme, and let \(\mathbf{E} = E_{1}^{(a_{1})} \oplus E_{2}^{(a_{2})} \oplus \dotsc \oplus E_{r}^{(a_{r})}\) be a weighted direct sum on \(X\). Then
	\medskip
	\begin{enumerate}[label=(\arabic*)]
		\item the dual \(\mathbf{E}^{\ast}\) is the weighted direct sum
			\[
				\mathbf{E}^{\ast}
				:=
				(E_{1}^{\ast})^{(a_{1})}
				\oplus
				(E_{2}^{\ast})^{(a_{2})}
				\oplus
				\cdots
				\oplus
				(E_{r}^{\ast})^{(a_{r})};
			\]
		\item the symmetric algebra \(S^{\bullet}(\mathbf{E})\) is the \(\mathbb{N}\)-graded \(\mathcal{O}_{X}\)-algebra equal to \(\mathrm{Sym}(E_{1} \oplus \dotsc \oplus E_{r})\), where sections of \(E_{j}\) are assigned weight \(a_{j}\). In other words, for any \(m \in \mathbb{N}\), one has 
			\[
			S^{\bullet}(\mathbf{E}) 
			= 
			\bigoplus_{m \in \mathbb{N}}
			S^{m}(\mathbf{E}),
			\] 
			with, for each \(m \in \mathbb{N}\):
			\[
				S^{m}(\mathbf{E})
				:=
				\bigoplus_{a_{1} l_{1} + a_{2} l_{2} + \dotsc + a_{r} l_{r} = m}
				S^{l_{1}} E_{1}
				\otimes
				S^{l_{2}} E_{2}
				\otimes
				\dotsc
				\otimes
				S^{l_{r}} E_{r}
			\]
	\end{enumerate}
\end{defi}

The graded algebra \(\mathrm{Gr}_{F}(E_{k, \bullet}^{HS} \Omega_{X})\) is actually isomorphic to the graded algebra of a natural weighted direct sum on \(X\). In the rest of the text, we will use the following two notations: 

\begin{defi}
	Let \(X\) be a smooth scheme. For any \(k \in \mathbb{N}\), we let
	\[
		\mathbf{\Omega}^{k}_{X}
		:= 
		\Omega_{X}^{(1)} \oplus \Omega_{X}^{(2)} \oplus \dotsc \oplus \Omega_{X}^{(k)}.
	\]
	and
	\[
		\mathbf{T}^{k}_{X}
		:=
		T_{X}^{(1)} \oplus T_{X}^{(2)} \oplus \dotsc \oplus T_{X}^{(k)}.
	\]
	where \(T_{X}\) is the vector bundle associated with the locally free sheaf \(\mathcal{H}om_{X}(\Omega_{X}, \mathcal{O}_{X})\).
\end{defi}

The main result of this section is the following.

\begin{thm} \label{thm:grading}
	Let \(X\) be a smooth scheme, and let \(k \in \mathbb{N}\). Then we have an isomorphism of graded \(\mathcal{O}_{X}\)-algebras 
	\[
		\mathrm{Gr}_{F}(E^{HS}_{k, \bullet}\Omega_{X})
		\cong
		S^{\bullet}(\mathbf{\Omega}^{k}_{X}).
	\]
\end{thm}

\begin{proof} We will show that the isomorphism holds for the \(m\)-th graded piece, for any \(m \in \mathbb{N}\). We proceed in two steps: first, we define the isomorphism locally using étale coordinates, and then we show that these isomorphisms glue together using the Fa\`{a} di Bruno formula.
	\smallskip

	{\em Local definition.} Let \(U \subset X\) be an affine open subset sitting in a diagram as in Proposition~\ref{prop:standardopen}; let us use the same notation as in this proposition. Remark first that since \(\pi\) is \'{e}tale, we have the following isomorphism of \(\mathcal{O}_{U}\)-modules
	\[
		\Omega_{U} \cong \pi^{\ast} \Omega_{\mathbb{A}_{\mathbbm{k}}^{n}}.
	\]
	In other words, \(\Omega_{U}\) is generated as a free \(\mathcal{O}_{U}\)-module by the elements
	\[
		dx_{i} \quad (\text{for}\; i \in \llbracket 1, n \rrbracket),
	\]
	where we implied the pullback by \(\pi\) to \(U\).
	\smallskip

	Let \(m \in \mathbb{N}\). It follows from the above that \(S^{m}(\mathbf{\Omega}_{U}^{k})\) is generated as a free \(\mathcal{O}_{U}\)-module by the sections
	\begin{equation} \label{eq:monomialkahler}
	\bigotimes_{p = 1}^{k} \bigodot_{q=1}^{n} dx_{q}^{\alpha_{p, q}}
	\end{equation}	
	where \(\alpha = (\alpha_{p, q}) \in \mathbb{N}^{\llbracket 1, k \rrbracket \times \llbracket 1, n \rrbracket}\) is such that \(\sum_{p,q} p\, \alpha_{p, q} = m\).
	\medskip

	On the other hand, we also have
	\[
		E_{k,m}^{HS} \Omega_{X}|_{U}
		=
		\pi^{\ast}(E_{k,m}^{HS} \Omega_{\mathbb{A}_{\mathbbm{k}}^{n}})
	\]
	and it follows from Definition~\ref{defi:canonicalfilt} that the graded term \(\mathrm{Gr}_{F}(E_{k,m}^{HS} \Omega_{X})|_{U}\) is generated by the monomial classes
	\begin{equation} \label{eq:monomialHS}
		\left[
		\prod_{p=1}^{l} \prod_{q=1}^{l} (x_{q}^{(p)})^{\alpha_{p,q}},
		\right]
	\end{equation}
	where, again, we imply the pullback to \(U\); the \(\alpha_{p, q}\) are elements of \(\mathbb{N}\) subject to the condition \(\sum_{p, q} p \, \alpha_{p, q} = m\). The brackets above denote the class in \(\mathrm{Gr}^{\beta}_{F}(E_{k,m}^{HS}\Omega_{U})\), for the multi-index 
	\[
		\beta 
		= 
		(\sum_{q} \alpha_{1, q}, \sum_{q} \alpha_{2, q}, \dotsc, \sum_{q} \alpha_{n, q}).
		\]
	\smallskip

	We let \(\Psi_{U} : S^{m} (\mathbf{\Omega}_{U}^{k}) \longrightarrow \mathrm{Gr}_{F} (E_{k,m}^{HS} \Omega_{U})\) be the isomorphism of \(\mathcal{O}_{U}\)-modules that sends the monomial \eqref{eq:monomialkahler} onto \eqref{eq:monomialHS}. Note that this isomorphism is compatible with products on both sides, and thus realizes an isomorphism of graded algebras
	\[
		\Psi_{U} 
		:
		S^{\bullet} (\mathbf{\Omega}_{U}^{k})
		\overset{\cong}
		{\longrightarrow}
		\mathrm{Gr}_{F}(E_{k,\bullet}^{HS} \Omega_{U}).
	\]
\medskip

	{\em Glueing.} Let us now consider the setup of Proposition~\ref{prop:faadibruno}, and let \(U \subset U_{1} \cap U_{2}\) be an open affine subset on which the conclusion of the proposition holds. Then, since \(x_{i}^{(1)}\) (resp. \(y_{i}^{(1)}\)) identifies with \(dx_{i}^{(1)}\) (resp. \(dy_{i}^{(1)}\)) under the identification \(E_{1, \bullet}^{HS}\Omega_{U} \cong S^{\bullet} \Omega_{U}\), we have
	\[
		dy_{i} = \sum_{j=1}^{n} J_{ij} \, d x_{j}
	\]
	for all \(i \in \llbracket 1, n\rrbracket\). On the other hand if we pick \((\alpha_{p, q}) \in \mathbb{N}^{\llbracket 1, n \rrbracket \times \llbracket 1, l\rrbracket}\), then we may compare the monomial classes in terms of \(x\) and \(y\) as follows. On the one hand, the equality above yields
	\begin{align*}
	\bigotimes_{p = 1}^{l} \bigodot_{q=1}^{n} dy_{q}^{\alpha_{p, q}}
	=
	\left[
		\bigotimes_{p = 1}^{l} \bigodot_{q=1}^{n} 
		(\sum_{j=1}^{n} J_{qj}\, dx_{j})^{\alpha_{p, q}}
	\right]
	\end{align*}
	On the other hand, we have
	\begin{align*}
		\left[
		\prod_{p=1}^{l} \prod_{q=1}^{l} (y_{q}^{(p)})^{\alpha_{p,q}}
		\right]
		& =
		\left[
			\prod_{p=1}^{l} \prod_{q=1}^{l} (\sum_{j = 1}^{n} J_{qj} x_{j}^{(p)} 
			+
			Q_{q}^{p}(x^{(1)}, x^{(2)}, \dotsc, x^{(p-1)})
			)^{\alpha_{p,q}}
		\right] \\
		& =  
		\left[
			\prod_{p=1}^{l} \prod_{q=1}^{l} 
			(\sum_{j = 1}^{n} J_{qj} x_{j}^{(p)})^{\alpha_{p,q}}
		\right]
	\end{align*}
	where the second equality comes from the fact that all the terms involved in the polynomials \(Q_{q}^{p}\) yield monomials with strictly smaller exponents for the lexicographic order of Definition~\ref{defi:lexicographic}. These two equalities show that \(\Psi_{U_{1}}\) and \(\Psi_{U_{2}}\) coincide on \(U\), and since \(U\) runs in a cover of \(U_{1} \cap U_{2}\), we have
	\[
		\Psi_{U_{1}}|_{U_{1} \cap U_{2}}
		=
		\Psi_{U_{2}}|_{U_{1} \cap U_{2}}.
	\]
	This gives the required glueing, and we get a global isomorphism \(\Psi : S^{m}(\mathbf{\Omega}_{X}^{k}) \to \mathrm{Gr}_{F}(E_{k,m}^{HS}\Omega_{X})\).
\end{proof}

\section{Jet spaces}

Using the Hasse-Schmidt algebra on a smooth scheme, we may define the affine jet spaces in the standard manner.

\begin{defi}
	Let \(X\) be a smooth scheme, and let \(k \in \mathbb{N}\). The {\em \(k\)-th jet space} of \(X\) is the scheme
	\[
		J_{k} X := \mathbf{Spec}_{X}(E_{k, \bullet}^{HS}\Omega_{X}).
	\]
\end{defi}

It follows from Theorem~\ref{thm:defiuniversalHS} and the subsequent remark that \(\mathbbm{k}\)-points on \(J_{k}X\) are in bijection with the \(k\)-th order jets on \(X\) that are defined over \(\mathbbm{k}\).
\smallskip

\begin{rem}[Jet spaces of affine spaces] \label{rem:jetaffine} Let \(n, k \in \mathbb{N}\) be integers. Denote by \(x = (x_{1}, x_{2}, \dotsc, x_{n})\) the standard coordinates on \(X\). The Hasse-Schmidt algebra of \(\mathbb{A}^{n}_{\mathbbm{k}}\) is the sheaf of algebras associated with the \(\mathbbm{k}[x]\)-algebra \(\mathrm{HS}_{\mathbbm{k}[x]/\mathbbm{k}}\), isomorphic to
	\[
		\mathbbm{k}[x][(x_{i}^{(j)})_{1 \leq i \leq n, 1 \leq j \leq k}].
	\]
	Taking the spectrum, we see that 
	\[
		J_{k}\mathbb{A}_{\mathbbm{k}}^{n}
		\cong
		\mathbb{A}_{\mathbbm{k}}^{n}
		\times
		\mathbb{A}_{\mathbbm{k}}^{nk},
	\]
	where coordinates on the vertical factor are given by the \(x_{i}^{(j)}\), for \((i, j) \in \llbracket 1, n \rrbracket \times \llbracket 1, k \rrbracket\).
\end{rem}

\begin{rem} [Description in \'{e}tale coordinates]\label{rem:descretale} 
Using \'{e}tale coordinates, it is easy to give a description of the jet scheme as an affine bundle over any smooth scheme. Indeed, if \(U \subset X\) is an open subset admitting an \'{e}tale morphism \(\pi : U \to \mathbb{A}^{n}_{\mathbbm{k}}\), then we have
\[
	E_{k, \bullet}^{HS} \Omega_{X}
	\cong
	\pi^{\ast} 
	E_{k, \bullet}^{HS} \Omega_{\mathbb{A}_{\mathbbm{k}}^{n}}
\]
as \(\mathcal{O}_{U}\)-algebras. Thus, we have a square diagram
\[
	\begin{tikzcd}
		J_{k}X|_{U} \arrow[r] \arrow[d] \arrow[dr, phantom, "\square"] & J_{k}\mathbb{A}_{\mathbbm{k}}^{n} \arrow[d] \\
	U           \arrow[r, "\pi"]           & \mathbb{A}_{\mathbbm{k}}^{n}.
	\end{tikzcd}
\]
Using Remark~\ref{rem:jetaffine}, this shows that \(J_{k} X|_{U}\) is isomorphic to \(U \times \mathbb{A}_{\mathbbm{k}}^{kn}\).
\end{rem}

\subsection{Green-Griffiths jet stack} The scheme \(J_{k} X\) admits a natural action of the \(\mathbbm{k}\)-multiplicative group \(\mathbb{G}_{m}\) by reparametrization of jets. The corresponding morphism
\[
	J_{k} X \times \mathbb{G}_{m}
	\longrightarrow
	J_{k} X
\]
is dual to the morphism of \(\mathcal{O}_{X}\)-algebras
\[
	E_{k, \bullet}^{HS} \Omega_{X} 
	\longrightarrow
	E_{k, \bullet}^{HS} \Omega_{X} 
	\otimes_{\mathbbm{k}}
	\mathbbm{k}[\lambda, \lambda^{-1}]
\]
provided by the grading, where a section \(s\) of \(E_{k, m}\Omega_{X}^{HS}\) is sent to \(s \otimes \lambda^{m}\).
\medskip

\begin{defi}
	Let \(X\) be a smooth scheme. The {\em Green-Griffiths jet stack of order \(k\)} of \(X\) is the quotient stack of the complement of the zero section \(J_{k} X - \{0\}\) by the action of \(\mathbb{G}_{m}\). We denote this stack by the symbol \(X_{k}^{GG}\). The projection \(J_{k} X \to X\) yields a natural morphism \(\pi_{k} : X_{k}^{GG} \to X\).
\end{defi}

\begin{rem} For affine spaces, Green-Griffiths are the products between the base and {\em weighted projective stacks} (see Annex~\ref{annex:stacks} for more details about these objects). We see from Remark~\ref{rem:descretale} that if \(X = \mathbb{A}_{\mathbbm{k}}^{n}\), we have an isomorphism
	\[
		X_{k}^{GG}
		\cong
		\mathbb{A}_{\mathbbm{k}}^{n}
		\times
		\mathcal{P}(1, \dotsc, 1, 2, \dotsc, 2, \dotsc, k, \dotsc, k),
	\]
	where each number \(i \in \llbracket 1, k\rrbracket\) is repeated \(n\) times. Now, if \(X\) is any smooth scheme, and if \(U \subset X\) is an open subset admitting an \'{e}tale morphism \(\pi  : U \to \mathbb{A}^{n}_{\mathbbm{k}}\), then the pullback provides an isomorphism of stacks
	\[
		X_{k}^{GG}|_{U} 
		\cong 
		U \times \mathcal{P}(1, \dotsc, 1, 2, \dotsc, 2, \dotsc, k, \dotsc, k).
	\]
\end{rem}

\begin{prop}
	The coarse moduli space of the stack \(X_{k}^{GG}\) is the projectivized scheme
	\[
		X_{k}^{GG, \mathrm{coarse}} := \mathbf{Proj}_{X}(E_{k, \bullet}^{HS} \Omega_{X}).
	\]
\end{prop}

The stack \(X_{k}^{GG}\) supports a natural tautological line bundle \(\mathcal{O}(1)_{k} \to X_{k}^{GG}\), such that 

\[
	(\pi_{k})_{\ast} \mathcal{O}_{k}(m) = E_{k, m}^{HS} \Omega_{X}
\]
for all \(k, m \in \mathbb{N}\).

\begin{defi}
	The {\em singular locus} \(X_{k}^{GG, \mathrm{sing}}\) is the image in \(X_{k}^{GG}\) of the subscheme \(J_{k}^{\mathrm{sing}}X \subset J_{k}X\) whose points correspond to {\em singular jets}, i.e. jets with zero first derivative.  
\end{defi}

\begin{rem}
	If \(X = \mathbb{A}_{\mathbbm{k}}^{k}\), with coordinates \((x_{1}, \dotsc, x_{n})\), then \(J_{k}^{\mathrm{sing}}X\) is cut out in \(J_{k}X\) by the sheaf generated by the \(x_{j}^{(1)}\) for \(j = 1, \dotsc, n\). If \(U \to \mathbb{A}_{\mathbbm{k}}^{k}\) is an étale morphism, then \(J_{k}^{\mathrm{sing}}U\) is the pullback of  \(J_{k}^{\mathrm{sing}}\mathbb{A}_{\mathbbm{k}}^{k}\). In general, we may describe \(J_{k}^{\mathrm{sing}}X\) locally by covering \(X\) with open subsets as in Proposition~\ref{prop:standardopen}.
\end{rem}

\subsection{The wild locus} In this section, we assume that \(\mathbbm{k}\) is a field of characteristic \(p > 0\). Let \(X\) be a smooth scheme over \(\mathbbm{k}\), and consider the fiber product 
\[
	\begin{tikzcd}
	X^{(p)}
		\arrow[r, "F_{X}"]
		\arrow[d]
		\arrow[dr, phantom, "\square"]
	&
	X
		\arrow[d]
	\\
	\mathrm{Spec}\,\mathbbm{k}
		\arrow[r, "F"]
	&
	\mathrm{Spec}\,\mathbbm{k}
	\\
	\end{tikzcd},
\]
where \(F\) denotes the Frobenius morphism. Note that \(F_{X}\) is not a morphism of \(\mathbbm{k}\)-schemes in general, but is nonetheless a morphism of \(\mathbb{F}_{p}\)-schemes, where \(\mathbb{F}_{p} \subset \mathbbm{k}\) is the prime field. 

\begin{rem}
	Note that since \(\mathbb{F}_{p}\) is a perfect field, the extension \(\mathbb{F}_{p} \subset \mathbbm{k}\) is \'{e}tale, and thus \(X\) and \(X^{(p)}\) are smooth over \(\mathbb{F}_{p}\) as well.
\end{rem}
\medskip

For any \(k \in \mathbb{N}\), we have an induced morphism of \(\mathbb{F}_{p}\)-schemes
\[
	F_{k}: 
	J_{k} X^{(p)}
	\longrightarrow
	J_{k} X.
\]

We now have the following proposition.

\begin{prop} \label{prop:wild}
	Let \(X\) be a smooth scheme, and let \(k \in \mathbb{N}\). We have a commutative diagram of  \(\mathbb{F}_{p}\)-schemes and stacks	
	\begin{equation} \label{eq:Frobeniussquare}
		\begin{tikzcd}
		J_{k} X^{(p)}
			\arrow[d]
			\arrow[r]
		&
		J_{k} X
			\arrow[d]
		\\
		(X^{(p)})_{k}^{GG}
			\arrow[r]
		&
		X^{GG}_{k}
		\end{tikzcd}
	\end{equation}
	The image of the bottom arrow is the complement of the largest open substack on which \(X_{k}^{GG}\) is a Deligne-Mumford stack.
\end{prop}

\begin{rem} We refer to Annex~\ref{annex:stacks} and more particularly Section~\ref{sec:DMlocus} for some comments about the above terminology. Informally, the image of the bottom arrow above is the closed substack where closed points have a weight divisible by \(p\), for which the isotropy groups are {\em non-reduced} groups of roots of unity. On the complement of this locus, \(X_{k}^{GG}\) admits an \'{e}tale atlas; on the contrary finite maps covering the wild locus are necessarily inseparable.
\end{rem}

\begin{proof}
	{\em Step 1. The affine case.} We first prove the result in the case where \(X = \mathbb{A}_ {\mathbbm{k}}^{n}\). In this situation, we have \(X^{(p)} = X\), and the morphism \(X^{(p)} \to X\) is dual to the morphism of \(\mathbb{F}_{p}\)-algebras
	\[
		\begin{array}{ccc}
		\mathbbm{k}[x_{1}, \dotsc, x_{n}]
			&
		\longrightarrow
			&
		\mathbbm{k}[x_{1}, \dotsc, x_{n}]
		\\
			Q
			&
			\longmapsto
			&
			Q^{p}.
		\end{array}
	\]
	Letting \(A := \mathbbm{k}[x_{1}, \dotsc, x_{n}]\), the morphism \(J_{k} X^{(p)} \to J_{k} X\) is now dual to the induced endomorphism of
	\[
	\mathrm{HS}_{A/\mathbbm{k}}
	\cong
	\mathbbm{k}[x][x_{i}^{(j)}].
	\]

	By Lemma~\ref{lem:derivfrob}, this morphism is given by
	\[
		\begin{array}{ll}
		a \longmapsto a^{p}
		&
		(a \in \mathbbm{k})
		\\
		x_{i} \longmapsto x_{i}^{p}
		&
		(1 \leq i \leq n)
		\end{array}
	\]
	and
	\[
		x_{i}^{(l)}
		\longmapsto
		\left\{
			\begin{array}{ll}
				(x_{i}^{(j)})^{p} & \text{if}\quad l = jp \\
				0                 & \text{if}\quad p\;\; \text{does not divide}\;\; l.
			\end{array}
		\right.
	\]
	The schematic image of \(J_{k}X^{(p)} \to J_{k}X\) is the subscheme whose image is cut out by the kernel of this morphism, i.e. by the ideal
	\[
		\mathcal{I}
		=
		\left<
		\left.
		x_{i}^{(j)}
		\;\right|\;
		p\; \text{does not divide}\, j
		\right>.
	\]
	The complement of the image of \((X^{(p)})^{GG}_{k}\) in \(X_{k}^{GG}\) is then given by the image of the open subset of \(J_{k}X\) of points with at least one non-zero coordinate with a weight not divisible by \(p\). By the discussion of Section~\ref{sec:DMlocus}, this is exactly the Deligne-Mumford locus of \(X_{k}^{GG}\).
	\medskip
	
\noindent
{\em Step 2. The general case.} Let \(X\) be a smooth scheme above \(\mathbbm{k}\), and let \(U \subset X\) be an affine subset admitting an \'{e}tale morphism
\[
	\pi : U \to \mathbb{A}^{n}_{\mathbbm{k}}.
\]
The diagram \eqref{eq:Frobeniussquare} with \(X\) replaced by \(U\) can now be obtained by base change from the corresponding diagram for \(\mathbb{A}^{n}_{\mathbbm{k}}\). Since being Deligne-Mumford is a property of stacks that is invariant under \'{e}tale base change, the proposition with \(X\) replaced by \(U\) follows from Step 1 above. We get the result since the property is local over \(X\).
\end{proof}

\begin{lem} \label{lem:derivfrob}
	Write \(x = (x_{1}, \dotsc, x_{n})\). Using the isomorphism \(\mathrm{HS}_{\mathbbm{k}[x]/\mathbbm{k}} \cong \mathbbm{k}[x][x_{i}^{(j)}]\), let \(D_{j} : \mathbbm{k}[x] \to \mathrm{HS}_{\mathbbm{k}[x]/\mathbbm{k}}\) denote the standard derivatives, defined by
	\[
		D_{j} (x_{i}) = x_{i}^{(j)}
		\quad
		(1 \leq i \leq n,  1 \leq j \leq k).
	\]
	Then, for any \(a \in \mathbbm{k}[x]\) and any \(j \in \llbracket 1, k\rrbracket\) with \(jp \leq k\), we have
	\[
		D_{jp}(a^{p}) = (D_{j}(a))^{p}
	\]
	and \(D_{l}(x^{p}) = 0\) if \(p\) does not divide \(l\). In particular, 
	\[
		D_{jp}(x_{i}) = (x_{i}^{(j)})^{p}.
	\]
\end{lem}
\begin{proof}[Proof of Lemma~\ref{lem:derivfrob}]
	Let \(R = \mathbbm{k}[t]/(t^{k+1})\), and let 
	\[
		\begin{array}{cccc}
			\Psi : & \mathbbm{k}[x]/\mathbbm{k} & \longrightarrow & \mathrm{HS}_{\mathbbm{k}[x]/\mathbbm{k}} \otimes_{\mathbbm{k}} R \\
			       &    a & \longmapsto     & D_{0}(a) + D_{1}(a) t + \dotsc + D_{k}(a) t^{k}
		\end{array}
	\]
			be the standard morphism of \(\mathbbm{k}\)-algebras provided by the derivations. Now, we have, for any \(a \in \mathbbm{k}[x]\):
			\begin{align*}
				\Psi(a^{p})
				& =
				\Psi(a)^{p} \\
				& =
				(D_{0}(a) + D_{1}(a)t + \dotsc + D_{k}(a) t^{k})^{p} \\
				& =
				D_{0}(a)^{p} + D_{1}(a)^{p}t^{p} + \dotsc + D_{k}(a)^{p} t^{pk}
			\end{align*}
		Identifying the coefficients, this gives the result.
\end{proof}

\begin{defi}
	We call {\em wild locus} of \(J_{k} X\) (resp.\ of \(X_{k}^{GG}\)) the closed image of \(J_{k} X^{(p)}\) (resp. \((X^{(p)})_{k}^{GG}\)). We denote it by \(J_{k}^{w}X\) (resp.\ \(X_{k}^{GG, w}\)).
\end{defi}

\begin{rem}
	Let \(X\) be a smooth scheme. As explained in the proof of Proposition~\ref{prop:wild}, in local \'{e}tale coordinates \((x_{0}, x_{1}, \dotsc, x_{n})\) on \(X\), the wild locus of \(J_{k}X\) is given by the equations
	\begin{equation} \label{eq:wildlocal}
		x_{i}^{(j)} = 0,
	\end{equation}
	for all \(i, j\) such that \(p\) does not divide \(j\). It might not be entirely obvious why this locus should be independent of the choice of coordinates, but this is easy to check from the Fa\`{a} di Bruno formula \eqref{eq:faadibruno}. Indeed, if \((y_{1}, \dotsc, y_{n})\) is another set of coordinates, then for any \(j \in \mathbb{N}\) not divisible by \(p\), this formula reads
	\[
		y_{i}^{(j)}
		=
		R_{i}^{j}(x^{(1)}, \dotsc, x^{(j)})
	\]
	for some polynomial \(R_{i}^{j}\) of weighted degree \(j\). In a given monomial of \(R_{i}^{j}\), not all derivatives can be of order divisible by \(p\) (otherwise the degree \(j\) would be divisible by \(p\) as well). This shows that a point satisfying the equations \eqref{eq:wildlocal} also satisfies \(y_{i}^{(j)} = 0\).
\end{rem}

\section{Jets and hyperbolicity}

Let us discuss here very briefly about the link between Hasse-Schmidt differentials and algebraic hyperbolicity over an arbitrary field. Again, the result are closely related to what is known in the complex setting.
\medskip

{\em In this section, we assume for simplicity that \(\mathbbm{k}\) is algebraically closed, even though this is not strictly necessary.}

\begin{defi} Let \(X\) be a smooth projective variety, and let \(L\) be an ample line bundle over \(X\). We say that \(X\) is {\em algebraically hyperbolic} if there exists constants \(A, B > 0\) such that for any smooth projective curve \(C\) over \(\mathbbm{k}\), and any non-constant {\em separable} morphism \(f : C \to X\), we have
	\[
		-\chi(C) \geq A \deg_{C} (f^{\ast} L) + B.
	\]
	Similarly, we say that \(X\) is algebraically hyperbolic {\em modulo \(Z \subset X\)} if the inequality holds for any \(f\) as above as soon as \(f(C) \not\subset Z\).
\end{defi}

\begin{rem}
	\begin{enumerate}
		\item This definition is the analogue of Demailly's notion of algebraic hyperbolicity (see \cite[§3]{dem12a}). If the base field is \(\mathbb{C}\), it is implied by Brody hyperbolicity (i.e. the absence of non-constant holomorphic map \(f : \mathbb{C} \to X\)), and conjecturally equivalent to it.
		\item the assumption that \(f\) is {\em separable} in the above definition is required by the existence of the Frobenius morphism: by considering a morphism of the form
			\[
				C' \overset{F^{m}}{\longrightarrow} C \longrightarrow X,
			\]
			we can get \(\deg f^{\ast} L\) arbitrarily large, while keeping \(\chi(C') = \chi(C)\) constant.
	\end{enumerate}
\end{rem}

Exactly as in the complex case, we can show in general that the existence of "sufficiently many" jet differential equations implies the algebraic hyperbolicity of \(X\). Let us give some details.
\medskip

\begin{prop}
	Let \(X\) be projective variety endowed with an ample line bundle \(L\). Let \(p_{k} : X_{k}^{GG} \to X\) be the standard projection, and let
	\[
		\mathbb{B}_{+}(\mathcal{O}_{k}(1)) := \bigcap_{m > 0} \mathrm{Bs}(\mathcal{O}_{k}(m) \otimes p_{k}^{\ast} L^{-1}).
	\]
	be the augmented base locus of \(\mathcal{O}_{k}(1)\). Let \(B \subset \mathbb{B}_{+}(\mathcal{O}_{k}(1))\) be the union of the irreducible components that are not included in the singular substack \(X_{k}^{GG, \mathrm{sing}}\), and let
	\[
		\mathrm{GG}_{k}(X) := p_{k}(B).
	\]
	Then \(X\) is algebraically hyperbolic modulo \(\mathrm{GG}_{k}(X)\).
\end{prop}

\begin{rem}
	The definition of \(\mathrm{GG}_{k}(X)\) above mimics the standard definition of the Green-Griffiths locus in complex hyperbolicity (see e.g. \cite{DR15}). Note that in the standard definition, we usually take the full image \(p_{k}(\mathbb{B}_{+}(\mathcal{O}(1))\). Removing the singular components is completely free, as we will see in the proof.
\end{rem}

We need the following classical lemma.

\begin{lem} \label{lem:nonvanishinghyp}
	Let \(C\) be a smooth projective curve, endowed with a line bundle \(L\). Let \(k, m \geq 1\) be integers, and assume that
	\begin{equation} \label{eq:inequalityhyp}
		h^{0}(C, E_{k, m}^{HS} \Omega_{C} \otimes L^{-1}) \neq 0.
	\end{equation}
	Then
	\[
		- \chi(C) \geq m \deg L
	\]
\end{lem}
\begin{proof}

	Because of the assumption, at least one graded in \(\mathrm{Gr}_{F} (E_{k, m}^{HS} \Omega_{C}) \otimes L^{-1}\) has a global section. This means that there exists \(l_{1}, \dotsc, l_{k} \in \mathbb{N}\) such that \(\sum_{j} j l_{j} = m\) and
	\[
		h^{0}(C, S^{l_{1}}\Omega_{C} \otimes \dotsc \otimes S^{l_{k}} \Omega_{C} \otimes L^{-1}) \neq 0.
	\]
	The symmetric product above identifies with \(K_{C}^{l_{1} + l_{2} + \dotsc + l_{k}}\). Thus, we deduce that
	\[
		(l_{1} + l_{2} + \dotsc + l_{k}) \deg K_{C} - \deg L \geq 0.
	\]
	We have \(\deg K_{C} = - \chi(C)\). On the other hand, we have
	\[
		\sum_{j} l_{j} \leq \sum_{j} j l_{j} \leq m.
	\]
	This gives the result.
\end{proof}

\begin{proof}

\noindent
{\em Step 0. Setup.} By noetherianity of \(X_{k}^{GG}\), there exists a finitely many sections
	\[
		P_{j_{i}} \in H^{0}(X_{k}^{GG}, \mathcal{O}_{k}(m_{i}) \otimes L^{-1}) - \{0\}.
	\]
	for \(i \in \llbracket 1, r \rrbracket\) and \(m_{i} \in \mathbb{N}\), that cut out \(\mathbb{B}_{+}(\mathcal{O}_{k}(1))\). These sections can also be seen as global sections of \(q_{k}^{\ast} L^{-1}\) on the total space of \(q_{k} : J_{k} X \to X\), homogeneous with respect to the actions of \(\mathbb{G}_{\mathbbm{k}, m}\).
	\medskip

\noindent
	{\em Step 1. Proof.} Let \(C\) be a smooth curve, and let \(f : C \to X\) be a separable morphism, such that \(f(C) \not\subset \mathrm{GG}_{k}(X)\). The morphism \(f\) induces a diagram 
	\[
		\begin{tikzcd}
		J_{k} C
			\arrow[r, "\overline{f}"]
			\arrow[d]
		&
		J_{k} X
			\arrow[d, "q_{k}"]
		\\
		C
		\arrow[r, "f"]
		&
		X
		\end{tikzcd}
	\]
	We claim that \(\overline{f}(J_{k}C)\) is not-included in \(J_{k}^{\mathrm{sing}}X\). To see this, let \(C_{0} \subset X\) be the schematic image of \(C\). Since it is not reduced to a point, the restriction morphism \(\Omega_{X}|_{C_{0}} \to \Omega_{C_{0}}\) is non-trivial. By generic smoothness, there exists an open subset \(U \subset C_{0}\) which is smooth. On the other hand, since the morphism \(f : f^{-1}(U) \to U\) is separable, the morphism
	\[
		f^{\ast} \Omega_{U} \to \Omega_{f^{-1}(U)}
	\]
	is non-zero (see e.g. \cite[Lemma~0C1C]{stacks}). All in all, this shows that \(f^{\ast} \Omega_{X} \to \Omega_{C}\) is non-zero, and hence \(\overline{f}(J_{k}C) \not\subset J_{k}^{\mathrm{sing}}X\).
	\medskip

	Now, it follows from this and the hypothesis \(f(C) \not\subset \mathrm{GG}_{k}(X)\), that there exists at least one \(P_{j_{i}}\) which does not vanish on \(\overline{f}(J_{k}C)\). Thus, \(\overline{f}^{\ast} (P_{j_{i}})\) identifies with a non-zero section in
	\[
		H^{0}(C, E_{k, m_{i}}^{HS} \Omega_{C} \otimes f^{\ast} L^{-1}).
	\]
	Lemma~\ref{lem:nonvanishinghyp} implies that
	\[
		- \chi(C) \geq m_{i} \deg_{C} f^{\ast} L.
	\]

	This shows that the inequality \eqref{eq:inequalityhyp} holds, taking \(A = \min \{m_{i}\}_{1 \leq i \leq r}\) and \(B = 0\).
\end{proof}

\section{The Rees deformation} \label{ref:reesdeformation}

In this section, we give some recollections on the deformation of \(X_{k}^{GG}\) to a weighted projectivized bundle (see \cite{cad17} for a more precise description of this construction).
\medskip

{\em For simplicity, we assume again in this section that \(\mathbbm{k}\) is algebraically closed, so that closed points in \(\mathbb{A}_{\mathbbm{k}}^{k}\) can be associated with tuples of elements of \(\mathbbm{k}\).}

\subsection{Deformation of the Hasse-Schmidt algebra} Let \(X\) be a smooth variety of \(\mathbbm{k}\), and let \(k \in \mathbb{N}\) be an integer. The classical Rees construction (cf. \cite[Section 6.5]{Eis95}) permits to define a sheaf of algebras \(\mathcal{E}_{k, \bullet}^{HS}\) on \(X \times \mathbb{A}^{k}_{\mathbbm{k}}\) whose specializations  to \(X \times \{0\}\) (resp. \(X \times \{(\lambda_{1}, \dotsc, \lambda_{k})\}\), with \(\lambda_{i} \neq 0\)) identify with \(\mathrm{Gr}_{F}(E_{k, \bullet}^{HS}\Omega_{X})\) (resp. \(E_{k, \bullet}^{HS}\Omega_{X}\)).
\medskip

Let us recall that this construction prescribes to define the sheaf of \(\mathcal{O}_{X}\)-algebras \(\mathcal{E}_{k, \bullet}^{HS}\) as follows. Write \(\mathbb{A}_{\mathbbm{k}}^{k} = \mathrm{Spec}\, \mathbbm{k}[t_{1}, \dotsc, t_{k}]\). Now, we let \(\mathcal{E}_{k,m}^{HS} \subset E_{k, m}^{HS}\Omega_{X} \otimes_{\mathbbm{k}} \mathbbm{k}[t_{1}, \dotsc, t_{k}]\) be the subsheaf generated by sections of the form
\[
	t_{1}^{\beta_{1}} \dotsc t_{k}^{\beta_{k}} s,
\]
where \(s\) is a section of \(F^{\beta_{1}, \dotsc, \beta_{k}} E_{k, m}^{GG}\Omega_{X}\). The above claim concerning the specializations of this sheaf now follow from a standard computation.
\medskip

We can now perform one of the following three constructions:
\begin{enumerate}[label=(\roman*)]
	\item taking the relative spectrum above \(X \times \mathbb{A}^{k}_{\mathbbm{k}}\) yields a scheme \(\mathcal{J}_{k}X \to X \times \mathbb{A}^{k}_{\mathbbm{k}}\) such that for any \(\lambda = (\lambda_{1}, \dotsc, \lambda_{n})\}\) with \(\lambda_{i} \neq 0\), the specialization to \(X \times \left\{ \lambda \right\}\) identifies with
		\[
			J_{k}X \to X.
		\]
		The specialization to \(X \times \{0\}\) identifies with
		\[
			\mathbf{Spec}_{X}(S^{\bullet}( \mathbf{\Omega}_{X}^{k})).
		\]
		The latter identifies with a sum of \(k\)-copies of \(T_{X}\).
		\smallskip

		Note that there is a natural action of \(\mathbb{G}_{m, \mathbbm{k}}\) on \(\mathcal{J}_{k}X\) induced by the structure of graded algebra on \(\mathcal{E}_{k, m}^{GG}\).
		\medskip

	\item taking the relative \(\mathbf{Proj}\) above \(X \times \mathbb{A}_{\mathbbm{k}}^{k}\) yields a scheme \(\mathrm{P}_{k}X \to X \times \mathbb{A}_{\mathbbm{k}}^{k}\), such that for any \(\lambda = (\lambda_{1}, \dotsc, \lambda_{n})\) with \(\lambda_{i} \neq 0\), the specialization to \(X \times \left\{ \lambda \right\}\) identifies with
		\[
			X_{k}^{GG, \mathrm{coarse}} \to X.
		\]
		The specialization to \(X \times \{0\}\) identifies with
		\[
			\mathbf{Proj}_{X}(S^{\bullet}( \mathbf{\Omega}_{X}^{k})). 
		\]
		The scheme above is the schematic quotient of \(\mathcal{J}_{k}X|_{X \times \{0\}}\) by the natural action of \(\mathbb{G}_{m, \mathbbm{k}}\). In other words, we let \(\mathbb{G}_{m}\) act on
		\[
			T_{X} \oplus \dotsc \oplus T_{X}
		\]
		by giving it weight \(j\) on the \(j\)-th factor, which gives the relative projectivized bundle \(\mathrm{P}(\mathbf{T}_{X}^{k})\).
		\medskip

	\item taking the stacky quotient of \(\mathcal{J}_{k}X \times X \times \mathbb{A}_{\mathbbm{k}}^{k}\) by the action of \(\mathbb{G}_{m, \mathbbm{k}}\) gives a stack \(\mathcal{P}_{k} X \to X \times \mathbb{A}_{\mathbbm{k}}^{k}\) whose specialization to \(X \times \{\lambda\}\) with \(\lambda\) as above identifies with
		\[
			X_{k}^{GG} \to X.
		\]

		The specialization to \(X \times \{0\}\) identifies with the relative projective stack
		\[
			\mathcal{P}(\mathbf{T}^{k}_{X}),
		\]
		where we let \(\mathbb{G}_{m}\) act on each factor \(T_{X}^{(j)}\) with weight \(j\).
\end{enumerate}
\medskip

In other words, we get a commutative diagram
\begin{equation} \label{eq:diagramdeformation} 
	\begin{tikzcd}
		\mathcal{J}_{k} X
			\arrow[r, "q"]
			\arrow[dr, "p"]
		&
		\mathcal{P}_{k} X
			\arrow[d, "c"]
		\\
		&
		\mathrm{P}_{k} X
		\arrow[d] \\
		& 
		X \times \mathbb{A}_{\mathbbm{k}}^{k}.
	\end{tikzcd}
	\quad \quad
	\begin{array}{cl}
	q & \text{stacky quotient} \\
	p & \text{schematic quotient} \\
	c & \text{coarse moduli space}
	\end{array}
\end{equation}

Specializing it to \(X \times \{\lambda\}\) with \(\lambda\) having no zero coordinate, or to \(X \times \{0\}\) gives the following two identifications:
\medskip

\begin{equation}
\begin{tikzcd}
		{J}_{k} X
			\arrow[r, "q"]
			\arrow[dr, "p"]
		&
		X_{k}^{GG}	
			\arrow[d, "c"]
		\\
		&
		X_{k}^{GG, \mathrm{coarse}}
		\arrow[d] \\
		& 
		X \times \{\lambda\}.
\end{tikzcd}
\quad \text{and} \quad
\begin{tikzcd}
		T_{X}^{\oplus k}	
			\arrow[r, "q"]
			\arrow[dr, "p"]
		&
		\mathcal{P}(\mathbf{T}^{k}_{X})
			\arrow[d, "c"]
		\\
		&
		\mathrm{P}(\mathbf{T}^{k}_{X})
		\arrow[d] \\
		& 
		X \times \{0\}.
\end{tikzcd}
\end{equation}

In more informal terms, the left diagram {\em deforms}, or {\em specializes} to the diagram on the right. 
\medskip

The stack \(\mathcal{P}_{k}X\) supports a tautological line bundle \(\mathcal{O}_{\mathcal{P}_{k}X}(1)\) that restricts to the standard tautological line bundles of \(X_{k}^{GG}\) and \(\mathcal{P}(\mathbf{T}_{X}^{k})\). If \(m\) is divisible by \(\mathrm{lcm}(1, 2, \dotsc, k)\), then \(\mathcal{O}_{\mathcal{P}_{k}X}(m)\) is the pullback of line bundle on \(\mathrm{P}_{k}X\), again restricting to the standard line bundles on \(X_{k}^{GG, \mathrm{coarse}}\) and \(\mathrm{P}(\mathbf{T}_{X}^{k})\).
\medskip

Note that the previous deformation is trivial if the scheme \(X\) is such that \(E_{k, \bullet}^{HS}\Omega_{X}\) is isomorphic to its graded bundle: if this happens, we have an isomorphism \(J_{k} X \cong T_{X}^{\oplus k}\). This holds in particular for \(X = \mathbb{A}_{\mathbbm{k}}^{k}\); in general, this is also true for open subsets admitting étale coordinates:

\begin{prop}
	The deformation is locally trivial on \(X\), i.e. \(X\) is covered by open subsets on which the diagram \eqref{eq:diagramdeformation} restricts to the product with \(\mathbb{A}_{\mathbbm{k}}^{k}\) with 
	\[
		\begin{tikzcd}
		{J}_{k} U
			\arrow[r, "q"]
			\arrow[dr, "p"]
		&
		U_{k}^{GG}	
			\arrow[d, "c"]
		\\
		&
		U_{k}^{GG, \mathrm{coarse}}
		\arrow[d] \\
		& 
		U
		\end{tikzcd}.
	\]
	For such open sets \(U\), we have an isomorphism \(J_{k} U \cong T_{U}^{\oplus k}\).
\end{prop}
\begin{proof}
	Let \(U \subset X\) be as in Proposition~\ref{prop:standardopen}, so that we have \(E_{k, \bullet}^{HS} \Omega_{U} \cong \pi^{\ast} E_{k, \bullet}^{HS} \Omega_{\mathbb{A}^{n}_{\mathbbm{k}}}\). However, the Rees deformation for \(E_{k, \bullet}^{HS} \Omega_{\mathbb{A}_{\mathbbm{k}}^{n}}\) is trivial since \(J_{k}\mathbb{A}^{n}_{\mathbbm{k}}\) identifies with \(T_{U}^{\oplus k}\). This gives the result. 
\end{proof}

Note that the morphism \(\mathrm{P}_{k} X \to \mathbb{A}_{\mathbbm{k}}^{k}\) is {\em flat}, as the composition of to locally trivial morphisms \(\mathrm{P}_{k} X \to X \times \mathbb{A}_{\mathbbm{k}}^{k}\) and \(X \times \mathbb{A}_{\mathbbm{k}}^{k}\). As a consequence, we have the Euler characteristic equalities
\[
	\chi(X_{k}^{GG},
	\mathcal{O}_{k}(m)
	)
	=
	\chi(\mathrm{P}(\mathbf{T}^{k}_{X}),
	\mathcal{O}(m)
	),
\]
for \(m\) divisible enough. Note that since the higher direct images of \(\mathcal{O}(m)\) vanish for both projections \(X_{k}^{GG, \mathrm{coarse}} \to X\) and \(P(\mathbf{T}^{k}_{X}) \to X\) (see Proposition~\ref{prop:vanishingrel}), the previous inequality boils down to
\begin{equation} \label{eq:equalityEuler}
	\chi(X, E_{k, m}^{HS} \Omega_{X})
	=
	\chi(X, \mathrm{Gr}_{X}(E_{k, m}^{HS} \Omega_{X}))
	=
	\chi(X, S^{m}\mathbf{\Omega}_{X}^{k})
\end{equation}
which is of course simply a restatement of the fact that the Euler characteristic of a graded vector bundle and of its graded object are equal.

\chapter{Weighted Segre classes and asymptotic Euler characteristics} \label{chap:weighted}

The purpose of this chapter is twofold. First, we will explain how to present Green and Griffiths' classical asymptotic Euler computations \cite{GG80}, using {\em weighted Segre classes}. We will rederive an elementary proof of the Whitney formula~\eqref{eq:whitney}, as an equality between numerical classes. As we said above, the computations will be closely related to the original computations of Green-Griffiths \cite{GG80} (i.e. estimating Euler characteristics by integrals of polynomials on adequate simplexes). This leads to the second purpose of this chapter: its computations will serve to introduce the main method of Chapter~\ref{chap:morseineqweighted}, where we use once more these integrals on simplexes, but this time substituting the use of asymptotic Riemann-Roch with the generalized algebraic Morse inequalities.
\medskip

{\em In this chapter, \(\mathbbm{k}\) is a field of arbitrary characteristic, not necessarily algebraically closed.}

\section{The Whitney formula}

Let \(X\) be a smooth scheme, endowed with a weighted direct sum \(\mathbf{E} := E_{1}^{(a_{1})} \oplus \dotsc \oplus E_{r}^{(a_{r})}\), and let \(p : \mathrm{P}(\mathbf{E}) \to X\) be the associated projectivized scheme (see Section~\ref{sec:weighted}). If \(m_{0} = \gcd(a_{0}, a_{1}, \dotsc, a_{r})\), then \(\mathcal{O}(m_{0})\) can be seen as a line bundle on \(\mathrm{P}(\mathbf{E})\), and we may define the {\em weighted Segre class of \(\mathbf{E}\)} \index{Segre class!weighted} as the following endomorphism of the Chow ring with rational coefficients \(A_{\ast}(X)_{\mathbb{Q}}\):
\medskip

\[
	s_{k}(\mathbf{E}) \cap \alpha
	:=
	\frac{1}{m_{0}^{e + k - 1}}p_{\ast} (c_{1}(\mathcal{O}(m_{0}))^{e + k - 1 } \cap p^{\ast} \alpha)
	\quad
	(\alpha \in A_{\ast}(X)_{\mathbb{Q}})
\]
where \(e = \sum_{j} \mathrm{rk}\, E_{i}\). 

\medskip

We then have the following Whitney formula, that permits to express \(s_{\bullet}(\mathbf{E})\) in terms of the Segre classes of its direct factors.

\begin{prop} [{\cite[Proposition~3.2.11]{cad_thesis}}] \label{prop:whitneyformula}
We have the following equality of endomorphisms of \((A_{\ast} X)_{\mathbb{Q}}\):
\begin{equation} \label{eq:whitney}
	s_{\bullet}(E_{1}^{(a_{1})} \oplus \dotsc \oplus E_{s}^{(a_{s})})
	=
	\frac{\mathrm{gcd}(a_{1}, \dotsc, a_{s})}{a_{1} \dotsc a_{s}}
	\prod_{j} s_{\bullet}(E_{j}^{(a_{j})})
\end{equation}
where
	\(
	s_{\bullet}(E^{(a)})
	=
	\frac{1}{a^{\mathrm{rk} E - 1}} \sum_{l} \frac{s_{l}(E)}{a^{l}}.
	\)
	for a single vector bundle \(E\) and any integer \(a > 0\).
	\medskip
\end{prop}

The proof presented in \cite{cad_thesis} essentially copies Fulton's approach in \cite{ful98}. Our plan in the next sections will not be to reprove this formula directly, but rather to give a numerical version: in other words, we will show that this formula holds at least after quotienting by the numerical equivalence relation. Our idea will be to see the top Segre class as the leading coefficient in the asymptotic expansion of the Euler characteristic of the symmetric products of \(\mathbf{E}^{\ast}\). 
\smallskip

This type of computation is very close in spirit to the ones of \cite{GG80}. We will need several combinatorial lemmas, all gathered in Section~\ref{sec:combinatorial} below.

\section{Proof of the identity as endomorphisms of the numerical Chow group}

We will prove that for any integer \(k \in \llbracket 1, n \rrbracket\) and any cycle \(\alpha \in A_{k} X\):
\begin{equation} \label{eq:whitneynum}	
	\deg s_{k}(E_{1}^{(a_{1})} \oplus \dotsc \oplus E_{s}^{(a_{s})}) \cap \alpha
	=	
	\deg \left( \frac{\mathrm{gcd}(a_{1}, \dotsc, a_{s})}{a_{1} \dotsc a_{s}}
	\prod_{j} s_{\bullet}(E_{j}^{(a_{j})}) \right)_{k} \cap \alpha
\end{equation}
in \(\mathbb{Q}\), where the parenthesis in the right hand side means that we take the part of pure degree \(k\).

\subsection{Reduction steps} \label{sec:reductionsteps} Breaking \(\alpha\) into its irreducible components and using the projection formula, we see that we may assume \(\alpha = [X]\) and \(k = \dim X\), where \(X\) is a variety over \(\mathrm{Spec}\, \mathbbm{k}\). 
\medskip

We may also apply the splitting principle to assume that the \(E_{i}\) are split vector bundles. We can use the following elementary lemma.

\begin{lem} \label{lem:splittingprinciple}
	Let \(X\) be a variety of dimension \(n \in \mathbb{N}\). Let \(E \to X\) be a vector bundle. Then there exists a variety \(X'\) of the same dimension, and a generically finite dominant, projective morphism \(q : X' \to X\) such that \(q^{\ast} E\) admits a filtration by subbundles
	\[
		0 = E_{0} 
		\subsetneq 
		E_{1} 
		\subsetneq 
		\dotsc
		\subsetneq 
		E_{r} = q^{\ast} E
	\]
	where each graded term \(\quotientd{E_{i}}{E_{i-1}}\) is a line bundle.
\end{lem}
\begin{proof}
	First, we claim that there exits a proper dominant morphism \(u : Y \to X\) such that \(u^{\ast} E\) admits a filtration as above. To see this, we apply Fulton's classical construction (see \cite[p.\ 51]{ful98}). We let \(u_{1} : Y_{1} := \mathrm{P}(E) \to X\), and \(F_{1} := \mathcal{O}(-1)\) be the tautological line bundle. Now, apply inductively the construction to the quotient \(Q_{1}\) in the exact sequence
	\[
		0
		\longrightarrow
		F_{1}
		\longrightarrow
		u_{1}^{\ast} E
		\longrightarrow
		Q_{1}
		\longrightarrow
		0
	\]
	This gives a morphism \(v : Y \to Y_{1}\) on which \(v^{\ast} Q_{1}\) admits a total filtration. We then get the claim letting \(E_{1} := v^{\ast} F_{1}\). To get the final result, let \(X' \subset Y\) be any subvariety of dimension \(n\) whose projection to \(X\) is dominant: the filtration exists on \(X'\) by restriction.
\end{proof}

Using the previous lemma, we can assume that there exists a morphism \(q : X' \to X\) such that all the vector bundles \(p^{\ast} E_{i}\) admit such filtrations. We then have a diagram
\[
	\begin{tikzcd}
		\mathrm{P}_{X'}(q^{\ast}\mathbf{E})
			\arrow[r, "q'"]
			\arrow[dr, phantom, "\square"]
			\arrow[d, "p'"]
		&
		\mathrm{P}_{X}(\mathbf{E})
			\arrow[d, "p"]
		\\
		X'
			\arrow[r, "q"]
		&
		X
	\end{tikzcd}.
\]

The projection formula shows that the Whitney formula will hold for \(\mathbf{E}\) if it holds for \(q^{\ast} \mathbf{E}\). We may thus assume that all the \(E_{i}\) admit total filtrations.
\medskip

We may moreover assume that the filtrations are split. To see this, denote by \(\overline{E}_{i}\) the graded vector bundle associated to \(E_{i}\) for \(i \in \llbracket 1, r \rrbracket\), and let \(\overline{\mathbf{E}} := \overline{E}_{1}^{(a_{1})} \oplus \dotsc \oplus \overline{E}_{r}^{(a_{r})}\). Then, we have an induced filtration on the symmetric algebra \(\mathrm{Sym}^{\bullet}(\mathbf{E})\) whose graded object is \(\mathrm{Sym}^{\bullet} (\overline{\mathbf{E}})\). Applying the Rees deformation, we get a diagram
\[
	\begin{tikzcd}
		\mathrm{P}(\mathbf{E})
			\arrow[r, hook]
			\arrow[d]
		& 
		\mathcal{P}
			\arrow[d, "q"]
		& 
		\mathrm{P}(\overline{\mathbf{E}})
			\arrow[l, hook]
			\arrow[d]
		\\
		X \times \{1\} 
			\arrow[r, hook]
			\arrow[d]
		&
		X \times \mathbb{A}^{1}	
			\arrow[d]
		&
		X \times \{0\}
			\arrow[l, hook]
			\arrow[d]
		\\
		\{1\}
			\arrow[r, hook]
		&
		\mathbb{A}^{1}
		&
		\{0\}
			\arrow[l, hook]
	\end{tikzcd}
\]
as well as a line bundle \(\mathcal{O}_{\mathcal{P}}(m_{0})\) on \(\mathcal{P}\) restricting to the corresponding powers of the respective tautological bundles on \(\mathrm{P}(\mathbf{E})\) and \(\mathrm{P}(\overline{\mathbf{E}})\). The projection \(\mathcal{P} \to \mathbb{A}^{1}\) being flat, the cycles \([\mathrm{P}(\mathbf{E})]\) and \([\mathrm{P}(\mathbf{\overline{E}})]\) are rationally equivalent in \(\mathcal{P}\) (see \cite[Theorem~1.7]{ful98}). Thus, we have
\[
	\deg c_{1}(\mathcal{O}_{\mathcal{P}}(m_{0}))^{n+e-1} \cap [\mathrm{P}(\mathbf{E})] 
	= 
	\deg c_{1}(\mathcal{O}_{\mathcal{P}}(m_{0}))^{n+e-1} \cap [\mathrm{P}(\overline{\mathbf{E})}],
\]
and thus the left hand sides of \eqref{eq:whitneynum} coincide for \(\mathbf{E}\) and \(\overline{\mathbf{E}}\). Since the right hand sides also coincide by the usual Whitney formula, this shows that the formula will hold if we manage to prove it for \(\overline{\mathbf{E}}\). Thus, we may reduce to the case where each \(E_{i}\) is a direct sum of line bundles.

\subsection{Proof of the formula in the split case.} Let us then assume that each \(E_{i} = L_{i, 1} \oplus \dotsc \oplus L_{i, r_{i}}\) is split as a direct sum of line bundles for all \(1 \leq i \leq s\), and let \(\alpha_{i, j} = c_{1} (L_{i, j})\) denote the corresponding Chern roots. We will rather prove the dual formula for \(s_{n}(\mathbf{E}^{\ast})\) to make the signs easier to track.

	Let \(m_{0} := \gcd(a_{1}, \dotsc, a_{r})\). Then, applying the asymptotic Riemann-Roch theorem (see Section~\ref{sec:asymptoticRR}), one finds
	\begin{align*}
		\int_{X} s_{n}(\mathbf{E}^{\ast}) & = \frac{1}{m_{0}^{n+r-1}}\int_{\mathbb{P}(\mathbf{E})} c_{1} \mathcal{O}(m_{0})^{n+r-1} \quad \text{(by definition of } s_{n}(\mathbf{E}^{\ast}))\\
					& = \lim_{\stackrel{m \longrightarrow +\infty}{m_{0} | m}}
					\frac{\chi(\mathbb{P}(\mathbf{E}), \mathcal{O}(m))}{m^{n+r-1}/(n+r-1)!}
	\end{align*}
	The higher direct images \(R^{j} p_{\ast} \mathcal{O}(m)\) all vanish if \(j > 0\) by Proposition~\ref{prop:vanishingrel}, so the Leray spectral sequence gives
	\[
		\chi(\mathbb{P}(\mathbf{E}), \mathcal{O}(m))) = \chi(X, \pi_{\ast} \mathcal{O}(m)) = \chi(X, S^{m}\mathbf{E}).
	\]
	But then, since each \(E_{i}\) is split, Proposition~\ref{prop:symmetricsum} below implies that
	\[
		\int_{X} s_{n}(\mathbf{E}^{\ast}) =
		\frac{m_{0}}{a_{1}^{r_{1}} \dotsc a_{s}^{r_{s}}}
		\int_{X} \left( \prod_{i=1}^{s} \prod_{j=1}^{r_{i}} \sum_{p=0}^{n} \left(\frac{\alpha_{i, j}}{a_{i}}\right)^{p} \right)_{n}.
	\]
	Now, we see that for all \(i = 1, \dotsc, s\), the expression  \( \frac{1}{a_{i}^{r_{i} - 1}} \prod_{j=1}^{r_{i}} \sum_{p=0}^{n} \left(\frac{\alpha_{i, j}}{a_{i}}\right)^{p}\) identifies with the formula for \(s_{\bullet}((E_{i}^{\ast})^{(a_{i})})\) that is given in the statement of the proposition. Thus, one obtains
	\[
		\int_{X} s_{n}(\mathbf{E}^{\ast})
		= \frac{m_{0}}{a_{1} \dotsc a_{r}}
		\int_{X}
		\left( \prod_{i=1}^{s} s_{\bullet}((E_{i}^{\ast})^{(a_{i})}) \right)_{n},
	\]
	which gives the result.

\section{Simplexes and asymptotics of Euler characteristics} \label{sec:combinatorial}

The purpose of the following discussion is to prove Proposition~\ref{prop:symmetricsum}, that was used in the proof of the Whitney formula.

\begin{prop} \label{prop:symmetricsum}
	Let \(X\) be a variety of dimension \(n\). Let \(L_{1}, \dotsc, L_{r}\) be line bundles on \(X\), and fix integers \(a_{1}, \dotsc, a_{r} \in \mathbb{N}\). Then, one has the asymptotic expansion, as \(m\) goes to \(+\infty\) while being divisible by \(\gcd(a_{1}, \dotsc, a_{r})\):
\[
\chi(X, S^{m}(L_{1}^{(a_{1})} \oplus \dotsc \oplus L_{r}^{(a_{r})}))
	 =
	\frac{\gcd(a_{1}, \dotsc, a_{r})}{a_{1} \dotsc a_{r}}
	\int_{X}
	\left(
		\prod_{j = 1}^{r} \sum_{p=0}^{n} \frac{c_{1}(L_{j})^{p}}{a_{j}^{p}}
	\right)_{n}
	\frac{m^{n+r-1}}{(n+r-1)!} + o(m^{n+r-1}).
\]
	where \((\cdot)_{n}\) means we take the part of pure degree \(n\) of the class between the brackets.
\end{prop}

We will proceed quite similarly to \cite{GG80} and interpret the leading term of the left hand side as an integral of some polynomial on a given simplex in \(\mathbb{R}^{r}\). We need to introduce some notation.

\subsection{Notation} \label{sec:notationlattices} We introduce the following conventions.

\begin{enumerate}
	\setlength{\itemsep}{0.2em}
	\item For all \(n\), we let \(\mathrm{vol}_{n}\) denote the \(n\)-dimensional euclidean volume measure.
	\item Let \(m \in \mathbb N\). A \emph{\(m\)-dimensional simplex} \(\Delta\) is a metric space isomorphic to the convex envelop of \(m+1\) points in \(\mathbb{R}^{m}\), such that any \(p\) of them generate an affine \((p-1)\)-space. We will sometimes write \(\Delta \; \accentset{\circ}{\subset} \; \mathbb{R}^{m}\) to emphasize the fact that \(\Delta\) has non-empty interior in \(\mathbb{R}^{m}\) (or equivalently, that \(\dim \Delta = m\)) and to oppose this situation to the case of a \((m-1)\)-dimensional simplex included in \(\mathbb{R}^{m}\).

	\item For any \(m\)-dimensional simplex \(\Delta \accentset{\circ}{\subset} \mathbb{R}^{m}\), the \emph{uniform probability measure} of \(\Delta\) is the measure \(d \mathbf P_{\Delta} = \frac{1}{\mathrm{vol}_m(\Delta)} d \mathrm{vol}_{m}\).

Since this measure on $\Delta$ is the unique probability measure which is the restriction of a translation invariant measure on $\mathbb R^m$, we see that if $\Delta_1, \Delta_2 \subseteq \mathbb R^m$ are $m$-dimensional simplexes, and if $\Psi \in \mathrm{GL}(\mathbb R^m)$ is such that $\Delta_2 = \Psi(\Delta_1)$, then $\Psi$ sends the uniform measure of $\Delta_1$ on the uniform measure of $\Delta_2$.

\end{enumerate}

\subsection{Lattices and volumes of fundamental domains}

Let $a_1, ..., a_r \in \mathbb N$. Let
\[
	H = \{ (t_1, ..., t_r) \in \mathbb Z^r \; | \; \sum_i a_i t_i = 0\}.
\]
Let \(H_{\mathbb{R}} := H \otimes \mathbb{R}\) be the hyperplane in \(\mathbb{R}^{r}\) containing \(H\).
\medskip

Then $H \subseteq \mathbb Z^r$ is a \emph{primitive} sublattice, meaning that $\quotientd{\mathbb Z^r}{H}$ is torsion-free. Hence, by the adapted basis theorem, there exists a basis $(f_1, ..., f_r)$ of $\mathbb Z^r$ such that $(f_1, ..., f_{r-1})$ is in turn a basis for $H$. Let 
\[
	C_{H} = \sum_{1 \leq i \leq r-1} [0, 1] \cdot f_{i}
\]
denote the associated fundamental domain of \(H_{\mathbb{R}}\) for the action of \(H\). Note that all such fundamental domains are image of one another by an element of \(\mathrm{SL}(H)\) so they all have the same \((n-1)\)-volume.

\begin{lem} \label{lemlattice1} Any fundamental domain for $H$ has volume 
	\[
		\mathrm{vol}_{r-1} \left( C_H \right) 
		= 
		\frac{\sqrt{ \sum_{1 \leq i \leq r } a_i^2}}{\gcd(a_1, ..., a_r)}.
	\]
\end{lem}
\begin{proof}
	The lattice $H$ and the proposed formula for the volume do not change if we replace $a_i$ by $\frac{a_i}{\gcd(a_1, ..., a_r)}$, hence we can suppose that $\gcd(a_1, ..., a_r) = 1$. In this case, there exist $u_1, ..., u_r \in \mathbb Z$ such that $\sum_i a_i u_i = 1$. Replacing \(f_{r}\) by the vector \((u_{1}, \dotsc, u_{r})\) in the basis \((f_{1}, \dotsc, f_{r})\) does not change the fact that this is a basis of \(\mathbb{Z}^{r}\), so we can assume that $f_r = (u_1, ..., u_r)$.
	\medskip

Since $(f_1, ..., f_{r})$ is a basis of $\mathbb Z^r$, we have $\mathrm{vol}_r(\sum_{1 \leq i \leq r} [0, 1] \cdot f_i) = 1$. Moreover,
\begin{align*}
	\mathrm{vol}_r(\sum_{1 \leq i \leq r} [0, 1] \cdot f_{i}) 
	& 
	= \mathrm{vol}_{r-1}(\sum_{1 \leq i \leq r - 1} [0, 1] \cdot f_{i}) \, \cdot \, \norm{\pi_{H^\perp} (f_{r}) } \\
	& = \mathrm{vol}_{r-1}(C_{H}) \, \cdot\, \norm{\pi_{H^\perp} (f_{r}) }
\end{align*}
	where $\pi_{H^\perp} (f_r)$ is the orthogonal projection of $f_r$ on $H^\perp$, and $\norm{ \cdot }$ is the euclidean norm. We obtain
	\[
		\mathrm{vol}_{r-1}(C_{H}) = \frac{1}{\norm{\pi_{H^{\perp}}(f_{r})}}.
	\]

	Let us now compute \(\norm{\pi_{H^{\perp}}(f_{r})}\). Since $H^\perp = \mathbb R \cdot (a_1, \dotsc, a_r)$ by definition of $H$, one can write
	\[
		f_{r} 
		= 
		(u_{1}, \dotsc, u_{r}) 
		= 
		\underbrace{\lambda \cdot (a_{1}, \dotsc, a_{r})}_{=\, \pi_{H^{\perp}}(f_{r})} 
		\, + \, (b_{1}, \dotsc, b_{r}),
	\]
	with \(\lambda \in \mathbb{R}\) and \((b_{1}, \dotsc, b_{r}) \in H_{\mathbb{R}}\). Thus, one has
	\[
		0 = \sum_{j} a_{j}\, b_{j} = \sum_{j} a_{j}\, u_{j} - \lambda \sum_{j} a_{j}^{2}.
	\]
	 Since \(\sum_{j} a_{j} u_{j} = 1\), this gives \(\lambda = \frac{1}{\sum_{j} a_{j}^{2}}\) and thus \(\norm{\pi_{H^\perp} (f_r) }^{2} = \lambda^{2} \sum_{j} a_{j}^{2} = \frac{1}{\sum_{j} a_{j}^{2}}\). We obtain the result.
\end{proof}

\begin{lem} \label{lem:lattice2} Let \(\underline{a} = (a_1, ..., a_r) \in \mathbb{N}^{r}\). We introduce the simplex
	\[
		\Delta_{\underline{a}} = \{ (t_i) \in \mathbb R_+^r \; | \; \sum_i a_i t_i = 1 \}.
	\]
	Then the volume of \(\Delta_{\underline{a}}\) is 
	\[
		\mathrm{vol}_{r-1} (\Delta_{\underline{a}}) = 
		\frac{1}{(r-1)!} 
		\frac{\gcd(a_{1}, \dotsc, a_{r})}{a_{1} \dotsc a_{r}} \mathrm{vol}_{r-1}(C_H).
	\]
\end{lem}
\begin{proof}
	By Lemma \ref{lemlattice1}, it suffices to show that 
	\[
	\mathrm{vol}_{r-1}(\Delta_{\underline{a}}) 
	= 
	\frac{1}{(r-1)!} 
	\frac{\sqrt{\sum_{1 \leq i \leq r} a_i^2}}{a_1 \dotsc a_r}.
	\]

	This follows from Lemma~\ref{lem:volsimplex}.
	
	\end{proof}

\subsection{Integrating monomials on simplexes} We are now going to compute the integral of some monomial functions on simplexes with respect to the uniform probability measure. The goal is to prove the following.

\begin{lem} \label{lem:integral} Let \(\underline{a} = (a_{1}, \dotsc, a_{r}) \in \mathbb{N}^{r}\), and let \(\Delta_{\underline{a}} = \{ (t_{i}) \in \mathbb{R}_{+}^{r} \; | \; \sum_{i} a_{i} t_{i} = 1 \}\). Let \(p_{1}, \dotsc, p_{r} \in \mathbb{N}\). Then
	\[
		\int_{\Delta_{\underline{a}}} t_{1}^{p_{1}} \dotsc t_{r}^{p_{r}}
		d \mathbf{P}_{\Delta_{\underline{a}}}(t)
		=
		\frac{(r-1)!\, p_{1}! \dotsc p_{r}!}{(p_{1} + p_{2} + \dotsc + p_{r} + r - 1)!}
		\frac{1}{a_{1}^{p_{1}} \dotsc a_{r}^{p_{r}}}.
	\]
\end{lem}

First, let us remark that we can easily get back to the case where \(\underline{a}\) is equal to \(\underline{1} := (1, \dotsc, 1)\). Indeed, letting \(\Psi(t_{1}, \dotsc, t_{r}) = (a_{1} t_{1}, \dotsc, a_{r} t_{r})\), one has \(\Psi_{\ast} d\mathbf{P}_{\Delta_{\underline{1}}} = d \mathbf{P}_{\Delta_{\underline{a}}}\), so that
\begin{align*}
\int_{\Delta_{\underline{a}}} t_{1}^{p_{1}} \dotsc t_{r}^{p_{r}}
		d \mathbf{P}_{\Delta_{\underline{a}}}(t)
	& =
	\int_{\Delta_{\underline{1}}}
	\left(\frac{t_{1}}{a_{1}}\right)^{a_{1}}
	\dotsc
	\left(\frac{t_{r}}{a_{r}}\right)^{a_{r}}
	d \mathbf{P}_{\Delta_{\underline{1}}}(t) \\
	& =
	\frac{1}{a_{1}^{p_{1}} \dotsc a_{r}^{p_{r}}} \int_{\Delta_{\underline{1}}} t_{1}^{p_{1}} \dotsc t_{r}^{p_{r}} d \mathbf{P}_{\Delta_{\underline{1}}}(t)
\end{align*}

For any \(r \in \mathbb{N}\) and \(p_{1}, \dotsc, p_{r} \in \mathbb{N}\), let 
\[
	C_{p_{1}, \dotsc, p_{r}} 
	:= 
	\int_{\Delta_{\underline{1}}} 
	t_{1}^{p_{1}} \dotsc t_{r}^{p_{r}} 
	d \mathbf{P}_{\Delta_{\underline{1}}}(t).
\]
By the remark above, the proof of Lemma~\ref{lem:integral} will then be complete with the following result.
\begin{lem} \label{lem:computationint} We have the equality 
	\[
	C_{p_{1}, \dotsc, p_{r}} =
	\frac{(r-1)!\, p_{1}! \dotsc p_{r}!}{(p_{1} + p_{2} + \dotsc + p_{r} + r - 1)!}.
	\]
\end{lem}
\begin{proof}
	Letting \(\Delta_{r} := \Delta_{(1, \dotsc, 1)}\) to simplify the notation (\(1\) is repeated \(r\) times), one has
	\begin{align*}
		C_{p_{1}, \dotsc, p_{r}} & =
		\int_{t_{1} + \dotsc + t_{r} = 1}
		t_{1}^{p_{1}} \dotsc t_{r}^{p_{r}}
		\frac{dt_{1} \wedge \dotsc \wedge dt_{r-1}}{\mathrm{vol}_{dt_{1} \wedge \dotsc \wedge dt_{r-1}}(\Delta_{r})} \\
		& = \int_{t_{1} + s = 1} dt_{1}\, t_{1}^{p_{1}}\,
		\frac{\mathrm{vol}_{dt_{2} \wedge \dotsc \wedge dt_{r-1}}(s \Delta_{r-1})}{\mathrm{vol}_{dt_{1} \wedge \dotsc \wedge dt_{r-1}}(\Delta_{r})}
		\int_{t_{2} + \dotsc + t_{r} = s} t_{2}^{p_{2}} \dotsc t_{r}^{p_{r}}
		\frac{dt_{2} \wedge \dotsc \wedge dt_{r-1}}{\mathrm{vol}_{dt_{2} \wedge \dotsc \wedge dt_{r-1}}(s \Delta_{r-1})} \\
		& = \int_{t_{1} + s = 1} dt_{1}\, t_{1}^{p_{1}}\,
		\frac{s^{r-2}/(r-2)!}{1/(r-1)!}
		\cdot
		s^{p_{2} + \dotsc + p_{r}} \int_{t_{2} + \dotsc + t_{r} = 1} t_{2}^{p_{2}} \dotsc t_{r}^{p_{r}}
		\frac{dt_{2} \wedge \dotsc \wedge dt_{r-1}}{\mathrm{vol}_{dt_{2} \wedge \dotsc \wedge dt_{r-1}}(\Delta_{r-1})} \\
		& = (r-1) \int_{t_{1} + s = 1} dt_{1} \, t_{1}^{p_{1}}\, s^{r - 2 + p_{2} + \dotsc + p_{r}} C_{p_{2}, \dotsc, p_{r}}
		= (r-1) \frac{p_{1}!(p_{2} + \dotsc + p_{r} + r-2)!}{(p_{1} + p_{2} + \dotsc + p_{r} + r - 1)!} C_{p_{2}, \dotsc, p_{r}}
	\end{align*}
	where at the last line, we used Lemma~\ref{lem:Bfunction} below. This permits to prove the formula by induction on \(r\).
\end{proof}

\begin{lem} \label{lem:Bfunction}
	Let \(a, b \in \mathbb{N}\). One has
	\[
		\int_{0}^{1} t^{a} (1 - t)^{b} dt = \frac{a! b!}{(a + b + 1)!}.
	\]
\end{lem}
\begin{proof} This comes from the Beta function identity \(\int_{0}^{1} t^{a} (1 - t)^{b} dt = \frac{\Gamma(a+1) \Gamma(b+1)}{\Gamma(a+b+2)}\).
\end{proof}

\subsection{Riemann integrals and asymptotic estimates}

Using the previous results, we can now give the following asymptotic estimates. 

\begin{lem} \label{lem:asymptoticsriemann} Fix integers \(a_{1}, \dotsc, a_{r} \in \mathbb{N}_{>0}\), and \(p_{1}, \dotsc, p_{r} \in \mathbb{N}\). We have, for \(m \longrightarrow + \infty\) divisible by \(\gcd(a_{1}, \dotsc, a_{r})\):
	\[
		\sum_{a_{1} l_{1} +  \dotsc + a_{r} l_{r} = m}
		\frac{l_{1}^{p_{1}}}{p_{1}!} \dotsc \frac{l_{r}^{p_{r}}}{p_{r}!} 
		\quad = \quad
		\frac{\gcd(a_{1}, \dotsc, a_{r})}{a_{1}^{p_{1} + 1} \dotsc a_{r}^{p_{r} + 1}}
		\frac{m^{p_{1} + \dotsc + p_{r} + r - 1}}{(p_{1} + p_{2} + \dotsc + p_{r} + r - 1)!}
		+
		o(m^{p_{1} + \dotsc + p_{r} + r - 1}).
	\]
\end{lem}
\begin{proof} Let \(H_{m}\) be the set of \((l_{1}, \dotsc, l_{r}) \in \mathbb{Z}^{r}\) such that \(\sum_{j} l_{j} a_{j} = m\). It is non-empty if and only if \(\gcd(a_{1}, \dotsc, a_{r}) | m\), and if such is the case, it is then a translate of the lattice 
	\[
		H 
		= 
		\big\{ 
		(l_{1}, \dotsc, l_{r}) \in \mathbb{Z}^{r} 
		\; | \; 
		\sum_{j} a_{j} l_{j} = 0
		\big\}.
	\] 
	Let \(C_{H}\) denote a fundamental domain for \(H\).
	\medskip

	As \((l_{1}, \dotsc, l_{r})\) varies in \(H_{m} \cap \mathbb{N}^{r}\), the element \((\frac{l_{1}}{m}, \dotsc, \frac{l_{r}}{m})\) varies in \(\Delta_{\underline{a}}\), running in a lattice with cells isometric to \(\frac{1}{m}C_{H}\). Thus, one can use a Riemann sum to obtain
	\begin{align*}
		\mathrm{vol}_{r-1}(\frac{1}{m} C_{H}) \cdot
		\sum_{a_{1} l_{1} +  \dotsc + a_{r} l_{r} = m}
		\left(\frac{l_{1}}{m}\right)^{p_{1}}
		\dotsc
		\left(\frac{l_{r}}{m}\right)^{p_{r}}
		& \underset{m \longrightarrow + \infty}{\longrightarrow}
		\int_{\Delta_{\underline{a}}} t_{1}^{p_{1}} \dotsc t_{r}^{p_{r}} d \mathrm{vol}_{r-1}(t).
		\\
		& = \mathrm{vol}_{r-1}(\Delta_{\underline{a}})
		\int_{\Delta_{\underline{a}}} t_{1}^{p_{1}} \dotsc t_{r}^{p_{r}} d \mathbf{P}_{\Delta_{\underline{a}}}(t).
	\end{align*}
	Thus we deduce
	\[
		\frac{1}{m^{p_{1} + \dotsc + p_{r} + r - 1}}
		\sum_{a_{1} l_{1} +  \dotsc + a_{r} l_{r} = m}
		l_{1}^{p_{1}} \dotsc l_{r}^{p_{r}}
		\longrightarrow
		\frac{\mathrm{vol}_{r-1}(\Delta_{\underline{a}})}{\mathrm{vol}_{r-1}(C_{H})}
		\int_{\Delta_{\underline{a}}} t_{1}^{p_{1}} \dotsc t_{r}^{p_{r}} d \mathbf{P}_{\Delta_{\underline{a}}}(t).
	\]
	The right hand side can be computed using Lemmas~\ref{lem:lattice2} and \ref{lem:integral}. This gives the result.
\end{proof}

We will need another version of that lemma for our application to the asymptotic Riemann-Roch theorem.

\begin{lem} \label{lem:asymptfraction} Let \(n, r \in \mathbb{N}\) be two integers. Let \(\alpha_{1}, \dotsc, \alpha_{r}\) be indeterminates over \(\mathbb{Q}\). Fix integers \(a_{1}, \dotsc, a_{r} \in \mathbb{N}\), and \(p_{1}, \dotsc, p_{r} \in \mathbb{N}\). We have, for \(m \longrightarrow + \infty\) divisible by \(\gcd(a_{1}, \dotsc, a_{r})\):
	\begin{align*}
		\sum_{a_{1} l_{1} + \dotsc + a_{r} l_{r} = m}
		\frac{(\alpha_{1} l_{1} + \dotsc + \alpha_{r} l_{r})^{n}}{n!}
		& =
		\frac{\gcd(a_{1}, \dotsc, a_{r})}{a_{1} \dotsc a_{r}}
		\left[
			\sum_{p_{1} + \dotsc + p_{r} = n} \left(\frac{\alpha_{1}}{a_{1}}\right)^{p_{1}} \dotsc \left(\frac{\alpha_{1}}{a_{r}}\right)^{p_{r}}
		\right]
		\frac{m^{n + r -1}}{(n+ r - 1)!} \\
		& \quad
		+ o(m^{n+r-1}).
	\end{align*}
	where \(o(m^{n+r-1})\) means a homogeneous polynomial of degree \(n\) in \(\alpha_{1}, \dotsc, \alpha_{n}\), all of whose coefficients are negligible compared to \(m^{n+ r -1}\).
\end{lem}
\begin{proof}
	We expand the sum, using the Newton identity, and we apply Lemma~\ref{lem:asymptoticsriemann}:
	\begin{align*}
		\sum_{a_{1} l_{1} + \dotsc + a_{r} l_{r} = m}
		\frac{(\alpha_{1} l_{1} + \dotsc + \alpha_{r} l_{r})^{n}}{n!} & =
		\sum_{a_{1} l_{1} + \dotsc + a_{r} l_{r} = m}
		\;
		\sum_{p_{1} + \dotsc + p_{r} = n} \binom{n}{p_{1}, \dotsc, p_{r}}
		l_{1}^{p_{1}} \dotsc l_{r}^{p_{r}} \alpha_{1}^{p_{1}} \dotsc \alpha_{r}^{p_{r}} \\
		& =
		\frac{\gcd(a_{1}, \dotsc, a_{r})}{a_{1} \dotsc a_{r}}
		\left[
			\sum_{p_{1} + \dotsc + p_{r} = n} \left(\frac{\alpha_{1}}{a_{1}}\right)^{p_{1}} \dotsc \left(\frac{\alpha_{1}}{a_{r}}\right)^{p_{r}} \right]
		\frac{m^{n + r -1}}{(n+ r - 1)!}
		+ o(m^{n+r-1}).
	\end{align*}
\end{proof}

\subsection{Asymptotics of Euler characteristics} \label{sec:asymptoticEuler}

We are now ready to prove Proposition~\ref{prop:symmetricsum}, using the asymptotic Riemann-Roch theorem (see Section~\ref{sec:asymptoticRR}).
\medskip

\begin{proof}[Proof of Proposition~\ref{prop:symmetricsum}]
	Let \(\alpha_{i} = c_{1}(L_{i})\) for all \(1 \leq i \leq r\). One has then, using the asymptotic Riemann-Roch theorem:
\begin{align*}
	\chi(X, S^{m}(L_{1}^{(a_{1})} \oplus \dotsc \oplus L_{r}^{(a_{r})}))
	& = \chi(X, \bigoplus_{a_{1} l_{1} + \dotsc + a_{r} l_{r} = m}
	L_{1}^{\otimes l_{1}} \otimes \dotsc \otimes L_{r}^{\otimes l_{r}}) \\
	& = \sum_{a_{1} l_{1} + \dotsc + a_{r} l_{r} = m}
	\chi(X, L_{1}^{\otimes l_{1}} \otimes \dotsc \otimes L_{r}^{\otimes l_{r}}) \\
	& = \sum_{a_{1} l_{1} + \dotsc + a_{r} l_{r} = m}
	\left[
	    \int_{X} \frac{(\alpha_{1} l_{1} + \dotsc + \alpha_{r} l_{r})^{n}}{n!}
	    +
	    \sum_{j=1}^{n} \int_{X} \beta_{j} \cdot (\alpha_{1} l_{1} + \dotsc + \alpha_{r} l_{r})^{n-j}
	\right].
\end{align*}
where for all \(j = 1, \dotsc, n\), the symbol \(\beta_{j}\) denotes a cycle class class depending only on \(X\), but not on \(m\). One can now apply Lemma~\ref{lem:asymptfraction} to obtain the result.
\end{proof}

\section{An example of computation}

To illustrate the previous computations, let us retrieve Green and Griffiths' expression for the leading term 
\[
	\lim_{m \to \infty} \frac{\chi(X, E_{k, m}^{HS}\Omega_{X})}{m^{2}/2!}
\]
for a smooth surface \(X\), with \(k = 2, 3\) (see \cite[p.\ 52]{GG80}).
\bigskip

By \eqref{eq:equalityEuler}, we have
\[
	\chi(X, E_{k, m}^{HS} \Omega_{X})
	=
	\chi(X, S^{m}(\mathbf{\Omega}^{k}_{X})),
\]
so we have actually
\[
	\lim_{m \to \infty} \frac{\chi(X, E_{k, m}^{HS}\Omega_{X})}{m^{2}/2!} =
	s_{2}(\mathbf{T}_{k}).
\]

This term can now be computed as follows.
\bigskip

\subsection{Segre classes computations} In the following, we write \(c_{i} := c_{i}(T_{X})\). Recall that the total Segre class of \(X\) then reads 
\[
	s_{t}(T_{X}) = 1 - c_{1} t + (c_{1}^{2} - c_{2}) t^{2}.
\]

For \(k = 2\), we find:
\begin{align*}
	s_{2}(\mathbf{T}_{2}) \cap [X] & = s_{2} (T_{X}^{(1)} \oplus T_{X}^{(2)}) \cap [X] \\
			      & = \frac{1}{2} \left( \left[1 - c_{1} + (c_{1}^{2} - c_{2})\right] \cdot \left[1 - \frac{c_{1}}{2} + \frac{c_{1}^{2} - c_{2}}{4} \right]\right) \cap [X] \\
			      & = \frac{7 c_{1}^{2} - 5 c_{2}}{8} \cap [X]
\end{align*}

For \(k = 3\), this gives:
\begin{align*}
	s_{2}(\mathbf{T}_{3}) \cap [X] & = s_{2} (T_{X}^{(1)} \oplus T_{X}^{(2)} \oplus T_{X}^{(3)}) \cap [X] \\
			      & = \frac{1}{6} \left( \left[1 - c_{1} + (c_{1}^{2} - c_{2})\right] 
			      \cdot 
			      \left[1 - \frac{c_{1}}{2} + \frac{c_{1}^{2} - c_{2}}{4} \right] 
			      \cdot 
			      \left[1 - \frac{c_{1}}{3} + \frac{c_{1}^{2} - c_{2}}{9} \right] \right) \cap [X] \\
			      & = \frac{85c_{1}^{2} - 49 c_{2}}{216} \cap [X]
\end{align*}

\begin{rem}
We can push further this computation at any desired order. Expanding the formula 
\[
	s_{2}(\mathbf{T}_{k})
	=
	\frac{1}{k!}\prod_{1 \leq i \leq k}\left[ 1 - \frac{c_{1}}{i} + \frac{c_{1}^{2} - c_{2}}{i^{2}}\right] \cap [X]
\]
retrieves Green and Griffiths' general expression:

\[
	s_{2}(\mathbf{T}_{k}) = \frac{1}{k!} \left(\alpha_{k} c_{1}^{2} - \beta_{k} c_{2}\right) \cap [X],
\]
with 
\[
	\alpha_{k} = \sum_{1 \leq i \leq k} \frac{1}{i^{2}} + \sum_{1 \leq i < j \leq k} \frac{1}{ij}
	           = \sum_{1 \leq i \leq j \leq k} \frac{1}{ij}
\]
\[
	\beta_{k} = \sum_{1 \leq i \leq k} \frac{1}{i^{2}}	
\]
\end{rem}

\chapter{Truncated intersection theory and Morse inequalities} \label{chap:truncatedstrat}

This chapter presents the core result of this text: a generalized version of the algebraic Morse inequalities, formulated in terms of a combinatorial data coming from a stratification on a variety. After some definition and general facts, we will prove this result in Section~\ref{sectmorseineq}.
\medskip

{\em In this chapter, \(\mathbbm{k}\) is a field of arbitrary characteristic. Except in Section~\ref{sec:example}, we do not require it to be algebraically closed.}

\section{Stratifications} \label{sec:stratifications}

\subsection{Basic definitions}
It is time to introduce the concept on which all the rest of the text will be based: a basic notion of {\em stratification} on a given variety. Our stratifications will be constructed inductively, using the following two definitions. 

\begin{defi} \label{defistratum} Let \(X_{n}\) (resp. \(X_{n-1}\)) be a normal scheme of pure dimension \(n\) (resp. \(n-1\)). A morphism \(X_{n-1} \overset{f}{\longrightarrow} X_{n}\) is called a \emph{stratum} on \(X_{n}\), if it factors through a reduced Weil divisor \(D_{n} \subset X_{n}\), so that \(f : X_{n-1} \longrightarrow D_{n}\) is a proper birational morphism. In particular, \(f\) associates pairwise the irreducible components of \(X_{n-1}\) and of \(D_{n}\).
\end{defi}

\begin{defi} \label{defistratification} A \emph{stratification} \(\Sigma = (X_{\bullet}, f_{\bullet}, q)\) on a reduced scheme \(X\) of pure dimension \(n\), is the data of a proper birational morphism \(q: X_{n} \to X\) with \(X_{n}\) normal, and a sequence of strata of maximal length on \(X'\), of the form
\[
		X_{0} 
		\overset{f_0}{\longrightarrow} 
		X_{1} 
		\overset{f_1}{\longrightarrow}  
		\dotsc 
		\overset{f_{n-1}}{\longrightarrow} 
		X_{n}
		\overset{q}{\longrightarrow}
		X.
\]
\end{defi}


\begin{rem}
	\begin{enumerate}[label=(\alph*)]
		\item  In Definition~\ref{defistratification}, all the maps \(f_{i}\) and \(q\) are part of the data. We will sometimes omit \(q\) when introducing a given stratification, and write e.g. "Let \(\Sigma := (X_{\bullet}, f_{\bullet})\) be a stratification on \(X\)".
		\item If we are given a stratification \(\Sigma := (X_{\bullet}, f_{\bullet})\) on a reduced scheme \(X\), each scheme \(X_{i}\) comes with a natural map \(q_{i} : X_{i} \longrightarrow X\), obtained by composition.
	\end{enumerate}
\end{rem}

The following definition introduces a natural notion of trivialization of a vector bundle on a stratification.

\begin{defi} \label{defitrivialization} Let \(X\) be a reduced scheme of pure dimension \(n\), and let \(E \longrightarrow X\) be a vector bundle of rank \(r\). Consider a stratification
	\(
		\Sigma 
		:= 
		\left( 
		X_{\bullet}, f_{\bullet} 
		\right).
	\)
	For each \(i \in \llbracket 0, n \rrbracket\), let \(q_{i} : X_{i} \longrightarrow X\) be the naturally induced map, and let \(U_{i} := X_{i} \setminus f_{i-1}(X_{i-1})\).
	\begin{enumerate}[label=(\alph*)]
		\item A \emph{trivialization} \(\mathbf{e}\) of \(E\) over \(\Sigma\), is the data, for each \(i \in \llbracket 0, n \rrbracket\), of a trivialization \((e^i_j)_{1 \leq j \leq r}\) of \(q_i^\ast E\) on the dense open subset \(U_i \subset X_{i}\). 
	
		\item A \emph{trivialized stratification} $\underline{\Sigma} = (\Sigma, \mathbf e)$ for $E$ is the data of a stratification $\Sigma$ of $X$ and of a trivialization $\mathbf e$ of $E$ over $\Sigma$.
	\end{enumerate}

We say that $\Sigma$ is \emph{adapted} to $E$ if there exists a trivialization of $E$ over $\Sigma$.
\end{defi}

The main case we will have in mind is the situation of a line bundle, where \(\mathrm{rk}(E) = 1\). It will also be quite important to be able to deal with {\em \(\mathbb{Q}\)-line bundles}. This will be permitted by the following definition. 

\begin{defi} \label{defi:trivstrat}
	Let \(N\) be a \(\mathbb{Q}\)-line bundle on \(X\), and let \(\Sigma\) be a stratification of \(X\). A {\em trivialization} of \(N\) on \(\Sigma\) is the data of a pair \((d, \mathbf{e})\), where
	\begin{enumerate}[label=(\alph*)]
		\item \(d \in \mathbb{N}_{> 0}\) is a integer such that \(N^{\otimes d}\) comes from a standard line bundle on \(X\), i.e. is the image of \(\mathrm{Pic}(X)_{\mathbb{Q}}\) of an element \(L \in \mathrm{Pic}(X)\).
		\item \(\mathbf{e}\) is a trivialization of \(L\) on \(\Sigma\) in the sense of Definition~\ref{defitrivialization}.
	\end{enumerate}
	We will sometimes write the data under a fractional form, and say that \(\frac{1}{d}\mathbf{e}\) is a {\em fractional trivialization} (or simply, a {\em trivialization}) of \(N\) over \(\Sigma\). Again, we say that \(\Sigma\) is {\em adapted} to \(N\) if the latter admits such a trivialization on \(\Sigma\), and we introduce a notion of {\em trivialized stratification} \(\underline{\Sigma} = (\Sigma, \frac{1}{d}\mathbf{e})\) for a given \(\mathbb{Q}\)-line bundle.
\end{defi}

\begin{rem} \label{rem:powertriv}
	\begin{enumerate}
		\item The previous definition also makes sense if \(N\) itself comes from a line bundle, even if \(d \geq 2\). In other words, in the general case, we do not require that \(d\) is the smallest integer such that \(N^{\otimes d}\) is a standard line bundle. 
		\item Let \((\Sigma, \frac{1}{d} \mathbf{e})\) be a trivialized stratification adapted to a \(\mathbb{Q}\)-line bundle \(N\). Let \(L\) be a standard line bundle such that \(L \sim_{\mathbb{Q}} N^{\otimes d}\). Now, pick any integer \(f \in \mathbb{N}_{\geq 1}\). We can define the  {\em \(f\)-th power of \(\mathbf{e}\)} by replacing any trivialization \(e\) on a strata by its \(f\)-th power. Let us denote by \(\mathbf{e}^{\otimes f}\) this trivialization of \(L^{\otimes f}\). It follows directly from Definition~\ref{defi:trivstrat} that \((\Sigma, \frac{1}{df} \mathbf{e}^{\otimes f})\) (resp. \((\Sigma, \frac{1}{f} \mathbf{e}^{\otimes f})\)) is also a trivialized stratification of \(N\) (resp. \(L\)).
	\end{enumerate}
\end{rem}

It is essentially obvious that any vector bundle on a reduced scheme admits an adapted stratification. This is the content of the following proposition.

\begin{prop} Let \(E\) be a vector bundle (resp. a \(\mathbb{Q}\)-line bundle) on a reduced pure dimensional scheme \(X\). Then \(E\) admits an adapted stratification. 
\end{prop}

\begin{proof}
	It suffices to deal with the case of a vector bundle. Let us proceed by induction on \(n = \dim X\). Let first \(q : X_{n} \to X\) be the normalization, and let \(U_{n} \subset X\) be an open dense subset, above which \(q^{\ast} E\) admits a trivialization \((e^{n}_{1 \leq j \leq r})\). We may reduce \(U_{n}\) and assume that all irreducible components of \(D_{n} := X_{n} - U_{n}\) have codimension \(1\). In other words, \(D_{n}\) is a Weil divisor on \(X_{n}\). Let now \(X_{n-1} = D_{n}^{\mathrm{norm}}\) be its normalization, and let \(f_{n-1} : X_{n-1} \to X_{n}\) be the natural map. This map is a stratum on \(X\) according to Definition \ref{defistratum}. We may now produce the requested stratification by applying the induction hypothesis to the vector bundle \(f_{n-1}^{\ast} E \longrightarrow X_{n-1}\).
\end{proof}

The following definition introduces the combinatorial data that will be crucial in the statement of our Morse inequalities. Basically, if we are given a stratification \(\Sigma\) on a scheme \(X\), it is possible to build a {\em tree} by taking a node for every component of the strata, and connecting them in the obvious manner.

\begin{defi} \label{defitree} Let \(\Sigma = (X_{\bullet}, f_{\bullet})\) be a stratification on a reduced scheme \(X\) of pure dimension \(n\). The {\em associated tree} to \(\Sigma\) is the graph defined as follows: 
	\begin{enumerate}[label=(\alph*)]
		\item the {\em nodes} are indexed by the irreducible components of all the \(X_{i}\); 
		\item the {\em root} is the node corresponding to \(X_{n}\).
		\item if \(\nu\) is such a node, indexed by an irreducible component \(V \subseteq X_i\), then the {\em children} of \(\nu\) are indexed by the irreducible components \(W\) of \(X_{i-1}\) that satisfy \(f_{i-1} (W) \subseteq V\). 
	\end{enumerate}
\end{defi}

\begin{rem} \label{remtree} Let \(\mathcal{T}\) be associated to a stratification \(\Sigma\), as in Definition~\ref{defitree}. We see right away that the following properties are satisfied:

\begin{enumerate}[label=(\alph*)]
\item The \emph{leaves} are in bijection with the points of $X_0$.
\item to each node $\nu$ of $\mathcal T$ is associated an irreducible variety $V_\nu$;
\item to each arrow $\nu \longrightarrow \mu$ in $\mathcal T$ is associated a morphism $f_{\nu \mu} : V_\nu \longrightarrow V_\mu$;  
\item \label{rem:itemstratum} if $C_\mu$ is the set of children of $\mu$, the natural map 
	\[
		c_\mu : \bigsqcup_{\nu \in C_\mu} V_{\nu} \longrightarrow V_{\mu}
	\]
	is a stratum of $V_\mu$ accordingly to Definition \ref{defistratum}; 
\end{enumerate}
\end{rem}

\begin{rem} Let \(X\) be a reduced scheme of pure dimension \(n \in \mathbb{N}\), and let \(\Sigma = (X_{\bullet}, f_{\bullet})\) be a stratification on \(X\). Let \(j \in \llbracket 0, n\rrbracket\). Let \(V \subset X_{j}\) be an irreducible component corresponding to a node \(\mu\) in the tree of \(\Sigma\). Then there is {\em an induced stratification} on \(V\)
	\[
		V_{0}
		\longrightarrow
		V_{1}
		\longrightarrow
		\dotsc
		\longrightarrow
		V_{j} = V,
	\]
	that is obtained by discarding all the arrows in the tree of \(\Sigma\) that are not in the subtree based at \(\mu\). Note that since \(V\) is normal, we are allowed to take the identity for the map \(V_{j} \to V\).
\end{rem}

\subsection{Refinements of stratifications} It will be quite important to us to be able to {\em refine} a given stratification, to make them adapted to several line bundles at once. Basically, this process introduces new branches in the tree of the initial stratification. Let us make this notion more precise.

\begin{defi} \label{defirefinement} Let \(X\) be a normal scheme of pure dimension \(n\).
	\begin{enumerate}[label=(\alph*)]
		\item We say that a stratum \(f_{1}' : X_{1}' \longrightarrow X\) is a \emph{refinement} of another stratum \(f_{1} : X_{1} \longrightarrow X\), if \(X_{1}\) is isomorphic to a disjoint union of components of \(X_{1}'\), and if we have a commutative diagram
			\[
				\begin{tikzcd}
					X_{1} \arrow[r, hook] \arrow[dr, "f_{1}", swap] & X_{1}' \arrow[d, "f_{1}'"] \\
					                      & X
				\end{tikzcd}.
			\]

		\item Let \(\Sigma\) and \(\Sigma'\) be stratifications of \(X\), with their associated trees \(\mathcal{T}\) and \(\mathcal{T'}\). We say that \(\Sigma'\) is a \emph{refinement} of \(\Sigma\), if there is a embedding of rooted trees \(\varphi : \mathcal{T} \hookrightarrow \mathcal{T'}\) such that the following holds. 
			With the notation of Remark~\ref{remtree}, let \(V_{\nu}\), \(c_{\nu}, \dotsc\) (resp. \(V_{\nu}'\), \(c_{\nu}', \dotsc\)) be the objects associated to \(\mathcal{T}\) (resp. \(\mathcal{T}'\)). Then we require that for all \(\mu \in \mathcal T\), we have \(V_{\varphi(\mu)}' \cong V_{\mu}\). Moreover, we ask that for all \(\mu \in \mathcal{T}\) the stratum \(c_{\varphi(\mu)}'\) is a refinement of \(c_{\mu}\), in such a manner that the following diagram commutes 
	\[
	\begin{tikzcd}
		\displaystyle\bigsqcup_{\nu \in C_{\mu}} V_{\nu} \arrow[r, hook] \arrow[dr, "c_{\mu}", swap] & 
		 \displaystyle\bigsqcup_{\nu \in C_{\varphi(\mu)}} V'_{\nu} \arrow[d, "c_{\varphi(\mu)}"] \\
		& V_{\mu} \cong V_{\varphi(\mu)}'
	\end{tikzcd}.
	\]
	where the isomorphism below is induced by the identifications \(V_{\nu} \cong V'_{\varphi(\nu)}\).
\item \label{def:refinetriv} With the same notation, if we are given a \(\mathbb{Q}\)-line bundle \(L\) on \(X\),  we say that \(\underline{\Sigma'} = (\Sigma', \frac{1}{d} \mathbf{e}')\) refines \(\underline{\Sigma} = (\Sigma, \frac{1}{d} \mathbf{e})\) if on the one hand, \(\Sigma'\) refines \(\Sigma\), and on the other hand, if for any node \(\mu \in \mathcal T\), the trivializations \(e\) and \(e'\) on \(V_{\mu}\) and \(V_{\varphi(\mu)}'\) correspond under the identification \(V_{\mu} \cong V_{\varphi(\mu)}'\).
	\end{enumerate}
\end{defi}

Informally, $\Sigma'$ is a refinement of $\Sigma$ if it is obtained from it by adding more boundary components to the successive strata. The following elementary proposition will be quite useful to produce stratifications adapted to several vector bundles at once.

\begin{prop} \label{proprefine} Let $E, F \longrightarrow X$ be vector bundles (or \(\mathbb{Q}\)-line bundles) on a reduced scheme \(X\), and let $\Sigma$ be a stratification adapted to $E$. There exists a stratification $\Sigma'$, refining $\Sigma$, and adapted to both $E$ and $F$.
\end{prop}

The proof is straightforward and its details will be left to the reader. Essentially, it suffices to first add several components to the first strata \(f_{n-1} : X_{n-1} \to X_{n} \overset{q}{\to} X\) so that both \(q^{\ast} E\) and \(q^{\ast} F\) get trivialized on \(X_{n} \setminus f_{n-1}(X_{n-1})\), then, arguing by induction, refine the stratification induced on \(X_{n-1}\) to make it adapted to both \(f_{n-1}^{\ast} E\) and \(f_{n-1}^{\ast} F\), and finally, to pick any trivialization that is adapted again to both \(E\) and \(F\) on the components that were added to \(X_{n-1}\).

\section{Truncated first Chern classes}

The data of a trivialized stratification with respect to a \(\mathbb{Q}\)-bundle \(L\) contains all the information necessary to compute the top intersection number \(\int_{X} c_{1}(L)^{\dim X}\). The purpose of this section is to introduce {\em truncated versions} of these numbers. They will later play the same role in our algebraic Morse inequalities, than the truncated integrals in Demailly's holomorphic Morse inequalities.
\smallskip

To ease a bit the future proofs by induction, we will actually introduce a minimal subset of intersection theory, that will allow us to give a recursive definition of these numbers.
\medskip

\subsection{Cycle groups and stratifications. Definition of the truncated Chern classes.} Our next definition introduces a cycle group associated to a given stratification on $X$. We choose to use rational coefficients for our cycles: this will be well suited to prove Morse inequalities for \(\mathbb Q\)-line bundles.

\begin{defi} \label{deficyclegroup} Let \(X\) be a reduced scheme of pure dimension \(n\), and let \(\Sigma = (X_{\bullet}, f_{\bullet})\) be a stratification of $X$. For each \(k \in \llbracket 0, n\rrbracket\), the \emph{\(k\)-th cycle group} of \(\Sigma\) is the finite dimensional \(\mathbb{Q}\)-vector space
\[
	Z_{k}^{\Sigma}(X)_{\mathbb{Q}} = \bigoplus_{V \subseteq X_{k}} \mathbb{Q} \cdot [V],
\]
	where \(V\) runs through the irreducible components of \(X_{k}\). The \emph{total cycle group} of \(\Sigma\) is the direct sum 
\[
	Z_{\bullet}^{\Sigma}(X)_{\mathbb{Q}}
	=
	\bigoplus_{0 \leq k \leq n}
	Z_{k}^{\Sigma}(X)_{\mathbb{Q}}.
\]
\end{defi}

\begin{rem}
	With the setting of Definition~\ref{deficyclegroup}, there is a natural map \(\rho : Z_{\bullet}^{\Sigma}(X)_{\mathbb{Q}} \hookrightarrow Z_{k}^{\Sigma}(X)\) that sends every component of \(V \subset X_{k}\) to its image in \(X\) by the natural morphism \(X_{k} \to X\) induced by the stratification.
\end{rem}

For any \(\mathbb{Q}\)-line bundle on a pure dimensional reduced scheme \(X\), with an associated trivialized stratification in the sense of Definition~\ref{defi:trivstrat}, we can now construct a {\em truncated first Chern class} as a particular endomorphism of \(Z_{\bullet}^{\Sigma}(X)_{\mathbb{Q}}\). We need first to give the elementary step of our inductive construction.
\smallskip

\begin{defi} [First Chern class. Elementary step] \label{defi:chernelem}
	Let \(V\) be a normal variety of dimension \(n\), and let \(V_{n-1} \overset{f_{n-1}}{\longrightarrow} V_{n} = V\) be a stratum. Let \(L\) be a \(\mathbb{Q}\)-line bundle on \(V\). Assume that some power \(L^{\otimes d}\) is a line bundle admitting a trivialization \(e\) on the open dense subset \(U_{n} = V_{n} \setminus f_{n-1}(V_{n-1})\). Then \(e\) can be seen as a rational section \(s\) of \(L^{\otimes d}\) on \(V\). Since \(V\) is normal, it is smooth near the generic point of any component \(D_{i}\) of the Weil divisor \(f_{n-1}(V_{n-1})\), and \(s\) has a well defined multiplicity \(m_i \in \mathbb Z\) at this generic point. For each such component \(D_{i}\), let \(W_{i} \subseteq V_{n-1}\) be the unique irreducible component such that \(f_{n-1}(W_{i}) = D_i\) (see Definition \ref{defistratum}).
	\smallskip

	For any \(l \in \left\{0, 1 \right\}\), define the following \((n-1)\)-cycle in \(V_{n-1}\), as follows:
\begin{equation} \label{eqchernelem}
	c_{1}(L, f_{n-1})_{[l]} \cap [V] 
	\;
	:= 
	\;
	\sum_{(-1)^l m_i > 0} 
	\frac{m_i}{d} [W_i] 
	\quad
	\in 
	Z_{n-1}(V_{n-1})_{\mathbb{Q}}.
\end{equation}
	The previous sum is designed so that $W_i$ runs among the connected components of $Y_1$ such that $m_i$ has the same sign as \((-1)^l\). 

\end{defi}
\medskip

\begin{rem}
	With the notation of Definition~\ref{defi:chernelem}, recall that the first Chern class can be defined as an endomorphism of the Chow ring \(A_{\ast}(V)\), for which the image of the class of \([V]\) can be obtained as follows
\begin{equation} \label{eqdefimult}
	c_{1}(L) \cap [V] 
	= 
	\sum_{i} 
	\frac{m_{i}}{d} [D_{i}]
	\hspace{0.5cm}  \text{in} \; A_{n-1}(V)_{\mathbb Q}.
\end{equation}
 (see \cite[Section 2.5]{ful98}). The morphism \(f_{n-1} : V_{n-1} \longrightarrow V_{n} = V\) induces a composite map 
	\[
		Z_{n-1}(V_{n-1})_{\mathbb{Q}} 
		\overset{(f_{n-1})_{\ast}}{\longrightarrow}
		Z_{n-1}(V)_{\mathbb{Q}} 
		\overset{p}{\longrightarrow}
		A_{n-1}(V)_{\mathbb{Q}},
	\] 
	where \(p\) denotes the natural projection. Letting \(\rho : Z_{n-1}(V_{n-1})_{\mathbb{Q}} \to A_{n-1}(V)_{\mathbb{Q}}\) denote this composition, we now have:
	\begin{equation}\label{eq:contributionschern}
		c_{1}(L) \cap [V] 
		= 
		\rho
		\left( 
		c_{1}(L, f_{n-1})_{[1]} \cap [V] 
		\; 
		+ 
		\; 
		c_{1}(L, f_{n-1})_{[0]}) \cap [V]
		\right) 
		\quad
		\text{in}
		\;
		A_{n-1}(Y)_{\mathbb{Q}},
	\end{equation}
	which follows directly from the formula \(\rho([V_{i}]) = (f_{n-1})_{\ast}([V_{i}]) = [D_{i}]\) that holds for all \(i\), since \(f_{n-1}|_{V_{i}} : W_{i} \to D_{i}\) is proper {\em birational} for any such \(i\).
\end{rem}

Equation~\eqref{eq:contributionschern} above allows us to interpret the \(c_{1}(L, f_{n-1})_{[l]} \cap [V]\) as truncated components in a Chern intersection formula. In the next proposition, we will build up on this interpretation and use the elementary step of Definition~\ref{defi:chernelem} to define a truncated first Chern class as a genuine endomorphism of the total cycle group of a given scheme.
\medskip

To avoid repetition, let us fix the following context for the rest of this section.

\begin{nota}  \label{nota:stratification}
Let \(X\) a reduced scheme of pure dimension \(n\). Let \(L\) be a \(\mathbb{Q}\)-line bundle on \(X\), and let \(\Sigma = (X_{\bullet}, f_{\bullet})\) be a stratification adapted to \(L\), endowed with a trivialization \(\frac{1}{d} \mathbf{e}\). Denote by \(\underline{\Sigma} = (\Sigma, \frac{1}{d} \mathbf{e})\) the associated trivialized stratification in the sense of Definition~\ref{defi:trivstrat}.
\end{nota}

\begin{defi} \label{defistratfirstchern} 	
	Setting as in Notation~\ref{nota:stratification}. Let \(l \in \left\{0, 1 \right\}\). The {\em truncated first Chern class of level} \(l\) of the pair \((L, \underline{\Sigma})\) is the endomorphism 
	\[
		c_{1}(L, \underline{\Sigma})_{[l]}
		\in
		\mathrm{End}(Z^{\Sigma}_{\bullet}(X)_{\mathbb{Q}}),
	\]
	whose action on the pure cycles $[V] \in Z_{\bullet}^{\Sigma}(X)_{\mathbb{Q}}$ is defined as follows:
	\smallskip

	Let \(V\) be an irreducible component of \(X_{k}\), and let \(g_{k} : f_{k}^{-1}(V) \longrightarrow V\) be the morphism induced by $f_k$. By Remark~\ref{remtree} \ref{rem:itemstratum}, \(g_{k}\) is a stratum of $V$. For \(l \in \left\{0, 1\right\}\), we then define, following \eqref{eqchernelem}:
\begin{equation} \label{eqstratfirstchernclass}
	c_{1}
	(L, \underline{\Sigma})_{[l]} 
	\cap 
	[V] 
	= 
	c_{1}(L, g_{k})_{[l]} 
	\cap [V] 
	\;\; \in \;\;  
	Z_{k-1}(f_k^{-1}(V))_{\mathbb{Q}}.
\end{equation}
	By Definition \ref{deficyclegroup}, \(Z_{k-1}(f_k^{-1}(V))_{\mathbb{Q}}\) is a direct summand of \(Z_{k-1}^{\Sigma}(X)_{\mathbb{Q}}\). This permits to see the above cycle as an element of \(Z^{\Sigma}_{k-1}(X)_{\mathbb{Q}}\).
	\smallskip

	If \(l \notin \left\{ 0, 1 \right\}\), we let \(c_{1}(L, \underline{\Sigma})_{[l]}\) be the null endomorphism of \(Z_{\bullet}^{\Sigma}(X)_{\mathbb{Q}}\).
\end{defi}

We are now ready to define the notion of truncated {\em powers} of the first Chern classes that we will use in the rest of these notes: again, they will be constructed inductively using our last definition. 

\begin{defi} \label{defitrunchigherpowers}
	Setting as in Notation~\ref{nota:stratification}. Let \(k, l \in \mathbb{N}\). We let \(c_{1}(L, \underline{\Sigma})^{k}_{[l]}$ be the endomorphism of \(Z_{\bullet}^{\Sigma}(X)_{\mathbb{Q}}\) defined inductively as follows.
\begin{itemize}
	\item If \(k = 0\), we let \(c_{1}(L, \underline{\Sigma})^{0}_{[l]}\) be the identity for any \(l \geq 0\);

\item If \(k \geq 1\), we let
\begin{equation} \label{eqinductform}
	c_{1}(L, \underline{\Sigma})^{k}_{[l]} 
	= 
	c_{1}(L, \underline{\Sigma})^{k-1}_{[l]} \, \circ \, c_1(L, \underline{\Sigma})_{[0]} 
	+ 
	c_{1}(L, \underline{\Sigma})^{k-1}_{[l-1]} \, \circ \, c_1(L, \underline{\Sigma})_{[1]}
\end{equation}
\end{itemize}
\end{defi}

\begin{rem}
	Note that the definition above provides \(c_{1}(L, \underline{\Sigma})^k_{[l]} = 0\) if \(l \notin \llbracket 0, k \rrbracket\).
\end{rem}
\medskip

The Morse inequalities of Theorem \ref{thmmorse} will be stated in terms of the truncated Chern classes, that we introduce in the next definition.

\begin{defi} \label{deficherntrunc} 
	Setting as in Notation~\ref{nota:stratification}. Let \(l, k \in \mathbb{N}\). The {\em \(l\)-truncated \(k\)-th power of the first Chern class} of \((L, \underline{\Sigma})\) is the following endomorphism of \(Z_{\bullet}^\Sigma(X)_{\mathbb{Q}}\):
\[
	c_{1}(L, \underline{\Sigma})^{k}_{[\leq l]} 
	= 
	\sum_{0 \leq j \leq l} c_{1}(L, \underline{\Sigma})^{k}_{[j]}.
\]
\end{defi}

The terminology of \emph{truncated} power can be justified by the following proposition, which follows from applying inductively \eqref{eq:contributionschern}.

\begin{prop} \label{prop:truncatedversion} Setting as in Notation~\ref{nota:stratification}. Let \(\rho : Z_\bullet^\Sigma(X)_{\mathbb Q} \longrightarrow Z_\bullet(X)_{\mathbb Q}\) be the natural map. For any \(k \in \llbracket 0, n \rrbracket\), the \((n-k)\)-cycle
\[
	\rho \left( \; c_1(L, \underline{\Sigma})^k_{[\leq k]} \cap [X] \; \right) \; \; \in Z_{n-k}(X)_{\mathbb Q}
\] 

	is a representative of the cycle class $c_1(L)^k \cap[X] \in A_{n-k}(X)_{\mathbb Q}$.
\end{prop}

The next proposition shows that truncated powers behave as expected if we replace the trivialization by one of its powers. We refer to Remark~\ref{rem:powertriv} for the definition of the power of a trivialialization on a stratification.

\begin{prop} \label{prop:powertriv} Setting as in Notation~\ref{nota:stratification}. Let \(f \in \mathbb{N}_{\geq 1}\), and
	let \(\mathbf{e}^{\otimes f}\) be the \(f\)-th power of \(\mathbb{e}\). We let \(\underline{\Sigma}_{1} = (\Sigma, \frac{1}{d} \mathbf{e}^{\otimes f})\) and \(\underline{\Sigma}_{2} = (\Sigma, \frac{1}{df} \mathbf{e}^{\otimes f})\). Then, for all integers \(k\) and \(j\), one has :
	\[
		c_{1}(L^{\otimes f}, \underline{\Sigma}_{1})^{k}_{[\leq l]} 
		= 
		f^{k} c_{1}(L, \underline{\Sigma})^{k}_{[\leq j]}
	\]
	and
	\[
		c_{1}(L, \underline{\Sigma}_{2})^{k}_{[\leq l]} 
		= 
		c_{1}(L, \underline{\Sigma})^{k}_{[\leq j]}
	\]
\end{prop}
\begin{proof}
	This is essentially obvious : it follows from Definition that the left hand side of the first equality is obtained by replacing \(m_{i}\) by \(f m_{i}\) in equation~\ref{eqchernelem}. This gives the result after an immediate computation. For the second equality, one has to replace \(m_{i}\) be \(\frac{f m_{i}}{f} = m_{i}\), so the endomorphism is unchanged.
\end{proof}

As usual, we have a degree map that compute the total multiplicity of a $0$-cycle. Let us recall the following notation.

\begin{nota} \label{nota:degreepoint}
If \(V\) is a \(0\)-dimensional variety of finite type over \(\mathbbm{k}\), then it is isomorphic over \(\mathrm{Spec}\, \mathbbm{k}\) to a scheme \(\mathrm{Spec}\, \mathbb{K}\), where \(\mathbbm{k} \subset \mathbb{K}\) is a finite field extension. We then let
\[
	\mathrm{deg}_{\mathbbm{k}}
	(V) := \deg_{\mathbbm{k}}(\mathbb{K}).
\]
\end{nota}

\begin{defi} Setting as in Notation~\ref{nota:stratification}. The \emph{degree map} of \(\Sigma\) is defined as follows:
\bgroup
\renewcommand{\arraystretch}{1.2}
\[
	\begin{array}{cccc}
		\mathrm{deg} : & Z_{0}^{\Sigma}(X)_{\mathbb{Q}} & \longrightarrow & \mathbb{Q} \\
			       & \sum_{V} \alpha_{V} [V]        & \longmapsto     & \sum_{V} \alpha_{V}\, \mathrm{deg}_{\mathbbm{k}}(V).
	\end{array}
\]
\egroup
	 If \(\Psi\) is an endomorphism of \(Z_{\bullet}^{\Sigma}(X)_{\mathbb{Q}}\), we will sometimes write \(\deg \Psi\) instead of \(\deg \Psi([X])\), to lighten a bit the notation.
\end{defi}

Before closing this section, let us give a final recursive formula that will prove useful in our proofs by induction.

\begin{lem} \label{lemformulainduct}
	Setting as in Notation~\ref{nota:stratification}. Let \(e\) be the trivialization of \(L^{\otimes d}\) on \(X \setminus f_{n-1}(X_{n-1})\) provided by \(\mathbf{e}\). Let \((W_{1 \leq i \leq r})\) be the family of components of \(X_{n-1}\), and \(m_{k}\) be the multiplicities of \(e\) along the images of these components by \(f_{n-1}\). For each \(i \in I\), let \(\underline{\Sigma}_{i}\) be the trivialized stratification on \(W_{i}\) that is induced by \(\underline{\Sigma}\). Denote by \(q_i : W_{i} \longrightarrow X\) the natural maps. 
	\smallskip
		 
	Then, we have, for any \(l \in \llbracket 0 , n \rrbracket\):
\[
	\deg c_{1}(L, \underline{\Sigma})^{n}_{[\leq l]} 
	= 
	\sum_{m_{i} > 0} 
	\frac{m_{i}}{d} 
	\deg c_{1}(q_{i}^{\ast} L, \underline{\Sigma}_{i})^{n-1}_{[\leq l]} 
	+ 
	\sum_{m_{i} < 0} 
	\frac{m_{i}}{d} 
	\deg c_{1}(q_{i}^{\ast} L, \underline{\Sigma}_{i})^{n-1}_{[\leq l - 1]},
\]
	where the first (resp. second) sum runs among all indexes \(i \in \llbracket 1, r \rrbracket\) such that \(m_{i} > 0\) (resp. \(m_{i} < 0\)). 
\end{lem}

\begin{proof}
	If we sum~\eqref{eqinductform} letting \(l\) vary from \(0\) to a given integer, we get by Definition~\ref{deficherntrunc}:
	\[
	c_{1}(L, \underline{\Sigma})^{k}_{[\leq l]} 
	= 
	c_{1}(L, \underline{\Sigma})^{k-1}_{[\leq l]} \, \circ \, c_1(L, \underline{\Sigma})_{[0]} 
	+ 
	c_{1}(L, \underline{\Sigma})^{k-1}_{[\leq l-1]} \, \circ \, c_1(L, \underline{\Sigma})_{[1]}
	\]

	Applying this endomorphism to the cycle \([X] \in Z_{n}^{\Sigma}(X)_{\mathbb{Q}}\) yields, by Definition~\ref{defistratfirstchern} and equation~\eqref{eqchernelem}:
	\begin{align*}
	c_{1}(L, \underline{\Sigma})^{k}_{[\leq l]} \cap [X]
	& = 
		c_{1}(L, \underline{\Sigma})^{k-1}_{[\leq l]} \cap 
		\left( \sum_{m_{i} > 0} \frac{m_{i}}{d} [W_{i}] \right) 
	+ 
		c_{1}(L, \underline{\Sigma})^{k-1}_{[\leq l-1]} \cap
		\left( \sum_{m_{i} < 0} \frac{m_{i}}{d} [W_{i}] \right) 
		\\
		& = \sum_{m_{i} > 0} \left( \frac{m_{i}}{d} c_{1}(q_{i}^{\ast} L, \underline{\Sigma}_{i})^{k-1}_{[\leq l]} \cap 
		[W_{i}] \right) 
	+ 
 		\sum_{m_{i} < 0}
		\left(
		\frac{m_{i}}{d}	 
		c_{1}(q_{i}^{\ast} L, \underline{\Sigma})^{k-1}_{[\leq l-1]} \cap
		[W_{i}]
		\right) 
	\end{align*}
	where at the second line, we used the equalities 
	\(
	c_{1}(L, \underline{\Sigma})^{k-1}_{[\ast]} \cap [W_{i}] = 
	c_{1}(q_{i}^{\ast} L, \underline{\Sigma}_{i})^{k-1}_{[\ast]} \cap [W_{i}],
	\)
	valid in each subset \(Z_{0}^{\underline{\Sigma}_{i}}(W_{i}) \subset Z_{0}^{\underline{\Sigma}}(X)\). To conclude, it suffices to take the sum of the multiplicities of these \(0\)-cycles on both sides of the equation.
\end{proof}

\subsection{Paths and indexes} The purpose of this section is to present another manner of computing degrees of truncated Chern classes, using the tree data that has been introduced in Definition~\ref{defitree}. It is actually the way that will be the most useful in the applications to the existence of jet differential equations.
\medskip

\begin{nota}
	For the remainder of this section, we fix a trivialized stratification \(\underline{\Sigma} = (\Sigma, \frac{1}{d} \mathbf{e})\) for a \(\mathbb{Q}\)-line bundle \(L\) on a variety \(X\). If \(\Sigma = (X_{\bullet}, f_{\bullet})\), let \(U_{n} = X_{n} - f_{n-1}(X_{n-1})\), and let \(e\) be the trivialization of \(L^{\otimes d}\) on \(U_{n}\) given by \(\mathbf{e}\). Let \(\mathcal T\) be the tree associated to \(\Sigma\) by Definition~\ref{defitree}, and let \(\mathbf{r}\) be its root.
	
	We also fix the following data regarding the components of \(X_{n-1}\):
	\begin{itemize}
		\item If \(W_{i} \subset X_{n-1}\) is any irreducible component, let \(m_{i} \in \mathbb{Z}\) be the multiplicity of \(e\), seen as a rational section of \(L^{\otimes d}\), at the generic point of \(f_{n-1}(W_{i})\). 
		\item For any \(i \in \llbracket 1, m \rrbracket\), \(\underline{\Sigma}\) induces a trivialized stratification on \(W_{i}\), by keeping only the components of the \(X_{j}\) whose nodes descends from the node of \(W_{i}\) in \(\mathcal{T}\). Let us denote that trivialized stratification by \(\underline{\Sigma}_{i} = (\Sigma_{i}, \frac{1}{d}\mathbf{e}_{i})\).
	\end{itemize}
\end{nota}

\begin{defi} \label{defi:markingtree}
	The {\em marking of \(\mathcal{T}\) associated with \(\frac{1}{d}\mathbf{e}\)} is the marking of all edges of \(\mathcal{T}\) by rational integers, defined as follows.
	
	Let \((W_{i})_{1 \leq i \leq r}\) be the irreducible components of \(X_{n-1}\). Denote by \(\mathbf{w}_{i}\) the associated nodes of \(\mathcal T\), which are exactly the children of \(\mathbf{r}\).  Now, mark the edges of \(\mathcal{T}\) with the following rational multiplicities:
	\begin{enumerate}[label=(\alph*)]
		\item mark the edges \(\mathbf{w}_{i} \to \mathbf{r}\) with the rational number \(\mu_{i} = \frac{m_{i}}{d}\), for all \(i \in \llbracket 1, m\rrbracket\) ;
		\item for any \(i \in \llbracket 1, m\rrbracket\), the subtree of \(\mathcal{T}\) based at \(\mathbf{w}_{i}\) identifies with the tree associated to the induced trivialized stratification \(\underline{\Sigma}_{i}\). Mark the edges of these trees by induction.
	\end{enumerate}
\end{defi}

Using these markings will provide a natural way to compute truncated top intersections. We will need the following definition.

\begin{defi} Let \(\sigma\) be a root-to-leaf path in \(\mathcal{T}\). We say that \(\sigma\) is a path of \emph{index} \(l\), if there are \emph{exactly} \(l\) negative markings on the edges of \(\sigma\).
\end{defi}

Then, by Lemma \ref{lemformulainduct}, and the construction of the markings on $\mathcal T$, the following proposition is straightforward.

\begin{prop} \label{propgraph}  For any root-to-leaf path in \(\mathcal T\), let \(C_{\sigma}\) denote the product of the markings along the edges of \(\sigma\). Let \(V_{\sigma}\) be the \(0\)-dimensional variety labeling the leaf of \(\sigma\), and let \(\deg(\sigma) := \deg_{\mathbbm{k}}(V_{\sigma})\).
	\smallskip
	
	Then, for all \(l\), we have
	\begin{equation} \label{eq:degreeindexes}
	\deg c_1(L, \underline{\Sigma})^n_{[l]} = \sum_{\mathrm{index}(\sigma) = l} C_\sigma\, \deg(\sigma).
		\end{equation}
	where, in the sum above, $\sigma$ runs among all root-to-leaf paths in \(\mathcal{T}\) of index \(l\).
\end{prop}

Let us finish this section with a proposition showing that we can use arbitrarily refinements of a stratification $\Sigma$ to compute the truncated Chern classes.

\begin{prop} \label{proprefine2} Let \(\underline{\Sigma} = (\Sigma, \frac{1}{d} \mathbf{e})\) be a trivialized stratification, and let \(\underline{\Sigma}' = (\Sigma', \frac{1}{d} \mathbf{e'})\) be a refinement of \(\underline{\Sigma}\). Then, for any \(l\), we have
\[
	\deg c_{1}(L, \underline{\Sigma})^n_{[l]} 
	= 
	\deg c_{1}(L, \underline{\Sigma}')^n_{[l]}.
\]
\end{prop}
\begin{proof}
	With the notation of Definition \ref{defirefinement}, we have an embedding of trees \(\varphi : \mathcal{T} \longrightarrow \mathcal{T'}\). By item~\ref{def:refinetriv} of this definition, we see that two edges of \(\mathcal{T}\) and \(\mathcal{T}'\) correspond under this embedding, then they have the same markings (see Definition~\ref{defi:markingtree}).
	\smallskip
	
	Let \(\mathbf{v} \overset{\mathfrak{s}}{\longrightarrow} \mathbf{w}\) be an edge of \(\mathcal{T'}\) which does not belong to \(\varphi(\mathcal{T})\), and that satisfies \(\mathbf{w} \in \mathcal T\). We claim that $\mathfrak{s}$ is attributed the multiplicity \(0\) for the marking associated to \(\frac{1}{d} \mathbf{e'}\). Indeed, if \(V\) (resp. \(W\)) is the irreducible variety labeling \(\mathbf{v}\) (resp. \(\mathbf{w}\)), then \(V\) does not appear among the irreducible components of the stratum of \(W\) provided by \(\Sigma\), since we must have \(\mathbf{w} \notin \varphi(\mathcal{T})\) in our situation. Thus, if \(e\) is the trivialization of \(L^{\otimes d}\) given by \(\mathbf{e}\) on an open subset of \(W\), then \(e\) must be invertible near the generic point of \(\mathrm{Im}(V \longrightarrow W)\). This shows that \(\mathfrak{s}\) is marked with \(0\) in \(\mathcal{T}'\).  Now, if \(\sigma\) is a root-to-leaf path of \(\mathcal{T}'\) not included in \(\varphi(\mathcal T)\), it has one edge \(\mathfrak{s}\) in the same situation as above. This implies that the product of all markings along the edges of \(\sigma\) is \(C_{\sigma} = 0\), and thus this path does not contribute to the sum defining \(\deg c_{1}(L, \underline{\Sigma}')^n_{[l]}\) in Proposition \ref{propgraph}. Since all paths in \(\mathcal{T'}\) that are also included in \(\mathcal{T}\) have their markings provided by \(\frac{1}{d}\mathbf{e}\), the latter proposition ends the proof.
\end{proof}

\subsection{Finite morphisms and stratifications} \label{sec:finitemorphism} Let \(p : X' \to X\) be a finite morphism between two varieties. In this section, we explain how to produce a stratification on \(X'\), starting from a stratification on \(X\), and we compare the truncated Chern classes obtained in this way. The straightforward idea is to consider the fiber products and to take the adequate normalizations. Let us give some details. 
\smallskip

{\subsubsection{Basic construction.} Let \(X\) be a normal variety of dimension \(n\), and let \(f : Y \to X\) be a stratum on \(X\), where \(Y\) is a normal scheme of pure dimension \(n-1\). Let \(D\) be the schematic image of \(f\) in \(X\): this is a Weil divisor in \(X\), such that the corestriction \(f : Y \to D\) is proper birational. 

Consider now a finite morphism \(p : X' \to X\), where \(X'\) is a normal variety as well. The scheme \(Y \times_{X} X'\) has pure dimension \(n-1\). Denote by \(Y'\) the scheme obtained by taking the disjoint union of the normalizations of all irreducible components of \(Y \times_{X} X'\): this is again a purely dimensional normal scheme of dimension \(n-1\). There is a naturally induced morphism \(f' : Y' \to X'\).

Now, \(D' := p^{-1}(D)\) is also a Weil divisor in \(X'\). For each component \(D_{0}'\) of \(D'\), there is exactly one component \(Y_{0}'\) of \(Y'\) that dominates it, and for which the induced morphism \(Y_{0}'\to Y'\) realizes a birational morphism. Indeed, this can be checked after shrinking \(X\) around the generic point of \(D_{0} := p(D_{0}')\), so that \(f : Y \to X\) identifies with the embedding \(D_{0} \hookrightarrow X\). In this case, \(Y \times_{X} X'\) identifies with the inverse image \(p^{-1}(D_{0})\), so the claim is obvious.
\medskip

With these definitions, we have then proved the following:

\begin{lem} \label{lem:stratumfinite}
	The morphism \(f' : Y' \to X'\) is a stratum on \(X'\).
\end{lem}

Suppose now that we are given a line bundle \(L\) on \(X\), with a trivialization \(e\) on the Zariski open subset \(X \setminus D\). This pulls-back to define a trivialization \(e'\) of \(p^{\ast} L\) on \(X' - D'\). The following result is an immediate consequence of the projection formula.
\medskip

\begin{lem} \label{lem:projectionformula} Let \(d := \mathrm{deg}(p)\). Let \(D_{0}\) be an irreducible component of \(D\), and let \(m_{D_{0}} \in \mathbb{Z}\) be the multiplicity of \(e\) around \(D_{0}\), seen as a rational section of \(L\) (see Definition~\ref{defi:chernelem}). For each irreducible \(D_{0}'\) of \(D'\), let \(m_{D_{0}'} \in \mathbb{Z}\) be the multiplicity of \(e'\) around \(D_{0}'\). Then we have the following.
	\begin{enumerate} [label = (\alph*)]
		\item if \(m_{D_{0}} \geq 0\) (resp. \(m_{D_{0}} \leq 0\)), then \(m_{D_{0}'} \geq 0\) (resp. \(m_{D_{0}'} \leq 0\)) for all \(D_{0}'\) above \(D_{0}\);
		\item we have
			\[
				\sum_{D_{0}'} m_{D_{0}'} \deg (D_{0}'/D_{0}) = d m_{D_{0}},
			\]
			where the sum runs over all components \(D_{0}'\) above \(D_{0}\).
	\end{enumerate}
\end{lem}
\begin{proof}
	For the first point, it suffices to remark that if \(e\) (resp \(e^{-1}\)) is a regular section near the generic point of \(D_{0}\), then \(e'\) (resp. \(e'\)) is regular near the generic point of all \(D_{0}'\). The second point is provided by looking at the multiplicity of \(D_{0}\) in the projection formula \(p_{\ast} D(e') = d D(e)\).
\end{proof}

\subsubsection{Construction of a stratification} \label{sec:constructionstrat} Let now \(X\) be a reduced scheme of pure dimension \(n\), and consider a stratification \(\Sigma = (X_{\bullet}, f_{\bullet}, q)\) on \(X\), with the notation of Definition~\ref{defistratum}. Let \(p : X' \to X\) be a finite morphism, and let \(X_{n}'\) be the normalization of \(X'\). Applying inductively the previous construction starting from the strata \(X_{n-1} \to X_{n}\), and using the finite morphism \(X_{n}' \to X_{n}\), we get a diagram as follows, where the top line is a stratification of \(X'\):

\[
	\begin{tikzcd}
		X_{0}'
			\arrow[r, "f_{0}'"]
			\arrow[d, "p_{0}"]
		&
		X_{1}'
			\arrow[r, "f_{1}'"]
			\arrow[d, "p_{1}"]
		&
		\dotsc
			\arrow[r]
		&
		X_{n-1}'
			\arrow[r, "f_{n-1}'"]
			\arrow[d, "p_{n-1}"]
		&
		X_{n}'
			\arrow[r, "q'"]
			\arrow[d, "p"]
		&
		X'
			\arrow[d, "p"]
		\\
		X_{0}
			\arrow[r, "f_{0}"]
		&
		X_{1}
			\arrow[r, "f_{1}"]
		&
		\dotsc
			\arrow[r]
		&
		X_{n-1}
			\arrow[r, "f_{n-1}"]
		&
		X_{n}
			\arrow[r, "q"]
		&
		X
	\end{tikzcd}
\]

Denote by \(\Sigma'\) the stratification on \(X'\) that is obtained in this manner. Now, if \(L\) is a line bundle on \(X\), and \(d \in \mathbb{N}_{\geq 1}\), let \(\mathbf{e}\) be a trivialization of \(L^{\otimes d}\) on \(\Sigma\), the remark following Lemma~\ref{lem:stratumfinite} shows that we may pull back each trivialization appearing in \(\mathbf{e}\) to obtain a trivialization \(\mathbf{e}'\) for \((L')^{\otimes d}\) on \(\Sigma'\), where \(L' = p^{\ast} L\).  We now have the following claim :

\begin{prop} \label{prop:coveringdegree}
	With the notation above, let \(\underline{\Sigma} = (\Sigma, \frac{1}{d} \mathbf{e})\) (resp. \((\underline{\Sigma}' := (\Sigma', \frac{1}{d} \mathbf{e}')\)) be the corresponding trivialized stratification for \(L\) on \(X\) (resp. for \(L'\) on \(X'\)). Then we have, for any \(l \in \llbracket 0, n \rrbracket\):
	\[
		\deg c_{1}(L', \underline{\Sigma}')_{[\leq l]}^{n}
		=
		\deg(p)
		\deg c_{1}(L, \underline{\Sigma})_{[\leq l]}^{n}.
	\]
\end{prop}
\begin{proof}

	Let us argue by induction on \(\dim X\). If \(\dim X = 0\), then \(X \cong \mathrm{Spec}\, \mathbb{K}\) for some finite extension \(\mathbbm{k} \subset \mathbb{K}\), and \(X' \cong \mathrm{Spec}\, \mathbb{K}'\) for some other finite extension \(\mathbb{K} \subset \mathbb{K'}\). We must then take \(l = 0\), for which
	\begin{align*}
		\deg c_{1}(L', \underline{\Sigma}')_{[0]}^{0}
		& =
		\deg [X'] \\
		& = \dim_{\mathbbm{k}} \mathbb{K'} \\
		& =  \dim_{\mathbb{K}}{\mathbb{K'}} \cdot \dim_{\mathbbm{k}} \mathbb{K} \\
		& = \deg(p) \deg [X] = \deg(p)\, \deg c_{1}(L, \underline{\Sigma})_{[0]}^{0}.
	\end{align*}

	If \(\dim X \geq 1\), let \(e\) (resp. \(e'\)) be the trivialization of \(L^{\otimes d}\) (resp. of \(L'^{\otimes d}\)) on \(X \setminus f_{n-1}(X_{n-1})\) (resp. \(X \setminus f'_{n-1}(X'_{n-1})\)) provided by \(\mathbf{e}\) (resp. by \(\mathbf{e}'\)). Let \((W_{i})_{1 \leq i \leq r}\) (resp. \((W_{j}')_{1 \leq j \leq s}\)) denote the family of components of \(X_{n-1}\) (resp. \(X_{n-1}'\)), and \(m_{k}\) (resp. \(m_{k}'\)) be the multiplicities of \(e\) (resp. \(e'\)) along the images of these components by \(f_{n-1}\) (resp. \(f_{n-1}'\)). Then we are in the setup of Lemma~\ref{lemformulainduct}, for both varieties \(X\) and \(X'\). For each \(W_{i}\) (resp \(W_{j}')\), denote by \(\underline{\Sigma}_{i}\)  (resp. \(\underline{\Sigma}'_{j}\)) the trivialized stratification induced on \(W_{i}\) (resp. \(W_{j}'\)).
	\medskip

	Fix a component \(W_{i}\), and denote by \((W'_{j})_{j \in I_{j}}\) the components among the \(W'_{j}\) that dominate it. For any \(j \in I_{i}\), the induction hypothesis applied to the varieties \(W_{i}\) and \(W'_{j}\) yields:
	\[
		\deg c_{1}(q_{j}'^{\ast}L', \underline{\Sigma}'_{j})^{n}_{[\leq l]}
		=
		\deg(W_{j}'/W_{i}) 
		\deg c_{1}(q_{i}^{\ast} L, \underline{\Sigma}_{i})^{n}_{[\leq l]}
	\]

	This gives
	\begin{align*}
		\sum_{j \in I_{i}}
		\frac{m_{j}'}{d} \deg c_{1}(q_{j}'^{\ast}L', \underline{\Sigma}_{j}')^{n}_{[\leq l]} 
		& = 
		\sum_{j \in I_{i}}
		\frac{m_{j}'}{d} 
		\deg(W_{j}'/W_{i}) 
		\;
		\deg c_{1}(q_{i}^{\ast} L, \underline{\Sigma}_{i})^{n}_{[\leq l]} \\
		& = \deg(p)\; \frac{m_{i}}{d}
		\deg c_{1}(q_{i}^{\ast} L, \underline{\Sigma}_{i})^{n}_{[\leq l]}
	\end{align*}
	where we applied Lemma~\ref{lem:projectionformula}, {\em (b)} at the last line (remark that \(\deg(p) = \deg (X_{n}'/X_{n})\). Now, by Lemma~\ref{lem:projectionformula}, {\em (a)}, if \(W_{j}'\) dominates \(W_{i}\), we have \(m_{j}' > 0\) if and only if \(m_{i} > 0\). This, joint with the previous computation, shows that
\begin{align*}
	\sum_{m_{j} > 0} 
	\frac{m_{j}'}{d} 
	\deg c_{1}((q_{j})'^{\ast} L', \underline{\Sigma}_{j}')^{n-1}_{[\leq l]} 
	& + 
	\sum_{m_{j}' < 0} 
	\frac{m_{j}'}{d} 
	\deg c_{1}((q_{j}')^{\ast} L', \underline{\Sigma}'_{j})^{n-1}_{[\leq l - 1]} \\
	& =
	\deg(p)
	\left(
	\sum_{m_{i} > 0} 
	\frac{m_{i}}{d} 
	\deg c_{1}(q_{i}^{\ast} L, \underline{\Sigma}_{i})^{n-1}_{[\leq l]} 
	 + 
	\sum_{m_{i} < 0} 
	\frac{m_{i}}{d} 
	\deg c_{1}(q_{i}^{\ast} L, \underline{\Sigma}_{i})^{n-1}_{[\leq l - 1]}
	\right).
\end{align*}

This gives the result, by Lemma~\ref{lemformulainduct}.
\end{proof}

\section{Morse inequalities} \label{sectmorseineq}

In this section, \(X\) will be a proper variety of dimension \(n\). In this situation, for any coherent sheaf \(\mathcal{F}\) on \(X\), the cohomology groups
\[
	H^{i}(X, \mathcal{F})
	\quad
	(i \geq 0)
\]
have a structure of finite dimensional \(\mathbbm{k}\)-vector spaces. We denote as usual
\[
	h^{i}(X, \mathcal{F}) = \dim_{\mathbbm{k}} H^{i}(X, \mathcal{F}).
\]

\begin{rem} Note that this definition takes into account the possibility that \(X\) might be defined over a finite extension of \(\mathbbm{k}\). For example, if \(X \cong \mathrm{Spec}\, \mathbb{K}\) associated with a finite field extension \(\mathbbm{k} \subset \mathbb{K}\), then the sheaf \(\mathcal{O}_{X}\) has rank one, while
	\begin{equation} \label{eq:valueh0}
		h^{0}(X, \mathcal{O}_{X})
		=
		\dim_{\mathbbm{k}} H^{0}(X, \mathcal{O}_{X})
		=
		\dim_{\mathbbm{k}}{\mathbb{K}}.
	\end{equation}
\end{rem}

\subsection{Truncated Euler characteristics} In essence, the algebraic Morse inequalities are a means to estimate the asymptotic behaviour of truncated Euler characteristics, which are defined as follows.  

\begin{defi} Let \(\mathcal{F}\) be a coherent sheaf on \(X\), and let \(l \in \llbracket 0, n \rrbracket\) be an integer. The {\em \(l\)-th truncated} Euler characteristic of \(\mathcal{F}\) is the integer
\[
	\chi^{[l]} (X, \mathcal{F}) = \sum_{j = 0}^{l} (-1)^{j+l} h^{j}(X, \mathcal{F}).
\]
\end{defi}

Note that in this definition, the number \(h^{l}(X, \mathcal{F})\) comes with a positive sign. A very easy but nonetheless crucial remark is that these numbers satisfy a subadditivity property with respect to short exact sequences, that generalizes the additivity property of the usual Euler characteristic.

\begin{lem} \label{lem:subadditivity}
	Consider the following exact sequence of coherent sheaves on \(X\):
	\[
	0
	\longrightarrow
	\mathcal{F}
	\longrightarrow
	\mathcal{E}
	\longrightarrow
	\mathcal{G}
	\longrightarrow
	0
	\]
	Then, for all integer \(l\), one has the inequalities:
	\[
		\chi^{[l]}(X, \mathcal{E})
		\leq
		\chi^{[l]}(X, \mathcal{F})
		+
		\chi^{[l]}(X, \mathcal{G}).
	\]
	and
	\[
		\chi^{[l]}(X, \mathcal{F})
		\leq
		\chi^{[l]}(X, \mathcal{E}) 
		+
		\chi^{[l-1]}(X, \mathcal{G})
	\]

\end{lem}

\begin{proof}
	Write down the long exact sequence in cohomology, and truncate it after the term \(H^{l}(X, \mathcal{G})\): for some finite dimensional vector space \(Z\), we get a resolution
	\[
		\dotsc
		\longrightarrow
		H^{l-1}(X, \mathcal{G})
		\longrightarrow
		H^{l}(X, \mathcal{F})
		\longrightarrow
		H^{l}(X, \mathcal{E})
		\longrightarrow
		H^{l}(X, \mathcal{G})
		\longrightarrow
		Z
		\longrightarrow
		0.
	\]
	The result then follows from writing \(\dim Z \geq 0\), and using that the Euler characteristic of the exact sequence above equals
	\[
		0
		=
		\dim Z 
		-
		\chi^{[l]}(X, \mathcal{G})
		+
		\chi^{[l]}(X, \mathcal{E})
		-
		\chi^{[l]}(X, \mathcal{F}).
	\]
	The second inequality follows from the same argument, writing this time:
	\[
		\dotsc
		\longrightarrow
		H^{l-1}(X, \mathcal{G})
		\longrightarrow
		H^{l}(X, \mathcal{F})
		\longrightarrow
		H^{l}(X, \mathcal{E})
		\longrightarrow
		Z
		\longrightarrow
		0.
	\]
\end{proof}

We can iterate the lemma above, and derive the following consequence for filtered sheaves on \(X\):

\begin{lem} \label{lem:filtration}
	Let \(\mathcal{E}\) be a coherent sheaf on \(X\), and let \(F^{\bullet} \mathcal{E}\) be a finite filtration on \(\mathcal{E}\) by subsheaves:
	\[
		0 = F^{s} \mathcal{E}
		\subset
		F^{s-1} \mathcal{E}
		\subset
		\dotsc
		\subset
		F^{0} \mathcal{E}
		=
		\mathcal{E}.
	\]
	Let \(\mathrm{Gr}^{F}_{\bullet}(\mathcal{E}) := \displaystyle\bigoplus_{j} \quotientd{F^{j} \mathcal{E}}{F^{j+1} \mathcal{E}}\) be the graded object. Then, for all \(l \in \llbracket 0, n \rrbracket\), we have the inequality~:
	\[
		\chi^{[l]}(X, \mathcal{E}) \leq \chi^{[l]}(X, \mathrm{Gr}_{F}(\mathcal{E})).	
	\]
\end{lem}
\begin{proof}
	For all \(j \in \llbracket 0, s-1\rrbracket\), the exact sequence
	\[
		0 
		\longrightarrow 
		F^{j+1} \mathcal{E} 
		\longrightarrow
		F^{j} \mathcal{E}
		\longrightarrow
		\mathrm{Gr}_{j}^{F}(\mathcal{E})
		\longrightarrow
		0
	\]
	gives the inequality
	\[
		\chi^{[l]}(F^{j} \mathcal{E}) - \chi^{[l]}(F^{j+1}\mathcal{E})
		\leq
		\chi^{[l]}(\mathrm{Gr}_{F}^{j}(\mathcal{E})),
	\]
	so we get the result summing for all possible \(j\).
\end{proof}

\subsection{Statement of the Morse inequalities}

We are now in position to state and prove the following algebraic version of the Morse inequalities.

\begin{thm}[Morse inequalities] \label{thmmorse} Let \(X\) be a proper variety of dimension \(n\). Let \(L\) be a \(\mathbb{Q}\)-line bundle on \(X\), and let \(\underline{\Sigma} := (\Sigma, \frac{1}{d} \mathbf e)\) be a trivialized stratification adapted to \(L\). Let \(M\) be another line bundle on \(X\). Then, for each integer \(i\), and for any \(m\) such that \(L^{\otimes m}\) is a line bundle, we have
	\medskip

\begin{enumerate}[label=(\roman*)]
\item (Strong Morse inequalities) 
	\[
		\chi^{[i]} (X, M \otimes L^{\otimes m}) 
		\leq (-1)^{i} 
		\left( 
		\deg c_{1}(L, \underline{\Sigma})^{n}_{[\leq i]} 
		\right) 
		\frac{m^{n}}{n!} + O(m^{n-1})
\]

\item (Weak Morse inequalities) 
\[
		h^{i} (X, M \otimes L^{\otimes m}) 
		\leq (-1)^i 
		\left( 
		\deg c_{1}(L, \underline{\Sigma})^{n}_{[i]} 
		\right) 
		\frac{m^n}{n!} + O(m^{n-1})
\]

\item (Asymptotic Riemann-Roch formula)
\[
		\chi(X, M \otimes L^{\otimes m}) 
		= 
		\left( 
		\deg c_{1}(L, \underline{\Sigma})^{n}_{[\leq n]} \right) 
		\frac{m^n}{n!} + O(m^{n-1})
\]
\end{enumerate}
\end{thm}

	\begin{rem}
		We included this truncated version of the asymptotic Riemann-Roch formula above for completeness, but of course this statement is none other than the usual, untruncated version (see e.g. Proposition~\ref{prop:truncatedversion}).
	\end{rem}

Before starting the proof, let us make the following simplifying remark.

\begin{lem}
	Theorem~\ref{thmmorse} is implied by the same statement, where \(L\) is assumed to be a line bundle.
\end{lem}
\begin{proof}
	Let \(f \in \mathbb{N}_{\geq 1}\) be the smallest positive integer such that \(L^{\otimes f}\) is a standard line bundle. Now, assume that the results holds for \(L^{\otimes f}\), with the trivialized stratification \(\Sigma' := (\Sigma, \frac{1}{d}\mathbf{e}^{\otimes f})\) (see Proposition~\ref{prop:powertriv}). Then, for any \(m \in \mathbb{N}\), one has
	\[
		\chi^{[i]}(X, M \otimes L^{\otimes f m})
		\leq 
		(-1)^{i}
		(\deg c_{1}(L, \underline{\Sigma}')^{n}_{[\leq i]}) \frac{m^{n}}{n!} + O(m^{n-1})
	\]
	Since \(c_{1}(L, \underline{\Sigma}')^{n}_{[\leq i]} = f^{n} c_{1}(L, \underline{\Sigma})^{n}_{[\leq i]}\) again by Proposition~\ref{prop:powertriv}, this gives the result by making the change of variables \(m \leftarrow fm\).
\end{proof}

\subsection{Proof in the case where \(L\) is a line bundle} \label{sec:proofmorselb}

As usual, the weak Morse inequalities follow from the strong ones, remarking that \(h^{i} = \chi^{[i]} + \chi^{[i-1]}\). Also, the asymptotic Riemann-Roch formula can be obtained from the strong Morse inequalities, using \(\chi = (-1)^n \chi^{[n]} = (-1)^{n+1} \chi^{[n+1]}\). Thus, it suffices to prove the first point. 
\medskip

\begin{rem}
	Let us make some comments on the strategy. To simplify the exposition, assume that \(X\) is normal.
	\begin{enumerate}
		\item The proof will be induction on \(\dim X\), the base case being essentially trivial. The general induction argument will be a standard telescopic argument of the sort that can be used to prove the asymptotic Riemann-Roch theorem (see e.g. the presentation of \cite[Section 1.2]{Deb01}). The idea is to write
			\[
				L = \mathcal{O}(A - B)
			\]
			where \(A - B\) is the divisor of poles and zeros of a trivialization \(e\) induced by \(\mathbf{e}\) on \(X\). Then, we can write short exact sequences such as in Lemma~\ref{lem:subadditivity} to get the relevant inequality (see \eqref{eqtelescopic} below). There are however two main technical issues.
		\item The first problem is that even if \(A - B\) is a Cartier divisor, its components are {\em a priori} merely Weil divisors. This leads to problems when trying to estimate the cohomology of \(\mathcal{O}_{A} \otimes L^{\otimes m}\), for example (our plan is to proceed by exhibiting a filtration on this sheaf, that works only if the divisor has been decomposed in a sum of Cartier divisors; see Lemma~\ref{lemfilt} below). For this reason, we might have to perform blowing-ups on \(X\).
		\item The second issue comes from the fact that our data is a priori adapted to \(L^{\otimes d}\) and not \(L\). We resolve this discrepancy by taking a cyclic cover on \(X\).
		\item Because of the clean-up steps required by the two points above, we naturally end up with diagrams of the form
\[
	\begin{tikzcd}
		& & X' \arrow[d, "p"] \\
		\dotsc \arrow[r, "f_{n-2}"] & X_{n-1} \arrow[r, "f_{n-1}"] & X
	\end{tikzcd}
\]
	where we want to prove the Morse inequalities for \(p^{\ast} L\) on \(X'\), with our stratification nonetheless still defined on \(X\). This relative situation forces us to use a slightly more general induction step, that will be given by Proposition~\ref{propredmorse}.
	\end{enumerate}
\end{rem}

\smallskip

We are now ready to give the proof of Theorem~\ref{thmmorse} when \(L\) is a line bundle. It will be complete with the following result.

\begin{prop} \label{propredmorse} Under the hypotheses of Theorem \ref{thmmorse}, assume moreover that \(L\) is a line bundle. Let \(p: X' \longrightarrow X\) be a proper dominant, generically finite morphism of degree \(\delta \in \mathbb{N}_{\geq 1}\), where \(X'\) is a variety, and let \(M'\) be any line bundle on \(X'\). Then, for each \(i \in \llbracket 0, n \rrbracket\), and any integer \(m\) divisible by \(d\), we have
\begin{equation} \label{eqmorseupperboundbir}
\chi^{[i]} (X', M' \otimes p^\ast L^{\otimes m}) 
	\leq 
	\delta (-1)^{i} 
	\left( 
	\deg c_{1}(L, \underline{\Sigma})^{n}_{[\leq i]} 
	\right) 
	\frac{m^n}{n!} + O(m^{n-1}).
\end{equation}
\end{prop}

The end of the section is devoted to the proof of this result. 
\medskip

\setcounter{stepcc}{-1}

\noindent
{\em \refstepcounter{stepcc} Step~\thestepcc. \label{step:0} We may assume that \(X\) and \(X'\) are normal varieties.} Let \(q : X_{n} \to X\) be the arrow provided by the stratification \(\Sigma\). Then form the square product
\[
	\begin{tikzcd}
		X_{n}' 
		\arrow[d, "p'"]
		\arrow[r, "q'"]
		\arrow[dr, phantom, "\square"]
		& 
		X' 
		\arrow[d, "p"] \\
		X_{n}  
		\arrow[r, "q"] 
		& 
		X
	\end{tikzcd}
\]
Since \(q\) is proper birational, so there is a unique component irreducible component \(Z \subset X_{n}'\) such that \(X_{n}'\) is reduced at the generic point of \(Z\), and such that \(d : Z \to X_{n}'\) is proper birational. Using Lemma~\ref{lem:compasympt} with \(X\) replaced with \(X'\) (resp. \(X'\) replaced with \(X_{n}'\), resp. \(X_{0}'\) replaced with \(Z\)), shows that

\[
	\chi^{[i]}(Z, (q')^{\ast} (M' \otimes p^{\ast} L^{\otimes m})) 
	= 
	\chi^{[i]}(X', M' \otimes p^{\ast} L^{\otimes m}) + O(m^{n-1})
\]
Moreover, the stratification \(\Sigma\) can be also seen trivially as a stratification on \(X_{n}\): the right hand side of \eqref{eqmorseupperboundbir} tautologically assumes the same value for this stratification. Thus, we see that it suffices to prove the result with \(X\) (resp. \(X'\)) replaced with \(X_{n}\) (resp. \(Z\)), to prove the theorem. This shows that it suffices to consider the case where \(X\) is normal. More precisely, we have the following reduction step.
\medskip

\begin{lem}
	Assumptions as in Proposition~\ref{propredmorse}. We may assume without loss of generality that \(X\) and \(X'\) are normal, and that the morphism \(q : X_{n} \to X\) is the identity.
\end{lem}

To prove this lemma, it only remains to show that \(X'\) can be supposed normal. But this follows from an application of Lemma~\ref{lemmodification}, using the fact that the normalization map \((X')^{\rm norm} \to X'\) is proper birational.
\medskip

Let us start the induction process.

\medskip

\noindent
\emph{\refstepcounter{stepcc} Step~\thestepcc. \label{step:initialization} Initialization of the induction.} If \(\dim X = 0\), then \(X \cong \mathrm{Spec}\, \mathbb{K}\) is a scheme point with \(\mathrm{deg}_{\mathbbm{k}}\, \mathbb{K} < + \infty\). This implies that \(X' = \bigsqcup_{j} X'_{j}\) is a union of scheme points, all finite over \(\mathrm{Spec}\, \mathbbm{K}\), and satisfying
\begin{equation} \label{eq:egalitedegre}
	\sum_{j} 
	\mathrm{deg}_{\mathbbm{k}}\, X_{j}'
	=
	(\sum_{j} 
	\mathrm{deg}_{\mathbbm{K}}\, X_{j}'
	)
	\cdot
	\mathrm{deg}_{\mathbbm{k}} \mathbb{K}
	=
	\delta \cdot \deg_{\mathbbm{k}} \mathbb{K}.
\end{equation}
In this case, for each \(m \in \mathbb{N}\), we have \(L' \otimes p^{\ast} L^{\otimes m} \cong \mathcal{O}_{X'}\) as \(\mathcal{O}_{X'}\)-modules, so 
	\begin{align*}
		\chi^{[i]} (X', M' \otimes p^\ast L^{\otimes m}) 
		& = (-1)^{i} h^{0} (X', M' \otimes p^\ast L^{\otimes m})  \\
		& = (-1)^{i} h^{0}(X, \mathcal{O}_{X'}) \\
		& = (-1)^{i} \delta \cdot \deg_{\mathbbm{k}}{\mathbb{K}}, 
	\end{align*}
	where at the last line, we used \eqref{eq:valueh0} and \eqref{eq:egalitedegre}.  
	On the other hand, 
	\[
		\deg c_1(L, \underline{\Sigma})^{0}_{[\leq 0]} \cap [X] 
		= \deg_{\mathbbm{k}} (X) 
		= \deg_{\mathbbm{k}} \mathbb{K},
	\]
	so the result holds in this case (one may even take \(O(m^{-1}) = 0\)).
	\medskip

	\noindent
\emph{\refstepcounter{stepcc} Step~\thestepcc. Setup for the induction step.} Suppose now that the result has been proved for \(\dim X \leq n-1\). Write \(\Sigma = (X_{\bullet},f_{\bullet})\), and let \(e\) be the trivialization of \(L^{\otimes d}\) over the open dense subset \(U_{n} = X_{n} \setminus f_{n-1}(X_{n-1})\) given by \(\mathbf{e}\).  
\medskip

	Denote by \((D_{i})_{1 \leq i \leq m}\) the irreducible components of \(f_{n-1}(X_{n-1})\), and let \(D = \sum_{1 \leq i \leq r} m_{i} D_{i}\) be the \emph{Cartier} divisor of zeros and poles of the section \(e\), seen as a rational section of \(L^{\otimes d}\). Note that in this sum, some of the \(m_{j}\) might actually be equal to \(0\). For each \(j \in \llbracket 1, r \rrbracket\), let \(W_{j}\) be the unique connected component of \(X_{n-1}\) such that \(f_{n-1}(W_{j}) = D_{j}\), and let \(q_{j} : W_{j} \longrightarrow X\) be the natural map.
\bigskip

\noindent
\emph{\refstepcounter{stepcc} Step~\thestepcc. We pass to a simpler model of $X'$.} The next two steps consist in doing some clean-up on the variety \(X'\), to later be able to do our induction step. First, we arrange the situation so that \(p^{\ast} L\) itself admits a trivialization, and not merely \(p^{\ast} L^{\otimes d}\).
\medskip

\begin{lem} \label{lem:covertriv}
	Assumptions as in Proposition~\ref{propredmorse}. We can assume without loss of generality that \(p^{\ast} L\) is trivial above \(p^{-1}(U_{n})\), with a trivialization \(e'\) such that \((e')^{\otimes d} = p^{\ast} e\). 
\end{lem}

\begin{proof}
	By Lemma \ref{lemmodification}, it suffices to prove Proposition \ref{propredmorse} with \(X'\) replaced by any proper dominant generically finite morphism \(q : X'' \longrightarrow X'\), \(M'\) replaced by \(q^\ast M'\), and \(p\) replaced by \(p \circ q\). Thus, it suffices to show that we can find a model \(X''\) on which \(q^{\ast} p^{\ast} L\) satisfies the property described above. Now, this can be obtained by a standard argument of cyclic covers: apply Lemma~\ref{lem:cycliccover} in Annex A, with \(X\) replaced with \(X'\), \(e\) replaced with \(p^{\ast} e\) and \(U\) replaced with \(p^{-1}(U_{n})\).
\end{proof}

Now, we will transform the inverse images of the \(D_{i}\) into Cartier divisors, in a suitable manner. We introduce the following terminology:
\medskip

\begin{defi} \label{defi:adapteddivisor}
	Let \(p : Y' \to Y\) be a dominant morphism between two varieties, and let \(D'\) be a {\em Cartier} divisor on \(Y'\). We say that \(D'\) is {\em adapted} to a prime Weil divisor \(D \subset Y\) if we have \(D' \subset p^{-1}(D)\) and if we can write
	\[
		D' = \widehat{D}' + E
	\]
	where \(\widehat{D}'\) is an irreducible Weil divisor such that \(p : \widehat{D}' \to D\) is dominant (and thus generically finite), and \(E\) is a \(p\)-exceptional sum of Weil divisors.
\end{defi}

\begin{rem}
	Note that in particular, \(D'\) is reduced near the generic point of \(\widehat{D}'\).
\end{rem}

We can now state our next clean-up step as follows.

\begin{lem}
	Assumptions as in the conclusion of Lemma~\ref{lem:covertriv}. We can assume without loss of generality that we have an equality 
	\[
		D(e')
		=
		\sum_{1 \leq j \leq r'}
		m_{j}' D_{j}'
		+
		E
		-
		F,
	\]
	where:
	\begin{enumerate}
		\item each \(D_{j}'\) is an effective Cartier divisor on \(X'\), {\em adapted} to exactly one component \(D_{i}\) of \(D\);
		\item \(E\) and \(F\) are {\em effective, \(p\)-exceptional Cartier} divisors.
	\end{enumerate}
\end{lem}

\begin{proof}
	Again, it suffices to find a dominant morphism \(q : X'' \to X'\) such that \(D(q^{\ast} e')\) assumes the form above. Note that the reduction of Lemma~\ref{lem:covertriv} will still hold on \(X''\).
	\smallskip

	On \(X'\), we may already write	
	\[
		D(e') = \sum_{1 \leq j \leq r} m_{j}' D_{j}' + E - F,
	\]
	where the \(D_{j}'\) are effective irreducible {\em Weil} divisors each dominating a component of \(D\), and \(E\), \(F\) are effective \(p\)-exceptional {\em Weil} divisors. To obtain the morphism \(X'' \to X'\), it suffices to apply Lemma~\ref{lem:WeiltoCartier} to \(X'\), with \(D\) replaced with \(D(e')\), and the \(D_{i}\) replaced with the \(D_{j}'\) together with the irreducible components of \(E\) and \(F\).
\end{proof}

Denote by \(\widehat{D}_{j}'\) the unique component of \(D_{j}'\) that dominates one of the \(D_{i}\), accordingly to Definition~\ref{defi:adapteddivisor}. Note that we have the projection formula 
	\[
		p_{\ast} (d\, D(e'))
		= p_{\ast} D(p^{\ast} e) 
		= \delta D
	\]
	which yields, for all \(i \in \llbracket 1, r\rrbracket\): 
	\begin{align} \label{eqsumproj1}
		\sum_{j} m_{j}' 
		\, 
		\deg (\widehat{D}_{j}' / D_{i}) 
		& = \frac{\delta}{d}\, m_{i},	
	\end{align}
where the sum runs among all $j \in \llbracket 1, r'\rrbracket$ such that $D_{j}'$ dominates \(D_{i}\).
\medskip

\noindent
\emph{\refstepcounter{stepcc} Step~\thestepcc. \label{step:boundAB} We bound from above the difference of two consecutive \(\chi^{[i]}(M' \otimes p^{\ast} L^{\otimes m})\)}. Letting 
\begin{equation} \label{eq:defA}
	A = \sum_{m_{j}' > 0} 
	m_{j}' D_{j}' + E 
\end{equation} 
and 
\begin{equation} \label{eq:defB}
	B 
	= 
	\sum_{m_{j}' < 0} (-m_{j}') D_{j}' + F,
\end{equation}
we obtain the two exact sequences

\begin{tikzcd}[column sep=1em]
 	0 
		\ar[r]  
	& M' \otimes p^{\ast} L^{\otimes(m+1)} \otimes \mathcal O_{X'}(-A) 
		\ar[r] \ar[d, equal] 
	& M' \otimes p^\ast L^{\otimes(m+1)} 
		\ar[r] 
	&  M' \otimes p^\ast L^{\otimes(m+1)} \otimes \mathcal O_{A} 
		\ar[r] 
	&  0 \\
 	0 
		\ar[r] & 
	  M' \otimes p^\ast L^{\otimes m} \otimes \mathcal O_{X'}(-B) 
	  	\ar[r] 
	& M' \otimes p^\ast L^{\otimes m} 
		\ar[r] 
	& M' \otimes p^\ast L^{\otimes m} \otimes \mathcal O_{B} 
		\ar[r] 
	&  0 
\end{tikzcd}
\medskip

Applying the first inequality of Lemma~\ref{lem:subadditivity} to the first line gives

\begin{align*}
	\chi^{[i]} (X',  M' \otimes p^\ast L^{\otimes(m+1)}) 
	\; \leq \;
	& \chi^{[i]} (X', M' \otimes p^\ast L^{\otimes(m+1)} \otimes \mathcal O_{X'}(-A)) 
	  + \chi^{[i]}(X', M' \otimes p^\ast L^{\otimes(m+1)} \otimes \mathcal O_A)
\end{align*}
\medskip

Similarly, the second inequality of the same lemma, applied to the second line yields:
\begin{align*}
	- \chi^{[i]} & (X',  M' \otimes p^\ast L^{\otimes m}) 
	           & \leq - \chi^{[i]} (X', M' \otimes p^\ast L^{\otimes m} \otimes \mathcal O_{X'}(-B)) 
		   + \chi^{[i - 1]}(X', M' \otimes p^\ast L^{\otimes m} \otimes \mathcal O_{B})
\end{align*}
\medskip

Summing these two equations, we obtain
\begin{align} \label{eqtelescopic}
	\chi^{[i]} (X',  M' \otimes p^{\ast} L^{\otimes(m+1)})  - & \chi^{[i]} (X',  M' \otimes p^{\ast} L^{\otimes m}) \\ \nonumber
	& \leq \chi^{[i]}(X', M' \otimes p^{\ast} L^{\otimes(m+1)} 
	\otimes 
	\mathcal O_{A}) 
	+ \chi^{[i - 1]}(X', M' \otimes p^\ast L^{\otimes m} \otimes \mathcal O_{B})
\end{align}
\medskip

\noindent
{\em \refstepcounter{stepcc} Step~\thestepcc. \label{step:5} "D\'{e}vissage" on \(\mathcal{O}_{A}\).} Since we have the identity of \emph{Cartier} divisors \eqref{eq:defA} on the variety \(X'\), Lemma~\ref{lemfilt} gives a filtration of \(\mathcal{O}_{A}\) by \(\mathcal{O}_{X'}\)-modules as
\[
	0 = \mathcal F_0 \subseteq ... \subseteq \mathcal{F}_{N} = \mathcal O_{A},
\]
where \(N = \sum_{m_{j}' > 0} m_{j}' + 1\). The graded terms are given by short exact sequences
\begin{equation} \label{eq:SES}
	0 
	\longrightarrow 
	\mathcal{F}_{j} 
	\longrightarrow 
	\mathcal F_{j+1} 
	\longrightarrow 
	\mathcal{L}_{j}	
	\longrightarrow 0.
\end{equation}
where \(\mathcal{L}_{j}\) is an invertible sheaf on either one of the effective Cartier divisors \(D_{j}' \subset A\) (each such divisor appearing exactly \(m_{j}'\) times) or on \(E\) (appearing exactly once). Note that the filtration does not depend on \(m\).
\smallskip

\noindent
\emph{\refstepcounter{stepcc} Step~\thestepcc. \label{step:6} We write an upper bound on the right hand side of \eqref{eqtelescopic}}.
Using Lemma~\ref{lem:filtration} on the filtration discussed in the previous step, we obtain 
\begin{align} \nonumber 
	\chi^{[i]}
		(X', M' \otimes q^{\ast} L^{\otimes(m+1)} \otimes \mathcal O_{A}) 
		& \leq \sum_{j} \; \sum_{1 \leq k \leq m'_{j}}
		\,
			\mathbf{1}_{[m_{j}' > 0]}
			\;
			\chi^{[i]}(D_{j}', \;
			\mathcal{L}_{j,k}^{D}
			\otimes 
			p^\ast L^{\otimes (m+1)} \otimes M') \\ \label{eqstepA} 
		& + 
			\chi^{[i]}(E, 
				\mathcal{L}^{E} \otimes q^\ast L^{(m+1)} \otimes M')
\end{align}
where the indicator functions are used to make the first sum run through the indexes \(j\) such that $m_{j}' >0$. The sheaves \(\mathcal{L}^{D}_{\bullet,\bullet}\) (resp. \(\mathcal{L}^{E}\)) are the invertible sheaves \(\mathcal{L}_{j}\) appearing in the exact sequences \eqref{eq:SES}. Again, neither these sheaves nor the number of terms in each sum depend on \(m\).
\medskip

Note that since \(E\) is \(p\)-exceptional, we have :
\[
	\chi^{[i]}(E, \mathcal{L}^{E} \otimes p^\ast L^{(m+1)} \otimes M') = O(m^{n-2}).
\]

We may also also apply the second equality of Lemma~\ref{lem:compasympt} to each term of the first sum (with \(X_{0}\) replaced with \(\widehat{D}_{j}'\)). This gives, for all \(j\) : 
\[
	\chi^{[i]}(D_{j}', \;
			\mathcal{L}_{j,k}^{D}
			\otimes 
			p^\ast L^{\otimes (m+1)} \otimes M')
	=
	\chi^{[i]}(\widehat{D}_{j}', \;
			\mathcal{L}_{j,k}^{D}
			\otimes 
			p^\ast L^{\otimes (m+1)} \otimes M')
	+
	O(m^{n-2}),
\]
where in the second term, the restriction to \(\widehat{D}_{j}' \subset D_{j}'\) is implied.
\medskip

Since the number of terms in the sum does not depend on \(m\), we may gather all the \(O(m^{n-2})\) terms, and obtain
\begin{align} \nonumber 
	\chi^{[i]}
		(X', M' \otimes p^{\ast} L^{\otimes(m+1)} \otimes \mathcal O_{A}) 
		 \leq \sum_{j} \; \sum_{1 \leq k \leq m'_{j}} 
		& \,
			\mathbf{1}_{[m_{j}' > 0]}
			\;
			\chi^{[i]}(\widehat{D}_{j}', \;
			\mathcal{L}_{j,k}^{D}
			\otimes 
			p^\ast L^{\otimes (m+1)} \otimes M') \\ \label{eqstepA1} 
		& + O(m^{n-2});
\end{align}

By a completely parallel argument following from the expression \eqref{eq:defB}, one has as well:

\begin{align} \nonumber 
	\chi^{[i-1]}
		(X', M' \otimes p^{\ast} L^{\otimes(m+1)} \otimes \mathcal O_{B}) 
		 \leq \sum_{j} \; \sum_{1 \leq k \leq (- m'_{j})} 
		& \,
			\mathbf{1}_{[m_{j}' < 0]}
			\;
			\chi^{[i-1]}(\widehat{D}_{j}', \;
			\mathcal{L}_{j,k}^{D}
			\otimes 
			p^\ast L^{\otimes (m+1)} \otimes M') \\ \label{eqstepB1} 
		& + O(m^{n-2});
\end{align}

\noindent
{\em \refstepcounter{stepcc} Step~\thestepcc. \label{step:7} We use the induction hypothesis.}  For each component \(D_{j}'\), the morphism \(p\) induces a projection \(\widehat{D}_{j}' \longrightarrow D_{l}\) to a unique component of \(D\): in the equations below, the dependence of \(l\) in the index \(j\) will only be implied. All these projection maps are generically finite dominant morphisms, so we can apply the induction hypothesis to each $D_{j}'$: for any $l \in \llbracket 1, r \rrbracket$, let $\underline{\Sigma}_{l}$ be the trivialized stratification induced on $D_{l}$ by $\underline{\Sigma}$, and let \(q_{l} : D_{l} \to X\) be the natural map. Then, the induction hypothesis with \(M'\) replaced with the \(M' \otimes \mathcal{L}_{j, k}^{D}\), gives for each term in the sums appearing in either \eqref{eqstepA1} or \eqref{eqstepB1} :
\[
	\chi^{[\bullet]}(\widehat{D}_{j}',
	p^{\ast} L^{\otimes (m+1)} \otimes M' \otimes \mathcal{L}_{j,k}^{D})
	\quad \leq \quad
	(-1)^{\bullet} \mathrm{deg}(\widehat{D}_{j}'/D_{l}) \; c_1(q^{\ast}_{l} L, \underline{\Sigma}_{l})^{n-1}_{[\leq \bullet]}
	\frac{m^{n-1}}{(n-1)!} + O(m^{n-2}),
\]
where the constant in the \(O(m^{n-2})\) can be chosen uniformly for all possible \(D_{j}'\). Once more, let us emphasize the fact that neither these divisors nor the line bundles \(M' \otimes \mathcal{L}_{j,k}^{D}\) depend on \(m\).
\smallskip

We may now combine this set of inequalities with \eqref{eqtelescopic}, \eqref{eqstepA1} and \eqref{eqstepB1}, which yields :
\begin{align*}
	\chi^{[i]} & (X',  M' \otimes p^{\ast} L^{\otimes(m+1)})  
	-  
	\chi^{[i]} (X',  M' \otimes p^{\ast} L^{\otimes m}) 
	\\
	& \leq \left(
		\sum_{m_{j}' > 0} m_{j}'\,  
		(-1)^{i} \mathrm{deg}(\widehat{D}_{j}'/D_{l}) \; 
		c_{1}(q_{l}^{\ast} L, \underline{\Sigma}_{l})^{n-1}_{[\leq i]} 
		+  
		\sum_{m_{j}' < 0} (-m_{j}')\,  
		(-1)^{i-1}
		\mathrm{deg}(\widehat{D}_{j}'/D_{l}) \; 
		c_1(q_{l}^{\ast} L, \underline{\Sigma}_{l})^{n-1}_{[\leq i - 1]} \right) 
		\frac{m^{n-1}}{(n-1)!}  \\
	& \hspace{5em}+ O(m^{n-2}),
\end{align*}
where for all \(j\), we have again denoted by \(D_{l}\) the component of \(D\) dominated by \(\widehat{D}_{j}'\).
\smallskip

Now, \eqref{eqsumproj1} give \(\sum_{j} {m_{j}'} \mathrm{deg}(\widehat{D}_{j}'/D_{l}) = \frac{\delta}{d} m_{l}\), where the sum runs among all \(j\) such that \(\widehat{D}_{j}'\) dominates a fixed component \(D_{l}\). Summing the telescopic sum finally yields:
\begin{align} \nonumber
		\chi^{[i]} &  (X',   M' \otimes p^\ast L^{\otimes(m+1)})  \\ \nonumber
	& \leq  \delta \, (-1)^i \left( \sum_{m_{l} > 0} \frac{m_{l}}{d}\, \deg c_1(q_{l}^{\ast} L, \underline{\Sigma}_{l})_{[\leq i]}  \label{eq:telescop}
	-  \sum_{m_{l} < 0} (-\frac{m_{l}}{d})\,   \deg c_1{}(q_{l}^{\ast} L, \underline{\Sigma}_l)_{[\leq i -1]} \right) \sum_{j \leq m} \frac{j^{n-1}}{(n-1)!} \\ \nonumber
&  \hspace{50pt}  + \sum_{j \leq m} O(l^{n-2}).\\
\end{align}

Since $\sum_{1 \leq j \leq m} \frac{j^{n-1}}{(n-1)!} = \frac{m^n}{n!} + O(m^{n-1})$ and $\sum_{l \leq m} l^{n-2} = O(m^{n-1})$, the conclusion then comes immediately from Lemma \ref{lemformulainduct}.

\medskip

\begin{lem} \label{lemfilt} Let \(X\) be an integral scheme on \(X\), and let \(D_{1}, ..., D_{r}\) be Cartier divisors on \(X\). For all \(i \in \llbracket 1, r\rrbracket\), let \(m_{i} \in \mathbb N\), and define \(D = \sum_{i} m_{i} D_{i}\). 
	
	Then there exists a filtration 
	\(
	\mathcal F_{1} \subseteq ... \subseteq \mathcal F_{N} = \mathcal O_{D}
	\)
	(where \(N = \sum_{i} m_{i}\)), with successive quotients given as follows:
	\[
	0 
	\longrightarrow 
	\mathcal{F}_{i} 
	\longrightarrow 
	\mathcal F_{i+1}  
	\longrightarrow 
	\mathcal{L}_{i, j} 
	\longrightarrow 
	0,
	\]
	where \(j \in \llbracket 1, r \rrbracket\) is some index depending on \(i\), and \(\mathcal{L}_{i,j}\) is an invertible \(\mathcal{O}_{D_{j}}\)-module, seen as a torsion \(\mathcal{O}_{X}\)-module. Each component \(D_{j}\) appears exactly \(m_{j}\) times in the graded terms of this filtration.
\end{lem}

\begin{proof}
	Assume first that \(X = \mathrm{Spec}\, A\), where \(A\) is an integral ring, and that each \(D_{i}\) is given by some equation \(f_{i} \in A\). The \(\mathcal{O}_{X}\)-module \(\mathcal{O}_{D}\) is induced by the \(A\)-module 
	\[
		M := \quotientd{A}{(f_{1}^{m_{1}} \dotsc f_{r}^{m_{r}})}.
	\]	
	Choose an ordering of \(\prod_{1 \leq j \leq r} \llbracket 0 , m_{j}-1 \rrbracket\), and define accordingly the different \(\mathcal{F}_{i} \subset \mathcal{O}_{D}\) so they are associated to the following submodules of \(M\): 
\[
	M_{a_1, ..., a_r} 
		= \quotientd{(f_1^{a_1} ... f_r^{a_r})}{(f_1^{m_1} ... f_r^{m_r})}
\]
	with \(0 \leq a_i \leq m_i\) for all \(i\). Note that since \(A\) is integral, we indeed have \(M_{a_1, ..., a_r} \hookrightarrow M_{b_1, ..., b_r}\) if \(a_{i} \geq b_{i}\) for all \(i\). 
	
	The local result follows from the following natural isomorphisms of \(A\)-modules:
	\[
		\begin{array}{ccc}
		\quotientd{M_{a_{1}, ..., a_{i}, ..., a_{r}}}{M_{a_{1}, ..., a_{i} + 1, ..., a_{r}}} 
			&
			\longrightarrow
			&
			\quotientd{A}{(f_{i})} \\
			{[}f_{1}^{a_{1}} \dotsc f_{r}^{a_{r}}{]}
			&
			\longmapsto
			& 
			1	
		\end{array}.
	\] 
	
	For the global case, note that the filtration is actually canonical once an ordering of all possible indexes \((b_{1}, \dotsc, b_{r})\) has been chosen. Indeed, this filtration does not change if we replace the \(f_{i}\) by \(g_{i} f_{i}\), with \(g_{i}\) invertible. This shows that the local definitions will glue together on any integral scheme.
\end{proof}

\begin{rem}
	Before ending this section, we want to emphasize the importance of stating Proposition~\ref{propredmorse} without necessarily assuming that \(X\) is normal, since in our induction steps, the divisorial subschemes \(D_{j} \subset X\) may well be non-normal.
	\medskip

	However, it is crucial that \(X\) is normal to be able to define the multiplicities of a meromorphic section, hence the necessity of Step~\ref{step:0} at the beginning of the proof.
\end{rem}

\section{An example : the classical algebraic Morse inequalities} \label{sec:example}

It is interesting to remark that the classical algebraic Morse inequalities of Demailly and Angelini can be retrieved now as a particular case of Theorem \ref{thmmorse}.
\medskip

{\em In the following theorem, we assume that \(\mathbbm{k}\) to be {\em algebraically closed} for simplicity, even though this hypothesis is not strictly necessary.}

\begin{thm}[Demailly \cite{dem96}, Angelini \cite{ang96}] \label{thm:demang} Let $X$ be projective variety of dimension $n$ over an algebraically closed field \(\mathbbm{k}\), and let $L$ be a line bundle on $X$. Assume $L \sim_{\mathbb{Q}} \mathcal O(F - G)$, where $F, G$ are nef \(\mathbb{Q}\)-Cartier divisors. Then, for any $i \in \llbracket 0, n \rrbracket$, and \(m\) divisible enough, we have
\[
	\limsup_{m \to \infty}
	\,
	\frac{\chi^{[i]} (X, L^{\otimes m})}{m^{n}/n!}
	\quad
	\leq 
	\quad
	\sum_{ 0 \leq j \leq i} (-1)^{i + j} \binom{n}{j} F^{n-j} \cdot G^j.
\]
\end{thm}

\begin{rem} If \(\mathbbm{k}\) is not algebraically closed, it is not hard to see that the result still holds: simply take the base change to an algebraic closure \(\mathbbm{k} \subset \mathbbm{K}\) and use the invariance of intersection numbers and the cohomology dimensions.
\end{rem}

Before proving this result, we will take some time to exhibit a type of construction of stratifications adapted to a given ample line bundle, that will be quite useful in the rest of this text.

\subsection{Stratifications adapted to an ample line bundle} Taking hyperplane sections, it is not difficult to construct sections adapted to a given ample line bundle:

\begin{prop} \label{prop:conststratample}
	Let \(X\) be a projective variety, endowed with an ample line bundle \(L\). Then there exists a fractional trivialized stratification \(\underline{\Sigma} = (\Sigma, \frac{1}{d} \mathbf{e})\) for \(L\) such that the multiplicities of the trivializations are all positive. We have, for all \(i \in \llbracket 0, n \rrbracket\):
	\[
		\deg c_{1}(L, \underline{\Sigma})^{n}_{[\leq i]}
		=
		(L^{n}).
	\]
\end{prop}
\begin{proof}
	We construct the stratification by induction on \(\dim X\), the case \(\dim X = 0\) being empty. First of all, let \(q_{n} : X_{n} \to X\) be the normalization of \(X\), and let \(L_{n} := q^{\ast} L\). Since \(q_{n}\) is finite, \(L_{n}\) is ample, and we can find \(d \in \mathbb{N}\) be such that \(L_{n}^{\otimes d_{0}}\) is very ample. Let \(D_{n}\) be a hyperplane section in \(|L_{n}^{\otimes d_{0}}|\), associated to a section \(e_{n} \in \Gamma(X_{n}, L_{n}^{\otimes d})\). The multiplicity of \(e_{n}\) along each component of \(D_{n}\) is positive.
	\medskip

	Now, apply the induction hypothesis to all the irreducible components of \(D_{n}\). For each such component \(D'\), we get a trivialized stratification on \(D'\) for some power \(L_{n}|_{D'}^{\otimes d_{D'}}\), with only positive multiplicities. Putting all this stratifications together, we get the requested \(\Sigma\). Finally, let
	\[
		d := \mathrm{lcm}(d_{0}, \{ d_{D'}\}_{D' \subset D_{n}}),
	\]
	and replace each trivialization by the adequate power to get a trivialization for \(L^{\otimes d}\). 
	\medskip

	Since all multiplicities are positive, we have
	\[
		\deg c_{1}(L, \underline{\Sigma})^{n}_{[\leq i]}
		=
		\deg c_{1}(L, \underline{\Sigma})^{n}_{[\leq n]}
	\]
	for all \(i \in \llbracket 0, n \rrbracket\). We now get the result since the latter term is always equal to \((L^{n})\) (see Proposition~\ref{prop:truncatedversion}).
\end{proof}

In the next section, we will actually use the following variant, that can be proved in exactly the same manner.

\begin{prop} \label{prop:stratampleseveral}
	Let \(X\) be a projective variety, endowed with a finite set of ample line bundles \(L_{1}, \dotsc, L_{r}\). Then there exists a stratification \(\Sigma\) and fractional trivializations \(\frac{1}{d_{i}}\mathbf{e}_{i}\) for each \(L_{i}\) (for \(i = 1, \dotsc, n\)), satisfying the following:
	\begin{enumerate}[label=(\roman*)]
		\item the multiplicities of the trivializations along the components of \(\Sigma\) are nonnegative;
		\item we may assume that for each component \(V\) in \(\Sigma\), the multiplicities of the \(\frac{1}{d_{i}} \mathbf{e}_{i}\) along \(V\) are zero for all but one \(i \in \llbracket 1, r\rrbracket\);
		\item  for all \(j \in \llbracket 0, n \rrbracket\) and all \(i \in \llbracket 1, r \rrbracket\).
	\begin{equation} \label{eq:identityjn}
		c_{1}(L_{i}, (\Sigma, \frac{1}{d_{i}} \mathbf{e}_{i}))_{[\leq j]}^{n}
		=
		c_{1}(L_{i}, (\Sigma, \frac{1}{d_{i}} \mathbf{e}_{i}))_{[\leq n]}^{n}
	\end{equation}
	\item For each \(i \in \llbracket 1, r \rrbracket\), we consider the endomorphism \(c_{i} := c_{1}(L_{i}, \frac{1}{d_{i}} \mathbf{e}_{i})_{[\leq n]} \in \mathrm{End}\left((Z_{\bullet}^{\Sigma}(X)_{\mathbb{Q}}\right)\). Now, for each sequence of indexes \(j_{1}, \dotsc, j_{n}\) in \(\llbracket 1, r \rrbracket\), we have
	\begin{equation} \label{eq:endomorphisms}
		\deg (c_{j_{1}}	\cap c_{j_{2}} \cap \dotsc \cap c_{j_{n}} \cap [X])
		= 
		(L_{j_{1}} \cdot L_{j_{2}} \cdot \dotsc \cdot L_{j_{n}}).
	\end{equation}
	\end{enumerate}
\end{prop}
\begin{proof}
	The construction of the trivialized stratification is the same as in Proposition~\ref{prop:conststratample}, except at the induction step, we now take \(D_{n}\) to be the union of supports of trivializing sections for very ample powers  of the pull-backs of the \(L_{1}, \dotsc, L_{r}\). Item {\em (i)} follows immediately.
	\medskip

	\noindent
	{\em (ii)} To see this, it suffices to use the following easy claim at the induction step. 

	\begin{claim} Let \(X\) be a variety over an algebraically closed field \(\mathbbm{k}\), and let \(N_{1}, \dotsc, N_{r}\) be very ample divisors on \(X\). Then there exists Weil divisors \(D_{1}, \dotsc, D_{r}\) on \(X\) such that each \(U_{i} := X - D_{i}\) is an affine open subset on which \(N_{i}\) is trivial. We have
		\[
			D_{i} \cap U_{j} \neq \varnothing
		\]
		for each \(i \neq j\).
	\end{claim}
	
	\noindent
	{\em (iii)} The equality \eqref{eq:identityjn} follows immediately from the fact that all multiplicities are positive.
	\smallskip

	\noindent
	{\em (iv)} This follows from an inductive application of \eqref{eq:contributionschern}.
\end{proof}

\subsection{Proof of Demailly-Angelini Morse inequalities}

\begin{proof}[Proof of Theorem~\ref{thm:demang}]
\noindent
{\em Step 0. From nef to ample.} Replacing $F$ (resp. $G$) by $F + \epsilon A$ (resp. $G + \epsilon A$ with \(\epsilon \in \mathbb{Q}_{>0}\) small and $A$ ample divisor does not change the left hand side of the equality, while the right hand side can be retrieved by letting \(\epsilon \to 0\). Thus we can assume that both $F$ and $G$ are ample.
	\medskip

\noindent
{\em Step 1. Stratifications and trivializations.} Apply Proposition~\ref{prop:stratampleseveral} to \(X\), endowed with the two ample line bundles \(F\) and \(G\). This gives a stratification \(\Sigma\) on \(X\), integers \(d_{F}\) and \(d_{G}\), and fractional trivializations \(\frac{1}{d_{F}} \mathbf{e}_{F}\) and \(\frac{1}{d_{G}} \mathbf{e}_{G}\) on \(\Sigma\) for \(F\) and \(G\) respectively. We may replace \(d_{F}\) and \(d_{G}\) by \(\mathrm{lcm}(d_{F}, d_{G})\) to assume that \(d_{F} = d_{G}\). Let us call \(d\) this common power.

	Letting
	\[
		\mathbf{e} := \mathbf{e}_{F} \otimes \mathbf{e}_{G}^{-1},
	\]
	we get a fractional trivialization \(\frac{1}{d} \mathbf{e}\) for \(\mathcal{O}(F - G)\). Let \(\underline{\Sigma} := (\Sigma, \frac{1}{d}\mathbf{e})\) (resp. \(\underline{\Sigma}_{F} := (\Sigma, \frac{1}{d}\mathbf{e}_{F})\), resp. \(\underline{\Sigma}_{G} := (\Sigma, \frac{1}{d}\mathbf{e}_{G})\)).
	\medskip

{\em Step 2. Computation.} Now, by Item {\em (ii)} of Proposition~\ref{prop:stratampleseveral}, we have
	\begin{equation} \label{eq:distinction}
		c_{1}(L, \underline{\Sigma})_{[i]} = 
		\left\{
			\begin{array}{cc}
				c_{1}(\mathcal{O}(F), \underline{\Sigma}_{F})_{[0]}
				\quad
				\text{if} \; i = 0 \\
				- c_{1}(\mathcal{O}(G), \underline{\Sigma}_{G})_{[0]}
				\quad
				\text{if} \; i = 1 \\
			\end{array}
		\right.
	\end{equation}
	as endomorphisms of \(Z(X)_{\mathbb{Q}}^{\Sigma}\).
	\medskip

	Consequently, iterating Definition \ref{defitrunchigherpowers} yields, for any $j \in \llbracket 0, n \rrbracket$: 
\begin{align*}
	c_{1}(L, \underline{\Sigma})^{n}_{[j]} \cap [X] & 
	= \sum_{\overset{T \in \llbracket 1, n \rrbracket}{|T| = j}}  
	c_{1}(L, \underline{\Sigma})_{[\mathbbm{1}_{1 \in T}]} 
	\cap ... \cap  
	c_{1}(L, \underline{\Sigma})_{[\mathbbm{1}_{n \in T}]} \cap [X] \\
	& = \sum_{\overset{T \in \llbracket 1, n \rrbracket}{|T| = j}} ( c_1(F))^{n - |T|} \cap (-c_1(G))^{|T|} \cap [X] \\
& = (-1)^{j} \binom{n}{j} F^{n- j} \cdot G^j
\end{align*}
	in $A_0(X)$ (here $\mathbbm{1}_{j \in T}$ is equal to $1$ if $j \in T$, and to $0$ otherwise). To pass from the first line to the second line, we used jointly \eqref{eq:distinction} and \eqref{eq:endomorphisms}.
	\medskip
	
	By Theorem \ref{thmmorse}, this gives the result.
\end{proof}

\chapter{Morse inequalities and weighted symmetric products} \label{chap:morseineqweighted}

The main theorem of this chapter is an asymptotic estimate on the growth of quantities of the form \(\chi^{[i]} (X, N^{\otimes m} \otimes S^m\, (\mathbf{E}))\), where \(\mathbf{E}\) is a \emph{weighted direct sum} of line bundles, and \(N\) is an arbitrary \(\mathbb{Q}\)-line bundle; it will be the central step in our proof of existence of jet differentials on varieties of general type. The main strategy will be to proceed as in the proof of Proposition~\ref{prop:symmetricsum}, but substituting the use of the asymptotic Riemann-Roch theorem with the algebraic Morse inequalities of the previous chapter.
\smallskip

{\em In this chapter, we do not assume that \(\mathbbm{k}\) is algebraically closed.}

\section{Notation and statement of the result} \label{sectstatement}

We refer to Section~\ref{sec:notationlattices} for some basic notation about simplexes.

\medskip

\begin{defi} \label{defisimplex}
	For all \(\underline{a} = (a_1, \dotsc, a_r) \in \mathbb N_{\geq 1}^{r}\), we define the \((r-1)\)-dimensional simplex 
	\[
		\Delta_{\underline{a}} 
		= 
		\{ (t_1, .., t_r) \in \mathbb{R}^{r} 
		\; | \; 
		\sum_{i} a_{i} t_{i} = 1 
		\} 
		\subseteq 
		\mathbb{R}^{r}.
	\]
	For any \(m \in \mathbb{N}\), we will simply denote by \(\Delta^{m}\) the \(m\)-dimensional simplex \(\Delta_{(1, \dotsc, 1)}\) (\(1\) repeated \(m+1\) times).
\end{defi}

\begin{setup} \label{setup:chapter3}
We fix a projective variety \(X\) of dimension \(n\), and we let \(L_{1}, \dotsc, L_{r}\) be \(\mathbb{Q}\)-line bundles over \(X\). Let \(\Sigma\) be a stratification adapted to all \(L_{1}, \dotsc, L_{r}\), and for each $i \in \llbracket 1, r \rrbracket$, choose a fractional trivialization $\frac{1}{d_{i}} \mathbf e_i$ of $L_i$ over $\Sigma$, where \(d_{i}\) is such that \(L_{i}^{\otimes d_{i}}\) is a standard line bundle. Let $\mathcal T$ be the tree associated to $\Sigma$. For each edge $\mathfrak{s}$ in $\mathcal T$, denote by $m_{i}^\mathfrak{s}$ the marking of $\mathfrak{s}$ associated to the trivialization $\mathbf e_i$. 
\end{setup}

\begin{rem}
	We have chosen to state the results of this chapter for a projective \(X\) (as opposed to proper) for some technical reasons, e.g. the need to take Bloch-Gieseker coverings. 
\end{rem}

\subsection{Symmetric powers for \(\mathbb{Q}\)-line bundles} \label{eq:symmqline} In order to state our main result, we will need to make sense of the notation \(S^{m}(L_{1}^{(a_{1})} \oplus \dotsc \oplus L_{r}^{(a_{r})})\), which was introduced when all the \(L_{i}\) are standard line bundles, but needs an additional assumption to be properly defined for \(\mathbb{Q}\)-line bundles. Let us give some details.
\medskip

\begin{nota}
	Data as in Setup~\ref{setup:chapter3}. Denote by \(T(X) \subset \mathrm{Pic}(X)\) the subgroup of torsion elements. Then we have an exact sequence
	\[
		0
		\longrightarrow
		T(X)
		\longrightarrow
		\mathrm{Pic}(X)
		\longrightarrow
		\mathrm{Pic}_{\mathbb{Q}}(X)
		\longrightarrow
		0.
	\]
	Denote by \(P(X) \subset \mathrm{Pic}_{\mathbb{Q}}(X)\) the image of \(\mathrm{Pic}(X)\). Let \(\Lambda \subset \mathrm{Pic}_{\mathbb{Q}}(X)\) be the subgroup generated by \(L_{1}, \dotsc, L_{r}\), and let \(\Lambda_{0}\) be the set of all \((b_{1}, \dotsc, b_{r}) \in \mathbb{Z}^{r}\) such that \(L_{1}^{\otimes b_{1}} \otimes \dotsc \otimes L_{r}^{\otimes b_{r}}\) belongs to \(P(X)\). This is an additive subroup of \(\mathbb{Z}^{r}\). 
	\medskip
	
	We let
	\[
	\Lambda_{L}
	:=
	\left\{
		\quad
		L_{1}^{\otimes b_{1}}
		\otimes
		L_{2}^{\otimes b_{2}}
		\otimes
		\dotsc
		\otimes
		L_{r}^{\otimes b_{r}}
		\in P(X)
		\quad | \quad
		(b_{1}, b_{2}, \dotsc, b_{r}) \in \Lambda_{0}
		\quad
	\right\}.
	\]
\end{nota}

The next proposition says that under a natural arithmetic condition, any element appearing in the expansion of \(S^{m}(L_{1}^{(a_{1})} \oplus \dotsc \oplus L_{r}^{(a_{r})})\) can be associated uniquely with an line bundle in a fixed, finitely generated subroup of \(\mathrm{Pic}(X)\).

\begin{prop} \label{prop:condgroup}
	There exists a finitely generated free subgroup \(H \subset \mathrm{Pic}(X)\) such that the projection \(H \to \mathrm{Pic}(X)_{\mathbb{Q}}\) realizes an isomorphism onto \(\Lambda_{L}\).
	\smallskip

Assume moreover that we have the inclusion of groups
			\begin{equation} \label{eq:inclgroups}
				\big\{
					(b_{1}, \dotsc, b_{r}) \in \mathbb{Z}^{r}
				\quad | \quad
				\sum_{i} a_{i} b_{i} = 0
				\big\}
				\subset
				\Lambda_{0}.
			\end{equation}

	Then there exists \(m_{0} \in \mathbb{N}_{> 0}\) such that for any \(m\) divisible by \(m_{0}\), the following holds. Let \((b_{1}, \dotsc, b_{r}) \in \mathbb{Z}^{r}\) satisfying \(\sum_{i} a_{i} b_{i} = m\). Then \(L_{1}^{\otimes b_{1}} \otimes \dotsc \otimes L_{r}^{\otimes b_{r}}\) belongs to \(P(X)\), and is the image of a unique element in \(H \subset \mathrm{Pic}(X)\). 
\end{prop}

\begin{proof}
	Remark that \(\Lambda_{L}\) is finitely generated subgroup of \(P(X)\), since it is the image of \(\Lambda_{0}\) under the natural morphism \(\mathbb{Z}^{r} \twoheadrightarrow \Lambda\). Thus, \(\Lambda_{L}\) is a free abelian group; pick a basis \(\mathcal{B}\) if this group. For each member of this basis, choose a preimage in \(\mathrm{Pic}(X)\), and let \(H\) be the subgroup of \(\mathrm{Pic}(X)\) generated by these preimages. We check immediately that the induced morphism
	\[
		H \longrightarrow \Lambda_{L}
	\]
	is an isomorphism.
	\smallskip

	It remains to check that we may pick \(m_{0}\) such that any \(L_{1}^{\otimes b_{1}} \otimes \dotsc \otimes L_{r}^{b_{r}}\) as in the hypothesis actually belongs to \(\Lambda_{L}\). This is easy: pick \(m_{1}\) such that \(L_{1}^{\otimes m_{1}} \in P(X)\), and let \(m_{0} = a_{1} m_{1}\). Then, for any \((b_{1}, \dotsc, b_{r})\) such that \(\sum_{i} a_{i} b_{i} = m_{0} k\) (with \(k \in \mathbb{Z}\)), we have:
	\[
		L_{1}^{\otimes b_{1}} \otimes \dotsc \otimes L_{r}^{\otimes b_{r}}
		=
		(L_{1}^{\otimes (b_{1} - km_{1})} \otimes L_{2}^{\otimes b_{2}} \otimes \dotsc \otimes L_{r}^{\otimes b_{r}})
		\otimes
		(L_{1}^{\otimes m_{1}})^{\otimes k}.
	\]
	By the assumption \eqref{eq:inclgroups}, we have \((b_{1} - k m_{1}, b_{2}, \dotsc, b_{r}) \in \Lambda_{0}\) so the first factor in the right hand side belongs to \(P(X)\). The second factor belongs to this group as well, which gives the result.
\end{proof}

Under the assumption of Proposition~\ref{prop:condgroup}, we can now give meaning to the symmetric powers \(S^{m}(L_{1}^{(a_{1})} \oplus \dotsc \oplus L_{r}^{(a_{r})})\).

\begin{defi} \label{defi:symmetric}
	Data as in Setup~\ref{setup:chapter3}. Fix \(a_{1}, \dotsc, a_{r} \in \mathbb{N}^{\ast}\), and assume that the inclusion of groups \eqref{eq:inclgroups} is satisfied. Let \(H \subset \mathrm{Pic}(X)\) be a subgroup as in the conclusion of Proposition~\ref{prop:condgroup}. Then, for all \(m\) divisible enough, we let
	\begin{equation} \label{eq:symmpower_Q}
		S_{H}^{m}(L_{1}^{(a_{1})} \oplus \dotsc \oplus L_{r}^{(a_{r})})
		:=
		\bigoplus_{a_{1} b_{1} + \dotsc + a_{r} b_{r} = m}
		L_{1}^{\otimes b_{1}}
		\otimes
		\dotsc
		\otimes
		L_{r}^{\otimes b_{r}},
	\end{equation}
	where each term on the right is to be understood as the corresponding element of \(H\).
\end{defi}

\begin{rem}
	\begin{enumerate}
		\item As we said above, our definition depends on the group \(H\) : we may change our choice by multiplying a given basis of \(H\) by torsion elements in \(\mathrm{Pic}(X)\). This will not impact the final asymptotic expansions.
		\item The condition~\eqref{eq:inclgroups} may seem a bit artificial and painful to check. Let us mention that there is of course nothing to check if all the \(L_{i}\) are line bundles. We will see in Section~\ref{sec:twist} that it is allows to compute objects of the form
			\[
				S^{m}(L_{1}^{(a_{1})} \oplus \dotsc L_{r}^{(a_{r})}) \oplus N^{\otimes m},
			\]
			where the \(L_{i}\) are standard line bundles, and \(N\) is a \(\mathbb{Q}\)-line bundle.
			\smallskip

	\end{enumerate}
\end{rem}

\subsection{Combinatorial data}
Our next ingredient in the statement of Theorem~\ref{thmineqintegral} is the data, for all \(i \in \llbracket 1, n \rrbracket\), of a piecewise polynomial function on \(\mathbb R^r\) defined as follows.

\begin{defi} \label{defiupsilon} Let $(t_1, \dotsc, t_r) \in \mathbb R^r$. Mark each edge $\mathfrak{s}$ in $\mathcal T$ with the real number 
	\[
		t_{1} \frac{m_{1}^{\mathfrak{s}}}{d_{1}}
		+ 
		\dotsc 
		+ 
		t_{r} \frac{m_{r}^\mathfrak{s}}{d_{r}}.
	\]
	For each root-to-leaf path \(\sigma\) in \(\mathcal{T}\), denote by \(C_{\sigma}\) the product of all markings along the edges of \(\sigma\), and let \(V_{\sigma}\) denote the \(0\)-dimensional variety marking the leaf of \(\sigma\). Then, for all \(i \in \llbracket 1, n \rrbracket\), we let
\[
	\upsilon_{[\leq i]}(t_{1}, \dotsc, t_r) 
	= 
	\sum_{\mathrm{index}(\sigma) \leq i} C_{\sigma}\, \mathrm{\deg}_{\mathbbm{k}} (V_{\sigma}),
\]
	where the sum runs among all the complete paths of index \(\leq i\) in \(\mathcal{T}\), i.e. among the paths with less than \(i\) negative markings.
\end{defi}

It is easy to check that the functions \(\upsilon_{[\leq i]}\) are piecewise polynomial, and satisfy the following homogeneity property: 
	\[
		\upsilon_{[\leq i]} 
		(\lambda \cdot u) = \lambda^n \; \upsilon_{[\leq i]} (u)
		\quad
		(\lambda \in \mathbb{R}_{+},\; u \in \mathbb{R}^{r})
	\] 

\subsection{Main theorem} We can now state the main theorem of this chapter.

\begin{thm} \label{thmineqintegral} Assumptions as in Setup~\ref{setup:chapter3}. Let $\underline{a} = (a_1, \dotsc, a_r) \in \mathbb N_{\geq 1}^r$, and consider the $(r-1)$-dimensional simplex $\Delta_{\underline{a}}$, accordingly to Definition \ref{defisimplex}. Assume that we have the inclusion \eqref{eq:inclgroups}, and let \(H \subset \mathrm{Pic}(X)\) be a subgroup as in the conclusion of Proposition~\ref{prop:condgroup}. 
	\smallskip
	
	Then, for all $i \in \llbracket 0, n \rrbracket$, we have the following asymptotic upper bound, as \(m\) goes to infinity while being divisible by some fixed integer:
\begin{align*}
	\chi^{[i]} (X, S_{H}^{m} (&L_1^{(a_1)} \oplus \dotsc  \oplus L_r^{(a_r)}) )  \\
& \leq  \frac{\mathrm{gcd}(a_1, \dotsc, a_r)}{a_1 \dotsc a_r} \binom{n + r - 1}{r - 1} \left[ \int_{\Delta_{\underline{a}}} \upsilon_{[\leq i]} dP \right] \frac{m^{n+ r - 1}}{(n+ r - 1)!} + o(m^{n+r - 1}),
\end{align*}
where $P$ is the uniform probability measure on the simplex $\Delta_{\underline{a}}$.
\end{thm}

\begin{rem}
	Let us a make a few comments about the proof of the theorem.

	\begin{enumerate}[label=(\arabic*)]
		\item We will first make a reduction step to pass from the situation where the \(L_{i}\) are standard line bundles. It will suffice to perform Bloch-Gieseker coverings, and to imitate the discussion of Section~\ref{sec:finitemorphism}; characteristic \(p\) phenomena will not make the process especially complicated.
		\item After this first reduction step, the idea is to proceed as in the proof of Proposition~\ref{prop:symmetricsum} (see Section~\ref{sec:asymptoticEuler}), but substituting the use of Riemann-Roch theorem with the algebraic Morse inequalities of Theorem~\ref{thmmorse}.
		\item The symmetric product \(S^{m}(L_{1}^{(a_{1})} \oplus \dotsc \oplus L_{r}^{(a_{r})})\)  will decompose as a direct sum of line bundles, each coming with a natural trivialization of \(\Sigma\). One first difficulty is that instead of dealing with a unique line bundle \(L\) for which we want to estimate \(\chi^{[i]}(X, L^{\otimes m})\), each line bundle in the direct sum is a product of powers of some of the \(L_{i}\) with total (weighted) degree \(m\). This forces us to use a variant of Theorem~\ref{thmmorse}, adapted to this situation where we deal with a product of fixed line bundles (see Proposition~\ref{prop:variantmorse}).
		\item Now, we want to apply Proposition~\ref{prop:variantmorse} to each of the line bundles in the expansion of the symmetric product, and sum the inequalities. We need to make the right hand side converge to an integral, but we are faced with another problem, as the Morse inequalities may assume quite different forms from one line bundle to another. To circumvent this problem, we will need to better estimate the right hand sides that can appear this way: actually, we will first give an upper bound on the \(\chi^{[i]} (X, L_1^{\otimes l_{1}} \otimes \dotsc \otimes L_{r}^{\otimes l_{r}})$, holding for any tuple $(l_{1}, \dotsc, l_{r})$ allowed to vary in a small angular sector of \(\mathbb{N}^{r}\) (see Proposition~\ref{prop:cone}). Finally, we will sum over a covering of \(\mathbb{N}^{r}\) by arbitrarily small angular sectors: the previous estimates will then converge nicely to the integral appearing in Theorem \ref{thmineqintegral} (see Section~\ref{sec:proofthmineqintegral}).
	\end{enumerate}
\end{rem}

\section{Reduction step. Bloch-Gieseker coverings} \label{sec:redBGcoverings}

The goal of this section is to get back to the case where the \(\mathbb{Q}\)-line bundles have been replaced with standard line bundles. This will be achieved in Proposition~\ref{prop:reductionBG}, after some clean-up using Bloch-Gieseker coverings.
\bigskip

{\em In this section, we work under the hypothesis of Theorem~\ref{thmineqintegral}. We advise the reader to look back at Section~\ref{sec:finitemorphism} on stratifications induced on finite morphisms, and maybe to read again the notation of Setup~\ref{setup:chapter3}.}
\medskip

The following result is the well-known statement provided by the Bloch-Gieseker covering construction.

\begin{lem} \label{lem:BGnaive}
	There exists a finite morphism \(p : X' \to X\) and line bundles \(N_{1}, \dotsc, N_{r} \in \mathrm{Pic}(X')\) such that for each \(i \in \llbracket 1, r \rrbracket\), \(p^{\ast} L_{i}\) is equal to the image of \(N_{i}\) in \(\mathrm{Pic}(X')_{\mathbb{Q}}\).
\end{lem}

This follows from the classical Bloch-Gieseker construction (see Lemma~\ref{lem:BGcoveringannex}). In the rest of this section, we let \(p : X' \to X\) be as provided by Lemma~\ref{lem:BGcovering}.

\medskip

Let \(\Sigma'\) be the stratification induced on \(X'\) as in Section~\ref{sec:constructionstrat}. We have a diagram as follows.

\begin{equation} \label{eq:diagramproj}
\begin{tikzcd}
		X_{0}'
			\arrow[r, "f_{0}'"]
			\arrow[d, "p_{0}"]
		&
		X_{1}'
			\arrow[r, "f_{1}'"]
			\arrow[d, "p_{1}"]
		&
		\dotsc
			\arrow[r]
		&
		X_{n-1}'
			\arrow[r, "f_{n-1}'"]
			\arrow[d, "p_{n-1}"]
		&
		X_{n}'
			\arrow[r, "q'"]
			\arrow[d, "p_{n}"]
		&
		X'
			\arrow[d, "p"]
		\\
		X_{0}
			\arrow[r, "f_{0}"]
		&
		X_{1}
			\arrow[r, "f_{1}"]
		&
		\dotsc
			\arrow[r]
		&
		X_{n-1}
			\arrow[r, "f_{n-1}"]
		&
		X_{n}
			\arrow[r, "q"]
		&
		X
	\end{tikzcd}
\end{equation}

For each \(i \in \llbracket 1, r \rrbracket\), note that \(N_{i}^{\otimes d_{i}} = p^{\ast}(L_{i}^{\otimes d_{i}})\) in \(\mathrm{Pic}(X')_{\mathbb{Q}}\). In other words, there exists \(d'_{i} \in \mathbb{N}_{\geq 1}\) such that \(N_{i}^{\otimes d'_{i} d_{i}} = (p^{\ast} L_{i}^{\otimes d_{i}})^{\otimes d'_{i}}\) {\em in the standard Picard group} \(\mathrm{Pic}(X')\). We can replace each trivialization in \(\mathbf{e}_{i}\) by its \(d_{i}\)-th power to get a trivialization \(\mathbf{e}_{i}^{\otimes d_{i}'}\) for \((L_{i}^{\otimes d_{i}})^{\otimes d_{i}'}\). Pulling back to \(X'\) and dividing by \(d_{i}d_{i}'\), we get a fractional trivalization \(\frac{1}{d_{i}d_{i}'} \mathbf{e}_{i}'\) for \(N_{i}\), adapted to \(\Sigma'\).
\medskip

The next lemma claims that the multiplicities associated to the stratified trivializations on \(X'\) can be taken to be multiples of the ones on \(X\), in a {\em uniform manner}.  
	
\begin{lem} \label{lem:BGcovering}
	Under the conclusion of Lemma~\ref{lem:BGnaive}, the following holds as well. For any edge \(\mathfrak{s} := W \overset{\sigma}{\to} V\) in the tree of \(\Sigma\), (with \(W \subset X_{j}\) and \(V \subset X_{j+1}\)), let \(\mathfrak{s}' := W' \overset{\sigma'}{\to} V'\) be an arrow in the tree of \(\Sigma'\), such that \(W'\) dominates \(W\) and \(V'\) dominates \(V\). Then there exists a constant \(N_{\mathfrak{s}'}\) such that, for all \(i \in \llbracket 1, r \rrbracket\), we have
	\[
		m_{i}^{\mathfrak{s'}} = N_{\mathfrak{s}'} m_{i}^{\mathfrak{s}},
	\]
	where \(m_{i}^{\mathfrak{s}}\) (resp. \(m_{i}^{\mathfrak{s'}}\)) is the multiplicity of \(\mathfrak{s}\) (resp. \(\mathfrak{s}')\) associated with the trivialization \(\mathbf{e}_{i}\) (resp. \(\mathbf{e}_{i}')\).
.\end{lem}

\begin{proof}
	Let \(A := \mathcal{O}_{V, \sigma(W)}\) (resp. \(A' := \mathcal{O}_{V', \sigma'(W')}\)) be the local ring of \(V\) at the generic point of the image of \(V\) (resp. of \(V'\) at the image of \(W'\)). Since \(V\) and \(V'\) are normal, these rings are discrete valuation rings. Let \(t \in \mathfrak{m}_{A}\) (resp. \(t' \in \mathfrak{m}_{A'}\)) be a uniformizing element. Using the injection \(A \hookrightarrow A'\), we can write \(t = u (t')^{N_{\mathfrak{s}'}}\) with \(u \in A^{\times}\) and \(N_{\mathfrak{s}'} \in \mathbb{N}_{>0}\). In particular, for any \(g \in \mathrm{Frac}(A)\), we have \(\mathrm{val}_{\mathfrak{m}'}(g) = N_{\mathfrak{s'}}\, \mathrm{val}_{\mathfrak{m}}(g)\), uniformly in \(g\). This gives the result.
\end{proof}

\begin{lem} \label{lem:upsilonred}
		\smallskip

	For each \(j \in \llbracket 0, n\rrbracket\), let \(\upsilon'_{[\leq j]}\) be the function on \(\Delta_{\underline{a}}\) constructed from \(\Sigma'\) and the \(\frac{1}{d_{i} d_{i}'} \mathbf{e}_{i}'\), following Definition~\ref{defiupsilon}.

	Then we have the equality
	\[
		\upsilon_{[\leq j]}' = \mathrm{deg}(p) \; \upsilon_{[\leq j]}.
	\]
\end{lem}
\begin{proof}
	The proof will essentially follow from Proposition~\ref{prop:coveringdegree}, but we need to do some clean-up first.
	\medskip

\noindent
{\em Step 1. Reduction.} Fix some real numbers \(t_{1}, \dotsc, t_{r} \in \mathbb{R}\). Let \(\sigma\) (resp. \(\sigma'\)) be a root-to-leaf path in \(\mathcal{T}_{\Sigma}\) (resp. \(\mathcal{T}_{\Sigma}'\)), and assume that \(\sigma'\) is sent onto \(\sigma\) by the natural morphism of trees associated with the diagram~\eqref{eq:diagramproj}. Mark \(\sigma\) (resp. \(\sigma'\)) with the real numbers
	\[
	t_{1} \frac{m_{1}^{\mathfrak{s}}}{d_{1}}
	+ 
	\dotsc 
	+ 
	t_{r} \frac{m_{r}^\mathfrak{s}}{d_{r}}
	\quad
	\quad
	\left(
	\text{resp.} \;
	t_{1} \frac{m_{1}^{\mathfrak{s'}}}{d_{1}}
	+ 
	\dotsc 
	+ 
	t_{r} \frac{m_{r}^\mathfrak{s'}}{d_{r}}
	\right).
	\]
	Remark now that the second number is equal to \(N_{\mathfrak{s}'}\) times the first one, and hence has the same sign. Since this is true for any \(\mathfrak{s}\), this implies that \(\sigma\) and \(\sigma'\) have the same index. We thus see from Definition~\ref{defiupsilon} that \(\sigma\) appear in the sum defining \(\upsilon_{[\leq j]}\) if and only \(\sigma'\) appears in the sum defining \(\upsilon'_{[\leq j]}\).
	
	\medskip
	Writing \(C_{\sigma'}\) for the product of the multiplicities of \(\sigma'\), this implies that:
	\begin{align*}
		\upsilon_{[\leq j]}'(t_{1}, \dotsc, t_{r}) & = \sum_{\mathrm{index}(\sigma') \leq j} C_{\sigma'}\; \mathrm{deg}_{\mathbbm{k}}(V_{\sigma'} )\\
		& = \sum_{\mathrm{index}(\sigma) \leq j} \left( \sum_{\sigma' \; \text{dominating} \; \sigma} C_{\sigma'}\; \mathrm{deg}_{\mathbbm{k}}(V_{\sigma'})\right)
	\end{align*}

	We will then have the result if we can prove that
	\[
\sum_{\sigma' \; \text{dominating} \; \sigma} C_{\sigma'}\; \mathrm{deg}_{\mathbbm{k}}(V_{\sigma'})
=
	\mathrm{deg}(p)\, C_{\sigma}\; \mathrm{deg}_{\mathbbm{k}}(V_{\sigma}).
	\]
	
	{\em Step 2. End of proof.} Note that the index of \(\sigma\) does not play any role in the previous formula. Also, we may shrink \(X\) around the image of the point at the leaf of \(\sigma\), to assume that \(X_{0}\) contains only one point. For each \(i \in \llbracket 1, r \rrbracket\), let \(\underline{\Sigma}_{i} := (\Sigma, \frac{1}{d_{i}}\mathbf{e}_{i})\) (resp. \(\underline{\Sigma}'_{i} := (\Sigma', \frac{1}{d_{i} d_{i}'}\mathbf{e}'_{i})\)). In this case, we see right away that we have
	\[
	C_{\sigma}\, \mathrm{deg}_{\mathbbm{k}}(V_{\sigma})
	=
	\mathrm{deg}(\sum_{1 \leq i \leq r} t_{i} \; c_{1}(L_{i}, \underline{\Sigma}_{i})_{[\leq n]})^{n}
	\]
	where the degree is extended by \(\mathbb{R}\)-linearity. On the other hand, the left hand side is now the sum over all paths \(\sigma'\), so it equal to
	\[
		\sum_{\sigma'} C_{\sigma'}\, \mathrm{deg}_{\mathbbm{k}}(V_{\sigma'})
		=
		\mathrm{deg}(\sum_{1 \leq i \leq r} t_{i} \; c_{1}(N_{i}, \underline{\Sigma}_{i}')_{[\leq n]})^{n}
	\]

	Expanding both sides gives the result with the same proof as in Proposition~\ref{prop:coveringdegree}.
\end{proof}

The next lemma is the last piece that will allow us to pull-back our data to \(X'\).

\begin{lem} \label{lem:reduction}
	Let \(j \in \llbracket 0, n \rrbracket\). Then, for \(m \in \mathbb{N}\) divisible enough and going to infinity, we have the equality
	\begin{align*}
		h^{j}(X', S^{m}(N_{1}^{(a_{1})} \oplus \dotsc \oplus N_{r}^{(a_{r})}))
		 =
		\deg(p)\, & h^{j}(X, S_{H}^{m}(L_{1}^{(a_{1})} \oplus \dotsc \oplus L_{r}^{(a_{r})}))
		\\
		& + O(m^{r + n-1}).
	\end{align*}
\end{lem}

Before starting the proof, let us take the time to show that the asymptotic growth of the truncated Euler characteristic in Theorem~\ref{thmineqintegral}, does not depend on the choice of \(H\), which is also a good sanity check.

\begin{lem} \label{lem:indepH}
	Let \(j \in \llbracket 0, n \rrbracket\). Let \(H, H'\) be as in the conclusion of Proposition~\ref{prop:condgroup}. Then for \(m \in \mathbb{N}\) divisible enough and going to infinity, we have the equality
	\begin{align*}
		h^{j}(X, S_{H}^{m}(L_{1}^{(a_{1})} \oplus \dotsc \oplus L_{r}^{(a_{r})}))
		= 
		h^{j}(X, S_{H'}^{m}(L_{1}^{(a_{1})} \oplus \dotsc \oplus L_{r}^{(a_{r})}))
		+ O(m^{r + n-1}).
	\end{align*}
\end{lem}

\begin{proof}
	Let us first remark that since \(H\) and \(H'\) are free finitely generated subgroups of \(\mathrm{Pic}(X)\) isomorphic to the same image in \(\mathrm{Pic}(X)_{\mathbb{Q}}\), there exists a {\em finite} family \(\mathcal{S} \subset \mathrm{Pic}(X)\) of torsion elements such that for any \(L \in H\) and \(L' \in H'\) with the same image in \(\mathrm{Pic}(X)\), we have \(L^{-1} \otimes L \in \mathcal{S}\).
	\medskip

\noindent
{\em Step 1. Comparison between the pieces of the two sums.} Let us now introduce the following notation: if \((b_{1}, \dotsc, b_{r})\) are such that \(\sum_{i} a_{i} b_{i} = m \in \mathbb{N}\), we write
	\[
		[L_{1}^{\otimes b_{1}} \otimes \dotsc \otimes L_{r}^{\otimes b_{r}}]_{H}
	\]
	for the unique preimage in \(H \subset \mathrm{Pic}(X)\) of the element in brackets. We use a similar notation for \(H'\). By what has been said just above, for any such \(\underline{b} = (b_{1}, \dotsc, b_{r})\), we have
	\[
	[L_{1}^{\otimes b_{1}} \otimes \dotsc \otimes L_{r}^{\otimes b_{r}}]_{H}
	=
	[L_{1}^{\otimes b_{1}} \otimes \dotsc \otimes L_{r}^{\otimes b_{r}}]_{H'} \otimes T_{\underline{b}}
	\]
	for some torsion element \(T_{\underline{b}} \in \mathcal{S}\).
	\medskip

\noindent
	{\em Step 2. The torsion elements do not modify the asymptotic expansion.} Remark first the the function \((b_{1}, \dotsc, b_{r}) \in \Lambda_{0} \mapsto \sum_{j} a_{j} b_{j}\) induces a linear  form \(m : H \to \mathbb{R}\). We may apply Lemma~\ref{lem:asymptseveral} (3) with \(L_{1}, \dotsc, L_{r}\) being replaced by generators of the inverse image in \(H\) of \(\Lambda_{L}\), and \(\mathcal{F}\) replaced by any of the \(T \in \mathcal{S}\).	Since this set is finite, we may take a uniform constant in the \(O\)-term, which gives
	\begin{align*}	
	h^{j}(X, [L_{1}^{\otimes b_{1}} \otimes \dotsc \otimes L_{r}^{\otimes b_{r}}]_{H})
	& =
		h^{j}(X, [L_{1}^ {\otimes b_{1}} \otimes \dotsc \otimes L_{r}^{\otimes b_{r}}]_{H'} \otimes T_{\underline{b}}) \\
		&  = h^{j} (X, [L_{1}^ {\otimes b_{1}} \otimes \dotsc \otimes L_{r}^{\otimes b_{r}}]_{H'}) 
		+ O(m^{n-1})
	\end{align*}

\noindent
{\em Step 3. We take the sum.} Summing everything as \((b_{0}, \dotsc, b_{r})\) runs in \(\Lambda_{0}\), we get the following:
	\begin{align*}
	h^{j}(X, S_{H}^{m}(L_{1}^{(a_{1})} \oplus \dotsc \oplus L_{r}^{(a_{r})}))
		& 
		= 
		\sum_{\sum_{j} a_{j} b_{j} = m} 
		h^{j}(X, [L_{1}^{\otimes b_{1}} \otimes \dotsc \otimes L_{r}^{\otimes b_{r}}]_{H}) \\
		&
		=
		\sum_{\sum_{j} a_{j} b_{j} = m} 
		\left( h^{j}(X, [L_{1}^{\otimes b_{1}} \otimes \dotsc \otimes L_{r}^{\otimes b_{r}}]_{H'}) + O(m^{n-1}) \right) \\
		& 
		=
		\sum_{\sum_{j} a_{j} b_{j} = m} 
		 h^{j}(X, [L_{1}^{\otimes b_{1}} \otimes \dotsc \otimes L_{r}^{\otimes b_{r}}]_{H'}) + O(m^{n+r-1}) \\
		&
		= h^{j}(X, S_{H'}^{m}(L_{1}^{(a_{1})} \oplus \dotsc \oplus L_{r}^{(a_{r})}) + O(m^{n+r-1})
	\end{align*}

	To pass from the second to the third line, note that the constant in the \(O\) term is independent of \((b_{1}, \dotsc, b_{r})\) and of course of \(m\).
\end{proof}

The proof of the reduction step is quite similar :

\begin{proof}[Proof of Lemma~\ref{lem:reduction}] {\em Step 1. Torsion elements.} Again, let us denote by \([L_{1}^{\otimes b_{1}} \otimes \dotsc \otimes L_{r}^{\otimes b_{r}}]_{H}\) the unique preimage of the corresponding element in \(H\). By our choice of \(N_{i}\), note that for any \((b_{1}, \dotsc b_{r}) \in \Lambda_{0}\), note that the line bundle
	\[
		p^{\ast} [L_{1}^{\otimes b_{1}} \otimes \dotsc \otimes L_{r}^{\otimes b_{r}}]_{H}
		\otimes
		(N_{1}^{\otimes b_{1}} \otimes \dotsc \otimes N_{r}^{\otimes b_{r}})^{-1}
	\]
	is torsion on \(X'\). Since \(\Lambda_{0}\) is finitely generated, there exists a {\em finite} family of torsion line bundles \(\mathcal{S} \subset \mathrm{Pic}(X')\) that contains all such differences.
	Apply now Lemma~\ref{lem:asymptseveral} (3), this time on \(X'\), with the \(L_{i}\) replaced by the \(N_{i}\), and \(\mathcal{F}\) replaced by any member of the finite family \(\mathcal{S}\). Taking a constant uniform in \(\mathcal{S}\) yields:
	\[
		h^{j}(X', N_{1}^{\otimes b_{1}} \otimes \dotsc \otimes N_{r}^{\otimes b_{r}})
		=
		h^{j}(X', p^{\ast} [L_{1}^{\otimes b_{1}} \otimes \dotsc \otimes L_{r}^{\otimes b_{r}}]_{H}) + O(m^{n-1}).
	\]
	for any \((b_{0}, \dotsc, b_{r}) \in \Lambda_{0}\).

	\medskip
	\noindent {\em Step 2. Comparisons between \(X\) and \(X'\).} We apply Lemma~\ref{lem:modificationseveral}, with the \(L_{i}\) being replaced by generators of the inverse image of \(\Lambda_{L}\) in \(H\) . This gives
	\[
	h^{j}(X, [L_{1}^{\otimes b_{1}} \otimes \dotsc \otimes L_{r}^{\otimes b_{r}}]_{H})
	=
	h^{j}(X', p^{\ast} [L_{1}^{\otimes b_{1}} \otimes \dotsc \otimes L_{r}^{\otimes b_{r}}]_{H}) + O(m^{n-1}).
	\]
	Again, we took \(m : H \to \mathbb{R}\) to be the function induced by \((b_{j}) \mapsto \sum_{j} a_{j} b_{j}\) on \(\Lambda_{0}\). We obtain
	\[
	h^{j}(X, [L_{1}^{\otimes b_{1}} \otimes \dotsc \otimes L_{r}^{\otimes b_{r}}]_{H})
	=
	h^{j}(X', N_{1}^{\otimes b_{1}} \otimes \dotsc \otimes N_{r}^{\otimes b_{r}})
	+
	O(m^{n-1})
	\]
	We finish the proof by summing on all \((b_{0}, b_{1}, \dotsc, b_{r})\) that satisfy \(\sum_{j} a_{j} b_{j} = m\), as in the proof of Proposition~\ref{lem:indepH}.
\end{proof}

Summarizing everything, we have proved the following reduction step.

\begin{prop} \label{prop:reductionBG} To prove Theorem~\ref{thmineqintegral}, it suffices to deal with the case where :
	\begin{enumerate}
		\item all \(L_{i}\) are standard line bundles on \(X\), each endowed with fractional trivializations \(\frac{1}{d_{i}}\mathbf{e}_{i}\) on \(\Sigma\);
		\item \(S^{m}_{H}\) has been replaced with the standard symmetric product \(S^{m}\).
	\end{enumerate}
\end{prop}

Indeed, if the result holds in this narrower generality, this will prove the result for \(X'\), endowed with the line bundles \(N_{i}\) and the fractional trivializations \(\frac{1}{d_{i} d_{i}'} \mathbf{e}_{i}\) on \(\Sigma'\). But by Lemma~\ref{lem:upsilonred} and \ref{lem:reduction}, the corresponding Morse inequalities for \(X\) and \(X'\) are related by the same factor \(\mathrm{deg}(p)\) on both sides. This gives the result for \(X\) as well.
\medskip

{\em In the following, we will then assume that we have made this reduction step.}

\section{Morse inequalities for a product of line bundles. Direct sums of line bundles}

After the reduction step of the previous section, there is another result we need in order to prove our main theorem: a variant of the Morse inequalities, valid for products of members of a given finite family of line bundles (see Proposition~\ref{prop:variantmorse}). Let us introduce some notation first. 
\medskip

\begin{setup} \label{setup:variant}
	Let \(X\) be a projective variety of dimension \(n\), and let \(\mathcal L\) be a finite set of line bundles on \(X\). Assume \(\Sigma\) is a stratification on \(X\), adapted to all \(L \in \mathcal L\). For each \(L \in \mathcal L\), we let \(\frac{1}{d_{L}} \mathbf e_{L}\) be a fractional trivialization of \(L\) on \(\Sigma\) (\(\mathbf{e}_{L}\) being a trivialization of \(L^{\otimes d_{L}}\)), and we let \(\underline{\Sigma} = ( \Sigma, (\frac{1}{d_{L}} \mathbf{e}_L)_{L \in \mathcal L} )\). 
\end{setup}

In the next definition, we introduce the quantity that will serve as leading coefficient for an upper bound on the numbers $\chi^{[i]} (X, L_1^{\otimes a_1} \otimes  \dotsc \otimes L_r^{\otimes a_r})$ as $a_1, \dotsc, a_r \longrightarrow + \infty$, where $L_1, \dotsc, L_r \in \mathcal L$ are arbitrary elements. The idea is to mimic the formula \eqref{eq:degreeindexes}, but allowing to choose markings for the edges arbitrarily among the ones provided by the \(L \in \mathcal{L}\), and finally taking the maximum for all such choices. 
\medskip

\begin{defi} \label{defileadingcoeff}  \begin{enumerate}[label=(\arabic*)]
	\item Let \(\mathcal T\) be the tree associated to \(\Sigma\), and let \(\mathcal{E}\) be the set of its edges. For all \(\mathfrak{s} \in \mathcal E$, we let \(\frac{1}{d_{L}} m_{L}^{\mathfrak{s}}\) be the marking of \(\mathfrak{s}\) provided by the trivialization \(\frac{1}{d_{L}} \mathbf{e}_L\). Now, for any mapping \(\phi : \mathcal {E} \longrightarrow \mathcal L\), we let \(\mathcal{T}_{\phi}\) be the {\em marked} tree where each edge \(\mathfrak{s}\) is labeled by the marking \(m_{\phi(\mathfrak{s})}^\mathfrak{s}\).
		\medskip

	\item Let \(l \in \llbracket 0, n \rrbracket\). For all \(\phi : \mathcal{E} \longrightarrow \mathcal{L}$, and any root-to-leaf path \(\sigma\) in \(\mathcal T\), we let \(C_{\sigma, \phi}\) be the product of the labels along the edges of \(\sigma\) in \(\mathcal T_{\phi}\). We let
\[
c(\phi, \underline{\Sigma})_{[ \leq i]} 
		= 
		\sum_{\mathrm{index} (\sigma) \leq  i} C_{\sigma, \phi}
		\;
		\deg_{\mathbbm{k}}(V_{\sigma})
\]
		where the sum runs among all complete paths in $\mathcal T_{\phi}$ with less than \(i\) negative labels, and \(V_{\sigma}\) denotes the \(0\)-dimensional variety labeling the leaf of \(\sigma\). 

We let
\begin{equation} \label{eq:defic}
		(-1)^i c(\mathcal L, \underline{\Sigma})_{[\leq i]} 
		= 
		\underset{\phi : \mathcal{E} \longrightarrow \mathcal L}{\max} 
		(-1)^i c(\phi, \underline{\Sigma})_{[\leq i]}.
\end{equation}
\end{enumerate}
\end{defi}

Before stating the result, let us introduce the following simplifying notation.

\begin{nota} Let \(\underline{l} = (l_{L})_{L \in \mathcal L}\) be a family of integers. We write \(\mathcal L^{\otimes \underline{l}} = \bigotimes_{L \in \mathcal L} L^{\otimes l_{L}}\). For all \(L \in \mathcal{L}\), we write \(\underline{\delta}_{L} = (\delta_{L, L'})_{L' \in \mathcal{L}}\), where \(\delta_{L, L'} = 1\) if \(L = L'\) and \(\delta_{L, L'} = 0\) otherwise.
\end{nota}
\medskip

The variant of the algebraic Morse inequalities can be stated as follows.

\begin{prop} \label{prop:variantmorse} With the notation of Setup~\ref{setup:variant}, let \(M\) be any line bundle on \(X\).
	\medskip

Then for any $i \in \llbracket 0, n \rrbracket$, any $m \in \mathbb N$, and any $\underline{l} = (l_L)_{L \in \mathcal L}$ such that $\sum_L l_L = m$, we have
	\[
	\chi^{[i]} (X, \mathcal L^{\otimes \underline{l}} \otimes M) 
	\leq 
	(-1)^i 
	c(\mathcal L, \underline{\Sigma})_{[\leq i]} \frac{m^n}{n!} + O(m^{n-1}).
	\] 
\end{prop}

To prove the result, we argue as in the proof of Theorem \ref{thmmorse}: it suffices to prove the seemingly more general proposition. 

\begin{prop} \label{propredmorsegen} Under the hypotheses of Proposition \ref{prop:variantmorse}, let $p: X' \longrightarrow X$ be a projective, generically finite morphism, and let $M'$ be any line bundle on $X'$. 

	Then for any $i \in \llbracket 0, n \rrbracket$, any $m \in \mathbb N$, and any $\underline{l} = (l_L)_{L \in \mathcal L}$ such that $\sum_L l_L = m$, we have
$$
\chi^{[i]} (X', p^\ast \mathcal L^{\otimes \underline{l}} \otimes M') 
	\leq 
	\deg (p)\;
	(-1)^i c(\mathcal L, \underline{\Sigma})_{[\leq j]} \frac{m^n}{n!} + O(m^{n-1}).
$$ 
\end{prop}

\begin{proof}
	The proof is very close to the one of Proposition~\ref{propredmorse}. We reason by induction on \(\dim X\). 
	\medskip

\setcounter{stepcc}{-1}

\noindent
{\em \refstepcounter{stepcc} Step~\thestepcc. We may assume that \(X\) and \(X'\) are normal.} Exactly as in the proof of Proposition~\ref{propredmorse}, we may assume that \(X\) and \(X'\) are normal, and that the morphism \(q : X_{n} \to X\) is the identity.
\medskip

\noindent
	\emph{\refstepcounter{stepcc} Step~\thestepcc. Initialization of the induction.} When \(\dim X = 0\), we have \((-1)^{i} c(\mathcal{L}, \underline{\Sigma})_{[\leq i]} = (-1)^{i} \deg_{\mathbbm{k}}(X')\), and the result is trivial, with \(O(m^{-1}) = 0\), exactly as in Step~\ref{step:initialization} of the proof of Proposition~\ref{propredmorse}.
	\medskip

\noindent
{\em \refstepcounter{stepcc} Step~\thestepcc. Induction step.} Let us now assume that the result has been proved for all dimensions up to $n - 1$. Write \(\Sigma = (X_{\bullet}, f_{\bullet})\), with
	\[
		X_0 
		\overset{f_0}{\longrightarrow} 
		\dotsc 
		\longrightarrow 
		X_{n-1} 
		\overset{f_{n-1}}{\longrightarrow} 
		X_n 
		= X
	\]

	For each $L \in \mathcal L$, let $s_L$ be the trivialization of $L^{\otimes d_{L}}$ over the open dense subset $U_n = X_n \setminus f_{n-1}(X_{n-1})$ provided by $\mathbf{e}_L$. Denote by \((D_{i})_{1 \leq i \leq m}\) the irreducible components of \(f_{n-1}(X_{n-1})\). For each \(j \in \llbracket 1, r \rrbracket\), let \(W_{j}\) be the unique connected component of \(X_{n-1}\) such that \(f_{n-1}(W_{j}) = D_{j}\), and let \(q_{j} : W_{j} \longrightarrow X\) be the natural map.
	\smallskip

For each \(L \in \mathcal{L}\), we let \(D = \sum_{1 \leq j \leq r} m_{j}^{L} D_{j}\) be the \emph{Cartier} divisor of zeros and poles of the section \(s_{L}\), seen as a rational section of \(L\). We now proceed to the following simultaneous clean up steps for all line bundles.
\medskip

\begin{lem} \label{lem:covertrivseveral}
	Assumptions as in Proposition~\ref{propredmorsegen}. We can assume without loss of generality that for each \(L \in \mathcal{L}\), the line bundle \(p^{\ast} L\) is trivial above \(p^{-1}(U_{n})\), with a trivialization \(e'_{L}\) such that \((e')^{\otimes d_{L}} = p^{\ast} e_{L}\). 
\end{lem}

	To see this, we argue as in Lemma~\ref{lem:covertriv}, and replace \(p : X' \to X\) by a tower of finite tower of cyclic covers provided by Lemma~\ref{lem:cycliccover}.

	\begin{lem} \label{lem:cleanupstep}
	For each \(L \in \mathcal{L}\), denote by \(s'_{L}\) be the rational section of \(p^{\ast} L\) induced by the trivialization \(e_{L}'\). We can assume without loss of generality that for all \(L \in \mathcal{L}\), we have an equality 
	\[
		D(s'_{L})
		=
		\sum_{1 \leq j \leq r}
		m_{j}'{}^{L} D_{j}'
		+
		E^{L}
		-
		F^{L},
	\]
	where:
	\begin{enumerate}
		\item each \(D_{j}' = \widehat{D}_{j}' + E_{j}\) is an effective {\em Cartier} divisor on \(X'\), {\em adapted} to exactly one component \(D_{j}\) of \(X\), with the notation of Definition~\ref{defi:adapteddivisor};
		\item \(E^{L}\) and \(F^{L}\) are {\em effective, \(p\)-exceptional Cartier} divisors.
	\end{enumerate}
\end{lem}

	\begin{proof}
		With the assumptions of Proposition~\ref{prop:variantmorse}, we can write for any \(L \in \mathcal{L}\):
		\[
			D(s_{L}') = \sum_{1 \leq j \leq r'} m_{j}' D_{j}' + E - F,
		\]
		where \(D_{j}', E\) and  \(F\) are effective Weil divisors. Since \(X\) is normal and \(p\) is birational, it establishes a bijection between Weil divisors on \(X\) and \(X'\). Each \(D_{j}'\) dominates exactly one component \(D_{i}\) in \(D\), and the multiplicities \(m_{j}\) and \(m_{i}\) are identified. To conclude, it suffices to apply Lemma~\ref{lem:WeiltoCartier} successively to each \(D_{s_{L}'}\) while \(L\) varies in \(\mathcal{L}\).
	\end{proof}
\medskip

By the same argument as in \eqref{eqsumproj1}, we have, for any component \(D_{i}\):
\begin{equation} \label{eq:sumproj1sev}
	\sum_{j} m_{j}' 
		\, 
		\deg (\widehat{D}_{j}' / D_{i}) 
		= \deg(p)\, \frac{m_{i}}{d_{L}},	
\end{equation}
where the sum runs among all $j \in \llbracket 1, r'\rrbracket$ such that $D_{j}'$ dominates \(D_{i}\).
\medskip

\noindent
\emph{\refstepcounter{stepcc} Step~\thestepcc. We bound the difference between terms given by two close $r$-uples.} Let $\underline{l} = (l_L)_{L \in \mathcal L}$ be such that $\sum_{L \in \mathcal L} l_L = m$, with $m \geq 1$. Then, for some $L \in \mathcal L$, we have $\underline{l} - \underline{\delta}_{L} \in \mathbb N^\mathcal L$, i.e. all coordinates of \(\underline{l} - \underline{\delta}_{L}\) are non-negative.
	\smallskip

	Write
	\[
		A = \sum_{m_{j}'{}^{L} > 0} m_{j}'{}^{L} D_{j}' + E^{L},
		\quad
		B = \sum_{m_{j}'{}^{L} < 0} (-m_{j}'{}^{L}) D_{j}' + F^{L},
	\]

	We can now follow the arguments of Step~\ref{step:boundAB} in the proof of Proposition~\ref{propredmorse}, to obtain
	\begin{align*}
		\chi^{[i]} (X', p^{\ast} \mathcal{L}^{ \otimes \underline{l}} \otimes M)
		& - \chi^{[i]}  (X', p^{\ast} \mathcal{L}^{ \otimes (\underline{l} - \underline{\delta}_{L})} \otimes M)  \\ 
		& \leq \chi^{[i]}(X', M \otimes p^{\ast} \mathcal{L}^{\otimes \underline{l}} 
	\otimes 
	\mathcal O_{A}) 
		+ \chi^{[i - 1]}(X', M' \otimes p^{\ast} \mathcal{L}^{\otimes (\underline{l}- \underline{\delta}_{L})} \otimes \mathcal O_{B})\\
	 & \quad + O(m^{n-2}).
\end{align*}

\noindent
{\em \refstepcounter{stepcc} Step~\thestepcc. "D\'{e}vissage", upper bound and induction hypothesis.} We now reason exactly as in Steps~\ref{step:5}, \ref{step:6} and \ref{step:7} of the proof of Proposition~\ref{propredmorse}. First, using "d\'{e}vissage" on the sheaves \(\mathcal{O}_{A}\) and \(\mathcal{O}_{B}\) yields the following inequality, parallel to \eqref{eqstepA1} and \eqref{eqstepB1}:

\begin{align*}
	\chi^{[i]} & (X, \mathcal L^{ \otimes \underline{l}} \otimes M)  - \chi^{[i]}  (X, \mathcal L^{ \otimes (\underline{l} - \underline{\delta_L})} \otimes M) 
\\  
	& \leq  
	\sum_{ 1 \leq j \leq r'} 
	\left(
	\mathbf{1}_{[m_{j}'{}^{L} > 0]}  \;  
	\sum_{1 \leq k \leq m_{j}'{}^{L}}
	\; \chi^{[i]} (\widehat{D}_{j}', \mathcal{L}_{j,k} \otimes \mathcal L^{\otimes \underline{l}} \otimes M) 
	\; + \; 
	\mathbf{1}_{[m'{}^{L}_{j} < 0]} 
	\sum_{1 \leq k \leq (-m_{j}'{}^{L})}
	\chi^{[i - 1]} 
	(\widehat{D}_{j}', \mathcal{L}_{j, k} \otimes \mathcal{L}^{\otimes (\underline{l} - \underline{\delta_{L}})} \otimes M) 
	\right) \\
 & + O(m^{n-2}),
\end{align*}
where the \(\mathcal{L}_{j,k}\) are line bundles on the Weil divisors \(\widehat{D}_{j}'\).
\medskip

For each component \(D_{j}\) of \(D\), we denote by  $\Sigma_{j}$ be stratification induced on $D_{j}$ by $\Sigma$, and for each line bundle \(N \in \mathcal{L}\), we let \(\frac{1}{d_{N}} \mathbf{e}_{N}^{j}\) be the fractional trivialization of \(N \in \mathcal L\) induced by \(\frac{1}{d_{N}}\mathbf{e}_N\) on \(\Sigma_{j}\). 

Writing \(\underline{\Sigma}_{j} = (\Sigma_{j}, (\frac{1}{d_{N}} \mathbf{e}^{j}_{N})_{N \in \mathcal L})\), for each \(j\), we can apply the induction hypothesis to each morphism \(\widehat{D}_{j}' \to D_{j}\), which gives:

\begin{align} \nonumber
	\chi^{[i]} & (X, \mathcal L^{ \otimes \underline{l}} \otimes M)  - \chi^{[i]}  (X, \mathcal L^{ \otimes (\underline{l} - \underline{\delta}_L)} \otimes M) 
\\  \nonumber
	& \leq  
	(-1)^{i}
	\sum_{ 1 \leq j \leq r'} 
	\left(
	\mathbf{1}_{[m'_{j}{}^{L} > 0]} \deg(\widehat{D}_{j}'/D_{i})  \;  
	m_{j}'{}^{L}	
	\; 
	c(\mathcal L, \underline{\Sigma}_{j})_{[\leq i]} \frac{m^{n-1}}{(n-1)!}  
	\; + \; \right. \\
	& \hspace{4cm}
	\left.
	\mathbf{1}_{[m_{j}'{}^{L} > 0]} \deg(\widehat{D}_{j}'/D_{i})  \;  
	m_{j}'{}^{L}	
	c(\mathcal L, \underline{\Sigma}_{j})_{[\leq i-1]} \frac{(m-1)^{n-1}}{(n-1)!}  
	\right) \\ 
	& \hspace{4cm} + O(m^{n-2}) \nonumber \\
	& = (-1)^{i} \deg(p)  
	\sum_{ 1 \leq j \leq r} 
	\left(
	\mathbf{1}_{[m^{L}_{j} > 0]}  \;  
	\frac{m_{j}^{L}}{d_{L}}
	\; 
	c(\mathcal L, \underline{\Sigma}_{j})_{[\leq i]} \frac{m^{n-1}}{(n-1)!}  
	\; + \;	
	\mathbf{1}_{[m^{L}_{j} < 0]} 
	\, \frac{m_{j}^{L}}{d_{L}}\,
	c(\mathcal L, \underline{\Sigma}_{j})_{[\leq i-1]} \frac{(m-1)^{n-1}}{(n-1)!}  
	\right)
	\nonumber
	\\ 
 	& \quad \quad + O(m^{n-2}) \label{eqinductseveral},
\end{align}
\medskip

To obtain the previous inequality, it is important to remark that the number of terms in the sums is independent of \(m\), so the constant in the \(O(m^{n-2})\) term can be gathered in a global term. Moreover, this constant can also be taken uniformly in \(L \in \mathcal{L}\). Note that at the last line, we have made use of \eqref{eq:sumproj1sev} to pass from a sum on the \(D_{j}'\) to a sum on the \(D_{i}\).
\medskip

The proof will then be complete after will sum the terms, using the following lemma.
\medskip

\begin{lem} \label{lem:value} For a component \(D_{j}\) of \(X\), let $\mathfrak{s}_{j}$ be the edge linking the corresponding component of \(X_{n-1}\) to $X_n$ in the tree $\mathcal T$. Then
\begin{align} \nonumber 
	(-1)^i \sum_{1 \leq j \leq r}
	( 
	\mathbf{1}_{[m_{L}^{j} > 0]} \; \frac{m_{L}^{j}}{d_{L}} \; c(\mathcal L, \underline{\Sigma}_{j})_{[\leq i]}  
	+ 
	\mathbf{1}_{[m_{L}^{j} < 0]} & \; \frac{m_{L}^{j}}{d_{L}} \; c(\mathcal L, \underline{\Sigma}_{j})_{[\leq i - 1]})  \\ \label{eqinductphi}
	& = \underset{\phi : \mathcal E \longrightarrow \mathcal L \; | \; \forall{j}, \phi(\mathfrak{s}_{j}) = L }{\max} (-1)^i c(\phi, \underline{\Sigma})_{[\leq i]}.
\end{align}
In particular, by Definition \ref{defileadingcoeff}, the left hand side is bounded from above by $(-1)^i c(\mathcal L, \underline{\Sigma})_{[\leq i]}$.
\end{lem}
 
\begin{proof}[Proof of Lemma~\ref{lem:value}]
	This comes right away from Definition \ref{defileadingcoeff}. If we denote by \(\mathcal{E}_{1}, \dotsc, \mathcal{E}_{r}\) the respective sets of paths for the subtrees \(\mathcal{T}_{j} \subset \mathcal{T}\) based at all the components of \(X_{n-1}\), we can write
	\begin{align*}
	\underset
	{\phi : \mathcal E \longrightarrow \mathcal L \; | \; \forall{j}, \phi(\mathfrak{s}_{j}) = L }{\max} 
		(-1)^i c(\phi, \underline{\Sigma})_{[\leq i]} & = 
		\underset{\phi_{j} : \mathcal{E}_{j} \longrightarrow \mathcal L}{\max} 
		(-1)^{i}
		\sum_{1 \leq j \leq r} 
		\frac{m_{j}^{L}}{d_{L}}
		c(\phi_{j}, \underline{\Sigma}_{j})_{[\leq \lambda(i)]} 
	\end{align*}
	where in the second maximum, we range over all possible choices of markings \(\phi_{1}, \dotsc, \phi_{r}\) for the edges \(\mathcal{E}_{j}\) by elements of \(\mathcal{L}\). The index \(\lambda(i)\) is equal to \(i\) if \(m_{j}^{L}\) and to \(i-1\) if \(m_{j}^{L} < 0\), since in this case, the edge \(\mathfrak{s}_{j}\) is already marked by a negative sign. Thus, breaking the sum accordingly to whether \(m_{j}^{L}\) is positive or not, shows that it is equal to the sum of two terms: 
	\begin{align*}
		\frac{1}{d_{L}} \sum_{m_{j}^{L} > 0} 
		\;
		\underset{\phi_{j}}{\max}
		\;
		\left(
		m_{j}^{L} 
		\;
		(-1)^{i} c(\phi_{j}, \underline{\Sigma}_{j})_{[\leq i]}
		\right)
		& =
		\frac{1}{d_{L}}
		\sum_{m_{j}^{L} > 0} 
		\;
		m_{j}^{L} 
		\;
		\underset{\phi_{j}}{\max}
		(-1)^{i} c(\phi_{j}, \underline{\Sigma}_{j})_{[\leq i]}
		\\
		& =
		\frac{1}{d_{L}}
		\sum_{m_{j}^{L} > 0}
		\;
		m_{j}^{L}
		\;
		(-1)^{i} c(\mathcal{L}, \underline{\Sigma}_{j})_{[\leq i]} \\
	\end{align*}
	and
	\begin{align*}
		\frac{1}{d_{L}}
		\sum_{m_{j}^{L} < 0} 
		\;
		\underset{\phi_{j}}{\max}
		\;
		\left(
		m_{j}^{L} 
		\;
		(-1)^{i} c(\phi_{j}, \underline{\Sigma}_{j})_{[\leq i - 1]}
		\right)
		& =
		\frac{1}{d_{L}}
		\sum_{m_{j}^{L} < 0} 
		\;
		(-m_{j}^{L})
		\;
		\underset{\phi_{j}}{\max}
		(-1)^{i-1} c(\phi_{j}, \underline{\Sigma}_{j})_{[\leq i - 1]}
		\\
		& =
		\frac{1}{d_{L}}
		\sum_{m_{j}^{L} < 0}
		\;
		(-m_{j}^{L})
		\;
		(-1)^{i-1} 
		c(\mathcal{L}, \underline{\Sigma}_{j})_{[\leq i]}.
	\end{align*}
	These two terms are precisely the ones that appear in the left hand side of \eqref{eqinductphi}.
\end{proof}
Now, since $(m-1)^{n-1} = m^{n-1} + O(m^{n-2})$, inserting \eqref{eqinductphi} in \eqref{eqinductseveral} gives
$$
\chi^{[i]} (X, \mathcal L^{ \otimes \underline{l}} \otimes M)  - \chi^{[i]}  (X, \mathcal L^{ \otimes (\underline{l} - \underline{\delta_L})} \otimes M) \leq (-1)^i c(\mathcal L, \underline{\Sigma})_{[\leq i]} \frac{m^{n-1}}{(n-1)!} + O(m^{n-2}).
$$

\noindent
\emph{\refstepcounter{stepcc} Step~\thestepcc. We sum these differences over a path leading to a given $\underline{l}$.} Consider now a sequence $\underline{l}_1, \dotsc, \underline{l}_m = \underline{l}$, where for all $j$, we have $\underline{l}_j = \underline{l}_{j-1} + \underline{\delta_{L_j}}$ for some ${L}_{j} \in \mathcal L$. For all $j$, we have $\sum_{L} (l_j)_{L} = j$, so we can sum the last inequality for all $\underline{l}_j$, to get
$$
\chi^{[i]}(X, \mathcal L^{\otimes \underline{l}}) \leq (-1)^i c(\mathcal L, \underline{\Sigma})_{[\leq i]} \sum_{j = 1}^m \frac{j^{n-1}}{(n-1)!} + O(m^{n-1}).
$$
Remark that we used the fact that the constant in the \(O(m^{n-1})\) does not depend on \(j\). This gives the result, since $\sum_{1 \leq j \leq m} \frac{j^{n-1}}{(n-1)!} = \frac{m^n}{n!} + O(m^{n-1})$.

\end{proof}

\section{Estimates on angular sectors}

To prove Theorem~\ref{thmineqintegral}, which will give a better estimate in the right hand side of Proposition~\ref{prop:variantmorse} when $\underline{l}$ belongs to a narrow angular sector of $\mathbb N^{\mathcal L}$.
\medskip

\begin{setup}
	Same notation as in Setup~\ref{setup:chapter3}. Assume that $\mathcal L = \{ L_{1}, \dotsc, L_{r} \}$ is a finite set of line bundles on \(X\). Pick \(\underline{u}_{1}, \dotsc, \underline{u}_{p} \in \mathbb{N}^{r}$. For each \(i \in \llbracket 1, p \rrbracket\), we write $\underline{u}_{i} = (u_{i,1}, \dotsc, u_{i, r})$, and we let \(M_{i} = \mathcal L^{\otimes \underline{u}_i}\). Each \(M_{i}\) is endowed with a natural fractional trivialization \(\frac{1}{d}\mathbf{f}_{i}\) on \(\Sigma\), where \(d = \mathrm{lcm} \{d_{i}\}_{1 \leq i \leq p}\), given on a particular stratum by
\begin{equation} \label{eqdefifi}
	f_{i}
	= (e_{1})^{\otimes u_{i, 1}} 
	\otimes \dotsc \otimes 
	(e_{r})^{\otimes u_{i, r}},
\end{equation}
	where the \(e_{j}\) are the associated trivializations of the \(L_{j}^{\otimes d}\) on the specified stratum. Let \(\mathcal M = \{M_1, \dotsc, M_p \}\), and \(\underline{\Sigma}^{\mathcal M} = (\Sigma, (\mathbf f_j)_j)\).
\end{setup}

Using Definition \ref{defileadingcoeff} with $\mathcal L$ replaced by $\mathcal M$, we can define the quantities $c(\mathcal M, \underline{\Sigma}^{\mathcal M})_{[\leq i]}$. It is clear by construction that the latter are continuous piecewise polynomial functions in the $u_{i}$. More precisely, we have the following.

\begin{lem} \label{lemcontfunction} For any \(i \in \llbracket 0, n \rrbracket\), there exists a continuous piecewise polynomial function 
	\[
		\varphi_{[\leq i]} 
		: 
		(\mathbb{R}^{r}_{+})^{p} 
		\longrightarrow 
		\mathbb{R},
	\]
homogeneous of degree \(n\), such that for all \(\underline{u}_1, \dotsc, \underline{u}_p \in \mathbb{N}^{r}\) as above, we have
\[
	c(\mathcal M, \underline{\Sigma}^{\mathcal M})_{[\leq i]} 
	= 
	\varphi_{[\leq i]} 
	(\underline{u}_{1}, \dotsc, \underline{u}_{p}).
\]
	Moreover, the function \(\varphi_{[\leq j]}\) only depends on \(p\) and on the data of $\Sigma, \mathbf{e}_{1}, \dotsc, \mathbf{e}_{r}$.
\end{lem}
\begin{proof}
	By Definition \ref{defileadingcoeff}, the real number $c(\mathcal M,  \underline{\Sigma}^{\mathcal M})_{[\leq i]}$ is a sum of maxima and minima of $n$-homogeneous piecewise polynomials functions in the markings of the trees $\mathcal T_{\phi}$. Each one of these markings being a linear form in $(\underline{u}_1, \dotsc, \underline{u}_r) \in (\mathbb R_+^r)^p$ by definition of the $\mathbf f_i$ (see \eqref{eqdefifi}), we get the result.
\end{proof}

Note that we have a very simple relation between this function \(\varphi_{[\leq i]}\) and the function \(\upsilon_{[\leq i]}\) introduced in Definition~\ref{defiupsilon}. 
\begin{lem} For any $t \in \Delta_{\underline{a}}$, we have
$$
\upsilon_{[\leq i]} (t) = \varphi_{[\leq i]} (t, \dotsc, t).
$$
\end{lem}

\begin{proof}
	By continuity, we may assume that \(t\) is rational, and by homogeneity, we may assume its entries are integers. It suffices to compare the definition of the right hand side with Definition~\ref{defileadingcoeff} and Lemma~\ref{lemcontfunction}. Since all elements of \(\mathcal{M} = (M_{1}, \dotsc, M_{p})\) are equal to the same line bundle in the situation of this lemma, inspection of the maximum defining \(c(\mathcal{M}, \underline{\Sigma}^{\mathcal{M}})_{[\leq i]}\) (see \eqref{eq:defic}), shows that all its the candidate terms are equal to \((-1)^{i} \upsilon_{[\leq i]}(t)\).
\end{proof}

By continuity, if \(\underline{u}_{1}, \dotsc, \underline{u}_{p}\) tend simultaneously to a given \(t \in \Delta_{\underline{a}}\), then \(\varphi_{[\leq i]}(t, \dotsc, t)\) tends to \(\varphi_{[\leq i]}(t)\). This will provide a refinement of Proposition \ref{prop:variantmorse} when $\underline{l}$ belongs to a narrow cone of the form $\sum_j \mathbb R_+ \cdot \underline{u}_{j}$, for $\underline{u}_1, \dotsc, \underline{u}_p \in \mathbb N^{r}$:

\begin{prop} \label{prop:cone}
	Let \(\underline{a} = (a_{1}, \dotsc, a_{r}) \in \mathbb N_{\geq 1}^{r}\) 
	and \(\underline{u}_{1}, \dotsc, \underline{u}_{p} \in \mathbb N^{r}\) be as before. For any \(m \in \mathbb{N}\), we let \(H_{m} = \{ \underline{l} = (l_{j}) \in \mathbb{N}^{r} \; | \; \sum_{j} a_{j} l_{j} = m \}\). Assume that there exists \(t \in \Delta_{\underline{a}}\) and \(\lambda, \epsilon > 0\) such that for all \(j \in \llbracket 1, p \rrbracket\), we have 
	\[
		\norm{t - \frac{1}{\lambda}\underline{u}_j}_{\infty} < \epsilon.
	\]
	Then, for any \(i \in \llbracket 0, n \rrbracket\), any \(m \in \mathbb{N}\), and any \(\underline{l} \in (\sum_j \mathbb{R}_{+} \cdot \underline{u}_{j}) \cap H_{m}\), we have
\begin{equation} \label{eqsmallcone}
\chi^{[i]} (X, \mathcal L^{\otimes \underline{l}} \otimes M) 
	\leq 
	(-1)^{i} 
	(\varphi_{[\leq j]}(t) + O(\epsilon) ) \frac{m^{n}}{n!} + O(m^{n-1}),
\end{equation}
	where the constant appearing in the \(O(\epsilon)\) term depends only on \(\Sigma, \mathbf{e}_1, \dotsc, \mathbf{e}_r\) and the \(a_j\) (and in particular, not on \(m\) nor the \(\underline{u}_{j}\)).
\end{prop}
\begin{proof}

\setcounter{stepcc}{0}

\noindent
	\emph{\refstepcounter{stepcc} Step~\thestepcc. We determine a first asymptotic expansion for the right hand side of \eqref{eqsmallcone}}. In this step, the \(\underline{u}_{j}\) are \emph{fixed}, and only \(\underline{l}\) and \(m\) are allowed to vary.
\smallskip

Since \(\underline{l} \in (\sum_{j} \mathbb R_+ \cdot \underline{u}_{j}) \cap \mathbb{N}^{r}\), we can write \(\underline{l} = \underline{d} + \underline{v}_{0}\), with \(\underline{d} \in \sum_{j} \mathbb{N} \cdot \underline{u}_{j}\) and \(\underline{v}_{0} \in \left( \sum_{j} [0, 1] \cdot \underline{u}_{j}\right) \cap \mathbb{N}^{r}\). In particular \(\norm{\underline{v}_{0}}_{\infty}\) is bounded by a constant independent of \(m\). Write $\underline{d} = {q}_{1} \, \underline{u}_1 + \dotsc + {q}_{p} \, \underline{u}_{p}$ and let $\underline{q} = (q_{1}, \dotsc, q_{p})$. We have then
\begin{align*}
\mathcal L^{\underline{l}} \otimes M 
	& = 
	M_{1}^{\otimes q_{1}} 
	\otimes 
	\dotsc 
	\otimes 
	M_{p}^{\otimes q_{p}} 
	\otimes 
	\mathcal{L}^{\otimes \underline{v}_{0}}
	\otimes M \\
	&  = 
	\mathcal{M}^{\otimes \underline{q}} \otimes
	( \mathcal L^{\otimes \underline{v}_{0}} \otimes M) 
\end{align*}
	(recall that \(M_{j} = \mathcal{L}^{\otimes \underline{u}_{j}}\)). 
	\smallskip
	
	Since $\underline{v_0}$ is bounded, there is only a finite number of possible \(\mathcal L^{\otimes \underline{v}_{0}} \otimes M\) that can appear in the previous equation. Thus, we can apply Proposition \ref{prop:variantmorse} applied with $\mathcal L$ (resp. $\underline{\Sigma}$, resp. $M$) replaced by $\mathcal M$ (resp. $\underline{\Sigma}^{\mathcal M}$, resp. $\mathcal L^{\otimes \underline{v_0}} \otimes M$), to obtain
$$
\chi^{[i]} (X, \mathcal L^{\otimes \underline{l}} \otimes M) \leq (-1)^i c(\mathcal M, \underline{\Sigma}^{\mathcal M})_{[\leq i]} \frac{(\sum_j q_j)^n}{n!} + O((\sum_j q_j)^{n-1}).
$$
Lemma \ref{lemcontfunction} yields in turn:

\begin{equation} \label{eqasymptoticstep1}
	\chi^{[i]} 
	(X, \mathcal L^{\otimes \underline{l}} \otimes M) 
	\leq (-1)^{i} \varphi_{[\leq i]} 
	(\underline{u}_{1}, \dotsc, \underline{u}_{p}) 
	\frac{(\sum_{j} q_{j})^{n}}{n!} 
	+ 
	O((\sum_j q_j)^{n-1}).
\end{equation}

\noindent
\emph{\refstepcounter{stepcc} Step~\thestepcc. \label{step:asympt} Keeping the \(\underline{u}_{j}\) fixed, we give an asymptotic expansion of the upper bound \eqref{eqasymptoticstep1} in terms of \(m\).} A direct computation shows that 
	\[
		\sum_{j=1}^{r}
		q_{j} 
		\leq 
		\frac{1}{\max \norm{\underline{u}_{j}}_{\infty}} 
		\frac{1}{\min_{k} a_{k}} ( \sum_{j} a_{j} l_{j}) = O(m),
	\]
	hence we have a bound \(O((\sum_j q_j)^{n-1}) = O(m^{n-1})\).
	\smallskip

Moreover, we have
\begin{align*}
	m = \sum_{j=1}^{r} a_{j} l_{j} 
	& = \sum_{j=1}^{r} a_{j} d_{j} + \sum_{j=1}^{r} a_{j} (v_{0})_{j} \\
	& = \sum_{k=1}^{p} q_{k} (\sum_{j=1}^{r} a_{j} u_{k, j} ) + \sum_{j=1}^{r} a_{j} (v_{0})_{j} \\
	& = \sum_{k=1}^{p} q_{k} (\sum_{j=1}^{r} a_{j} \lambda (t_{j} + s_{k, j})) + O(1)
\end{align*}
	where we let \(s_{k, j} = \frac{1}{\lambda} u_{k, j} - t_j\). Note that the \(O(1)\) term may depend on the \(\underline{u}_j\) and the \(a_{j}\), but not on \(m\). Also, we have \(|s_{k,j}| \leq \epsilon\) by hypothesis. Thus, since \(\sum_{j} a_j t_j = 1\), still by hypothesis, we have
	\[
	m 
	= 
	(\sum_{k} q_{k}) \, \lambda \, (1 + O(\epsilon)) + O(1),
	\]	
where the constant appearing in the \(O(\epsilon)\) term may depend on the \(a_{j}\), but not on \(m\) nor on the \(u_{j}\), and the constant \(O(1)\) may not depend on \(m\).
\medskip

Inserting this in \eqref{eqasymptoticstep1}, we get
\[
	\chi^{[i]} 
	(X, \mathcal L^{\otimes \underline{l}} \otimes M) 
	\leq 
	(-1)^i \varphi_{[\leq i]} (\underline{u}_{1}, \dotsc, \underline{u}_{p}) 
	\frac{1 + O(\epsilon)}{\lambda^n} 
	\frac{m^n}{n!} + O(m^{n-1})
\]
as \(m\) goes to infinity.
\medskip

\noindent
\emph{\refstepcounter{stepcc} Step~\thestepcc. We show that the leading coefficient obtained at Step~\ref{step:asympt} is close to $(-1)^i \upsilon_{[\leq j]} (u)$.} By homogeneity, we have \(\frac{1}{\lambda^{n}} \varphi_{[\leq i]} (\underline{u}_{1}, \dotsc, \underline{u}_{p}) = \varphi_{[\leq i]} (\frac{1}{\lambda} \underline{u}_{1}, \dotsc, \frac{1}{\lambda} \underline{u}_{p})\). Since the definition of the function \(\varphi_{[\leq j]}\) depends only on \(\Sigma, \mathbf{e}_{1}, \dotsc, \mathbf{e}_{r}\), and since this function is uniformly continuous in a compact neighborhood of \(\Delta_{\underline{a}}\), we have finally
\[
	\varphi_{[\leq i]} 
	(\frac{1}{\lambda} 
	\, \underline{u}_{1}, 
	\dotsc, 
	\frac{1}{\lambda} 
	\, \underline{u}_{p}) 
	= 
	\varphi_{[\leq i]} 
	(t, \dotsc, t) 
	+ 
	O(\epsilon),
\] 
where the constant appearing in $O(\epsilon)$ may only depend on $\Sigma, \mathbf{e}_1, \dotsc, \mathbf{e}_r$ and the $a_j$. Since $\upsilon_{[\leq i]} (t) = \varphi_{[\leq i]} (t, \dotsc, t)$, this ends the proof.
\end{proof}

The line bundle $\mathcal L^{\otimes \underline{l}}$ is one of the many line bundles appearing in the natural decomposition of the symmetric product $S^m (L_1^{(a_1)} \oplus \dotsc \oplus L_r^{(a_r)})$. To prove Theorem \ref{thmineqintegral}, i.e. to obtain an upper bound on $\chi^{[i]} (X, S^m (L_1^{(a_1)} \otimes \dotsc \otimes L_r^{(a_r)}))$, we will cover \(\mathbb N^r\) by narrow cones of the form $\sum_{j} \mathbb R_+ \cdot \underline{u}_{j}$, and then apply inequality \eqref{eqsmallcone} to every line bundle appearing in the decomposition. Summing over all the cones, and then letting their width tend to \(0\), will yield the result at the next section.

\section{The Riemann sum} \label{sec:proofthmineqintegral}

This section is finally devoted to the proof of Theorem~\ref{thmineqintegral}. We introduce the following notation.

\begin{setup} Let $v_1, \dotsc, v_{n-1}$ be a basis of the primitive sublattice 
$$
H = \{(z_1, \dotsc, z_r) \in \mathbb Z^r \; | \; \sum_i a_i z_i = 0 \} \subseteq \mathbb Z^r.
$$
For $m \in \mathbb N$, let $\mathcal H_m = \{ (t_i) \in \mathbb R^r \; | \; \sum_i a_i t_i = m \}$: with the notation of Proposition~\ref{prop:cone}, we have then $H_m = \mathcal H_m \cap \mathbb N^r$. Let $\epsilon > 0$. Let $m_0 \in \mathbb N$ be large enough so that $m_0 > \frac{\max_i \norm{v_i}_{\infty}}{\epsilon}$. 
\end{setup}

\bigskip

\setcounter{stepcc}{0}

\medskip

\noindent
\emph{\refstepcounter{stepcc} Step~\thestepcc. We construct a partition of $m_0 \cdot \Delta_{\underline{a}}$ in elementary polyhedral cells.}
\medskip
 
For all \(u \in H_{m_{0}}\), we let 
\[
	C^{\circ}_{u} 
	:= 
	u 
	+ 
	\sum_{j=0}^{r-1} [0, 1[ \cdot \, v_{j} \subseteq \mathcal{H}_{m_0}.
\]
This is a fundamental domain for the action \(H_{m_{0}}\) by translation, based at \(u\). We intersect this tile with the \(m_{0}\)-dilation of \(\Delta_{\underline{a}}\), to define:
\[
	C_{u} 
	:= 
	C^{\circ}_{u} 
	\cap 
	(m_{0} \cdot \Delta_{\underline{a}}).
\]

The following facts are easy to check.

\begin{lem}  \label{lemauxiliary}
\begin{enumerate}
	\item Each \(C_{u}\) is a rational polyhedron contained in \(m_{0} \cdot \Delta_{\underline{a}}\). They cover \(m_{0} \cdot \Delta_{\underline{a}}\) and are pairwise disjoint. 
\item We have, for any fixed $m_0$, and any $m \geq m_0$
\[
	\mathrm{Card} 
		\left((\mathbb R_+ \cdot C_u) \; \cap \; H_m \right) 
		= 
		O \left( \left(\frac{m}{m_0}\right)^{r-1} \right).
\]
		as \(m \longrightarrow +\infty\). In fact, if \(C_{u} \subseteq m_{0} \cdot \overset{\circ}{\Delta_{\underline{a}}}\), the cardinal above is equivalent to \((m/m_0)^{r-1}\) as \(m \longrightarrow + \infty\).
\item  $\mathrm{card} \left( \{ u \in H_{m_0} \; | \; C_u \cap \partial (m_0 \cdot \Delta_{\underline{a}}) \neq \emptyset \} \right) = O(m_0^{r-2})$ as $m_0 \longrightarrow + \infty$.
\end{enumerate}
\end{lem}
\begin{proof}[Proof of the lemma] The first point is clear. The third point is easy to check since $\partial ( m_0 \cdot \Delta_{\underline{a}})$ is a union of $r - 2$ dimensional polyhedrons, and since all the $C_u$ are isometric, of diameter independent of $m_0$. 

	Let us prove the second point. Let $C = \sum_{j=1}^{r-1} [0,1[ \cdot\, v_{j}$: this is a fundamental domain for the lattice $H$. If $C_u \subseteq m_{0} \cdot \overset{\circ}{\Delta_{\underline{a}}}$, we have $C_u = C_u^\circ = u + C$, hence
\begin{align*}
(\mathbb R_+ \cdot C_u) \; \cap \; H_m & =   \mathbb R_+ \cdot \left( u + C \right) \cap H_m\\
& = \left[ \frac{m}{m_0} u +  \frac{m}{m_0} C \right] \cap H_m 
\end{align*} 

Thus, applying a dilation of ratio \(\frac{1}{m}\), we get
\begin{align*}
\mathrm{card}
	\left((\mathbb R_+ \cdot C_u) \; \cap \; H_m\right) & =
	\mathrm{card} 
	\left( 
	\left[ 
	\frac{1}{m_0} u 
	+ 
	\frac{1}{m_{0}} C \right] 
	\cap 
	\frac{1}{m} H_{m} 
	\right)
	\,
	\underset{m \to \infty}{\sim}
	\,
	\left(\frac{m}{m_{0}}\right)^{r-1} \\
\end{align*} 
	where we use the fact that the polytope in the brackets is an \((r-1)\)-dimensional parallelepiped of volume equal to \(\frac{1}{m_{0}^{r-1}}\). This ends the proof in the case where $C_u \subseteq m_0 \cdot \overset{\circ}{\Delta_{\underline{a}}}$. The proof of the general claim follows in the same lines, using $C_u \subseteq u + C$.
\end{proof}

\noindent
\emph{\refstepcounter{stepcc} Step~\thestepcc. \label{step:boundlb} We apply Proposition \ref{prop:cone} to each cone $\mathbb R_+ \cdot C_u$.}
\medskip

Let \(u \in H_{m_{0}}\). By construction, if \(\underline{u}_{1}, \dotsc, \underline{u}_{p}\) are the vertices of the polyhedron \(C_{u}\), we have 
\[
	\norm{\frac{u}{m_{0}} - \frac{1}{m_{0}} 
	\, 
	\underline{u}_{i}}_{\infty} 
	\leq 
	\frac{\max \norm{v_{i}}_{\infty}}{m_{0}} 
	\leq 
	\epsilon.
\]
Thus, we are in the setting of Proposition \ref{prop:cone}. For any \(m\), and any \((l_{1}, \dotsc, l_{r}) \in (\mathbb{R}_{+} \cdot C_{u} )\; \cap \; H_{m}  = (\sum_{j} \mathbb{R}_{+} \cdot \underline{u}_{j}) \cap H_{m}\), this gives
\begin{equation} \label{eqapplpropcone}
\chi^{[i]}(X, L_1^{\otimes l_1} \otimes \dotsc \otimes L_r^{\otimes l_r}) \leq (-1)^i \left(\upsilon_{[\leq j]}(\frac{u}{m_0}) + O(\epsilon)\right) \frac{m^n}{n!} + O(m^{n-1}). 
\end{equation} 
where the constant in $O(\epsilon)$ do not depend on $m_0$ nor on $m$.
\medskip

\noindent
\emph{\refstepcounter{stepcc} Step~\thestepcc. We sum over all cones $\mathbb R_+ \cdot C_u$.} We can sort the $\underline{l} \in H_m$ among the cones $\mathbb R \cdot C_u$ to which they belong, and get:
\begin{align*}
\chi^{[i]} (X,  S^m (L_1^{(a_1)} & \oplus \dotsc \oplus L_r^{(a_r)}))   = \sum_{\underline{l} \in H_m} \chi^{[i]} (X, L_1^{\otimes l_1} \otimes \dotsc \otimes L_r^{\otimes l_r})  \\
 &  = \sum_{u \in H_{m_0}} \left(  \sum_{\underline{l} \in (\mathbb R_+ \cdot C_u) \; \cap \; H_m} \chi^{[i]} (X, L_1^{\otimes l_1} \otimes \dotsc \otimes L_r^{\otimes l_r}) \right) \\
\end{align*}

Now, Step~\ref{step:boundlb} permits to bound this from above by
	\begin{equation} \label{eqintermediate1}
(-1)^i \sum_{u \in H_{m_0}}  \left( \sum_{\underline{l} \in (\mathbb R_+ \cdot C_u) \cap H_m} 1 \right) \cdot\left[ (\upsilon_{[\leq j]} \left(\frac{u}{m_0} \right) + O(\epsilon) ) \frac{m^n}{n!} + O(m^{n-1}) \right]. 
\end{equation}

We can apply Lemma \ref{lemauxiliary} (2), (3) to get an upper bound by
	\begin{equation} \label{eqintermediate2}
		(-1)^i \sum_{u \in H_{m_0}}  \left(\frac{m}{m_0} \right)^{r-1} \cdot\left[ (\upsilon_{[\leq j]} \left(\frac{u}{m_0} \right) + O(\epsilon) ) \frac{m^n}{n!} + O(m^{n-1}) \right] + O \left( \frac{m}{m_0} \right)^{r-1} O(m_0^{r-2}) O(m^n) 
\end{equation}
Indeed, we can split the sum over $u \in H_{m_0}$ in formula \eqref{eqintermediate1} in two, distinguishing among the $u$ for which $C_u \cap \partial (m_0 \cdot \Delta_{\underline{a}}) \neq \emptyset$ of not. Lemma \ref{lemauxiliary} (3) bounds the first part of the sum by the second member of the formula above. On the other hand, the sum over the elements $u \in H_{m_0}$ for which \(C_{u} \subset (m_{0} \cdot \overset{\circ}{\Delta}_{\underline{a}})\) can be bounded from above using Lemma \ref{lemauxiliary} (2) and (3): the obtained sum assumes the form \eqref{eqintermediate2}.
\medskip

Thus, we have proved that for any fixed $\epsilon > 0$, and any $m_0 > \frac{C}{\epsilon}$, we have
\begin{equation} \label{eqlimsup}
	\limsup_{m \longrightarrow + \infty} 
	\frac{\chi^{[i]} (X,  S^m (L_1^{(a_1)} \oplus \dotsc \oplus L_r^{(a_r)}))}{m^{n+ r - 1}} 
	\leq 
	\frac{(-1)^i}{n!} 
	\left( 
	\frac{1}{m_0^{r-1}} 
	\sum_{u \in H_{m_0}} 
	\upsilon_{[\leq i]}
	\left( \frac{u}{m_0} \right) 
	\right) 
	+ 
	C_1 \epsilon 
	+ 
	\frac{C_2}{m_0}.
\end{equation}
The constant \(C_{1}\) does not depend on \(m_{0}\). Indeed, in \eqref{eqintermediate2}, the constant in the \(O(\epsilon)\) term is independent of \(m_{0}\), and we have \(\frac{1}{m_{0}^{r-1}} \left( \sum_{u \in H_{m_{0}}} 1 \right) = \frac{1}{m_{0}^{r-1}} \mathrm{card}(H_{m_{0}}) \leq D\) for some constant \(D\) depending only on \(\underline{a}\). Also, the constant $C_2$ comes from the second member of \eqref{eqintermediate2} and does not depend on $\epsilon$.
\medskip

\noindent
\emph{\refstepcounter{stepcc} Step~\thestepcc. We recognize a Riemann sum in the upper bound \eqref{eqlimsup}.}

As the element $u$ runs among $H_{m_0}$, the element $t = \frac{u}{m_0}$ runs among a lattice in $\Delta_{\underline{a}}$, with fundamental domain isometric to $\frac{1}{m_0} C$. The latter has euclidean volume $\frac{1}{m_0^{r-1}} \mathrm{vol}_{r-1}(C)$. Thus, as $m_0 \longrightarrow + \infty$, we have
\begin{align*}
\frac{1}{m_0^{r-1}} \sum_{u \in H_{m_0}} \upsilon_{[\leq j]} \left(\frac{u}{m_0}\right) & = \frac{ \mathrm{vol}_{r-1}(\frac{1}{m_0} C)}{\mathrm{vol}_{r-1}(C)} \sum_{u \in H_{m_0}} \upsilon_{[\leq j]} \left(\frac{u}{m_0}\right) \\
&  \underset{m_0 \longrightarrow + \infty}{\longrightarrow} \mathrm{vol}_{r-1}(C)^{-1} \int_{\Delta_{\underline{a}}} \upsilon_{[\leq j]} \, d \mathrm{vol}_{r-1}
\end{align*} 
This shows that first letting $m_0 \longrightarrow + \infty$ and then \(\epsilon \longrightarrow 0\) in \eqref{eqlimsup}, gives
\[
	\limsup_{m \longrightarrow + \infty} \frac{\chi^{[i]} (X,  S^m (L_1^{(a_1)} \oplus \dotsc \oplus L_r^{(a_r)}))}{m^{n+ r - 1}} 
	\leq 
	\frac{(-1)^{i}}{n!} 
	\mathrm{vol}_{r-1}(C)^{-1} 
	\left( \int_{ \Delta_{\underline{a}}} \upsilon_{[\leq i]} 
	\, d \mathrm{vol}_{r-1} \right)
\]

To conclude, it suffices to use Lemma \ref{lem:lattice2}, joint to the fact that $dP = \frac{1}{\mathrm{vol}_{r-1}(\Delta_{\underline{a}})} d \mathrm{vol}_{r-1}$. 

\medskip

\section{Twist by an auxiliary $\mathbb Q$-line bundle} \label{sec:twist}

In this section, we will present a variant of Theorem~\ref{thmineqintegral} in which we allow the possibility of taking an auxiliary twist by a \(\mathbb{Q}\)-line bundle. We will state our result in a quite simple setup: in the end, all the cumbersome related to the definition of the symmetric product \(S^{m}_{H}\) will vanish, but the result will nonetheless illustrate the necessity of stating Theorem~\ref{thmineqintegral} in this level of generality.

\subsection{Statement of the result}

\begin{setup} Let \(X\) be a projective variety of dimension \(n\), and let \(L_{1}, \dotsc, L_{r}\) be {\em standard line bundles} on \(X\). Pick a stratification \(\Sigma\), with trivializations \(\mathbf{e}_{i}\) for each \(L_{i}\). Let \(N\) be a \(\mathbb Q\)-line bundle on \(X\), so that \(\Sigma\) is also adapted to \(N\), and let \(\frac{1}{d} \mathbf{g}\) be a fractional trivialization of \(N\) over \(\Sigma\). For each edge \(\mathfrak{s}\) in the tree \(\mathcal{T}\) associated to \(\Sigma\), we let \(m_{i}^{\mathfrak{s}}\) (resp. \(p^{\mathfrak{s}}\)) be the marking of \(\mathfrak{s}\) associated to \(\mathbf{e}_{i}\) (resp. \(\mathbf{g}\)). 
\end{setup}

We want to estimate the quantities \(\chi^{[i]}(X, S^{m}(L_{1}^{(a_{1})} \oplus \dotsc \oplus L_{r}^{(a_{r})}) \otimes N^{\otimes m})\), as \(m\) is sufficiently divisible. To do this, we need to adapt Definition \ref{defiupsilon} to take into account our supplementary data.

\begin{defi} \label{defiupsilonN} Let $(t_1, \dotsc, t_r) \in \mathbb R^r$. Mark each edge $\mathfrak{s}$ in $\mathcal T$ with the real number 
	\begin{equation} \label{eq:upsilonN}
		t_{1} m_{1}^{\mathfrak{s}} 
		+ 
		\dotsc 
		+ 
		t_{r} m_{r}^{\mathfrak{s}} 
		+ 
		\frac{1}{d} p^{\mathfrak{s}}.
	\end{equation}
		
		For each complete path $\sigma$ in $\mathcal T$, denote by $C_{\sigma}$ the product of all markings along the edges of $\sigma$. Then, for all $i \in \llbracket 1, n \rrbracket$, we let
	\[
	\upsilon_{[\leq i]}^{N} (t_1, \dotsc, t_r) 
	= 
	\sum_{\mathrm{index}(\sigma) \leq i} 
	C_{\sigma}\, \deg(V_{\sigma}),
	\]
	where the sum runs among all the complete path with a number of negative markings $\leq i$.
\end{defi}

We can now state the following corollary to Theorem \ref{thmineqintegral}.

\begin{corol} \label{corolineintegral} Let $\underline{a} = (a_1, \dotsc, a_r) \in \mathbb N^r$, and let $P$ denote the uniform probability measure on $\Delta_{\underline{a}}$. Then, for all $i \in \llbracket 0, n \rrbracket$, and any $m$ divisible enough, we have the asymptotic upper bound
\begin{align*}
	\chi^{[i]} (X, N^{\otimes m} \otimes S^m (&L_1^{(a_1)} \oplus \dotsc  \oplus L_r^{(a_r)}) )  \\
& \leq  \frac{\mathrm{gcd}(a_1, \dotsc, a_r)}{a_1 \dotsc a_r} \binom{n + r - 1}{r - 1} \left[ \int_{\Delta_{\underline{a}}} \upsilon_{[\leq i]}^N dP \right] \frac{m^{n+ r - 1}}{(n+ r - 1)!} + o(m^{n+r - 1}).
\end{align*}

\end{corol}

\begin{rem}
	To ease future reference, we made in the statement the slight abuse of language that we tried to avoid in the discussion of Sections~\ref{sectstatement} and \ref{sec:redBGcoverings}. In this statement, what is meant by \(N^{\otimes m}\) is a power of the form \(M^{\otimes m/d}\), where \(M\) is a standard line bundle so that \(N^{\otimes d}\) is the image of \(M\) in \(\mathrm{Pic}(X)_{\mathbb{Q}}\). If we make a different choice, this will add a torsion line bundle in the truncated Euler characteristic, ranging in a finite set; this will not affect the asymptotic expansion.
\end{rem}
\bigskip

The proof follows from the following fact.

\begin{lem}
	Let \(M_{j} := L_{j} \otimes N^{\otimes a_{j}} \in \mathrm{Pic}_{\mathbb{Q}}(X)\) for all \(j \in \llbracket 1, r \rrbracket\). Then :
	\begin{enumerate}
		\item in Proposition~\ref{prop:condgroup}, if we replace the \(L_{i}\) by the \(M_{i}\), the condition \eqref{eq:inclgroups} is satisfied;
		\item we may pick \(H\) in Proposition~\ref{prop:condgroup} so that for any \(m \in \mathbb{N}\) divisible enough, we have
	\[
		S^{m}_{H}
		(M_{1}^{(a_{1})}
		\oplus
		\dotsc
		\oplus
		M_{r}^{(a_{r})})
		=
		S^{m}
		(L_{1}^{(a_{1})} \oplus \dotsc \oplus L_{r}^{(a_{r})})
		\otimes
		N^{\otimes m}
	\]
	\end{enumerate}
\end{lem}
\begin{proof}
	Again, \(N^{\otimes m}\) in the statement of this result means \(M^{\otimes m/e}\), where \(M\) is an adequate standard line bundle, fixed once and for all.
	\medskip

\noindent
	{\em (1)} If \(\sum_{j} a_{j} b_{j} = 0\), we have
\[
	M_{1}^{\otimes b_{1}} \otimes \dotsc \otimes M_{r}^{\otimes b_{r}}
	=
	L_{1}^{\otimes b_{1}} \otimes \dotsc \otimes L_{r}^{\otimes b_{r}}
	\otimes
	N^{\otimes (\sum_{j} a_{j} b_{j})}
	=
	L_{1}^{\otimes b_{1}} \otimes \dotsc \otimes L_{r}^{\otimes b_{r}},
\]
which is the image of a standard line bundle in \(\mathrm{Pic}(X)_{\mathbb{Q}}\).
\medskip

\noindent
	{\em (2)} Let \(M \in \mathrm{Pic}(X)\) be such that \(N^{\otimes e}\) is the image of \(M\) for some \(e > 0\). Assume that \(m \in \mathbb{N}\) is divisible by \(e\). Then, any \((b_{1}, \dotsc, b_{r})\) satisfying \(\sum_{j} a_{j} b_{j} = m\) is such that
	\[
		M_{j}^{\otimes b_{1}} \otimes \dotsc \otimes M_{r}^{\otimes b_{r}}
	\]
	is the image of \(L_{1}^{\otimes b_{1}} \otimes \dotsc \otimes L_{r}^{\otimes b_{r}} \otimes M^{\otimes m/e}\) in \(\mathrm{Pic}(X)_{\mathbb{Q}}\). We may thus take a finite number of elements in the group generated by \(L_{1}, \dotsc, L_{r}, M\), that generate in turn a group \(H\) as required.
\end{proof}
\medskip

Thus, applying Theorem~\ref{thmineqintegral} to the \(M_{j}\) will yield the result: we just need to exhibit an adequate trivialization for these line bundles on \(\Sigma\). We can simply pick the trivialization \(\frac{1}{d} (\mathbf{e}_{i}^{d} \otimes \mathbf{g})\). For each \(\mathbb{Q}\)-line bundle \(M_{j}\), this gives to each edge \(\mathfrak{s}\) the weight
\[
	m_{j}^{\mathfrak{s}} + a_{j} \frac{1}{d} p^{\mathrm{s}}.
\]

Thus, for \((t_{1}, \dotsc, t_{r}) \in \Delta_{\underline{a}}\), Definition~\ref{defiupsilon} prescribes to mark each edge \(\mathfrak{s}\) with the weight
\[
	\sum_{1 \leq j \leq r}
	t_{j} (
	m_{j}^{\mathfrak{s}} + a_{j} \frac{1}{d} p^{\mathrm{s}}.
	)
\]

However, since \(\sum_{1 \leq j \leq r} t_{j} a_{j} = 1\), this is exactly equal to \eqref{eq:upsilonN}. This shows that the inequality provided by Theorem~\ref{thmineqintegral} coincides with the one of Corollary~\ref{corolineintegral}.

\chapter{Existence of jet differentials on varieties of general type} \label{chap:existence}

In this chapter, we prove Demailly's existence theorem for Hasse-Schmidt jet differentials on a smooth projective variety of general type.
\medskip

{\em In this section, we assume that \(\mathbbm{k}\) is a an arbitrary algebraically closed field.}
\medskip

The main result is the following:

\begin{thm} [\cite{dem11}, for \(\mathbbm{k} = \mathbb{C}\)] \label{thm:existence} Let \(X\) be a smooth projective variety over the field \(\mathbbm{k}\), and let \(A\) be an ample divisor on \(X\). Assume that \(X\) is of general type. Then, for a sufficiently large fixed \(k \gg 1\), the \(\mathcal{O}_{X}\)-algebra \(E_{k, \bullet}^{HS}\Omega_{X}\) is big. In other words, there exists a constant \(C_{k} > 0\) such that for \(m \gg 1\) sufficiently divisible, we have
	\[
		h^{0}(X, E_{k, m}^{HS}\Omega_{X}) \geq C_{k} m^{n+kn-1}
	\]
	This implies that for \(m \gg k \gg 1\), we also have
	\[
		H^{0}(X, E_{k, m}^{HS} \Omega_{X} \otimes \mathcal{O}_{X}(-A)) \neq 0.
	\]
\end{thm}

Recall that we say that a smooth projective variety \(X\) is of {\em general type} if its canonical bundle is {\em big}, i.e. satisfies
\[
	K_{X} \sim_{\mathbb{Q}} A + E
\]
where \(A\) and \(E\) are respectively ample and effective \(\mathbb{Q}\)-Cartier divisors.

\section{Weighted direct sums and Green-Griffiths vector bundles} \label{sectweightproj}

\medskip

\begin{setup}
	Let \(X\) be a smooth projective variety over \(\mathbbm{k}\). We resume the notation introduced in Section~\ref{sec:weighted}. In particular, we introduce the following weighted direct sums
	\[
		\mathbf{\Omega}_{X}^{k}
		:=
		\Omega_{X}^{(1)}
		\oplus
		\dotsc
		\oplus
		\Omega_{X}^{(k)}
	\]
\end{setup}
\medskip

Recall that the \(\mathcal{O}_{X}\)-algebra \(E_{k, \bullet}^{HS} \Omega_X\) is endowed with a canonical \(\mathbb N^{k}\)-filtration, that we denote by \(F_\bullet E_{k, \bullet}^{HS} \Omega_X\). By Theorem~\ref{thm:grading}, its associated graded algebra is isomorphic to 
\begin{equation} \label{eqgradedterm}
	\mathrm{Gr}_{F} \left( E_{k, \bullet}^{HS} \Omega_X \right) 
	\cong 
	S^{\bullet} 
	\mathbf{\Omega}^{k}_{X} 
\end{equation}
where the right hand term is defined as in \eqref{eq:deftensorweighted}.
\medskip

To prove Theorem \ref{thm:existence}, the strategy coming from \cite{dem11} is to bound from above $\chi^{[1]}(X, E_{k, m}^{HS} \Omega_X)$ for a sufficiently large \(k \in \mathbb{N}\), with \(m\) going to infinity. By the standard growth properties of the truncated Euler characteristic, it is actually enough to bound from above the graded term.

\begin{prop} \label{propcompjetvb} For any $k, m \in \mathbb N$, and any $0 \leq i \leq n$, we have
\begin{equation} \label{equpperboundjet}
\chi^{[i]} 
	\left( 
	X, E_{k, m}^{HS} \Omega_{X} 
	\right) 
	\leq 
	\chi^{[i]} 
	\left( 
	X, S^{m} \mathbf{\Omega}_{X}^{k}) 
	\right).
\end{equation}
\end{prop}

This is indeed just an application of Lemma~\ref{lem:filtration}.\footnote{This simple argument was communicated to me by L. Darondeau ; it seems that J. Merker was the first to note the relevance of this remark in this context (see \cite{mer15}).}

\begin{rem}
	When $m$ is divisible enough, we can draw a more geometric picture for the last proposition, using the Rees deformation discussed in Section~\ref{ref:reesdeformation}.  We can apply the upper semi-continuity property of $\chi^{[i]}$  (see Demailly \cite{dem95} for the case where \(\mathbbm{k} = \mathbb{C}\)) to obtain, for a generic $\lambda \neq 0$,
\[
	\chi^{[i]} ((P_{k}X)_{\lambda}, \mathcal{O}(m)|_{(P_{k}X)_{\lambda}}) 
	\leq 
	\chi^{[i]} ((P_{k}X)_{0}, \mathcal{O}(m)|_{(P_{k}X)_{0}}) 
\]
	Now the left hand side is equal to \(\chi^{[i]} (X_k^{HS}, \mathcal O_k^{HS}(m)) = \chi^{[i]} (X, E_{k, m}^{HS} \Omega_X)\), and the right hand side identifies similarly with the right hand side of \eqref{equpperboundjet}. 
\end{rem}

\bigskip

To prove the main result, we will actually need a more general version relative version of Proposition~\ref{propcompjetvb}, with an additional twist by an auxiliary line bundle.
\medskip

Let $p : X' \longrightarrow X$ be projective morphism, where \(X'\) is any variety, and let \(E\) be an effective Cartier divisor on \(X'\). Then, for any \(k, m \in \mathbb{N}\), the vector bundle \(p^{\ast} E_{k, m}^{HS}\) also admits a filtration with graded term
$$
p^\ast S^m (\Omega_X^{(1)} \oplus ... \oplus \Omega_X^{(k)}).
$$ 

This admits
\begin{equation} \label{eqsubsheaf}
	\mathcal O(-m E) \otimes p^\ast S^m (\Omega_X^{(1)} \oplus ... \oplus \Omega_X^{(k)})
\end{equation} 
as a subsheaf, and the $\mathcal O_{X'}$-module $\mathcal E_m = \mathcal O(-m E) \otimes p^\ast E_{k,m}^{HS} \Omega$ has a natural induced filtration for which the graded module is isomorphic to \eqref{eqsubsheaf}. 
\medskip

In this context, Lemma~\ref{lem:filtration} applied to $\mathcal E_m$ yields the following result.

\begin{prop} \label{prop:relative}
	Let $p : X' \longrightarrow X$ be a morphism of projective varieties of dimension $n$ with \(X\) smooth, and let $E$ be an effective Cartier divisor on $X'$. Fix $k \in \mathbb N$. Then, for any $m \in \mathbb N$, the subsheaf $\mathcal{E}_{m} := \mathcal{O}_{X'}(-mE) \otimes p^\ast E_{k, m}^{HS} \Omega_X$ satisfies
\[
	\chi^{[i]} 
	\left( X', \mathcal E_m \right) 
	\leq 
	\chi^{[i]} 
	\left( X',	
	\mathcal O(-m E) 
	\otimes 
	p^\ast S^{m} 
	(\Omega_X^{(1)} \oplus ... \oplus \Omega_X^{(k)}) \right)
\]
	for any \(i \in \llbracket 0, n \rrbracket\).
\end{prop}

In particular, if we specialize the result to \(i = 1\), we get the following.

\begin{prop} \label{proptwisteestimate}
	Let $p : X' \longrightarrow X$ be a generically finite morphism of projective varieties of dimension $n$, with \(X\) smooth, and let $E$ be an effective Cartier divisor on $X'$. Then, for any fixed \(k\) and \(m \gg 1\) divisible enough, we have
\[
	h^0(X, E_{k,m}^{HS} \Omega) 
	\geq 
	- 
	\frac{1}{\mathrm{deg}(p)}
	\chi^{[1]} \left( X',	\mathcal O(-m E) 
	\otimes 
	p^\ast 
	S^{m} (\Omega_X^{(1)} \oplus ... \oplus \Omega_X^{(k)}) 
	\right)
	+
	O(m^{n + kn - 1})
\]
as \(m \longrightarrow + \infty\).
\end{prop}
\begin{proof}
	Let us fix \(k \in \mathbb{N}\). First of all, remark that we have
	\[
	h^0(X, E_{k,m}^{HS} \Omega) 
	=
	\frac{1}{\mathrm{deg}(p)}
	h^0(X', p^{\ast} E_{k,m}^{HS} \Omega) + O(m^{n+nk-1}).
	\]
	as \(m \longrightarrow + \infty\) while being divisible by \(\mathrm{lcm}(1, 2, \dotsc, k)\). To see this, consider the fiber product
	\[
		\begin{tikzcd}
			X' \times_{X} X_{k}^{HS} 
				\ar[r, "q"] \ar[d, "\rho"] 
			&
			X_{k}^{HS}
				\ar[d, "\pi"]
				\\
			X'
				\ar[r, "p"]
			&
			X
		\end{tikzcd}
	\]
	The morphism \(q\) is generically finite, so we can apply the standard Lemma~\ref{lemmodification} to find:
	\[
	h^{0}(X_{k}^{HS}, \mathcal{O}_{k}(m))
	=
	\frac{1}{\deg p} h^{0}(X' \times_{X} X_{k}^{HS}, q^{\ast} \mathcal{O}_{k}(m))
	+
	O(m^{n+nr-1})
	\]
	On the other hand, we have, for \(m\) divisible enough
	\[
		H^{0}(X, E_{k, m}^{HS} \Omega_{X}) = H^{0}(X_{k}^{HS}, \mathcal{O}_{k}(m)) 
	\]
	and
	\[
		H^{0}(X', p^{\ast} E_{k, m}^{HS} \Omega_{X}) = H^{0}(X' \times_{X} X_{k}^{HS}, q^{\ast} \mathcal{O}_{k}(m)) 
	\]
	by the standard properties of projectivized algebras. This proves our claim.
	\smallskip

	To conclude, simply apply Proposition~\ref{prop:relative} to \(i=1\), and write
	\begin{align*}
		h^{0}(X', p^{\ast} E_{k, m}\Omega_{X}) & \geq h^{0}(X', \mathcal{E}_{m}) \\
		                                       & \geq - \chi^{[1]}(X', \mathcal{E}_{m}) \\
						       & \geq - \chi^{[1]} \left( X',	\mathcal O(-m E) 
	\otimes 
	p^\ast 
	S^{m} (\Omega_X^{(1)} \oplus ... \oplus \Omega_X^{(k)}) 
	\right)
	\end{align*}
\end{proof}

\section{An asymptotic expansion for a sum of copies of the same vector bundle}

The point of Proposition \ref{propcompjetvb} (or of the more general Proposition \ref{proptwisteestimate}) is to permit us to limit our study to the symmetric algebra \(S^{\bullet} \mathbf{\Omega}_{X}^{k}\). The goal of this section is now to perform this study for the more general case of symmetric products of the form $S^m(E^{(1)} \oplus ... \oplus E^{(k)})$, where $E$ is any vector bundle with $\mathrm{det} \, E$ big. 
\medskip

As we will see in the next section, it will actually be enough to deal with the case where \(E\) is itself a direct sum of line bundles.

\subsection{Splitting principle and reduction to a sum of line bundles} \label{sectredsum}

Let us give some details about the \emph{splitting principle} that will allow us to get back to the case where all the \(E_{i}\) are line bundles. The arguments are quite close to the discussion of Section~\ref{sec:reductionsteps}.

\begin{prop} \label{propredlb}
	Let \(\mathbf{E} := E_{1}^{(a_1)} \oplus ...\oplus E_{r}^{(a_r)}\) be a weighted direct sum over a projective variety \(X\). Then, there exists a generically finite dominant morphism \(p : X' \longrightarrow X\), where \(X'\) is a variety, and a weighted direct sum \(\mathbf{E}' = (E_{1}')^{(a_1)} \oplus ... \oplus (E_{r}')^{(a_r)}\) over \(X'\), such that
\begin{enumerate} \itemsep=0em
\item each \(E_{j}'\) is a direct sum of line bundles over \(X'\), of the same rank as \(E_{j}\) ;
\item for each \(j\) we have \(\det(E_{i}') \cong p^\ast \, \det (E_{i})\),
\item for any $i \in \llbracket 0, n \rrbracket$, we have the asymptotic inequality
\[
	\chi^{[i]} 
		(X, S^{m} \, \mathbf{E}) 
		\leq 
		\frac{1}{\mathrm{deg}(p)}
		\chi^{[i]} 
		(X', S^{m} \, \mathbf{E}') + O(m^{n + {\rm rk}\, \mathbf{E} - 1}).
\]
		as \(m\) goes to infinity while being divisible by \(\mathrm{lcm}(a_{1}, \dotsc, a_{r})\).
\end{enumerate}
\end{prop}
\begin{rem}
	The item (2) implies in particular that $\det(E_{i}')$ is big if $\det(E_{i})$ is.
\end{rem}

\begin{proof}
	By Lemma \ref{lem:splittingprinciple}, there exists a generically finite dominant, projective morphism $p : X' \longrightarrow X$ such that all the direct summands of \(p^\ast \mathbf{E} = p^\ast E_{1}^{(a_1)} \oplus ... \oplus p^\ast E_{r}^{(a_r)}\) admit filtrations with direct sums of line bundles as graded objects. The same argument as in Proposition~\ref{proptwisteestimate} yields
	\[	
	\chi^{[i]} 
	(X, S^m \, \mathbf{E}) 
	\leq 
	\frac{1}{\mathrm{deg}(p)}
	\chi^{[i]} 
	(\widetilde{X}, p^{\ast} S^{m} \mathbf{E}) + O(m^{n + \mathrm{rk} \mathbf E - 1}).
	\]
	Now, denote by \(E_{i}'\) the direct sum of the line bundles appearing in the graded object of \(p^{\ast} E_{i}\).  Note that we have \(\det E_{i}' \cong \det p^{\ast}E_{i}\), as required.
	\smallskip
	
	Also, taking \(\mathbf{E}'\) as in the statement of the proposition, we see that the pull-back \(p^{\ast} S^{m} \mathbf{E}\) admits a filtration with \(S^{m}\mathbf{E}'\) as its graded object. Now, an application of Lemma~\ref{lem:filtration} gives
\[
	\chi^{[i]} 
	(X',  p^{\ast} (S^{m} \mathbf{E})) 
	\leq  
	\chi^{[i]} (X',  S^{m}  \mathbf{E}'),
\]
This ends the proof.
\end{proof}

\subsection{Statement of the result} \label{sectstatement}

Our plan is now to apply Corollary \ref{corolineintegral} to the situation of a direct sum $E^{(1)} \oplus ... \oplus E^{(k)}$, where $E$ is a direct sum of line bundles. The main result of this section is an algebraic version of \cite[Theorem 2.37]{dem11}. Before stating it, let us introduce a simplifying notation. 
\medskip

\begin{nota} If \(E\) is a vector bundle, and \(k \in \mathbb{N}\) an integer, we introduce the following weighted direct sum
	\[
		\mathbf{E}_{k} := E^{(1)} \oplus ... \oplus E^{(k)}.
	\]
\end{nota}

The main result of this section can now be stated as follows. As explained above, we will deal with case where \(E\) is a direct sum of line bundles.

\begin{thm} \label{thmcomparison} Let $X$ be a projective variety of dimension $n$, and let \(E\) be a vector bundle on \(X\). Assume that \(E\) is a direct sum of line bundles. Let $N$ be an auxiliary line bundle on $X$. For each $k \in \mathbb{N}^{\ast}$, we introduce the $\mathbb Q$-line bundle
	$$
	N_k = \mathcal O_X \left( \frac{1}{kr} ( 1 + \frac{1}{2} + ... + \frac{1}{k} ) F \right).
	$$

Assume that $\Sigma$ is a stratification of $X$, adapted to $\mathrm{det}\, E \otimes N = L_1 \otimes ... \otimes L_r \otimes N$, and let $\mathbf{e}$ be a trivialization of $\det\, E \otimes N$ over $\Sigma$. Let $\underline{\Sigma} = (\Sigma, \mathbf{e})$.
\medskip

Then, for all $j \in \llbracket 0, n \rrbracket$, and all $m \gg k \gg 1$, with $m$ divisible enough, we have
\begin{align*}
	\chi^{[j]} (X, S^m \mathbf{E}_{k} \, \otimes \, N_k^{\otimes m })  \leq (-1)^j \frac{(\log k)^n}{n! (k!)^r} & \left( c_1 (\det E \otimes N, \underline{\Sigma})^n_{[\leq j]} + O(\frac{1}{\log k}) \right) \frac{m^{n+kr - 1}}{(n + kr - 1)!}  \\
	& + o(m^{n+kr-1})
\end{align*}
\end{thm}

The proof of this result will be completed in Section~\ref{sec:proofasympt}, after some setup and a few reduction steps.

\begin{setup} \label{defitrivial} By Propositions \ref{proprefine} and \ref{proprefine2}, we do not lose generality in assuming that $\Sigma$ is also adapted to $L_1, ..., L_r, N$. Under this hypothesis, we introduce trivializations $\mathbf{e}_1, ..., \mathbf{e}_r, \mathbf g$ of $L_1, ..., L_r, N$ on $\Sigma$, such that the following holds. For any irreducible component $V$ appearing in the stratification $\Sigma$, if $U \subseteq V$ is the complement of the natural strata on $V$, and if $e_i$ (resp. $g, e$) is the trivialization of $L_i$ (resp. $N$, resp. $\det E \otimes N$) on $U$ given by $\mathbf{e}_i$ (resp. $\mathbf{g}$, resp. $\mathbf{e}$), we ask that
	\[
e_1 \otimes ... \otimes e_r \otimes g = e.
	\]  

Note that this it always possible to find $\mathbf{e}_1, ..., \mathbf{e}_r, \mathbf{g}$ as above, by first fixing $e_1, ..., e_{r}, e$ on $U$, and then letting $g = e \cdot (e_1)^{-1} \cdot ... \cdot (e_{r})^{-1}$.
\end{setup}
\bigskip

Let now $\underline{k} := (1, ..., 1, 2, ..., 2, ...., k, ..., k)$, where each number is repeated $r$ times. Applying Corollary \ref{corolineintegral} to the weighted direct sum $\mathbf{E}_k = L_1^{(1)} \oplus ... \oplus L_r^{(1)} ... \oplus L_1^{(k)} \oplus ... \oplus L_r^{(k)}$ and to the $\mathbb Q$-line bundle $N_k$, we get, for \(m\) divisible enough:
\begin{align} \nonumber 
	\chi^{[i]} (X, N_k^{\otimes m } \; \otimes \; & S^m \mathbf{E}_{k}) \\ \label{equpperboundint}
	& \leq \frac{1}{(k!)^r} \binom{n+ kr - 1}{kr - 1} \left[ \int_{ \Delta_{\underline{k}}} \upsilon^{N_k}_{[\leq j]} \, d P \right] \frac{m^{n + kr - 1}}{(n+ kr - 1)!} + o(m^{n+ kr - 1})
\end{align}
where $\upsilon^{N_k}_{[\leq j]} : \mathbb R_+^{kr} \longrightarrow \mathbb R$ is the function provided by Definition \ref{defiupsilonN}, and $dP$ is the uniform probability measure on $\Delta_{\underline{k}}$.
\medskip

The theorem will come directly from the following asymptotic estimate of the integral term.

\begin{prop} \label{prop:asymptint} We have, as $k \longrightarrow + \infty$,
\begin{equation} \label{eqasymptint}
	\int_{ \Delta_{\underline{k}}} \upsilon^{N_k}_{[\leq j]} \, d P = \frac{(\log k)^n}{(kr)^n} \left[ c_1(\det E \otimes N, \underline{\Sigma})^n_{[\leq j]} + O(\frac{1}{\log k}) \right].
\end{equation}
\end{prop}

This proposition implies Theorem \ref{thmcomparison} right away: it suffices to insert \eqref{eqasymptint} in \eqref{equpperboundint}, and to remark that $\frac{1}{(kr)^n} \binom{n + kr - 1}{kr - 1} = \frac{1}{n!} (1 + O(\frac{1}{k}))$, with $n, r$ fixed, and $k \longrightarrow + \infty$.
\medskip

\subsection{Proof of Proposition~\ref{prop:asymptint}} \label{sec:proofasympt}

We propose to further elaborate on Demailly's Monte-Carlo approach, and to interpret the integral in \eqref{eqasymptint} as the mean value of the random variable $\upsilon^{N_k}_{[\leq j]}$ depending of a uniform sorting in $\Delta_{\underline{k}}$. The reader should compare \eqref{equpperboundint} with \cite[(2.17)]{dem11}: even though the computations are closely related, our asymptotic estimate is slightly different to the one of Demailly, as our random variables will depend on random sorting inside $\Delta_{\underline{k}}$, and not on a product $\Delta^{k-1} \times (S^{2r - 1})^k$. We refer to Annex~\ref{ann:simplexes} for some useful computations related to uniform random variables on simplexes.
\bigskip

Let \(\mathcal{T}\) the tree associated to \(\Sigma\), and let \(\sigma\) be a root-to-leaf path in \(\mathcal{T}\). For all \(i \in \llbracket 1, n \rrbracket\), we denote by $V_i^\sigma$ the irreducible $i$-dimensional variety that appears along the labels of $\sigma$ (see Remark \ref{remtree}). We also denote by $f_i : V_i^\sigma \longrightarrow V_{i+1}^\sigma$ the natural map provided by the stratification. Now, for all $i \in \llbracket 1, n \rrbracket$, and all $j \in \llbracket 1, r \rrbracket$, denote by $d_j^i(\sigma)$ the multiplicity along $f_{i-1}(V_{i-1})$ of the trivialization of $L_j$  provided by $\mathbf{e}_j$. Also, let $d^i(\sigma)$ (resp. $p^i(\sigma)$) denote the multiplicity of the trivialization of $\det E = L_1 \otimes ... \otimes L_r \otimes N$ (resp. $N$) provided by $\mathbf{e}$ (resp. $\mathbf{g}$) along $f_{i-1}(V_{i-1}^\sigma)$.
\smallskip

By our choice of $\mathbf{e}_1, ..., \mathbf{e}_r$ and $\mathbf{e}$ in Setup~\ref{defitrivial}, the following property is straightforward.

\begin{lem} \label{lemsum}
	For all complete path $\sigma$ in $\mathcal T$, and for all $i \in \llbracket 1, n \rrbracket$, we have 
	\[
		d^i(\sigma) = d_1^i(\sigma) + ... + d_r^{i}(\sigma) + p^i(\sigma).
	\]
\end{lem}

In this setting, Definition \ref{defiupsilonN} prescribes to compute $\upsilon^{N_k}_{[\leq j]}$ as follows. Let 
$$
t = (t_{j, l})_{1 \leq j \leq k, 1 \leq l \leq r} \in \Delta_{\underline{k}}.
$$
For all $i \in \llbracket 1, n \rrbracket$ and all root-to-leaf path $\sigma$, mark the edge from $V_{i-1}^\sigma$ to $V_i^\sigma$ with the real number 
$$
A_{i}^\sigma (t) = \sum_{ 1 \leq j \leq k, 1 \leq l \leq r} t_{j, l} d_l^{i}(\sigma) + p'{}^i(\sigma). 
$$
where $p'{}^i(\sigma) = \left[ \frac{1}{kr} \left( 1 + \frac{1}{2} + ... \frac{1}{k} \right) \right] p^i(\sigma)$.
\smallskip

Then, we have
\begin{equation} \label{equpsilon}
	\upsilon^{N_k}_{[\leq j]} (t) = \sum_{\sigma \; \text{root-to-leaf}} \left[ \mathbbm 1_{\{\mathrm{index}(\sigma) \leq j \}} \prod_{1 \leq i \leq n} A_{i}^\sigma(t) \deg_{\mathbbm{k}}\, (V_{0}^{\sigma}) \right],
\end{equation}
where $\mathbbm 1_{\{\mathrm{index}(\sigma) \leq j \}} = 1$ if there are less that $j$ negative values among the $A_i^\sigma(t)$ ($1 \leq i \leq n$), and $0$ otherwise. Note that this index depends on $t$: we will not write explicitly this dependence to lighten a bit the notations.
\medskip

We will now interpret each $A_{i}^\sigma(t)$, as well as $\upsilon^{N_k}_{[\leq j]}(t)$, as a {\em random variable}, using the probability measure $dP$ to draw a random element $t \in \Delta_{\underline{k}}$. To simplify the presentation, let us fix a complete path $\sigma$, and remove it for the time being from our notations. In the next lemma, we give estimates on the expectancy value and the variance of the $A_i$.

\begin{lem} \label{lemvariance} Let $i \in \llbracket 1, n \rrbracket$. Then,
\begin{enumerate}
	\item the expectancy value of $A_i$ satisfies $\mathrm{E}(A_i) \sim \frac{\log k}{kr} d^i$ as $k \longrightarrow + \infty$.
	\item There is a constant $C_i$ depending only on the $d^i_l$ (for $1 \leq l \leq r$), such that 
\[
	\mathrm{Var}(A_i) 
		\leq 
		\frac{C_i}{k^2} V.
\]
\end{enumerate}

\end{lem}
\begin{proof}
	We are in the situation of Section \ref{sect:simplexprob} : $t = (t_{j, l})_{1 \leq j \leq r\, 1 \leq l \leq k}$ is drawn with uniform law in the simplex $\Delta_{\underline{k}}$, and $A_i(t)$ is an affine function of $t$ of the form $A_i(t) = \sum_{1 \leq j \leq k} \sum_{1 \leq l \leq r} t_{j, l} d^i_l + p'{}^i$.

	\emph{(1)} Since $A_i$ is an affine function, Lemma~\ref{lem:averageaffine} shows that its mean value on $\Delta_{\underline{k}}$ is equal to the average value of the images of the vertices of $\Delta_{\underline{k}}$ by $A_i$. These vertices are the $\mathrm{v}_{j,l} = (\frac{1}{l} \delta_{j, j'} \delta_{l, l'})_{1 \leq j' \leq k, 1 \leq l' \leq r}$ for $1 \leq j \leq k$ and $1 \leq l \leq r$. The affine function $A_i$ takes the value $\frac{1}{j} d^i_l + p'{}^i$ on $\mathrm{v}_{j,l}$.

	Thus
	\begin{align*}
		\mathrm{E}[A_i] & = p'{}^{i} + \frac{1}{k r}  \sum_{1 \leq l \leq r} \sum_{1 \leq j \leq k} \frac{1}{j} d^i_l  \\
				& = \frac{1}{kr} \sum_{1 \leq j \leq k} \frac{1}{j} \left( \sum_{1 \leq l \leq r} d^i_l + p^i \right) 
	\end{align*}
	This gives the first point, since $\sum_{1 \leq l \leq r} d^i_l + p^i = d^i$ by Lemma \ref{lemsum}.

\emph{(2)} This follows right away from Lemma \ref{lemboundvariance}: if $M$ is the mean value of the random variable $\left( \sum_{1 \leq l \leq r} T_l d_l \right)^2$, with $T$ uniformly distributed in $\Delta^{r-1}$, the lemma provides the result with $C_i = \frac{\pi^2}{3} M$. 
\end{proof}

We now state the fundamental lemma that will allow us to end the proof of Proposition \ref{prop:asymptint}: it can be seen as a version of \cite[Lemma 2.25]{dem11}, adapted to our combinatorial context. Recall that we are working on a fixed path $\sigma$. We let $j_\sigma$ be the index of this complete path for the trivialization $\mathbf e$, i.e. the number of negative labels among the $d^i = d^i(\sigma)$.

\begin{lem} \label{lemupperbounddiff}
Let $j \in \llbracket 0, n \rrbracket$. Then, we have
\begin{align} \nonumber
	| \; \mathrm{E}( \mathbbm 1_{\{\mathrm{index(\sigma)} = j\} } \prod_{1 \leq i \leq n} A_i ) -  \delta_{j, j_\sigma} & \prod_{1 \leq i \leq n}   \mathrm{E}(A_i) \; | \\ \label{equpperbounddiff}
& \leq \left[ \sum_{1 \leq p \leq n} \; \; \left( \prod_{1 \leq q \leq p - 1} \mathrm{E} (A_q^2) \right)  \; \mathrm{Var}(A_p) \left( \prod_{p+1 \leq s \leq n} \mathrm{E}(A_s)^2 \right) \right]^{1/2}
\end{align}
\end{lem}

\begin{proof} The proof is essentially the same as \cite[Lemma 2.25]{dem11}, and is based on the following observation: let $(a_1, ..., a_t), (b_1, ..., b_t) \in \mathbb R^t$ be such that there are exactly $\alpha$ negative numbers among the $a_i$, and $\beta$ negative numbers among the $b_i$. Then we have, for any $j$:
$$
|\mathbbm{1}_{\{j=\alpha\}} \prod_{i} a_i - \mathbbm{1}_{\{j=\beta\}} \prod_{i} b_i| \leq \sum_{1 \leq p \leq t} \left( \prod_{1 \leq q \leq p-1} |a_q| \right) |a_p - b_p| \left( \prod_{p + 1 \leq s \leq t} |b_s| \right).
$$
This is easy to show by distinguishing among the possible values of $j$.

This observation gives:
\begin{align*}
| \mathbbm{1}_{\{\mathrm{index}(\sigma) = j \}} \prod_i A_i -  \mathbbm 1_{\{ j = j_\sigma \} } & \prod_{i} \mathrm E(A_i) | \leq \\
& \leq \sum_{1 \leq p \leq n} \; \; \left( \prod_{1 \leq q \leq p - 1} | A_q | \right)  \; | A_p - \mathrm{E}(A_p) | \; \left( \prod_{p+1 \leq s \leq n} \mathrm{E}(A_s) \right)
\end{align*}

Taking the expectancy value, we obtain
\begin{align*}
| \; \mathrm{E}( \mathbbm 1_{\{\mathrm{index(\sigma)} = j\} } \prod_i A_i) - \mathbbm 1_{\{ j = j_\sigma \} } & \prod_{i} \mathrm E(A_i) \; |^2\\
& \leq \mathrm{E} \left(| \mathbbm{1}_{\{\mathrm{index}(\sigma) = j \}} \prod_i A_i -  \mathbbm 1_{\{ j = j_\sigma \} } \prod_{i} \mathrm E(A_i) | \right)^2 \\
& \leq \mathrm{E} \left( \sum_{1 \leq p \leq n} \; \; \left( \prod_{1 \leq q \leq p - 1} | A_q | \right)  \; | A_p - \mathrm{E}(A_p) | \; \left( \prod_{p+1 \leq s \leq n} \mathrm{E}(A_s) \right) \right)^2 \\
& \leq \sum_{1 \leq p \leq n} \; \; \left( \prod_{1 \leq q \leq p - 1} \mathrm{E} (A_p^2) \right)  \; \mathrm{E} (|A_p - \mathrm{E}(A_p)|^2) \; \left( \prod_{p+1 \leq s \leq n} \mathrm{E}(A_s)^2 \right),
\end{align*}
where, at the last line, we used Cauchy-Schwarz inequality $\mathrm{E}(XY)^2 \leq \mathrm{E}(X^2) \mathrm{E}(Y^2)$. Since $\mathrm{Var}(A_l) = \mathrm{E} (|A_l - \mathrm{E}(A_l)|^2)$, we get the result.
\end{proof}

We are now ready to end the proof of Proposition \ref{prop:asymptint}, by summing the previous estimates over all paths $\sigma$.

\begin{proof}[Proof of Proposition \ref{prop:asymptint}]
	For each root-to-leaf path in \(\mathcal{T}\), let us write \(\delta_{\sigma} :=  \deg_{\mathbbm{k}}(V_{0}^{\sigma})\). This is a constant, independent of the random variable.
	\smallskip

	Then, using \eqref{equpsilon}, we can write
\begin{align*}
	\int_{\Delta_{\underline{k}}} \upsilon^{N_k}_{[\leq j]} d P & =  \mathrm{E} [ \upsilon^{N_k}_{[\leq j]} ]\\
	&  = \sum_{\sigma}  \mathrm{E} [ \mathbbm{1}_{\{\mathrm{index}(\sigma) \leq j\}} \prod_{1  \leq i \leq n} A_i^\sigma ]\, \delta_{\sigma} \\
	& \leq  \sum_{\sigma} \left[\mathbbm 1_{\{j_\sigma \leq j\}} \prod_{1 \leq i \leq n} \mathrm{E}[A_i^\sigma] \right] \delta_{\sigma}
	+ 
	\sum_\sigma \left[ \sum_{1 \leq l \leq j} \left| \mathrm{E}( \mathbbm 1_{\{\mathrm{index(\sigma)} = l\} } \prod_{1 \leq i \leq n} A_i^\sigma ) -  \delta_{l, j_\sigma} \prod_{1 \leq i \leq n}   \mathrm E(A_i^\sigma) \; \right| \right] \delta_{\sigma}
\end{align*}

We can apply Lemma \ref{lemupperbounddiff} to bound from above the second term of the right hand side. Moreover, by Lemma \ref{lemvariance}, (1), (2), the right hand side of \eqref{equpperbounddiff} is bounded from above by a term of the form $C \frac{(\log k)^{n-1}}{k^n}$, where the constant $C$ does not depend on $k$. Thus, we get:
	$$\int_{\Delta_{\underline{k}}} \upsilon^N_{[\leq j]} d P =  \sum_\sigma \mathbbm 1_{\{j_\sigma \leq j\}} \prod_{1 \leq i \leq n} \mathrm{E}[A_i^\sigma] \delta_{\sigma} + O\left(\frac{(\log k)^{n-1}}{k^n} \right) 
$$
Using again Lemma \ref{lemvariance}, (1), we then obtain
$$
	\int_{\Delta_{\underline{k}}} \upsilon^N_{[\leq j]} d P = \frac{(\log k)^n}{(kr)^n} \left[ \sum_\sigma \left( \mathbbm 1_{\{j_\sigma \leq j\}} \prod_{1 \leq i \leq n} d^i(\sigma)\, \delta_{\sigma} \right) \right] + O \left(\frac{(\log k)^{n-1}}{k^n}\right).
$$

	Now, the term between brackets is equal to $c_1(\det E \otimes N, \underline{\Sigma})_{[\leq j]}^n$, since the $d^i(\sigma)$ are the multiplicities of the trivialization $\mathbf e$ along the strata of $\Sigma$. This concludes the proof of Proposition \ref{prop:asymptint}, and of Theorem \ref{thmcomparison}.
\end{proof}

\section{End of the proof}

Theorem \ref{thm:existence} follows directly from Proposition \ref{proptwisteestimate} and the following result. Again, if $E$ is a vector bundle, we write $\mathbf{E}_k = E^{(1)} \oplus ... \oplus E^{(k)}$.

\begin{prop} \label{propfinal}
Let $X$ be a smooth projective variety of dimension $n$. Let $E \longrightarrow X$ be a vector bundle of rank $r$, such that $\det E$ is big. Then there exists:
	\begin{enumerate} \itemsep=0em
			\item a generically finite projective morphism $p : X' \longrightarrow X$, where \(X'\) is a variety;
		\item a decomposition $p^{\ast} \det E \sim_{\mathbb{Q}} A + G$ into ample and effective \(\mathbb{Q}\)-Cartier divisors on \(X'\);
		\item a trivialized stratification $\underline{\Sigma}$ on $X'$, adapted to $A$, such that \[
				\deg c_1(A, \underline{\Sigma})^n_{[\leq 1]} > 0;
	\]

		\item a sequence of effective $\mathbb Q$-divisors $(F_k)_{k \geq 1}$ on $X'$; 
	\end{enumerate}
			
			such that the following holds.  
\medskip

For $m \gg k \gg 0$, and $m$ divisible enough, we have,
\begin{align} \nonumber
	\chi^{[1]} \left( X', \right. & \left.  S^m \left( p^{\ast} \mathbf{E}_k \right) \otimes \mathcal O(-m \, F_k) \right) \\ \label{eqineqpropfinal}
	& \leq \; \frac{m^{n + kr - 1}}{(n + kr - 1)!} \frac{(\log k)^n}{(k!)^n} \left( \deg c_1(A, \underline{\Sigma})^n_{[\leq 1]} - O(\frac{1}{\log k}) \right) + o (m^{n + kr - 1})
\end{align}

	In the above result, the big \(O\)-term does not depend on \(m\), but the small \(o\)-term might depend on \(k\).
\end{prop}

Before proving the proposition above, let us explain how it permits to prove the main result.

\begin{proof}[Proof of Theorem~\ref{thm:existence}] 

Let $X$ be a projective smooth variety of dimension $n$, such that $K_X = \det \Omega_X$ is big.

 We now apply Proposition \ref{propfinal} with $E = \Omega_X$, to obtain $p : X' \longrightarrow X$ and $(F_k)_{k \geq 1}$ such that \eqref{eqineqpropfinal} holds. Then, Proposition \ref{proptwisteestimate} implies in turn that
\begin{align}
	h^0 (X, & E_{k,m}^{HS}\Omega_{X}) \geq  \nonumber
	\; \frac{1}{\deg(p)} \frac{m^{n + kr - 1}}{(n + kr - 1)!} \frac{(\log k)^n}{(k!)^n} \left( \deg c_1(A, \underline{\Sigma})_{[\leq 1]}^n - O((\log k)^{-1}) \right) \\ \label{eq:inequation}
 & \hspace{1cm} + o (m^{n + kr - 1}).
\end{align}    

	This implies that $E_{k, \bullet}^{HS} \Omega_{X}$ is big, if \(k\) is sufficiently large. This ends the proof.
\end{proof}

We now finish with the proof of Proposition \ref{propfinal}, which is based on Theorem \ref{thmcomparison}. We are essentially looking for a way to construct a natural stratification adapted to a line bundle of the form $A = \det E \otimes N$. If the latter were very ample, this would be easily done by taking successive generic hyperplane sections. In the ample case, we use a covering trick, joint with the construction already seen in Proposition~\ref{prop:conststratample}.

\begin{proof}[Proof of Proposition \ref{propfinal}]

	Since \(L := \det E\) is big, we can write, for a sufficiently large \(m \geq 0\):
	\[
		L^{\otimes m} \cong A_{0} \otimes \mathcal{O}_{X}(G_{0}),
	\]
	where \(A_{0}\) is an ample line bundle, and \(G_{0}\) is an effective Cartier divisor.
	\medskip

	Now, we can use Kawamata's covering lemma (see Proposition~\ref{prop:kawamatacov}) to find a finite dominant morphism $p : X' \longrightarrow X$ such that $p^\ast G_{0} = mG$, where $G$ is an effective Cartier divisor. Then we have
$$
	(p^\ast L \otimes \mathcal{O}_{X'}( - G))^{\otimes m} = p^\ast A_{0}
$$
	The divisor $p^\ast A_{0}$ is ample as the pullback of an ample divisor by a finite morphism, and $(p^\ast A_{0}^n) = (\deg p) (A_{0}^n)$. Thus, $A := p^\ast L\otimes \mathcal{O}(- G)$ is itself ample, and by Proposition~\ref{prop:conststratample}, there exists a (fractional) trivialized stratification $\Sigma$ on \(X'\), adapted to the ample divisor $A$, such that $\deg c_1(A, \underline{\Sigma})_{[\leq 1]}^{n} = (A^n) > 0$.

\medskip
	\medskip

	To conclude, we let $F_k = \frac{1}{kr} \left( 1 + \frac{1}{2} + ... + \frac{1}{k} \right) G$ for all $k \geq 1$. We now apply Theorem \ref{thmcomparison} with $X$ replaced by $X'$, $E$ replaced by $p^\ast E$ and letting $N = \mathcal O(- G)$. We have then $N_k = \mathcal O(-F_k)$ and $\det (p^\ast E) \otimes N = \mathcal O(A)$, so we get the result immediately.
\end{proof}

\subsection{Volume estimate in the complex case}

In this section only, we assume that \(\mathbbm{k} = \mathbb{C}\). We will retrieve the following result (implied by \cite[Corollary 2.38]{dem11}).  

\begin{corol} \label{thmvol} Let $X$ be a complex projective manifold, with $K_X$ big, and let \(\epsilon > 0\). For $k \gg 1$, we have:
$$
\mathrm{vol}(E_{k, \bullet}^{HS} \Omega_X) \geq \frac{(\log k)^n}{(k!)^n} \left( \mathrm{vol}(K_X) - \epsilon - O\left( \frac{1}{\log k} \right) \right).
$$
\end{corol}

\begin{proof}
	If we look at \eqref{eq:inequation}, we see that it suffices to show that in the statement of Proposition~\ref{propfinal}, the term \(\deg c_{1}(A, \underline{\Sigma})^{n}_{[\leq 1]}\) can be taken arbitrarily close to \(\deg(p) \mathrm{vol}(E)\). To see this, we will modify slightly the construction of \(X'\) presented in the previous section, by adding an additional step in the tower of coverings \(X_{2} \to X_{1} \to X\). 
\medskip

	Let \(\epsilon > 0\). Since $L = \det E$ is big, we can use Fujita's approximate Zariski decomposition theorem \cite{fuj94, DEL00}, to obtain a modification $p_{0} : X_0 \longrightarrow X$ and an integer $m$ such that
$$
	m (p_0^\ast L) = A_0 \otimes \mathcal{O}_{X}(G_0), 
$$
	where $A_0$ is ample with $(A_0^n) \geq m^n \left( \mathrm{vol} (\det E) - \epsilon \right)$, and where $G_0$ is an effective divisor. Now, add a Kawamata covering \(X' \overset{p_{1}}{\to} X_{0} \overset{p_{0}}{\to} X\) as in the previous proof, so that
	\[
		p_{1}^{\ast} G_{0} = m G,
	\]	
	Letting \(p := p_{0} \circ p_{1}\), the line bundle \(A := p^{\ast} L \otimes \mathcal{O}(-G)\) is again ample, and we have
	\begin{align*}
		\deg c_1(A, \underline{\Sigma})_{[\leq 1]}^n 
		& = (A^n) \\
		& = (\deg p_{1}) \frac{1}{m^{n}} (A_{0}^{n}) 
		\quad \quad \text{(since} \; A_{1}^{\otimes m} = p_{1}^{\ast} A_{0}) \\
		& = (\deg p) \frac{1}{m^{n}} (A_{0}^{n})
		\quad \quad \quad \quad \text{(since} \; \deg(p_{0}) = 1) \\
		& \geq  (\deg p) (\mathrm{vol}(\det E) - \epsilon).
	\end{align*}
	This is what we wanted.
\end{proof}

\appendix
\chapter{Algebraic geometry results} \label{annex:alggeom}

\section{Asymptotic cohomological computations}

In this section, \(\mathbbm{k}\) is a fixed field. A {\em scheme} denotes a separated scheme of finite type above \(\mathrm{Spec}\, \mathbbm{k}\). A {\em variety} is an integral scheme (see our section on notation).
\medskip

\begin{lem} \label{lem:torsion} Let \(X\) be a proper variety of dimension \(n\). Let \(L\) be a line bundle, and let \(\mathcal{F}\) be a coherent sheaf on \(X\). Then we have the following :
	\begin{enumerate}
		\item if \(\mathcal{F}\) is a torsion sheaf, then \(h^{i}(X, \mathcal{F} \otimes L^{\otimes m}) = O(m^{n-1})\) as \(m\) tends to \(+\infty\).
		\item if \(\mathcal{F}\) is of generic rank \(r\), then \(h^{i}(X, \mathcal{F} \otimes L^{\otimes m}) = r\, h^{i}(X, L^{\otimes m}) + O(m^{n-1})\) as \(m\) tends to \(+\infty\).
	\end{enumerate}
\end{lem}
\begin{proof}
	See e.g. \cite[Section 1.9]{Deb01}.
\end{proof}

\begin{lem} \label{lem:suppasympt} Let $X$ be a proper variety of dimension \(n\), and let \(X' \overset{p}{\longrightarrow} X\) be a morphism from any proper scheme. Let \(\mathcal{F}\) be a coherent sheaf on \(X'\), and let \(L\) be a line bundle on \(X\). Assume that for \(j \geq 1\), one has
	\[
		\mathrm{Supp}
		\left(
		R^{j} p_{\ast} \mathcal{F}
		\right)
		\subsetneq
		X.
	\]
	Then for any \(i \geq 0\), one has
	\[
		h^{i}(X', \mathcal{F} \otimes p^{\ast} L^{\otimes m})
		=
		h^{i}(X, p_{\ast}\mathcal{F} \otimes L^{\otimes m}) + O(m^{n-1}).
	\]
\end{lem}
\begin{proof}
	For each \(m \in \mathbb{N}\), the Leray spectral sequence associated to \(p\), reads: 
\[
	E_{2}^{r,s} := H^{s}(X, R^{r} p_{\ast} \mathcal{F} \otimes L^{\otimes m})
	\Rightarrow
	H^{s + r}(X', \mathcal{F} \otimes p^{\ast} L^{\otimes m}).
\]

By Lemma~\ref{lem:torsion}, for each \(r \in \llbracket 1, i \rrbracket\), there exists a constant \(C_{r}\) (depending on \(X, p, M, L\) and \(r\)) such that for all \(m \geq 0\) :
	\[
		\dim_{\mathbbm{k}} E_{2}^{r, i-r} \leq C_{r} m^{m-1}.
	\]
	Now, by definition of the Leray spectral sequence, the \(\mathbbm{k}\)-vector space \(H^{i}(X', \mathcal{F} \otimes p^{\ast} L^{\otimes m})\) has a filtration of length \(i\) with graded vector spaces \((E_{\infty}^{r, s})_{r + s = i}\) satisfying :
	\begin{enumerate}[label=(\alph*)]
		\item for \(r \in \llbracket 1, i\rrbracket\), \(E_{\infty}^{r,i-r}\) is a subquotient of \(E_{2}^{r, i-r}\) and thus
			\[
				\dim_{\mathbbm{k}} E_{\infty}^{r, i-r}
				\leq 
				\dim_{\mathbbm{k}} E_{2}^{r, i-r} 
				\leq 
				C_{r} m^{n-1}.
			\]
		\item \(E_{\infty}^{0, i}\) is the common kernel of maps starting from \(E_{2}^{0, i} = H^{i}(X, p_{\ast} \mathcal{F} \otimes L^{\otimes m})\) and arriving in subquotients of the \(E_{2}^{r, i-r}\), for \(r \geq 1\). Thus,
			\[
				| 
				\dim_{\mathbbm{k}} 
				E_{\infty}^{0, i} 
				- 
				h^{i}(X, p_{\ast} \mathcal{F} \otimes L^{\otimes m}) 
				| 
				\leq 
				\sum_{1 \leq r \leq i} 
				E_{2}^{r, i-r} 
				\leq 
				(\sum_{1 \leq r \leq i} C_{r}) m^{n-1} 
			\]
	\end{enumerate}
		Since \(h^{i}(X', \mathcal{F} \otimes p^{\ast} L^{\otimes m}) = \sum_{0 \leq r \leq i} \dim_{\mathbbm{k}} E_{\infty}^{r, i - s}\), this gives the result.
\end{proof}

\begin{lem} \label{lemmodification} Let $X$ be a proper variety of dimension \(n\), and let \(X' \overset{p}{\longrightarrow} X\) be a generically finite proper morphism of degree \(d\), where \(X'\) is any proper scheme. Then, for any line bundle \(L\) on \(X\), and any \(i \geq 0\), we have:
\begin{enumerate}[label=(\roman*)]
\item
\(
h^{i}(X',  p^\ast L^{\otimes m}) 
	= d \; h^{i}(X, L^{\otimes m}) + O(m^{n-1}).
\)
\item
	\(
\chi^{[i]}(X', p^\ast L^{\otimes m}) 
	= d \; \chi^{[i]}(X, L^{\otimes m}) + O(m^{n-1}),
\)
\end{enumerate}
\end{lem}
\begin{proof}
	Since finite morphisms are acyclic, the higher cohomology sheaves \(R^{j} p_{\ast} \mathcal{O}_{X}\) are torsion sheaves for \(j \geq 1\). On the other hand, \(p_{\ast} \mathcal{O}_{X}\) is of rank \(d\) under our hypotheses. The result is thus a consequence of Lemma~\ref{lem:suppasympt} with \(\mathcal{F} = \mathcal{O}_{X'}\).
\end{proof}

\begin{lem} \label{lem:compasympt}
	Let \(X\) be a proper variety of dimension \(n\), endowed with a line bundle \(L\). Let \(p : X' \to X\) be a proper morphism of schemes, and let \(\mathcal{F}\) be a sheaf on \(X'\). Assume that \(X'\) admits exactly one component \(X_{0}\) such that \(p\) is an isomorphism near the generic point of \(X_{0}\), all the other components being \(p\)-exceptional. Let \(\iota : X_{0} \hookrightarrow X'\) be the embedding morphism, and let \(r \in \mathbb{N}\) be the generic rank of \(\mathcal{F}\) on \(X_{0}\).
	\medskip

	Then one has, for all \(i \geq 0\):
	\[
		h^{i}(X', \mathcal{F} \otimes p^{\ast} L^{\otimes m})
		= r\, h^{i}(X, L^{\otimes m}) + O(m^{n-1})
		= h^{i}(X_{0}, \mathcal{F} \otimes p^{\ast} L^{\otimes m}) + O(m^{n-1}) \\
	\]
\end{lem}

\begin{proof}
	Let \(Z \subsetneq X\) be closed subset containing the exceptional components of \(X'\), such that \(p|_{X_{0}} : X_{0} \to X\) is an isomorphism above \(X - Z\). Then by our hypothesis, all higher cohomology sheaves \(R^{j} p_{\ast} \mathcal{F}\) are supported on \(Z\), so the first equality follows from Lemma~\ref{lem:suppasympt}. 
	\medskip

	The other inequality follows by applying the same Lemma~\ref{lem:suppasympt}, with \(X'\) replaced with \(X_{0}\) and \(\mathcal{F}\) with \(\iota^{\ast} \mathcal{F}\).
\end{proof}

\subsection{Variant. Uniform bound for a semi-group in \(\mathrm{Pic}(X)\)} It is also possible to state a version of Lemma~\ref{lem:torsion} for a finite family of line bundles. In this section, {\em we assume \(X\) to be projective}; this will be only used to make sure any line bundle \(L\) can be written \(L = \mathcal{O}(A-B)\) with \(A, B\) {\em Cartier} and effective. Note that the results of this section are only used in Chapter~\ref{chap:morseineqweighted}, where we make a similar restriction on \(X\).

\begin{lem} \label{lem:asymptseveral}
	Let \(X\) be projective variety of dimension \(n\), and let \(L_{1}, \dotsc, L_{r} \in \mathrm{Pic}(X)\). If \(\underline{a} = (a_{1}, \dotsc, a_{r}) \in \mathbb{Z}^{r}\), we write \(\mathcal{L}^{\otimes \underline{a}} := L_{1}^{\otimes a_{1}} \otimes \dotsc \otimes L_{r}^{\otimes a_{r}}\).
	
	Let \(m : \mathbb{R}^{r} \to \mathbb{R}\) be a linear form, taking positive values on the cone \(\mathbb{R}_{+}^{r} - \{0\}\). For any \(\underline{a} \in \mathbb{Z}^{r}\), write \(m_{\underline{a}} := m(\underline{a})\). Then we have the following.
	\begin{enumerate}
		\item for any \(i \in \llbracket 0, n \rrbracket\), and any coherent sheaf \(\mathcal{F}\) on \(X\) for any \(\underline{a} \in \mathbb{N}_{+}\), we have
			\[
				h^{i}(X, \mathcal{L}^{\otimes \underline{a}} \otimes \mathcal{F})	
				\;
				\leq	
				\;
				O(m_{\underline{a}}^{n}).
			\]
		\item for any \(i \in \llbracket 0, n \rrbracket\), and any torsion coherent sheaf \(\mathcal{G}\) on \(X\) for any \(\underline{a} \in \mathbb{N}_{+}\), we have
			\[
				h^{i}(X, \mathcal{L}^{\otimes \underline{a}} \otimes \mathcal{G})	
				\;
				\leq	
				\;
				O(m_{\underline{a}}^{n-1}).
			\]
		\item if \(\mathcal{F}\) is a coherent sheaf of generic rank \(r\), then \(h^{i}(X, \mathcal{F} \otimes {\mathcal{L}}^{\otimes \underline{a}}) = r\, h^{i}(X, \mathcal{L}^{\otimes \underline{a}}) + O(m_{\underline{a}}^{n-1})\).
	\end{enumerate}
	In both cases, the constant in the \(O\)-term depends only on \(X\), \(\mathcal{G}\) or \(\mathcal{F}\), \(L_{1}, \dotsc, L_{r}\) and the choice of function \(m\), but not on \(\underline{a}\).
\end{lem}

\begin{rem}
	In the previous lemma, the \(L_{i}\) can be taken completely arbitrarily : they do not have to generate a free group in \(\mathrm{Pic}(X)\). 
\end{rem}

\begin{proof} The proof is very close to the arguments of \cite[Section 1.9]{Deb01} as mentioned at the previous section, so we will just give an outline.
	\smallskip

	First of all, remark that for any choices of functions \(m_{1}\), \(m_{2}\) as above, we have a comparison
	\[
		\frac{1}{C} m_{1} \leq m_{2} \leq C m_{2},
	\]
	so it suffices to prove the result for any choice of \(m\). Let us take \(m_{\underline{a}} = \sum_{j} a_{j}\). All three items can proved simultaneously by induction on \(n\), the case \(n = 0\) being obvious.
	\smallskip

\noindent
	{\em (2)} By the classical {\em dévissage} argument, the second item follows from the validity of the first item in dimension \(n-1\).
	\smallskip

\noindent
{\em (1)} Using the standard telescopic argument, it suffices to show that if \(\underline{a}\) and \(\underline{b}\) differ only by adding one to a given coordinate, we have
	\begin{equation} \label{eq:boundseveral}
	|
	h^{i}(X, \mathcal{L}^{\otimes \underline{a}})	
	-
	h^{i}(X, \mathcal{L}^{\otimes \underline{b}})	
	|
	\leq C m_{\underline{a}}^{n-1},
	\end{equation}
	where \(C\) depends only on \(X\), \(\mathcal{F}\) and the \(L_{1}, \dotsc, L_{r}\). Let us show it for example with \(\underline{b} = \underline{a} + (1, 0, \dotsc, 0)\). Since \(X\) is projective, we can write \(L_{1} = \mathcal{O}_{X}(A - B)\), where \(A, B\) are effective Cartier divisors (it suffices to take \(B\) sufficiently ample). Then, we have the usual exact sequences

	\[
\begin{tikzcd}[column sep=1em]
 	0 
		\ar[r]  
	& \mathcal{F} \otimes \mathcal{L}^{\otimes \underline{b}} \otimes \mathcal{O}_{X}(-A) 
		\ar[r] \ar[d, equal] 
	& \mathcal{F} \otimes  \mathcal{L}^{\otimes \underline{b}} 
		\ar[r] 
	&  \mathcal{F} \otimes \mathcal{L}^{\otimes \underline{b}} \otimes \mathcal{O}_{A} 
		\ar[r] 
	&  0 \\
 	0 
		\ar[r] & 
	  \mathcal{F} \otimes \mathcal{L}^{\otimes \underline{a}} \otimes \mathcal{O}_{X}(-B) 
	  	\ar[r] 
	& \mathcal{F} \otimes \mathcal{L}^{\otimes \underline{a}} 
		\ar[r] 
	& \mathcal{F} \otimes \mathcal{L}^{\otimes \underline{a}} \otimes \mathcal O_{B} 
		\ar[r] 
	&  0 
\end{tikzcd}
\]

	Writing the long exact sequences in cohomology, we see the left hand side of \eqref{eq:boundseveral} can be bounded from above by a sum of terms of the form
	\[
		h^{j}(X, \mathcal{F} \otimes \mathcal{L}^{\otimes \underline{c}} \otimes \mathcal{O}_{C}),
	\]
	where \(\underline{c} \in \{ \underline{a}, \underline{b}\}\) and \(C \in \{A, B\}\). Since each \(\mathcal{O}_{C}\) is torsion, the result is provided by item {\em (2)}. Note that the constant appearing in \eqref{eq:boundseveral} may depend on a fixed way of writing each \(L_{j}\) as a difference of two effective Cartier divisors, but not on \(\underline{a}\) nor \(\underline{b}\).
\medskip

\noindent
{\em (3)} Let \(U \subset X\) be an open subset on which we have an isomorphism \(i : \mathcal{O}_{U}^{\oplus r} \to \mathcal{F}|_{U}\), and let
	\[
		\mathcal{G} \subset (\mathcal{O}_{X}^{\oplus r} \oplus \mathcal{F})
	\]
	be the subsheaf whose sections on open sets \(V \subset X\) are the ones that restrict to \(U \cap V\) to sections of the form \((i(s), s)\). It is coherent subsheaf; let \(\mathcal{H}\) be the quotient. Taking the induced projection maps, we get two exact sequences
	
	\[
		\begin{tikzcd}[column sep=1em]
		0 
			\ar[r]
		&
		\mathcal{T}_{1} 
			\ar[r]
		&
		\mathcal{O}_{X}^{\oplus r}
			\ar[r]
		&	
		\mathcal{H}
			\ar[r]
			\ar[d, equal]
		&
		0
		\\
		0 
			\ar[r]
		&
		\mathcal{T}_{2} 
			\ar[r]
		&
		\mathcal{F}	
			\ar[r]
		&	
		\mathcal{H}
			\ar[r]
		&
		0
		\end{tikzcd}	
	\]
	Note that \(\mathcal{T}\) and \(\mathcal{T'}\) are torsion sheaf, supported on \(X - U\). Again, writing the exact sequences in cohomology, we get that the difference
	\[
		\left|
		h^{i}(X, \mathcal{O}_{X}^{\oplus r} \otimes \mathcal{L}^{\otimes \underline{a}})
		-
		h^{i}(X, \mathcal{F} \otimes \mathcal{L}^{\otimes \underline{a}})
		\right|
	\]
	is bounded from above by a sum of terms of the form \(h^{j}(X, \mathcal{T}_{j} \otimes \mathcal{L}^{\underline{a}})\), so the result again follows from {\em (2)}. 
\end{proof}

With the exact same proofs as in Lemmas~\ref{lem:suppasympt}, \ref{lemmodification} and \ref{lem:compasympt}, we get the following three statements.

\begin{lem} \label{lem:suppasymptseveral} Let $X$ be a projective variety of dimension \(n\), on which we make the same assumptions as in Lemma~\ref{lem:asymptseveral}. Let \(X' \overset{p}{\longrightarrow} X\) be a morphism from any proper scheme. Assume that for \(j \geq 1\), one has
	\[
		\mathrm{Supp}
		\left(
		R^{j} p_{\ast} \mathcal{F}
		\right)
		\subsetneq
		X.
	\]
	Then for any \(i \geq 0\), one has, for any \(\underline{a} \in \mathbb{N}^{r}\) :
	\[
		h^{i}(X', \mathcal{F} \otimes p^{\ast} \mathcal{L}^{\otimes \underline{a}})
		=
		h^{i}(X, p_{\ast}\mathcal{F} \otimes L^{\otimes \underline{a}}) + O(m_{\underline{a}}^{n-1}).
	\]
\end{lem}

\begin{lem}  \label{lem:modificationseveral} Let \(X\) be a projective variety of dimension \(n\), as in Lemma~\ref{lem:asymptseveral}.  Let \(X' \overset{p}{\longrightarrow} X\) be a generically finite projective morphism of degree \(d\), where \(X'\) is any projective scheme. Then, for any \(i \geq 0\), and any \(\underline{a} \in \mathbb{N}\), we have:
\begin{enumerate}[label=(\roman*)]
\item
\(
	h^{i}(X',  p^\ast \mathcal{L}^{\otimes \underline{a}}) 
	= d \; h^{i}(X, \mathcal{L}^{\otimes \underline{a}}) + O(m_{\underline{a}}^{n-1}).
\)
\item
	\(
\chi^{[i]}(X',  p^\ast \mathcal{L}^{\otimes \underline{a}}) 
		= d \; \chi^{[i]}(X, \mathcal{L}^{\otimes \underline{a}}) + O(m_{\underline{a}}^{n-1}),
\)
\end{enumerate}
\end{lem}

\begin{lem} \label{lem:compasymptseveral}
	Let \(X\) be a projective variety of dimension \(n\), as in Lemma~\ref{lem:asymptseveral}. Let \(p : X' \to X\) be a projective morphism of schemes, and let \(\mathcal{F}\) be a sheaf on \(X'\). Assume that \(X'\) admits exactly one component \(X_{0}\) such that \(p\) is an isomorphism near the generic point of \(X_{0}\), all the other components being \(p\)-exceptional. Let \(\iota : X_{0} \hookrightarrow X'\) be the embedding morphism, and let \(r \in \mathbb{N}\) be the generic rank of \(\mathcal{F}\) on \(X_{0}\).
	\medskip

	Then one has, for all \(i \geq 0\):
	\[
		h^{i}(X', \mathcal{F} \otimes p^{\ast} \mathcal{L}^{\otimes \underline{a}})
		= r\, h^{i}(X, \mathcal{L}^{\otimes \underline{a}}) + O(m_{\underline{a}}^{n-1})
		= h^{i}(X_{0}, \mathcal{F} \otimes p^{\ast} \mathcal{L}^{\otimes \underline{a}}) + O(m_{\underline{a}}^{n-1}) \\
	\]
\end{lem}

\subsection{Asymptotic Riemann-Roch theorem} \label{sec:asymptoticRR} Let \(X\) be any proper variety over \(\mathbbm{k}\), and let \(n = \dim X\). It follows from \cite[Section 18.3]{ful98} that we can define the Todd class of \(X\) as a class \(\mathrm{Td}(X) \in A_{\ast}(X)_{\mathbb{Q}}\)
\[
	\mathrm{Td}(X) = [X] + \text{terms of dimension}\, < n.
\]

If \(E \to X\) is any vector bundle, we then have the following Riemann-Roch theorem (see \cite[Corollary 18.3.1 (a)]{ful98});
\[
	\chi(X, E) = \deg \left( \mathrm{ch}(E) \cap \mathrm{Td}(X) \right)_{n}.
\]

In particular, if \(L\) is a line bundle on \(X\), then we have the asymptotic expansion
\[
	\chi(X, L^{\otimes m})
	=
	\left(\deg c_{1}(L) \cap [X]\right) \frac{m^{n}}{n!} + O(m^{n-2}),
\]
where the \(O(m^{n-2})\) term can actually be expanded as a polynomial in \(m\) with coefficients depending only on \(\mathrm{Td}(X)\) and \(c_{1}(L)\).

\section{Cyclic covers and modifications}

We gather here a few results that are very classical in the complex setting (or are essentially trivial if we admit the existence of strong resolutions of singularities), but for which we could not find a satisfactory reference in the general case. Note that in most cases, the complex proof work without much change in the general setting.

\begin{lem} \label{lem:cycliccover}
	Let \(X\) be a variety, endowed with a line bundle \(L\). Let \(d \in \mathbb{N}^{\ast}\). Let \(U \subset X\) be an open subset on which \(L^{\otimes d}\) is trivial, and let \(e \in \Gamma(U, L^{\otimes d})\) be a trivializing section. Then there exists a finite dominant morphism \(p : X' \to X\), where \(X'\) is a variety, such that \(p^{\ast} L\) is trivial on \(U' := p^{-1}(U)\), with a trivializing section \(e'\) such that \((e')^{\otimes d} = p^{\ast} e\).
\end{lem}
\begin{proof}

	Let \(e\) be a trivializing section of \(L^{\otimes d}\) on \(U\). Introduce the completion of \(L^{\otimes d}\) at infinity as \(P = \mathrm{P}(L^{\otimes d} \oplus \mathcal{O}_{X})\), endowed with its projection \(q : P \longrightarrow X\). The total space of \(L^{\otimes d}|_{U} - \{0\}\) embeds in \(P\) as a Zariski open subset \(V\), on which \(p^{\ast} L\) admits a tautological trivializing section \(e_{0}\).
	\medskip
	
	Let \(Y_{U} \subset q^{-1}(U)\) be the graph of \(e\), and let \(Y := \overline{Y_{U}}\) be its closure in \(P\). The projection gives a dominant, birational morphism \(p : Y \to X\), isomorphic above \(U\). Note that \(e_{0}|_{Y}\) identifies with the pullback \(p^{\ast} e\) as a section in \(\Gamma(p^{-1}(U), p^{\ast} L)\).
		\medskip

		Now, we have a commutative diagram

		\[
			\begin{tikzcd}
				\mathrm{P}_{X}(L \oplus \mathcal{O}_{X})
				\arrow[r, "v"]
				\arrow[d, "q"]
				&
				\mathrm{P}_{X} (L^{\otimes d} \oplus \mathcal{O}_{X})
				\arrow[d, "p"]
				\\
				X 
				\arrow[r, "="]
				&
				X
			\end{tikzcd}
		\]
		where the arrow \(v\) can be expressed on a trivializing affine chart \(\mathrm{Spec}\, A \subset X\) by the following morphism of graded \(A\)-algebras
		\[
			\begin{array}{ccc}
			A[x_{0}, x_{1}]
				&
			\longrightarrow
				&
				A[y_{0}, y_{1}]\\
			
			x_{i}	
				&
				\longmapsto
				&
			y_{i}^{d}	
			\end{array}.
		\]
	Again, the variety \(P_{X}(L \oplus \mathcal{O}_{X})\) admits the total space of \(L|_{U} - \{0\}\) as an open subset \(V'\); on that subset, the pullback \(v^{\ast} p^{\ast} L\) admits a tautological trivializing section \(e_{0}'\), for which \((e_{0}')^{\otimes d} = v^{\ast} e_{0}\). Let \(X'_{0} := v^{-1}(Y)\). Remark that since \(e\) is non-vanishing on \(U\), one has \(X_{0}' \cap q^{-1}(U) \subset V\), i.e. \(X_{0}'\) does not intersect the zero section or the section at infinity above \(U\). Thus \(q^{\ast} L\) is trivial on \(q^{-1}(U) \cap X_{0}'\). Finally, let \(X'\) be the reduced subscheme associated to any irreducible component of \(X_{0}'\) that dominates \(X\), and let \(e'\) be the restriction of \(e_{0}'\) to this variety. 
\end{proof}

\begin{rem}
	If \(\mathbbm{k}\) is of characteristic \(p > 0\), and if \(p\) divides \(d\), note that the morphism \(v\) above factors through an inseparable morphism, so \(X_{0}'\) will be non-reduced even above \(U\). 
\end{rem}

As a special case, we get the following classical result. 

\begin{lem} \label{lem:cov1} Let \(X\) be a variety, endowed with a line bundle \(L\). Let \(d \in \mathbb{N}^{\ast}\) be such that \(L^{\otimes d}\) is effective, and let \(D\) be an effective associated Cartier divisor. Then there exists a finite dominant morphism \(p : X' \to X\), where \(X'\) is a variety, such that \(p^{\ast} D = d D'\) for some effective Cartier divisor \(D'\) on \(X\).
\end{lem}
\begin{proof}
	We prove the result with the same proof as in the previous lemma, taking \(U := X - |D|\), and \(e\) being the global section of \(L^{\otimes d}\) giving \(D\). In this case, we see from the proof that \(e'\) comes from a regular section of \(p^{\ast} L\), giving a divisor \(D'\) as intended.
\end{proof}

\begin{lem}[{Bloch-Gieseker coverings (see \cite{BG71}, \cite[Theorem~4.1.10]{lazpos1})}] \label{lem:BGcoveringannex}
	Let \(X\) be a projective variety, and let \(L\) be a line bundle on \(X\). Let \(d \in \mathbb{N}_{>1}\). Then there exists a finite dominant morphism \(p : X' \to X\) and a line bundle \(N\) on \(X'\) such that \(p^{\ast} L \cong N^{\otimes d}\).
\end{lem}
\begin{proof}

{\em Step 1.} Assume first that \(X = \mathbb{P}^{n}_{\mathbbm{k}}\), and that \(L = \mathcal{O}_{\mathbb{P}^{n}_{\mathbbm{k}}}(1)\). In this case, if \(f : \mathbb{P}^{n}_{\mathbbm{k}} \to \mathbb{P}^{n}_{\mathbbm{k}}\) is the finite morphism associated with the morphism of \(\mathbbm{k}\)-algebras
	\[
		\begin{array}{ccc}
			\mathbbm{k}[x_{0}, \dotsc, x_{n}]
			&
			\longrightarrow
			&
			\mathbbm{k}[x_{0}, \dotsc, x_{n}]
			\\
			P(x_{0}, \dotsc, x_{n})
			&
			\longmapsto
			&
			P(x_{0}^{d}, \dotsc, x_{n}^{d}),
		\end{array}
	\]
	then we have \(f^{\ast}\mathcal{O}_{\mathbb{P}^{n}_{\mathbbm{k}}}(1) \cong \mathcal{O}_{\mathbb{P}^{n}_{\mathbbm{k}}}(d)\). This gives the result in this case.
	\smallskip

	\noindent
	{\em Step 2.} Assume now that \(L\) is globally generated. We can find a morphism \(g : X \longrightarrow \mathbb{P}^{N}_{\mathbbm{k}}\) such that \(L \cong g^{\ast} \mathcal{O}_{\mathbb{P}^{n}_{\mathbbm{k}}}(1)\). Form the square product
	\[
		\begin{tikzcd}
			X_{1}
				\arrow[r]
				\arrow[d, "q"]
				\arrow[dr, phantom, "\square"]
			&
			\mathbb{P}_{\mathbbm{k}'}^{N}
				\arrow[d, "f"]
			\\
			X
				\arrow[r, "g"]
			&
			\mathbb{P}_{\mathbbm{k}}^{N},
		\end{tikzcd}
	\]
	where \(f\) is constructed as in Step~1. Then one gets the result by taking \(X'\) to be the reduced scheme underlying an irreducible component of \(X_{1}\), and \(p : X' \to X\) to be the restriction of \(q\).
	\medskip	

	\noindent
	{\em Step 3.} If \(X\) is projective, then we can find a very ample line bundle \(H\) on \(X\) such that \(N := H \otimes L\) is globally generated. Then apply Step~2 twice to get a finite morphism \(q : X' \to X\) on which \(p^{\ast} H\) (resp. \(p^{\ast} N\)) admit a \(d\)-th root \(H'\) (resp. \(N'\)). Then \(L' := N' \otimes (H')^{-1}\) is a \(d\)-th root of \(p^{\ast} L\), and it satisfies the requested condition with respect to the subvarieties \(W_{i}\).
\end{proof}

Combining the last two results, we get Kawamata's classical covering trick (\cite{kawa82}, see also \cite[Proposition 4.1.12]{lazpos1}).

\begin{prop} \label{prop:kawamatacov}
	Let \(X\) be a projective variety, and let \(D\) be an effective Cartier divisor on \(X\). Then for each \(m > 0\), there exists a finite dominant morphism \(p : X' \to X\) such that \(p^{\ast} D = m D'\) for some effective Cartier divisor \(D'\) on \(X\).
\end{prop}

Indeed, it suffices to apply first Lemma~\ref{lem:cov1} to take a \(m\)-th root \(L\) on \(p^{\ast} \mathcal{O}(D)\) on some finite cover \(p_{1} : X_{1} \to X\), and then Lemma~\ref{lem:BGcoveringannex} to this variety \(X_{1}\), using the effectivity of \(L^{\otimes m}\).

\begin{rem}
	Just to ease the mind of the suspicious reader, let us mention than things are even better behaved if \(p\) happens to divide \(m\). Assume for example that \(\mathbbm{k}\) is perfect and that \(m = p\). Then we can also produce a Kawamata covering as follows. Let \(K := \mathbbm{k}(X)\) be the fraction field of \(X\), and let \(K' := K^{[1/p]}\). Then \(K \subset K'\) is a purely inseparable field extension; it is also finite since \(\mathbbm{k}\) is perfect\footnote{Since \(\mathbbm{k}^{[1/p]} = \mathbbm{k}\), then writing \(K = \mathbbm{k}(\alpha_{1}, \dotsc, \alpha_{q})\), we see that \(K^{[1/p]} = \mathbbm{k}(\alpha_{1}^{1/p}, \dotsc, \alpha_{q}^{1/p})\) is of finite type above \(K\).}. Let \(q : X' \to X\) be the normalization of \(X\) in \(K'\). Then any regular section on \(X\) has a \(p\)-th root in \(\mathcal{O}_{X'}\): this implies that for any Cartier divisor \(D\) on \(X\), one has \(q^{\ast} D = p D'\) for some effective \(D'\) in \(X'\). In this case, this inseparable covering works for every divisor on \(X\): this is of course related to the inseparability of the map that makes it ramified everywhere.
\end{rem}

\begin{lem} \label{lem:WeiltoCartier}
	Let \(X\) be a normal variety, and let \(D\) be a Cartier divisor on \(X\). Assume that we have a decomposition with integer coefficients:
	\[
		D = \sum_{1 \leq i \leq r} m_{i} D_{i},
	\]
	where \(D_{1}, \dotsc, D_{r}\) are irreducible {\em Weil divisors} on \(X\). Then there exists a projective birational morphism \(p : X' \to X\), where \(X'\) is a normal variety, such that 
	\begin{enumerate}
		\item each \(D_{i}' = p^{-1}(D_{i})\) is an {\em effective Cartier} divisor, not necessarily irreducible, with 
	\[
		D_{i}' = \widehat{D}_{i} + E_{i}
	\]
			where \(\widehat{D}_{i}\) is an irreducible Weil divisor such that \(p : \widehat{D}_{i} \to D_{i}\) is dominant and birational, and \(E_{i}\) is an effective \(p\)-exceptional sum of irreducible Weil divisors.
		\item we have the equality of {\em Cartier} divisors
			\[
				p^{\ast} D
				=
				\sum_{1 \leq i \leq r}
				m_{i} D_{i}'.
			\]
	\end{enumerate}
	
\end{lem}
\begin{proof}
	We show by induction that there exists \(p : X' \to X\) satisfying {\em (1)}. Let \(q : X_{1} \to X\) be such that \(q^{-1}(D_{i}) = \widehat{D}_{i} + E_{i}\) are Cartier divisors satisfying the condition for all \(i \in \llbracket 1, r-1\rrbracket\).  
	\medskip

	Let \(q : X' \to X_{1}\) be the normalization of the blowing-up of the subscheme \(q^{-1}(D_{r})\) in \(X_{1}\), and let \(p : X' \to X\) be the composite map. Note that \(D_{r}' = p^{-1}(D_{r})\) is Cartier by construction. Since \(X'\), \(X_{1}\) and \(X\) are normal, the map \(p\) is an isomorphism above an open subset \(U \subset X\) with \(\mathrm{codim} (X - U) \geq 2\), so the decomposition of \(D_{r}'\) follows. Since the pullback of a Cartier divisor remains Cartier, checking that the \(q^{\ast} D_{i}\) satisfy the conditions for \(i < r\) is straightforward.

	Now, let us check that {\em (2)} is satisfied as well. This is local on \(X'\), so we can prove the equality on an open affine subset \(U' = \mathrm{Spec}\, A'\) of \(X'\), with \(p(U') \subset U\) for some other affine subset \(U = \mathrm{Spec}\, A\) in \(X\). If \(K\) is the fraction field of \(A\), we have the following inclusions of \(\mathbbm{k}\)-algebras:
	\[
		A \hookrightarrow A' \hookrightarrow K.
	\]
	Now, let \(\mathfrak{p}_{i} \subset A\) be the prime ideals corresponding to the \(D_{i}\). Since each \(p^{-1}(D_{i})\) is Cartier, we can assume (possibly after reducing \(U'\)) that there exists \(f_{i} \in A'\) such that for all \(i \in \llbracket 1, r \rrbracket\).
	\[
		\mathfrak{p}_{i} A' = (f_{i}) 
	\]
	Similarly (possibly after shrinking \(U\) and \(U'\)), we may assume that \(D\) is principal, given by an element \(f \in K\). We have to show that the element
	\[
		g := \frac{f}{\prod_{i} f_{i}^{m_{i}}}
	\]
	belongs to the invertible elements of \(A'\). However, the decomposition on \(D\) on \(X\) implies that, for any height one prime ideal \(\mathfrak{q} \subset A\), one has \(\mathrm{mult}_{\mathfrak{q}}(g) = 0\) i.e. \(g \in A_{\mathfrak{q}}^{\times}\). Now, since \(U\) is normal, \(A\) is integrally closed and we deduce that \(g \in A^{\times} = K\) by e.g. \cite[Lemma 031T]{stacks}. {\em A fortiori}, \(f\) is invertible in \(A'\).
\end{proof}

\chapter{Simplexes and lattices} \label{ann:simplexes}

\section{Volume computations}

In this section, we gather a few lemmas related to volume computations for simplexes in euclidean spaces. We refer to Section~\ref{sec:notationlattices} for the general notation.
\medskip

Let us recall that for any \((a_{1}, \dotsc, a_{r}) \in \mathbb{R}^{r}_{>0}\), we introduce the \((r-1)\)-dimensional simplex
	\[
		\Delta_{\underline{a}} = \{ (t_i) \in \mathbb R_+^r \; | \; \sum_i a_i t_i = 1 \}.
	\]

\begin{lem} \label{lem:simplextrivial}
	Let \(r \geq 1\). We consider the following open simplex in \(\mathbb{R}^{r}\) 
	\[
		\Delta_{r}
		=
		\left\{
			(t_{1}, \dotsc, t_{r}) \in \mathbb{R}_{+}^{r}
			\;|\;
			\sum_{j} t_{j} \leq 1.
		\right\}
	\]
	Then \(\mathrm{vol}_{r}(\Delta) = \frac{1}{r!}\).
\end{lem}
\begin{proof}
	This can be proved by induction. The case \(r = 1\) is obvious. In general, write
	\begin{align*}
		\mathrm{vol}_{r}(\Delta_{r}) & 
		= \int_{0}^{1} dt_{r} \int_{\sum_{j \leq r-1} t_{j} \leq 1 - t_{r}} dt_{1} \dotsc dt_{r-1} \\
		& = \int_{0}^{1} dt_{r}\, (1-t_{r})^{r-1} \mathrm{vol}_{r-1}(\Delta_{r-1}) \\
		& = \frac{1}{r} \mathrm{vol}_{r-1}(\Delta_{r-1}).
	\end{align*}
\end{proof}

\begin{lem} \label{lem:gram}
	Let \(s \leq r\) be two integers, and let \(\psi : \mathbb{R}^{s} \hookrightarrow \mathbb{R}^{r}\) be a linear embedding. Let
	\[
		G = (\left< \psi (e_{i}), \psi(e_{j}) \right>)_{1 \leq i, j \leq s}
	\]
	be the Gram matrix of \(\psi\), where \(e_{1}, \dotsc, e_{s}\) is the canonical basis for \(\mathbb{R}^{s}\), and \(\left<\bullet, \bullet \right>\) denotes the euclidean scalar product. Then we have
	\[
		\psi^{\ast}\, \mathrm{vol}_{\mathrm{Im} \psi} = \sqrt{\det G}\, \mathrm{vol}_{s}.
	\]
	where \(\mathrm{vol}_{\mathrm{Im} \psi}\) denotes the euclidean Lebesgue measure induced on the image of \(\psi\).
\end{lem}
\begin{proof}
	Let \(A \in \mathrm{M}_{rs}(\mathbb{R})\) denote the matrix of \(\psi\). Then \(G = {}^{t}A A\) is a symmetric matrix. There exists \(P \in \mathrm{O}_{s}(\mathbb{R})\) such that \({}^t P GP =: D\) is diagonal. In other words, the morphism associated to \(AP\) sends \((e_{1}, \dotsc, e_{r})\) to an orthogonal set of vectors in \(\mathbb{R}^{r}\). Since \(P^{\ast} \mathrm{vol}_{s} = \mathrm{vol}_{s}\), this reduces the situation to the case where \((\psi(e_{1}), \dotsc, \psi(e_{r}))\) is an orthogonal family and \(G\) is diagonal. In this case, the result is immediate.
\end{proof}

\begin{lem} \label{lem:volsimplex}
	Let \(\underline{a} = (a_1, \dotsc, a_r) \in \mathbb{R}_{>0}^{r}\). The \((r-1)\)-volume of \(\Delta_{\underline{a}}\) is equal to
	\[
	\frac{1}{(r-1)!} 
	\frac{\sqrt{\sum_{1 \leq i \leq r} a_i^2}}{a_{1} \dotsc a_{r}}.
	\]
\end{lem}

\begin{proof}
To perform this computation, we can for example use the parametrization of \(\Delta_{\underline{a}}\) given by 
	\[
		\begin{array}{cccc}
			\psi : & \Delta & \longrightarrow & \Delta_{\underline{a}} \\
			     &    t     & \longmapsto    & 
			     (\frac{t_{1}}{a_{1}}, \dotsc, \frac{t_{r-1}}{a_{r-1}} , 
			     \frac{1 - \sum_{1 \leq i \leq r - 1} t_{i}}{a_{r}}),
		\end{array}
	\]
	where 
	\(\Delta = \{ (t_i) \in \mathbb{R}_{+}^{r-1} \; | \; \sum_i t_i \leq 1 \}\) 
	is the standard \((r-1)\)-dimensional simplex in $\mathbb R^{r-1}$. 
	\smallskip
	
	Letting \((e_{1}, \dotsc, e_{r})\) be the canonical basis of \(\mathbb{R}^{r}\), Lemma~\ref{lem:gram} above shows that \(\psi^\ast( d\mathrm{vol}_{r-1} ) = \sqrt{\det G} \, d\mathrm{vol}_{r-1}\), where \(G= (\left< \psi_\ast(e_i), \psi_\ast(e_j)\right>)_{i,j}\). The computation of this determinant is given in Lemma~\ref{lem:determinant} below: we find \(\det G = \frac{1}{\prod_i a_i^2} \sum_{i} a_i^2$. Thus, we have $\mathrm{vol}_{r-1}(\Delta_{\underline{a}}) = \frac{\sqrt{\sum_i a_i^2}}{\prod_i a_i} \mathrm{vol}_{r-1}(\Delta)$. To conclude, it suffices to compute $\mathrm{vol}_{r-1}(\Delta) = \frac{1}{(r-1)!}$, which follows from Lemma~\ref{lem:simplextrivial}.
\end{proof}

\begin{lem} \label{lem:determinant}
	Let \(\alpha_{1}, \dotsc, \alpha_{r} \in \mathbb{R}^{\ast}\). Let
	\[
		A
		=
		\left(
		\begin{matrix}
			\alpha_{1} &  &     &   \\
			    & \alpha_{2} &  &   \\
				  &        &\ddots &  \\
			          &                &        & \alpha_{r-1}\\
			\alpha_{r} & \dotsc & \dotsc & \alpha_{r}
		\end{matrix}
		\right)
		\in M_{r, r-1}(\mathbb{R}).
	\]
	where the non-displayed terms are zero. Let \(G = {}^{t}A A \in \mathrm{M}_{r-1,r-1}(\mathbb{R})\). Then
	\[
		\det G = \prod_{i} \alpha_{i}^{2} \left(\sum_{i} \frac{1}{\alpha_{i}^{2}}\right). 
	\]
\end{lem}

\begin{proof}
Indeed, we have
	\[
		G 
		= 
		\mathrm{diag}(\alpha_{1}^{2}, \dotsc, \alpha_{r-1}^{2})
		+
		\alpha_{r}^{2} J,
	\]
	where \(J \in \mathrm{M}_{r-1, r-1}(\mathbb{R})\) is the matrix all of whose entries are equal to \(1\). Expanding the determinant using multilinearity yields
	\[
		\det G
		=
		\prod_{i=1}^{r-1} \alpha_{i}^{2}
		+
		\alpha_{r}^{2} \sum_{i=1}^{r-1}
		\prod_{j \notin \{i, r\}} \alpha_{j}^{2}. 
	\]
	This gives the result.
\end{proof}

\section{Uniform random variables on simplexes} \label{sect:simplexprob}

We present now a few estimates for the classical probability functional on random variables with values in affine simplexes, that follow quite closely Demailly's estimates in \cite{dem11}. The main result of this section is Lemma \ref{lemboundvariance}, which was used in the proof of Lemma \ref{lemvariance}.
\medskip

Again, we use the notations introduced in Section~\ref{sec:notationlattices}. Recall that for any $m$-dimensional simplex $\Delta \accentset{\circ}{\subset} \mathbb R^m$, the \emph{uniform probability measure} of $\Delta$ is the measure $d \mathbf P_{\Delta} = \frac{1}{\mathrm{vol}_m(\Delta)} d \mathrm{vol}_m$. If \(r \in \mathbb{N}_{>0}\), we also let
\[
	\Delta^{r}
	:=
	\big\{
	(x_{1}, \dotsc, x_{r+1}) \in \mathbb{R}_{+}^{r+1}
	\;|\;
	\sum_{i} x_{i} = 1
	\big\}.
\]
\medskip

The following result is quite simple to show.

\begin{lem} \label{lem:averageaffine}
	Let \(f : \mathbb{R}^{m} \to \mathbb{R}\) be an affine function. Let \(\Delta \subset \mathbb{R}^{m}\) be any simplex (open or not), and let \(X\) be a random variable drawn uniformly in \(\Delta\). Then \(\mathrm{E}[f(X)]\) is equal to the average of \(f\) on the vertices of \(\Delta\).
\end{lem}
\begin{proof} We may precompose \(f\) by any affine automorphism of \(\mathbb{R}^{m}\), to assume that \(\Delta = \Delta^{m}\). Also, since the result is obvious for constants, we may assume that \(f\) is linear; by linearity, it actually suffices to show the result when \(f\) is a form \(f(x) = x_{i}\) \((1 \leq i \leq m)\). But \(x_{i}\) takes the value \(0\) at all vertices except \(1\), so we have to show that
	\[
		\mathrm{E}[X_{i}] = \frac{1}{m}.
	\]
	This last result is obvious: all \(\mathrm{E}[X_{i}]\) are equal by the symmetry of the problem, and they sum up to \(1\) since \(\sum_{i} {x_{i}} = 1\) on \(\Delta^{m}\).
\end{proof}

Let now $r, k \in \mathbb N$, and consider a random variable $X$ drawn uniformly in the $(kr -1)$-dimensional simplex
$$
\Delta_{\underline{k}} = \Delta_{(1, \dotsc, 1, \dotsc, k, \dotsc, k)} \subseteq \mathbb R^{kr - 1}
$$ 
(each integer $i \in \llbracket 1, k \rrbracket$ being repeated $r$ times). We write $X = (X_{j,l})_{1 \leq j \leq k, 1 \leq l \leq r}$, and for each \(j \in \llbracket 1, k\rrbracket\), \(X_{j} := (X_{j, 1}, \dotsc, X_{j, r})\).

\begin{lem} \label{lem:randomvariables} For all $j \in \llbracket 1, k \rrbracket$ and all $i \in \llbracket 1, r \rrbracket$, we let $Y_j = \sum_{1 \leq l \leq r} X_{j, l}$, and $Z_{l}^j = \frac{X_{j,l}}{Y_j}$.
Then
\begin{enumerate}
	\item the random variables \(Z^j := (Z^j_1, \dotsc, Z^j_r)\) (where \(1 \leq j \leq k\)) are of uniform law with values in $\Delta^{r-1}$, and are pairwise independent. They are also independent of the $Y_j$ ;
\item the random variable $(Y_1', \dotsc, Y_k') = (Y_1, 2 Y_2, \dotsc, k Y_k)$ takes its values in $\Delta^{k-1} \subseteq \mathbb R^k$. Its density is
\[
	dP(y_1, \dotsc, y_{k}) = \binom{kr-1}{k-1, r-1, \dotsc, r-1} (y_1 \dotsc y_k)^{r-1}d\mathbf{P}_{\Delta^{k-1}},
\]
		where the constant is the multinomial coefficient \(\binom{kr-1}{k-1, r-1, \dotsc, r-1} = \frac{(kr-1)!}{(k-1)! (r-1)!^{k}}\).
\end{enumerate}
\end{lem}
\begin{rem}
	The density above is a particular case of Dirichlet distribution: it also appears naturally in Demailly's estimates (see \cite[(2.16)]{dem11}).
\end{rem}
\begin{proof}\footnote{We thank Rémi Peyre for a very useful discussion concerning this lemma.} 
\emph{(1)} 
	We consider the bijection
	\[
		\begin{array}{cccc}
			\psi:
			&
			\Delta^{k-1}
			\times
			\Delta^{r-1}
			\times
			\dotsc
			\times
			\Delta^{r-1}
			&
			\longrightarrow
			&
			\Delta_{\underline{k}}
			\\
			&
			(y, z_{1}, z_{2}, \dotsc, z_{k})
			&
			\longmapsto
			&
			(y_{1}z_{1}, \frac{1}{2} y_{2} z_{2}, \dotsc, \frac{1}{k} y_{k} z_{k})
		\end{array}
		\quad
		(\Delta^{r-1} \; \text{repeated}\; k \; \text{times}).
	\]

	Then with the notation of the Lemma, we have \(X = \psi(Y', Z_{1}, Z_{2}, \dotsc, Z_{k})\), and it suffices to show that the pullback of the probability \(d\mathbf{P}_{\Delta_{\underline{k}}}\) splits as a product of probabilities on each factor. To see this, denote by \(H \subset \mathbb{R}^{r}\) (resp.\ \(H_{\underline{k}} \subset \mathbb{R}^{kr}\)) the affine hyperplane containing \(\Delta^{r-1}\) (resp.\ \(\Delta_{\underline{k}}\)). The application \(\psi\) above is the restriction of the bijection
	\[
		\begin{array}{cccc}
			\Psi : & 
			\Delta^{k-1} \times H \times \dotsc \times H &
			\longrightarrow
			&
			H_{\underline{k}}
			\\
			&
			(y, z_{1}, z_{2}, \dotsc, z_{k})
			&
			\longmapsto
			&
			(y_{1}z_{1}, \frac{1}{2} y_{2} z_{2}, \dotsc, \frac{1}{k} y_{k} z_{k})
		\end{array}
		(H \; \text{repeated}\; k \; \text{times}).
	\]

	Now, we can write \(\Psi^{\ast} d\mathrm{vol}_{H_{\underline{k}}} = \rho(y, z_{1}, \dotsc, z_{k})\, d\mathbf{P}_{\Delta^{k-1}} d \mathrm{vol}_{H} \dotsc d \mathrm{vol}_{H}\) for some density function \(\rho\). Since for all \(h \in H_{\underline{k}}\), the function \((y,z) \mapsto \Psi(y, z + h)\) is a translate of \(\Psi\), and since both \(\mathrm{vol}_{H}\) and \(\mathrm{vol}_{H_{\underline{k}}}\) are translation invariant, this shows that \(\rho\) is actually independent of \(z_{1}, \dotsc, z_{r}\). Taking the restriction to the simplexes and normalizing yields
	\begin{equation} \label{eq:pullback}
		\psi^{\ast} d\mathbf{P}_{\Delta_{\underline{k}}}
		=
		\phi(y) d \mathbf{P}_{\Delta^{k-1}}
		\,
		d\mathbf{P}_{\Delta^{r-1}}
		\dotsc
		d\mathbf{P}_{\Delta^{r-1}}
	\end{equation}
	for some probability density \(\phi : \Delta^{k-1} \to [0,1]\). This gives the result.
\medskip

\noindent
\emph{(2)} We have to compute the density \(\phi\) above. Integrating \eqref{eq:pullback} with respect to the variables \(z_{1}, z_{2}, \dotsc z_{r}\) shows that for any \(y \in \Delta^{k-1}\), \(\phi(y)\) is proportional to the \(k(r-1)\)-dimensional volume of the image
\[
	\Psi(
	\{y\}
	\times
	\Delta^{r-1}
	\times
	\dotsc
	\times
	\Delta^{r-1}
	)
	\subset 
	\Delta_{\underline{k}}.
\]

This implies that there exists a constant \(C > 0\) such that 
\[
	\phi(y)
	=
	C
	y_{1}^{r-1}
	\dotsc
	y_{k}^{r-1}
\]
for all \(y \in \Delta^{k-1}\). To determine the value of \(C\), we use the normalization \(\int_{\Delta^{k-1}} \phi(y) = 1\). It follows from Lemma~\ref{lem:computationint} that \(C = \frac{(kr - 1)!}{(k-1)!(r- 1)!^{k}}\), as announced. 
\end{proof}

In the next lemma, we denote by \(\mathrm{E}[X]\) the expected value of a random variable \(X\).
\medskip

\begin{lem} \label{lemrandomy} With the same notations as in Lemma \ref{lem:randomvariables}, we have
\begin{enumerate}
	\item for all \(j \in \llbracket 1, k \rrbracket\), \(\mathrm{E}[Y_j] = \frac{1}{j k}\) and \(\mathrm{E}[Y_{j}^2] =  \frac{1}{j^2} \frac{r+1}{k(kr + 1)} \leq \frac{2}{j^2 k^2}\). 
	\item for all \(j, l \in \llbracket 1, k \rrbracket\), with \(j \neq l\), the variables \(Y_{j}\) and \(Y_{k}\) are negatively correlated i.e.
		\[
			\mathrm{E}[Y_j Y_l] 
			\leq 
			\mathrm{E}[Y_j]\, 
			\mathrm{E}[Y_l].
		\]
\end{enumerate}
\end{lem}
\begin{proof}
	{\em (1)}. We compute, using the density of Lemma~\ref{lem:randomvariables} above: 
	\begin{align*}
		\mathrm{E}[Y_j] 
		& 
		= 
		\frac{1}{j} \mathrm{E}[Y_{j}'] 
		= 
		\frac{1}{j} \int_{\Delta^{k-1}} y_j dP(y_1, \dotsc, y_k)  \\
		& 
		=
		\frac{1}{j}
		\binom{kr-1}{k-1, r-1, \dotsc, r-1} 
		\int_{\Delta_{\underline{k}}} 
		y_{1}^{r-1} \dotsc y_{j}^{r} \dotsc y_{k}^{r-1} d \mathbf{P}_{\Delta^{k-1}}(y).
	\end{align*}
	Applying Lemma~\ref{lem:computationint} gives the result. The computation of $\mathrm{E}[Y_j^2]$ is similar. 

\noindent
	{\em (2)} We write \(\mathrm{E}[Y_j Y_l] = \frac{1}{jl} \mathrm{E}[Y_j' Y_l'] = \frac{1}{jl} \int_{\Delta^{k-1}} y_j y_l \, dP(y)\). We get this time \(\mathrm{E}[Y_j Y_l] = \frac{1}{j l} \frac{r}{k(kr + 1)} \leq \frac{1}{j k} \frac{1}{l k}\), hence the result follows by {\em (1)}.
\end{proof}

Now, we let \(d_1, \dotsc, d_r \in \mathbb R\), and we let $T$ be a random variable of uniform law on the $(r-1)$-dimensional simplex $\Delta^{r-1} \subseteq \mathbb R^r$. Let \(S = \sum_{1 \leq l \leq r} d_l T_l\). 
\smallskip

We check easily that $\mathrm{E}[S] = \frac{1}{r} \sum_{1 \leq l \leq r} d_l$ since for all \(l \in \llbracket 1, r \rrbracket\):
\[
	\mathrm{E}[T_{l}]
	=
	\int_{\Delta^{r-1}}
	t_{l}
	\,
	d \mathbf{P}_{\Delta^{r-1}}(t)
	=
	\frac{1}{r}
\]
by Lemma~\ref{lem:computationint}. We now have

\begin{lem} \label{lemboundvariance} We let $X = (X_{j,l})_{1 \leq j \leq k, \, 1 \leq l \leq r}$ have the same meaning as before. 
	Let \[
		\begin{array}{cccc}
		A :
		& 
		\Delta_{\underline{k}} 
		&
		\longrightarrow 
		&
		\mathbb{R}
		\\
		& 
		t 
		&
		\longmapsto
		&
		\sum_{1 \leq j \leq  k} \sum_{1 \leq l \leq r} t_{j, l} d_l.
	\end{array}
	\]
			
		Then, we have an upper bound
\begin{equation} \label{eqvar}
	\mathrm{Var}[A(X)] 
	\leq 
	\frac{\pi^2}{3 k^2} \mathrm{E}[S^2].
\end{equation} 
\end{lem}
\begin{proof}
We will actually show that
\[
	\mathrm{Var}[A(X)] 
	\leq 
	\frac{2}{k^2} 
	\left( \sum_{1 \leq j \leq k} \frac{1}{j^2} \right) 
	\mathbf{E}[S^2]
\]

	We have \(A(X) = \sum_{j, l}\, d_{l}  X_{j, l}\), so, using the notations of Lemma \ref{lem:randomvariables}, we can rewrite the left hand side as follows:
\begin{align*}
\mathrm{Var}
	\left[
		\sum_{1 \leq l \leq r}
		d_{l}
		\left( 
		\sum_{1 \leq j \leq k} Y_{j} Z_{l}^{j} 
		\right)
	\right] 
	= 
	\mathrm{Var}[\sum_{1 \leq j \leq k} Y_{j} S_{j} ]
\end{align*}
where we let $S_j = \sum_{1 \leq l \leq r} d_l Z^j_l$. By definition of the variance, this is actually equal to
\begin{align} \label{eq:variance}
\mathrm{E}
	\big[
		(\sum_{1 \leq j \leq k} Y_{j} S_{j})^{2}
	\big]
	-
\mathrm{E} 
	\big[
		\sum_{1 \leq j \leq k} Y_{j} S_{j}
	\big]^{2}  
\end{align}

	By Lemma \ref{lem:randomvariables}, the variables \(S_{j}\) are pairwise independent for \(j \in \llbracket 1 , k \rrbracket\), of the same law as \(S\), and independent of the \(Y_j\).
	This implies that
	\[
		\begin{array}{rlr}
		\mathrm{E}[Y_{j} S_{j}] 
		& = 
		\mathrm{E}[Y_{j}] 
		\mathrm{E}[S_{j}] 
		= 
		\mathrm{E}[Y_{j}] 
		\mathrm{E}[S]
		&
		\forall j \in \llbracket 1, k \rrbracket
		\\
		\mathrm{E}[Y_{p} Y_{q} S_{p} S_{q}] 
		& = 
		\mathrm{E}[Y_{p} Y_{q}] 
		\mathrm{E}[S_{p}] 
		\mathrm{E}[S_{q}]
		= 
		\mathrm{E}[Y_{p} Y_{q}] 
		\mathrm{E}[S]^2
		&
		\forall p \neq q \in \llbracket 1, k \rrbracket
	\end{array}
	\]
	Thus, expanding \eqref{eq:variance} yields 
\begin{align*}
\mathrm{E}[ \sum_j Y_j^2 S_j^2] & + \mathrm{E}[\sum_{p \neq q} Y_p Y_q S_p S_q ] -  \sum_j \mathrm{E}[ Y_j S_j]^2 - \sum_{p \neq q} \mathrm{E}[Y_p S_p] \mathrm{E}[Y_q S_q]\\
& = \mathrm{E}[\sum_j Y_j^2]\, \mathrm{E}[S^2] + \sum_{p \neq q} \left( \mathrm{E}[Y_p Y_q] - \mathrm{E}[Y_p] \mathrm{E}[Y_q] \right) \, \mathrm{E}[S^2] - \sum_j \mathrm{E}[Y_j]^2 \, \mathrm{E}[S]^2 
\end{align*}

By Lemma \ref{lemrandomy}, (2), we get
\begin{align*}
\mathrm{Var}[A(X)] 
	& \leq 
	\sum_{j} \mathrm{E}[Y_j^2] \mathrm{E}[S^2] - \sum_{j} \mathrm{E}[Y_j]^2 \mathrm{E}[S]^2 \\
 & \leq  \sum_j \mathrm{E}[Y_j^2] \mathrm{E}[S^2],
\end{align*}
and then the conclusion follows immediately from Lemma \ref{lemrandomy}, (1).
\end{proof}

\chapter{Weighted projectivized bundles and stacks} \label{annex:stacks}

In this appendix, we present some general facts about the algebraic stacks that appear  in this text as weighted projective spaces: the material of this section is classical, and its main purpose is to help the reader fix their ideas without needing to revisit the whole theory. In particular, we will put some emphasis on the positive characteristic phenomena that can happen in this context: one of the main consequences is that the considered stacks will not be Deligne-Mumford in general, but merely {\em algebraic stacks} admitting an {\em fppf atlas}.
\medskip

Our main reference for this section will be the Stacks project and more particularly \cite[Part 0ELS]{stacks}.
\medskip

We fix a field \(\mathbbm{k}\). All schemes will be of finite type above \(\mathbbm{k}\), and all direct products will be taken above \(\mathrm{Spec}\, \mathbbm{k}\). Let \(p = \mathrm{char}(\mathbbm{k})\). The following discussion will be relevant for any value of \(p\), but several remarks will become essentially empty if \(p = 0\).

\section{Torsors in the fppf topology}

Algebraic stacks can be seen locally as classifying objects for torsors above schemes. To define torsors, we need to specify a {\em Grothendieck topology}: the {\em fppf topology} will the best suited for our purposes. We will not recall the axioms of a site or a Grothendieck topology, for which we refer to \cite[Definition 00VH]{stacks}. The following definition will be sufficient for our discussion.

\begin{defi}
	Let \(X\) be a scheme of finite type over \(\mathrm{Spec}\, \mathbbm{k}\). A {\em fppf \footnote{ Let us recall that {\em "fppf"} stands for {\em fidèlement plat de présentation finie} : "faithfully flat, of finite presentation".}
	covering} of \(X\) is the data of a collection of morphisms of \(\mathbbm{k}\)-schemes \(\{f_{i} : Y_{i} \to X\}_{i \in I}\) where each \(f_{i}\) is flat, locally of finite presentation, and such that \(X = \bigcup_{i \in I} f_{i}(Y_{i})\).
\end{defi}

\begin{rem} Using this terminology, we can now define a {\em fppf local} version for any  property \((P)\) of \(\mathbbm{k}\)-schemes that is invariant under base change: a given scheme \(X\) satisfies this property {\em locally for the fppf topology} if there exists a covering \(\{f_{i} : Y_{i} \to X\}\) as above such that the union \(Y = \bigsqcup_{i} Y_{i}\) satisfies \((P)\), possibly after pulling-back the relevant data from \(X\) to \(Y\).
\end{rem}
\medskip

A {\em fppf torsor} (or simply {\em torsor}, for short) can in particular be defined as a principal bundle that is locally trivial for the fppf topology:

\begin{defi}
	Let \(G\) be a group scheme of finite type above \(\mathbbm{k}\). Let \(X\) be a \(\mathbbm{k}\)-scheme of finite type. An {\em \(G\)-torsor above \(X\) for the fppf topology} is the data of 
	\begin{enumerate}[label=(\alph*)]
		\item a surjective morphism of \(\mathbbm{k}\)-schemes \(\pi : P \to X\);
		\item an group action \(\rho : G \times P \longrightarrow P\) that makes \(\pi\) an invariant morphism,
	\end{enumerate}
	such that there exists an fppf covering \(\{f_{i} : Y_{i} \to X\}_{i \in I}\) and isomorphisms of \(\mathbbm{k}\)-schemes \(P \times_{f_{i}} Y_{i} \cong G \times Y_{i}\), with \(\rho\) pulling back to the trivial \(G\)-action on \(G \times Y_{i}\) for each \(i \in I\).
\end{defi}

In the case where the group is the multiplicative group \(\mathbb{G}_{m} := \mathrm{Spec}\, \mathbbm{k}[\lambda, \lambda^{-1}]\), \(G\)-torsors are essentially the total spaces of line bundles where we have removed the zero section:

\begin{prop} \label{prop:torsorslines} Let \(X\) be a \(\mathbbm{k}\)-scheme. There is an equivalence of categories
	\[
		\left\{
			\text{line bundles}\; L \to X
		\right\}
		\leftrightarrow
		\left\{
			\begin{array}{c}
			\mathbb{G}_{m}\text{-torsors}\; P \to X\\
			\text{for the fppf topology}
			\end{array}
		\right\}
	\]
	where a line bundle \(L \to X\) is sent to the \(\mathbb{G}_{m}\)-torsor \(L - \{0\} := \mathrm{Spec}_{X} \left(\bigoplus_{m \in \mathbb{Z}} L^{\otimes m}\right)\), endowed with its standard \(\mathbb{G}_{m}\)-action by multiplication on the fibers.
\end{prop}

\begin{rem}
	In particular, \(\mathbb{G}_{m}\)-torsors are locally trivial in the Zariski topology as well. This property makes \(\mathbb{G}_{m}\) a {\em special group in the sense of Serre} (cf. \cite{serre58})
\end{rem}

\begin{proof}
	This is essentially a special case of \cite[Theorem 03P8]{stacks}. Let us simply recall how to construct a quasi-inverse for the functor described in the proposition, using fppf descent theory. Let \(P \to X\) be a \(\mathbb{G}_{m}\)-torsor that is trivialized on some fppf covering \(\{f_{i} : Y_{i} \to X\}_{i \in I}\). Now, for each \(i \in I\), we have an isomorphism
	\[
		Y_{i} \times_{f_{i}} P 
		\cong 
		Y_{i} \times \mathrm{Spec}\, \mathbbm{k}[\lambda, \lambda^{-1}].
	\]
	Denoting \(P_{i} :=  Y_{i} \times_{f_{i}} P\) and letting \(\pi_{i} : P_{i} \to Y_{i}\) be the canonical projection, let \(L_{i}\) be the \(\mathcal{O}_{Y}\)-submodule of \((\pi_{i})_{\ast} \mathcal{O}_{P_{i}}\), generated by the global coordinate \(\lambda\). Each \(L_{i}\) is a trivial line bundle above \(Y_{i}\). Now, we check that the \(L_{i}\) admit a descent data, induced by the natural descent data for the family \((P_{i})_{i \in I}\) coming from \(P\). Any descent data for coherent sheaves in the fppf topology is effective (this is even true for fpqc coverings: see \cite[Proposition 023T]{stacks}); this shows that there exists a coherent sheaf \(L\) on \(X\) pulling back to the \(L_{i}\). Finally, since Zariski local freeness of finite modules can be checked in the fppf topology (see \cite[Lemma 05B1]{stacks}), the freeness of the \(L_{i}\) implies that \(L\) is Zariski locally free as well. Finally, an isomorphism between \(L - \{0\}\) and \(P\) can be constructed first on the schemes \(Y_{i}\), and then descended to \(X\), using again a natural descent data.
\end{proof}

\section{Weighted projectived spaces and stacks}

Let \(n \in \mathbb{N}\), and let \(a_{0}, a_{1}, \dotsc, a_{n} \in \mathbb{N}_{>0}\). We introduce the quotient stack
\[
	\mathcal{P}(a_{0}, a_{1}, \dotsc, a_{n})
	:=
	\left[
		\quotientd{\mathbb{A}_{\mathbbm{k}}^{n+1} - \{0\}}{\mathbb{G}_{m}},
	\right]
\]
where \(\mathbb{G}_{m} \cong \mathrm{Spec} \mathbbm{k}[\lambda, \lambda^{-1}]\) acts on \(\mathbb{A}_{\mathbbm{k}}^{n+1}-\{0\}\) {\em via}
\[
	\lambda \cdot (x_{0}, x_{1}, \dotsc, x_{n}) 
	= 
	(\lambda^{a_{0}} x_{0}, 
	\lambda^{a_{1}} x_{1},
	\dotsc,
	\lambda^{a_{n}} x_{n}).
\]
\medskip

\begin{rem} Let \(\mathrm{Sch}_{\mathbbm{k}}\) denote the category of \(\mathbbm{k}\)-schemes. Formally, \(\mathcal{P}(a_{0}, a_{1}, \dotsc, a_{n})\) can be seen as a category \(\mathcal{P}\) with a projection functor \(p : \mathcal{P} \longrightarrow \mathrm{Sch}_{\mathbbm{k}}\), such that for any \(\mathbbm{k}\)-scheme \(X\), the full subcategory \(\mathcal{P}(X) := p^{-1}(X)\) is a groupoid, with objects given by diagrams
\[\
	\mathrm{Ob}\, \mathcal{P}(X)
	=
	\left\{
	\left.
	\begin{tikzcd}
		P \arrow[r, "u"] \arrow[d, "q"] & \mathbb{A}_{\mathbbm{k}}^{n+1} - \{0\} \\
		X 
		\end{tikzcd}
		\right|
		\;
		\;
		q \; \text{is a}\; \mathbb{G}_{m}\text{-torsor},
		\;
		u \; \text{is}\; \mathbb{G}_{m}\text{-equivariant}
	\right\}
\]

Morphisms in \(\mathcal{P}(X)\) are given by isomorphisms of torsors that are equivariant for the action of \(\mathbb{G}_{m}\). Given any arrow \(f : X \to Y\), there is also a pullback arrow \(f^{\ast} : \mathcal{P}(Y) \to \mathcal{P}(X)\), that can be defined naturally.
\medskip
\end{rem}

\begin{rem}
	It follows from \cite[Section 04UI]{stacks} that this definition gives a stack for the fppf topology. In particular, if we are given an fppf covering \(\{f_{i}: Y_{i} \to X\}\), as well as an elements \(P_{i} \in \mathrm{Ob}\, \mathcal{P}(Y_{i})\) and a {\em descent data} on \(\bigsqcup_{i,j} Y_{i} \times_{X} Y_{j}\), then there exists an object \(P \in \mathrm{Ob}\, \mathcal{P}(X)\) pulling back to the \(Y_{i}\).
\medskip

	Given a \(\mathbbm{k}\)-scheme \(X\), the \(2\)-Yoneda lemma \cite[Section 0GWH]{stacks} implies that \(\mathrm{Ob}\, \mathcal{P}(X)\) can be seen as the {\em groupoid of morphisms} from \(X\) (seen as a fppf stack) to \(\mathcal{P}\). From now on, we will rather use the notation \(\mathrm{Hom}(X, \mathcal{P})\) to denote this groupoid.
\end{rem}

\begin{rem} As we will recall below, \(\mathcal{P} := \mathcal{P}(a_{0}, a_{1}, \dotsc, a_{n})\) is an example of {\em algebraic stack}. As a defining property of these objects, the diagonal \(\mathcal{P} \longrightarrow \mathcal{P} \times \mathcal{P}\) is a {\em representable morphism}. This is essentially equivalent to the fact that for any pairs of morphisms \(f : X \to \mathcal{P}\) and \(g : Y \to \mathcal{P}\), the categorical fiber product
	\[
		X \times_{\mathcal{P}} Y
	\]
	is a stack underlied by some scheme. If \((P)\) is any property of morphisms of schemes that is invariant under base change, we can say that a morphism \( f: X \to \mathcal{P}\) (where \(X\) is a scheme) {\em satisfies \((P)\)} if for other such morphism \(g : Y \to \mathcal{P}\) scheme, the morphism of schemes
	\[
		X \times_{\mathcal{P}} Y
		\longrightarrow
		Y	
	\]
	satisfies \((P)\). In particular, there is a meaning for \(f\) to be {\em surjective, closed, open, \'{e}tale}, etc.
\end{rem}

\begin{rem}[Smooth atlas]
There is a natural morphism in \(\mathbb{A}_{\mathbbm{k}}^{n} - \{0\} \longrightarrow \mathcal{P}\), given by the trivial \(\mathbb{G}_{m}\)-torsor above \(\mathbb{A}_{\mathbbm{k}}^{n}\), sitting in the diagram
\[
\begin{tikzcd}
	(\mathbb{A}_{\mathbbm{k}}^{n+1} - \{0\}) \times \mathbb{G}_{m}
	\arrow[r, "u"] \arrow[d, "q"] & \mathbb{A}_{\mathbbm{k}}^{n+1} - \{0\} \\
	\mathbb{A}_{\mathbbm{k}}^{n+1} - \{0\}
\end{tikzcd}
\]
	where the morphism \(u\) is given by \(u(\lambda, x_{0}, x_{1} \dotsc, x_{n}) = (\lambda^{a_{0}} x_{0}, \lambda^{a_{1}} x_{1}, \dotsc, \lambda^{a_{n}} x_{n})\).
\medskip

	Given any scheme \(X\), and a morphism of stacks \(F : X \longrightarrow \mathcal{P}\) provided by a torsor \(P \in \mathrm{Hom}(X, \mathcal{P})\), it follows from the general definitions that we have a square diagram of stacks
	\[
		\begin{tikzcd}
			P \arrow[r, "u"] \arrow[d, "q"] \arrow[dr, phantom, "\square"] & \mathbb{A}^{n+1}_{\mathbbm{k}} - \{0\} \arrow[d, "v"]\\
			X \arrow[r, "F"]                                         & \mathcal{P}
		\end{tikzcd}
	\]

	In other words, we have \(X \times_{\mathcal{P}} (\mathbb{A}^{n+1}_{\mathbbm{k}} - \{0\}) \cong P\), and the projection map on \(X\) is a smooth, surjective morphism of schemes. The fact that this holds for any \(X\) is essentially the meaning behind the abstract statement that {\em the morphism of stacks \(v : \mathbb{A}^{n+1}_{\mathbbm{k}} - \{0\} \to \mathcal{P}\) is itself smooth and surjective}. This makes \(v\) a {\em smooth atlas} for \(\mathcal{P}\); the existence of such an atlas gives in turn a structure of {\em algebraic} (or {\em Artin}) stack to \(\mathcal{P}\).
\end{rem}

\subsection{The tautological line bundle}

The stack \(\mathcal{P} := \mathcal{P}(a_{0}, a_{1}, \dotsc, a_{n})\) supports a natural tautological line bundle \(\mathcal{O}(1)\). In other words, for any \(\mathbbm{k}\)-scheme \(X\), and any morphism \(f : X \longrightarrow \mathcal{P}\), there is an associated line bundle \(f^{\ast} \mathcal{O}(1)\) on \(X\), this data being functorial in \(X\) and \(f\).
\medskip

Actually, if \(f : X \to \mathcal{P}\) corresponds to a fppf torsor \(P \to X\) and a \(\mathbb{G}_{m}\)-equivariant morphism \(u : X \to \mathbb{A}_{\mathbbm{k}}^{n+1}-\{0\}\), then \(f^{\ast} \mathcal{O}(1)\) is the line bundle associated to \(P\) via Proposition~\ref{prop:torsorslines}.
\medskip

\subsection{The standard charts} \label{sec:standardcharts} In characteristic zero, the stack \(\mathcal{P} := \mathcal{P}(a_{0}, a_{1}, \dotsc, a_{n})\) is a Deligne-Mumford stack, i.e. there exists a \(\mathbbm{k}\)-scheme \(Y\) of finite type, and a surjective \'{e}tale morphism of stacks \(Y \to \mathcal{P}\): we can take \(Y\) to be the disjoint union of standard open coordinate charts. In positive characteristic \(p > 0\), we can do the same construction, but some inseparability phenomena will happen if \(p\) divides one of the multiplicities \(a_{0}, a_{1}, \dotsc, a_{n}\). Let us give a few details.
\medskip

For any \(i \in \llbracket 1, n\rrbracket\), let \(U_{i} := \mathbb{A}_{\mathbbm{k}}^{n}\), and define a morphism of stacks \(p_{i} : U_{i} \to \mathcal{P}\) by considering the diagram 
\begin{equation} \label{eq:standardchart}
		\begin{tikzcd}
			U_{i} \times \mathbb{G}_{m} 
				\arrow[r, "u_{i}"] 
				\arrow[d, "q_{i}"]  
		& \mathbb{A}^{n+1}_{\mathbbm{k}} - \{0\} \\
			U_{i}                                       
		& 
		\end{tikzcd}
\end{equation}
given by \(u_{i}(\lambda, x_{0}, \dotsc, \widehat{x}_{i}, \dotsc, x_{n}) = (\lambda^{a_{0}} x_{0}, \dotsc, \lambda^{a_{i}}, \dotsc, \lambda^{a_{n}} x_{n})\), where the hat symbol means that the coordinate has been omitted.
\medskip

The next lemma shows that the morphism \(p_{i} : \mathbb{A}_{\mathbbm{k}}^{n+1}-\{0\} \to \mathcal{P}\) is flat, finite, and open; we will see next that it is not \'{e}tale in general. Recall that for all \(a \in \mathbb{N}^{\ast}\), we denote by \(\mu_{a}\) the group subscheme of \(\mathbb{G}_{m}\) provided by the equation
\[
	\lambda^{a} - 1.
\]
It is not reduced if \(p\) divides \(a\).

\begin{lem} \label{lem:torsorproduct}
	Let \(X\) be a scheme, and let \(F : X \longrightarrow \mathcal{P}\) be a morphism of stacks, associated with a \(\mathbb{G}_{m}\)-torsor \(q : P \to X\) and a \(\mathbb{G}_{m}\)-equivariant morphism \(u : P \to \mathbb{A}_{\mathbbm{k}}^{n+1} - \{0\}\). Then the fiber product \(X \times_{\mathcal{P}} U_{i}\) is isomorphic to a scheme \(Q\), arising as a \(\mu_{a_{i}}\)-torsor \(Q \to U\), where \(U \subset X\) is some open subset of \(X\).
\end{lem}

\begin{proof}
	{\em Definition of the torsor.} First, let \(U \subset X\) be the subset
	\begin{equation} \label{eq:defU}
		U := \{ x \in X \;|\; u(P_{x}) \subset V_{i} \},
	\end{equation}	
	where \(V_{i} := \{x_{i} \neq 0\} \subset \mathbb{A}_{\mathbbm{k}}^{n+1} - \{0\}\), and where \(P_{x}\) denotes the fiber of \(q\) above \(x \in X\). Note that this set is open in \(X\). Indeed, if \(x \in U\), then Proposition~\ref{prop:torsorslines} allows to find Zariski open neighborhood \(V \subset X\) and a trivializing section \(e \in \Gamma(V, P)\). The equivariance of \(u\) implies that \(U\) contains the open neighborhood of \(x\) where \(x_{i} \circ u \circ e\) does not vanish.
	\smallskip

	Let \(H_{i} \subset \mathbb{A}^{n+1}_{\mathbbm{k}} - \{0\}\) be the affine hyperplane \(H_{i} := \{x_{i} = 1\}\). Let \(Q\) be the fiber product
	\[
		\begin{tikzcd}
			Q
				\arrow[r]
				\arrow[d, hook]
			&
			H_{i}
				\arrow[d, hook]
			\\
			P
				\arrow[r, "u"]
			&
			\mathbb{A}_{\mathbbm{k}}^{n+1} -\{0\}.
		\end{tikzcd}
	\]
	We claim that the composite map \(Q \to P \to X\) makes \(Q\) a \(\mu_{a_{i}}\)-torsor above \(U\) in the fppf topology. First, its image clearly lands in \(U\). Now, if \(V \subset U\) is a Zariski open subset where \(P\) admits a trivialization \(e\), then the morphism
	\[
		P|_{V} 
		\,
		\cong 
		\,
		\mathbb{G}_{m} \times V 
		\longrightarrow 
		\mathbb{A}_{\mathbbm{k}}^{n+1} - \{0\}
	\]
	identifies with the morphism \((\lambda, p) \longmapsto \lambda \cdot u(e(p))\). This shows that \(Q|_{V}\) is the scheme given by the equation
	\[
		\lambda^{a_{i}} = \frac{1}{x_{i} \circ u \circ e}.
	\]

	After a base change to the subscheme \(V' \subset V \times \mathbb{A}^{1}_{\mathbbm{k}}\) given by the equation \(\mu^{a_{i}} =  \frac{1}{x_{i} \circ u \circ e}\) (\(\mu\) being the coordinate on the second factor), we get a fppf morphism \(V' \to V\) such that \(Q \times_{V} V'\) is the subscheme of \(V' \times \mathbb{A}^{1}_{\mathbbm{k}}\) given by the equation
	\[
		(\frac{\lambda}{\mu})^{a_{i}} = 1.
	\]
	Taking \(\lambda' = \lambda/\mu\) as a trivializing section, this makes \(Q \times_{V} V'\) a trivial torsor above \(V'\). Finally, we can cover \(U\) by the images of such morphisms \(V' \to V\), and get an fppf cover of \(U\) trivializing \(Q\).
	\smallskip

\noindent
	{\em Representability.} Let \(Y\) be any \(\mathbbm{k}\)-scheme. We need to show that the groupoid of morphisms from \(Y\) to \(X \times_{\mathcal{P}} U_{i}\) admits an equivalence of categories to the set of morphisms \(\mathrm{Hom}(Y, Q)\) (seen as a discrete groupoid). Let \(Y \to X \times_{\mathcal{P}} U_{i}\) be a morphism of stacks. By definition of the stacky fiber product, this corresponds to the data of
	\begin{enumerate}[label=(\alph*)]
		\item morphisms \(f : Y \to X\) and \(g : Y \to U_{i}\) ;
		\item an isomorphism of torsors \(\varphi\) sitting in the following commutative diagram
			\[
				\begin{tikzcd}
					&
					f^{\ast} P 
						\arrow[d]
						\arrow[r]
						\arrow[dl, swap, "\varphi"]
						\arrow[dr, phantom, "\square"]
					&
					P
						\arrow[d, "q"]
						\arrow[dd, bend left, "u"]
					\\
					Y \times \mathbb{G}_{m}
						\arrow[r]
						\arrow[d, "g'"]
						\arrow[dr, phantom, "\square"]
					&
					Y
						\arrow[r, "f"]
						\arrow[d, "g"]
					&
					X \\
					U_{i} \times \mathbb{G}_{m}
						\arrow[r, "q_{i}"]
						\arrow[rr, bend right, "u_{i}"]
					&
					U_{i}
					& \mathbb{A}^{n+1}_{\mathbbm{k}} - \{0\}
				\end{tikzcd},
			\]
			where the squares indicate cartesian products.
	\end{enumerate}

	The isomorphism \(\varphi\) is actually equivalent to the data of a trivializing section \(s : Y \to f^{\ast} P\), identified by \(\varphi\) with the section \(\mathrm{id} \times \{1\} : Y \to Y \times \mathbb{G}_{m}\). The diagram shows that we have an inclusion of schemes \(u_{i} \circ g' (Y \times \{1\}) \subset H_{i}\). In other words, \(u \circ s(Y) \subset H_{i}\), and thus \(s\) lands in \(f^{\ast} Q\). Putting everything together, we see that the data of \(f, g\) and \(\varphi\) is equivalent to the data of a morphism \(s' : Y \to Q\) sitting in
	\[
		\begin{tikzcd}
			f^{\ast} Q
				\arrow[r]
				\arrow[d]
			& 
			Q 
				\arrow[d] \\
			Y
				\arrow[r, "f"]
				\arrow[u, "s", bend left]
				\arrow[ur, "s'", dashed]
			&
			X
		\end{tikzcd}
	\]
	Some bookkeeping permits to show that \(s'\) depends fonctorially in \(f, g, \varphi\), thus giving the required equivalence of categories between \(\mathrm{Hom}(Y, X \times_{\mathcal{P}} U_{i})\) and \(\mathrm{Hom}(Y, Q)\). We can check further that this equivalence is in turn functorial in \(Y\).
\end{proof}

\begin{rem} \label{rem:chart}
This proposition shows that for each \(i \in \llbracket 0, n \rrbracket\), the morphism \(U_{i} \to \mathcal{P}\) yields a surjective morphism onto an open substack \(\mathcal{P}_{i}\) of \(\mathcal{P}\). Moreover, if we let \(U := \bigsqcup_{i} U_{i}\) and \(p : U \to \mathcal{P}\) be the induced morphism of stacks, then \(p\) is faithfully flat, of finite presentation. It is not \'{e}tale unless the characteristic does not divide any of the \(a_{i}\). Proposition~\ref{prop:DMlocus} below will show that the locus where it is \'{e}tale is exactly the Deligne-Mumford open substack of \(\mathcal{P}\).
\end{rem}

\subsection{Isotropy groups of a geometric point} Let \(u : \mathrm{Spec}\, \mathbbm{K} \overset{u}{\longrightarrow} \mathcal{P}(a_{0}, a_{1}, \dotsc, a_{n})\) be a \(\mathbb{K}\)-geometric point. Since \(\mathbb{G}_{m}\)-torsors on \(\mathrm{Spec}\, \mathbb{K}\) are trivial, any such morphism corresponds to a diagram
\begin{equation} \label{eq:geometricpoint}
	\begin{tikzcd}
		\mathbb{G}_{m, \mathbb{K}}
			\arrow[r,"u"]
			\arrow[d]
		&
		\mathbb{A}_{\mathbbm{k}}^{n+1} - \{ 0 \}
		\\
		\mathrm{Spec}\, \mathbb{K}.
	\end{tikzcd}
\end{equation}

The {\em isotropy group} \(I(\mathcal{P}, u)\) is the \(\mathbb{K}\)-algebraic group representing the fiber product \(\mathrm{Spec}\, \mathbb{K} \times_{\mathcal{P}} \mathrm{Spec}\, \mathbb{K}\). In other words, if \(I\) denotes this isotropy group, for any \(\mathbbm{k}\)-algebra, there is an identification between \(I(R)\) and \(\mathrm{Hom}(\mathrm{Spec}\, R, \mathbb{K} \times_{\mathcal{P}} \mathbb{K})\), functorial in \(R\).

Now, for any \(\mathbb{K}\)-algebra \(R\), a morphism \(\Phi : \mathrm{Spec}\, R \to \mathrm{Spec}\, \mathbb{K} \times_{\mathcal{P}} \mathrm{Spec}\, \mathbb{K}\) is equivalent to the data of an isomorphism of torsors \(\varphi\) above \(\mathrm{Spec}\, R\) sitting in the diagram
\[
		\begin{tikzcd}
					&
					\mathbb{G}_{m, R} 
						\arrow[d]
						\arrow[r]
						\arrow[dl, swap, "\varphi"]
						\arrow[dr, phantom, "\square"]
					&
					\mathbb{G}_{m, \mathbb{K}}	
						\arrow[d, "q"]
						\arrow[dd, bend left=50, "u"]
					\\
					\mathbb{G}_{m, R}
						\arrow[r]
						\arrow[d]
						\arrow[dr, phantom, "\square"]
					&
					\mathrm{Spec}\, R	
						\arrow[r]
						\arrow[d]
					&
					\mathrm{Spec}\, \mathbb{K} \\
					\mathbb{G}_{m, \mathbb{K}}
						\arrow[r, "q"]
						\arrow[rr, bend right, "u"]
					&
					\mathrm{Spec}\, \mathbb{K}
					& \mathbb{A}^{n+1}_{\mathbbm{k}} - \{0\}
				\end{tikzcd},
\]
\medskip

If we write \(\mathbb{G}_{m, R} = \mathrm{Spec}\, R[\lambda, \lambda^{-1}]\), the morphism \(\varphi\) is equivalent to the data of an element \(\mu \in R\) such that \(\varphi^{\ast}(\lambda) = \mu \lambda\). Let us check what conditions \(\mu \in R\) must satisfy to define a morphism \(\varphi\) as above. 

The morphism \(u\) is dual to a morphism of \(\mathbbm{k}\)-algebras \(\psi : \mathbbm{k}[x_{0}, x_{1}, \dotsc, x_{n}] \to \mathbb{K}[\lambda, \lambda^{-1}]\). The \(\mathbb{G}_{m}\)-equivariance condition imposes that for each \(i \in \llbracket 0, n \rrbracket\):
\begin{equation} \label{eq:psixi}
	\psi(x_{i}) = c_{i} \lambda^{a_{i}}
\end{equation}
for some \(c_{i} \in \mathbb{K}\). Now, \(\mu \in R\) yields a morphism \(\Phi\) as above if and only if the following two dual diagrams commute:
\[
	\begin{tikzcd}
	\mathbb{G}_{m, R}
		\arrow[dr]
		\arrow[dd, "\varphi"]
	&
	\\
	&
	\mathbb{A}_{\mathbbm{k}}^{n+1}
	\\
	\mathbb{G}_{m, R}
		\arrow[ur]
	\end{tikzcd}
	,
	\quad
	\quad
	\begin{tikzcd}
	R[\lambda, \lambda^{-1}]	
	&
	\\
	&
	\mathbbm{k}[x_{0}, x_{1}, \dotsc, x_{n}]
		\arrow[ul, "x_{i} \mapsto c_{i} \lambda^{a_{i}}"]
		\arrow[dl, "x_{i} \mapsto c_{i} \lambda^{a_{i}}"]
	\\
	R[\lambda, \lambda^{-1}]	
		\arrow[uu, "\lambda \mapsto \mu \lambda"]
	\end{tikzcd}.
	\]

The diagram on the right commutes if and only if
\begin{equation} \label{eq:rootsunity}
	\mu^{a_{i}} = 1
	\quad
	\text{for all}
	\;
	i \in \llbracket 0, n \rrbracket
	\;
	\text{such that}
	\;
	c_{i} \neq 0.
\end{equation}
We see easily that this holds of \(\mu^{a} = 1\), where \(a\) is the gcd of all the \(a_{i}\) appearing in the equation above.
\smallskip

This discussion permits to describe the isotropy group of a \(\mathbb{K}\)-geometric point.

\begin{prop} \label{prop:isotropygroup}
	Let \(u : \mathrm{Spec}\, \mathbb{K} \to \mathcal{P}(a_{0}, a_{1}, \dotsc, a_{n})\) be a \(\mathbb{K}\)-geometric point, provided by a diagram as in \eqref{eq:geometricpoint}.	Let \(c_{0}\, \dotsc, c_{i} \in \mathbb{K}\) be the coordinates of the image of the composite morphism
	\begin{equation} \label{eq:composite}
	\mathrm{Spec}\, \mathbb{K}
	\times \{1\}
	\hookrightarrow
	\mathbb{G}_{m, \mathbb{K}}
	\overset{u}{\longrightarrow}
	\mathbb{A}_{\mathbbm{k}}^{n+1} - \{0\}.
	\end{equation}

	Let
	\[
		a
		=
		\mathrm{gcd}
		\left\{
			a_{i}
			\;
			|
			\;
			i \in \llbracket 0, n \rrbracket
			\;
			\text{such that}
			\;
			c_{i} \neq 0
		\right\}.
	\]
	Then the isotropy group of \(u\) is the \(\mathbb{K}\)-algebraic group
	\[
		\{\lambda^{a} = 1\}
	\subset
	\mathbb{G}_{m, \mathbb{K}}.
	\]
\end{prop}

\begin{rem}
	In particular, this group is smooth if and only if \(p \nmid a_{i}\) for some \(i \in \llbracket 0, n \rrbracket\) for which \(c_{i} \neq 0\).
\end{rem}

\begin{proof}
	Denote by \(I\) the isotropy group of \(u\). The proposition means that for any \(\mathbbm{k}\)-algebra, the \(R\)-points of \(I\) are identified with the subset
	\[
	\{
		\mu \in R
		\;|\;
		\mu^{a} = 1
	\}
	\;
	\subset 
	\;
	\mathbb{G}_{m}(R) = R^{\times}.
	\]
	functorially in \(R\). This comes from the discussion above, remarking that if \(\psi\) is defined as in \eqref{eq:psixi}, then \(c_{i}\) is the image of the composite morphism
	\[
		\mathbbm{k}[x_{0}, x_{1}, \dotsc, x_{n}]
		\overset{x_{i} \mapsto c_{i} \lambda^{-1}}{\longrightarrow}
		\mathbb{K}[\lambda, \lambda^{-1}]
		\overset{\lambda \mapsto 1}{\longrightarrow}
		\mathbb{K}
	\]
	dual of \eqref{eq:composite}, and thus is the same element of \(\mathbb{K}\) as in the statement of the proposition. 
\end{proof}

\subsection{The Deligne-Mumford locus} \label{sec:DMlocus} The {\em Deligne-Mumford locus} of \(\mathcal{P} := \mathcal{P}(a_{0}, a_{1}, \dotsc, a_{n})\) is the largest open substack \(\mathcal{P}^{\circ}\) that is Deligne-Mumford, i.e. admits a surjective \'{e}tale morphism \(U \to \mathcal{P}^{\circ}\) from some scheme \(U\). This locus can be described as the complement of the weighted projective stack where we have kept only the \(a_{i}\) divisible by \(p\). In particular, \(\mathcal{P}\) itself is Deligne-Mumford if \(p = 0\).
\medskip

More precisely, order the \(a_{i}\) such that \(a_{0}, \dotsc, a_{r}\) are divisible by \(p\), and \(a_{r+1}, \dotsc, a_{n}\) are not (with \(r = -1\) if \(p = 0\)). Then, we have a diagram
\[
	\begin{tikzcd}
	\mathbb{A}^{r+1} - \{0\}
		\arrow[r, hook]
		\arrow[d]
	&
	\mathbb{A}^{n+1}-\{0\}
		\arrow[d]
	&
	W
		\arrow[l, hook]
		\arrow[d]
	\\
	\mathcal{P}(a_{0}, a_{1}, \dotsc, a_{r})
		\arrow[r, hook]
	&
	\mathcal{P}(a_{0}, a_{1}, \dotsc, a_{n})
	&
	\mathcal{Q} := 
	\left[
	\quotientd{W}{\mathbb{G}_{m}}
	\right]
		\arrow[l, hook]
	\end{tikzcd}
\]
where \(W\) is the open subscheme of \(\mathbb{A}^{n+1} - \{0\}\), complement of \(\mathbb{A}^{r+1}-\{0\}\), and the vertical maps denote the quotients by \(\mathbb{G}_{m}\), arriving in the corresponding quotient stack.
\smallskip

Then, we have the following.

\begin{prop} \label{prop:DMlocus}
	With the previous notation, we have \(\mathcal{P}^{\circ} = \mathcal{Q}\).
\end{prop}
\begin{proof}

	Let \(\mathcal{P} := \mathcal{P}(a_{0}, a_{1}, \dotsc, a_{n})\). With the notation of Section~\ref{sec:standardcharts}, we let 
	\[
		V := \bigsqcup_{p \nmid a_{i}} U_{i},
	\]
	and let \(p : V \longrightarrow \mathcal{P}(a_{0}, a_{1}, \dotsc, a_{n})\) be the induced morphism. We first check that \(p(V) = \mathcal{Q}\). In other words, for any scheme \(X\) and any morphism of stacks \(f : X \to \mathcal{Q}\), the natural morphism 
	\[
		V \times_{\mathcal{P}} X \to \mathcal{Q} \times_{\mathcal{P}} X
	\]
	is onto. This can be seen easily from Proposition~\ref{lem:torsorproduct}, using the fact that \(U = X\) in \eqref{eq:defU} in our situation. Now, since any morphism
	\[
		\mu_{a} \to \mathrm{Spec}\, \mathbbm{k}
	\]
	is \'{e}tale unless \(p\) divides \(a\), Lemma~\ref{lem:torsorproduct} again shows that \(p : V \longrightarrow \mathcal{Q}\) is \'{e}tale and surjective. This shows that \(\mathcal{Q}\) is Deligne-Mumford, and thus \(\mathcal{Q} \subset \mathcal{P}^{\circ}\).
	\medskip

	Conversely, we claim that \(\mathcal{P}\) is not Deligne-Mumford on any open substack intersecting \(\mathcal{P}(a_{0}, a_{1}, \dotsc, a_{r})\). To see this, it suffices to remark that by Proposition~\ref{prop:isotropygroup}, the isotropy group of any geometric point of \(\mathcal{P}(a_{0}, a_{1}, \dotsc, a_{r})\) is non-reduced. Since a Deligne-Mumford stack has only smooth isotropy groups,  this gives the result.
\end{proof}

\subsection{The coarse moduli space} We consider the following scheme 
\[
	\mathbb{P}(a_{0}, a_{1}, \dotsc, a_{n})
	:=
	\mathrm{Proj}_{\mathbbm{k}}\, \mathbbm{k}[X_{0}, X_{1}, \dotsc, X_{n}],
\]
where in the free algebra above, each indeterminate \(X_{i}\) is assigned the weight \(a_{i} \in \mathbb{N}\). This weighted projective space can also be seen as a categorical GIT quotient
\[
	\quotientGIT{\mathbb{A}^{n+1}_{\mathbbm{k}} - \{0\}}{\mathbb{G}_{m}}.
\]
In particular, there is a \(\mathbb{G}_{m}\)-invariant morphism of schemes \(q: \mathbb{A}_{\mathbbm{k}}^{n+1}-\{0\} \to \mathbb{P}(a_{0}, a_{1}, \dotsc, a_{n})\), and any other such morphism factors uniquely through \(q\). 
\smallskip

Recall that for a given stack \(\mathcal{X}\) in the fppf topology, we say that a scheme \(X\) is a {\em coarse moduli space} for \(\mathcal{X}\) if there is a surjective morphism of stacks \(p : \mathcal{X} \to X\) such that any other such morphism \(\mathcal{X} \to Y\) factors uniquely through \(p\).
\medskip

In the next two propositions, we write \(\mathcal{P} := \mathcal{P}(a_{0}, a_{1}, \dotsc, a_{n})\) and \(\mathbb{P} := \mathbb{P}(a_{0}, a_{1}, \dotsc, a_{n})\).

\begin{prop}  Then there exists a morphism of stacks \(p : \mathcal{P} \to \mathbb{P}\) that makes \(\mathbb{P}\) a coarse moduli space for \(\mathcal{P}\).
\end{prop}

\begin{proof}
	To construct \(p : \mathcal{P} \to \mathbb{P}\), it suffices to define, for each scheme \(X\), a functor \(\mathrm{Hom}(X, \mathcal{P}) \to \mathrm{Hom}(X, \mathbb{P})\) that is in turn functorial in \(X\). 	Consider such a scheme \(X\), and take an element \(\mathrm{Hom}(X, \mathcal{P})\) given by a torsor \(q : N \to X\) and an equivariant morphism \(u : N \to \mathbb{A}_{\mathbbm{k}}^{n+1} - \{0\}\). Let \((X_{i})_{i \in I}\) be an open affine covering such that each \(N|_{X_{i}}\) is trivial. Then by the standard properties of Proj schemes, we get induced morphisms \(X_{i} \to \mathbb{P}\). It follows from a routine check that these glue together to define a morphism \(X \to \mathbb{P}\). The other parts of the definition follow directly from the discussion at the beginning of the section.
\end{proof}

By \cite{dol82}, the variety \(\mathbb{P}\) supports natural sheaves \(\mathcal{O}_{\mathbb{P}}(m)\), whose sections correspond to weighted homogeneous polynomials in the \(X_{i}\). If \(m\) is divisible by \(\mathrm{lcm}(a_{0}, a_{1}, \dotsc, a_{n})\), then \(\mathcal{O}_{\mathbb{P}}(m)\) is a line bundle, admitting the trivialization
\[
	X_{i}^{\frac{m}{a_{i}}}
\]
on the chart \(\{X_{i} \neq 0\} \subset \mathbb{P}\).

\begin{prop}
	Let \(m \in \mathbb{N}\) be divisible by \(\lcm(a_{0}, a_{1}, \dotsc, a_{n})\). Then the tautological line bundle \(\mathcal{O}(m)\) on \(\mathcal{P}\) is the pullback from the line bundle \(\mathcal{O}_{\mathbb{P}}(m)\).
\end{prop}
\begin{proof}
	We have to show that for any \(\mathbb{G}_{m}\)-equivariant diagram 
	\[
		\begin{tikzcd}
			N
				\ar[r]
				\ar[d]
			&
			\mathbb{A}_{\mathbbm{k}}^{n+1} \setminus \{0\}
				\ar[d]
			\\
			X
				\ar[r, "f"]
			&
			\mathbb{P}
		\end{tikzcd},
	\]
	where the left vertical arrow is a \(\mathbb{G}_{m}\)-torsor, we have a natural isomorphism \(N^{\otimes m} \cong f^{\ast} \mathcal{O}_{\mathbb{P}}(m)\), where we see \(N\) as a line bundle on \(X\).
Now, if we consider the morphism
	\[
		\begin{tikzcd}[row sep = large, column sep=large]
			\mathbb{A}_{\mathbbm{k}}^{n+1} \setminus \{0\}
				\ar[dr]
				\ar[r, "{X_{i} \mapsto X_{i}^{m/a_{i}}}"]
			&
			\mathbb{A}_{\mathbbm{k}}^{n+1} \setminus \{0\}
				\ar[d, "q"]
			\\
			&
			\mathbb{P}
		\end{tikzcd},
	\]
	we see that the right vertical arrow identifies with the total space of \(\mathcal{O}(m)\). To conclude, remark that we now get a well defined diagram, where the top horizontal arrow is \(\mathbb{G}_{m}\)-equivariant:
\[
		\begin{tikzcd}
			N^{\otimes m}
				\ar[r]
				\ar[d]
			&
			\mathbb{A}_{\mathbbm{k}}^{n+1} \setminus \{0\}
				\ar[d, "q"]
			\\
			X
				\ar[r, "f"]
			&
			\mathbb{P}
		\end{tikzcd},
	\]
	This can be seen for example in a chart where \(N\) is trivial, the existence and equivariance of \(u\) can then be read from the diagram:
	\[
		\begin{tikzcd}
			N \cong X \times \mathbb{G}_{m}
				\ar[rrr] 
				\ar[ddd]
			&
			&
			&
			\mathbb{A}_{\mathbbm{k}}^{n+1} \setminus \{0\} 
				\ar[ddd, "X_{i} \mapsto X_{i}^{m/a_{i}}"]
			&
			\\
			[-0.5cm]
			&
			\hspace{-1cm}
			(x, \lambda)
				\ar[r, mapsto]
				\ar[d, mapsto]
			&
			(g_{0}(x) \lambda^{a_{0}}, \dotsc, g_{n}(x) \lambda^{a_{n}})
				\ar[d, mapsto]
			&
			\hspace{-1cm}
			\\
			&
			(x, \lambda^{m})
				\ar[r, mapsto]
			&
			(g_{0}(x)^{m/a_{0}} \lambda^{m}, \dotsc, g_{n}(x)^{m/a_{n}} \lambda^{m})
			\\
			[-0.5cm]
			N^{\otimes m} \cong X \times \mathbb{G}_{m}
				\ar[rrr, "u"]
			&
			&
			&
			\mathbb{A}_{\mathbbm{k}}^{n+1} \setminus \{0\}
		\end{tikzcd}
	\]

	The arrow \(u\) then realizes the requested isomorphism.
\end{proof}

\subsection{Higher direct images}

The purpose of this section is to recall a standard vanishing cohomology result for weighted projective varieties, proved in particular by Dolgachev \cite{dol82} under the assumption that the characteristic of \(p\) is prime to all weights.
\smallskip

Let \(a_{1}, \dotsc, a_{n} \in \mathbb{N}_{>0}\), and let \(\mathbb{P} := \mathbb{P}(a_{0}, \dotsc, a_{n})\), with homogeneous coordinates \((x_{0}, \dotsc, x_{n})\) (resp. \(\mathbb{P}^{n} := \mathbb{P}(1, \dotsc, 1)\), with homogeneous coordinates \((y_{1}, \dotsc, y_{n})\)). We may define a finite dominant morphism
\[
	p : \mathbb{P}^{n} \longrightarrow \mathbb{P}(a_{0}, \dotsc, a_{n}),
\]
dual to the following injective morphism of graded \(\mathbbm{k}\)-algebras
\[
	\begin{array}{cccc}
		i : & \mathbbm{k}[x_{0}, x_{1}, \dotsc, x_{n}] 
		&
	\longrightarrow
		&
	\mathbbm{k}[y_{0}, y_{1}, \dotsc, y_{n}]
		\\
		& 
		x_{j}
		&
	\longmapsto
		&
		y_{j}^{a_{j}}.
	\end{array}
\]

If \(m\) is divisible by \(\mathrm{lcm}(a_{0}, a_{1}, \dotsc, a_{n})\), we have \(p^{\ast} \mathcal{O}_{\mathbb{P}}(m) \cong \mathcal{O}_{\mathbb{P}^{n}}(m)\); the pullback of the trivializing section \(X_{j}^{m/a_{j}}\) on \(\{x_{j} \neq 0\}\) identifies with the trivializing section \(Y_{j}^{m}\) on \(\{y_{j} \neq 0\}\).
\medskip

The next statement may be proven more abstractly, but let us give a direct proof.

\begin{lem} The standard morphism of \(\mathcal{O}_{\mathbb{P}}\)-modules \(\mathcal{O}_{\mathbb{P}} \to p_{\ast} \mathcal{O}_{\mathbb{P}^{n}}\) admits a {\em trace morphism}, i.e. a section \(\mathrm{tr} : p_{\ast} \mathcal{O}_{\mathbb{P}^{n}} \to \mathcal{O}_{\mathbb{P}}\) .
\end{lem}
\begin{proof}
	It suffices to show that \(i\), seen as a morphism of \(\mathbbm{k}[x_{0}, x_{1}, \dotsc, x_{n}]\)-modules, has a left inverse \(s\). Let \(P \in \mathbbm{k}[y_{0}, y_{1}, \dotsc, y_{n}]\). Then \(P\) can be written uniquely as
	\[
		P = 
		\sum_{l \in \llbracket 0, a_{0} - 1 \rrbracket \times \dotsc \llbracket 0, a_{n}-1\rrbracket}
		\,
		Q_{l}(y_{0}^{a_{0}}, \dotsc, y_{n}^{a_{n}}) y_{0}^{l_{0}} \dotsc y_{n}^{l_{n}},
	\]
	where the \(Q_{l}\) are polynomials. Then we may let \(s(Q) = Q_{0}(x_{0}, \dotsc, x_{n})\). It is straightforward to check that this gives a morphism of \(\mathbbm{k}[x_{0}, x_{1}, \dotsc, x_{n}]\)-modules, and that \(s \circ i\) is the identity.
\end{proof}

\begin{prop} \label{prop:vanishingweighted}
	For any \(m > 0\) divisible by \(\mathrm{lcm}(a_{0}, \dotsc, a_{n})\), and any \(j > 0\), we have
	\[
		H^{j}(\mathbb{P}, \mathcal{O}_{\mathbb{P}}(m)) = 0.
	\]
\end{prop}
\begin{proof}
	Since the morphism \(p\) is finite hence acyclic, and since we have the classical vanishing \(H^{j}(\mathbb{P}^{n}, \mathcal{O}_{\mathbb{P}^{n}}(m)) = 0\), we deduce that
	\begin{equation} \label{eq:vanishingPn}
		H^{j}(\mathbb{P}, p_{\ast} \mathcal{O}_{\mathbb{P}^{n}}(m)) = 0.
	\end{equation}
	However, by the previous lemma, we may write \(p_{\ast} \mathcal{O}_{\mathbb{P}^{n}}(m) = \mathcal{O}_{\mathbb{P}}(m) \oplus \mathcal{F}\) for some coherent sheaf \(\mathcal{F}\) on \(\mathbb{P}\). This shows that \(H^{j}(\mathbb{P}, \mathcal{O}_{\mathbb{P}}(m))\) is a direct factor of \eqref{eq:vanishingPn}, and hence vanishes.
\end{proof}

Let us finally state a relative version of the previous result, with the language of weighted projective spaces.

\begin{prop} \label{prop:vanishingrel}
	Let \(X\) be a variety over \(\mathbbm{k}\), and let \(\mathbf{E} := E_{1}^{(a_{1})} \oplus \dotsc \oplus E_{n}^{(a_{n})}\) be a weighted direct sum of vector bundles over \(X\). Let \(p : P \to X\) be the coarse moduli space for the associated weighted projective stack. For \(m\) divisible by \(\mathrm{lcm}(a_{1}, \dotsc, a_{n})\), and any \(j > 0\), we have the vanishing
	\[
		R^{j}p_{\ast} \mathcal{O}_{P}(m) = 0,
	\]
	where \(\mathcal{O}_{P}(m)\) denotes the \(m\)-th power of the tautological bundle on \(P\).
\end{prop}

\begin{proof}
	It suffices to show that for any closed point \(x\) in \(X\), we have
	\[
		R^{j} p_{\ast} \mathcal{O}_{P}(m) \otimes \kappa(x) = 0,
	\]
	where \(\kappa(x)\) is the residue field at \(x\). Since \(p : P \to X\) is locally trivial, the sheaf \(\mathcal{O}_{\mathbb{P}}\) is flat over \(X\), and the previous vanishing can be proved fiberwise (see e.g. \cite[Corollary 12.9]{har77}). We are then lead back to Proposition~\ref{prop:vanishingweighted}.
\end{proof}

\backmatter
\bibliographystyle{amsalpha}
\bibliography{biblio.bib}

\providecommand{\bysame}{\leavevmode\hbox to3em{\hrulefill}\thinspace}
\providecommand{\MR}{\relax\ifhmode\unskip\space\fi MR }
\providecommand{\MRhref}[2]{%
  \href{http://www.ams.org/mathscinet-getitem?mr=#1}{#2}
}
\providecommand{\href}[2]{#2}
\begin{thebibliography}{CDDR24}

\bibitem[Ang96]{ang96}
F.~Angelini, \emph{{An algebraic version of Demailly's asymptotic Morse
  inequalities}}, Proceedings of the American Mathematical Society \textbf{124}
  (1996), no.~11, 3265--3269.

\bibitem[BD19]{BD19}
D.~Brotbek and Y.~Deng, \emph{Kobayashi hyperbolicity of the complements of
  general hypersurfaces of high degree}, Geometric and Functional Analysis
  \textbf{29} (2019), no.~3, 690--750.

\bibitem[B{\'e}r15]{ber15}
G.~B{\'e}rczi, \emph{{Thom polynomials of Morin singularities and the
  Green-Griffiths-Lang conjecture}}, arXiv:1011.4710v2 (revised 2015), 2015.

\bibitem[BG71]{BG71}
S.~Bloch and D.~Gieseker, \emph{{The positivity of the Chern classes of an
  ample vector bundle}}, Inventiones mathematicae \textbf{12} (1971), no.~2,
  112--117.

\bibitem[BK24]{BK19}
G.~B{\'e}rczi and F.~Kirwan, \emph{Non-reductive geometric invariant theory and
  hyperbolicity}, Inventiones mathematicae \textbf{235} (2024), no.~1, 81--127.

\bibitem[Bog77]{bog77}
F.~Bogomolov, \emph{Families of curves on a surface of general type}, Doklady -
  Akademiya Nauk SSSR, Earth Science Sections \textbf{236} (1977), no.~5,
  1041--1044.

\bibitem[Bon98]{bonavero98}
L.~Bonavero, \emph{In\'egalit\'es de {Morse} holomorphes singuli\`eres}, J.
  Geom. Anal. \textbf{8} (1998), 409--425.

\bibitem[Bro17]{bro17}
D.~Brotbek, \emph{On the hyperbolicity of general hypersurfaces}, Publications
  math{\'e}matiques de l'IH{\'E}S \textbf{126} (2017), no.~1, 1--34.

\bibitem[Cad18]{cad_thesis}
B.~Cadorel, \emph{Hyperbolicité complexe et quotients de domaines symétriques
  bornés ({PhD} thesis)}, 2018.

\bibitem[Cad20]{cad17}
\bysame, \emph{Jet differentials on toroidal compactifications of ball
  quotients}, Annales de l'Institut Fourier \textbf{70} (2020), no.~6,
  2331--2359.

\bibitem[Cad24]{Cad24}
\bysame, \emph{Hyperbolicity of generic hypersurfaces of polynomial degree via
  {G}reen-{G}riffiths jet differentials}, arXiv:2406.19003, 2024.

\bibitem[CDDR24]{CDDR24}
F.~Campana, L.~Darondeau, J.-P. Demailly, and E.~Rousseau, \emph{On the
  existence of logarithmic and orbifold jet differentials}, Annales Henri
  Lebesgue \textbf{7} (2024), 1--67 (en). \MR{4765352}

\bibitem[CDR20]{CDR18}
F.~Campana, L.~Darondeau, and E.~Rousseau, \emph{Orbifold hyperbolicity},
  Compos. Math. \textbf{156} (2020), no.~8, 1664--1698 (English).

\bibitem[Dar15]{Dar15}
L.~Darondeau, \emph{On the logarithmic {G}reen–{G}riffiths conjecture},
  International Mathematics Research Notices \textbf{2016} (2015), no.~6,
  1871--1923.

\bibitem[Deb01]{Deb01}
O.~Debarre, \emph{Higher-dimensional algebraic geometry}, Universitext,
  Springer New York, 2001.

\bibitem[DEL00]{DEL00}
J.-P. Demailly, L.~Ein, and R.~Lazarsfeld, \emph{A subadditivity property of
  multiplier ideals.}, Michigan Math. J. \textbf{48} (2000), no.~1, 137--156.

\bibitem[Dem85]{dem85}
J.-P. Demailly, \emph{Champs magnétiques et inégalités de {M}orse pour la
  {$d''$}-cohomologie}, Ann. Inst. Fourier (Grenoble) \textbf{35} (1985),
  189--229.

\bibitem[Dem95]{dem95}
J.-P. Demailly, \emph{{Propriétés de semi-continuité de la cohomologie et de
  la dimension de Kodaira-Iitaka}}, Comptes Rendus de l'Académie des
  Sciences-S{é}rie I-Math{é}matique \textbf{320} (1995), no.~3, 341--346.

\bibitem[Dem96]{dem96}
J.-P. Demailly, \emph{{$L^2$ vanishing theorems for positive line bundles and
  adjunction theory}}, Lect. Notes in Math. \textbf{1464} (1996), 1–97.

\bibitem[Dem97]{dem97}
\bysame, \emph{Vari{é}t{é}s hyperboliques et {é}quations diff{é}rentielles
  alg{é}briques}, Gaz. Math. (1997), no.~73, 3--23.

\bibitem[Dem11]{dem11}
\bysame, \emph{Holomorphic {M}orse inequalities and the
  {G}reen-{G}riffiths-{L}ang conjecture}, Pure Appl. Math. Q. \textbf{7}
  (2011), no.~4, Special Issue: In memory of Eckart Viehweg, 1165--1207.

\bibitem[Dem12]{dem12a}
\bysame, \emph{Hyperbolic algebraic varieties and holomorphic differential
  equations}, Acta Math. Vietnam. \textbf{37} (2012), no.~4, 441--512.

\bibitem[Dem18]{dem18}
\bysame, \emph{{Recent results on the Kobayashi and Green-Griffiths-Lang
  conjectures}}, expanded version of talks given at the 16th Takagi Lectures in
  Tokyo, 2018.

\bibitem[Den16]{deng16}
Y.~Deng, \emph{Effectivity in the hyperbolicity related problems}, {Chap. 4 of
  the PhD memoir “Generalized Okounkov Bodies, Hyperbolicity Related and
  Direct Image Problems”, arXiv:1606.03831}, 2016.

\bibitem[Div08]{div08}
S.~Diverio, \emph{Differential equations on complex projective hypersurfaces of
  low dimension}, Compos. Math. \textbf{144} (2008), no.~4, 920--932.

\bibitem[Div09]{div09}
\bysame, \emph{Existence of global invariant jet differentials on projective
  hypersurfaces of high degree}, Mathematische Annalen \textbf{344} (2009),
  no.~2, 293--315.

\bibitem[DMR10]{DMR10}
S.~Diverio, J.~Merker, and E.~Rousseau, \emph{Effective algebraic degeneracy},
  Inventiones mathematicae \textbf{180} (2010), no.~1, 161--223.

\bibitem[Dol82]{dol82}
I.~Dolgachev, \emph{Weighted projective varieties}, Group Actions and Vector
  Fields (1982), 34--71.

\bibitem[DR15]{DR15}
S.~Diverio and E.~Rousseau, \emph{The exceptional set and the
  {Green}-{Griffiths} locus do not always coincide}, Enseign. Math. (2)
  \textbf{61} (2015), no.~3-4, 417--452 (English).

\bibitem[Eis95]{Eis95}
D.~Eisenbud, \emph{Commutative algebra. {With} a view toward algebraic
  geometry}, Grad. Texts Math., vol. 150, Berlin: Springer-Verlag, 1995
  (English).

\bibitem[Fuj94]{fuj94}
T.~Fujita, \emph{{Approximating Zariski decomposition of big line bundles}},
  Kodai Math. J. \textbf{17} (1994), no.~1, 1--3.

\bibitem[Ful98]{ful98}
W.~Fulton, \emph{Intersection theory.}, 2nd ed. ed., Ergeb. Math. Grenzgeb., 3.
  Folge, vol.~2, Berlin: Springer, 1998 (English).

\bibitem[GG80]{GG80}
M.~Green and P.~Griffiths, \emph{Two applications of algebraic geometry to
  entire holomorphic mappings}, The Chern Symposium 1979, Proc. Internal.
  Sympos. Berkeley, CA, 1979 (New York), Springer-Verlag, 1980, p.~41–74.

\bibitem[Har77]{har77}
R.~Hartshorne, \emph{Algebraic geometry}, 1st ed. 1977. ed., Graduate Texts in
  Mathematics, 52, Springer New York, New York, NY, 1977 (eng).

\bibitem[JV18]{JV18}
A.~Javanpeykar and A.~Vezzani, \emph{{Non-archimedean hyperbolicity and
  applications}}, https://hal.science/hal-01901289, October 2018.

\bibitem[Kaw82]{kawa82}
Y.~Kawamata, \emph{{A generalization of Kodaira-Ramanujam's vanishing
  theorem}}, Mathematische Annalen \textbf{261} (1982), no.~1, 43--46.

\bibitem[Lan87]{lang87}
S.~Lang, \emph{Introduction to complex hyperbolic spaces}, Springer-Verlag, New
  York, 1987.

\bibitem[Laz04]{lazpos1}
R.~Lazarsfeld, \emph{Positivity in algebraic geometry. {I}}, Ergebnisse der
  Mathematik und ihrer Grenzgebiete. 3. Folge, vol.~48, Springer-Verlag,
  Berlin, 2004, Classical setting: line bundles and linear series.

\bibitem[McQ98]{McQ98}
M.~McQuillan, \emph{Diophantine approximations and foliations}, Publications
  Math\'ematiques de l'IH\'ES \textbf{87} (1998), 121--174.

\bibitem[Mer15]{mer15}
J.~Merker, \emph{Algebraic differential equations for entire holomorphic curves
  in projective hypersurfaces of general type: Optimal lower degree bound},
  Progress in Mathematics (2015), 41--142.

\bibitem[RR13]{RR12}
X.~Roulleau and E.~Rousseau, \emph{On the hyperbolicity of surfaces of general
  type with small c12}, Journal of the London Mathematical Society \textbf{87}
  (2013), no.~2, 453--477.

\bibitem[Ser58]{serre58}
J.-P. Serre, \emph{Espaces fibr\'es alg\'ebriques}, S\'eminaire Claude
  Chevalley \textbf{3} (1958), 1--37 (fr), talk:1.

\bibitem[Siu93]{siu93}
Y.-T. Siu, \emph{An effective {Matsusaka} big theorem}, Ann. Inst. Fourier
  (Grenoble) \textbf{43} (1993), no.~5, 1387–1405.

\bibitem[Siu15]{siu15}
Y.-T. Siu, \emph{Hyperbolicity of generic high-degree hypersurfaces in complex
  projective space}, Inventiones mathematicae \textbf{202} (2015), no.~3,
  1069--1166.

\bibitem[{Sta}25]{stacks}
The {Stacks project authors}, \emph{The {S}tacks project},
  https://stacks.math.columbia.edu/, 2025.

\bibitem[SY96]{SY96}
Y.-T. Siu and S.~K. Yeung, \emph{Hyperbolicity of the complement of a generic
  smooth curve of high degree in the complex projective plane}, Invent. Math.
  \textbf{124} (1996), 573–618.

\bibitem[Voj04]{Voj04}
P.~Vojta, \emph{Jets via {H}asse-{S}chmidt derivations}, arXiv:math/0407113,
  2004.

\end{thebibliography}

\end{document}